\documentclass[twoside,12pt,letterpaper]{article}
\usepackage[margin=1in,centering]{geometry}
\usepackage[labelfont=bf,textfont=sf]{caption}
\usepackage{latexsym,graphicx}
\usepackage{graphbox}  
\usepackage{multirow}
\usepackage{amssymb,amsmath,amsthm,amsopn,pifont} 
\usepackage[abs, percent]{overpic}
\usepackage{mathrsfs}
\usepackage[usenames,dvipsnames,table]{xcolor} 
\usepackage[implicit=true,               
colorlinks=true,linkcolor=Black,citecolor=Black,urlcolor=Black
]{hyperref}
\usepackage{aliascnt} 
\usepackage{longtable} 
\usepackage{array}     
\definecolor{light-gray}{gray}{0.8}
\newcolumntype{:}{!{\color{light-gray}\vline}} 
\usepackage{enumitem}
\setlist[enumerate,1]{label=\textbf{(\arabic*)},ref=(\arabic*)}

\usepackage{placeins} 
\usepackage{xifthen}
\newcommand\measurepage{\dimexpr\pagegoal-\pagetotal-\baselineskip\relax}
\newtest{\IsThereSpaceOnPage}[1]{\lengthtest{\measurepage<#1}}
\let\stdsection\section
\renewcommand\section{\FloatBarrier%
  \ifthenelse{\IsThereSpaceOnPage{.15\textheight}}{\clearpage
  }{\relax}\stdsection}
\let\stdsubsection\subsection
\newcommand{\SubSecSkip}{\vspace{.5\baselineskip plus 10pt minus 5pt}}
\renewcommand\subsection{\FloatBarrier%
  \ifthenelse{\equal{\arabic{subsection}}{0}}%
  {\vspace{0pt plus 10pt}}{\SubSecSkip}\stdsubsection}
\input xy
\xyoption{all}
\def\r{{\bf r}}
\def\a{{\bf a}}
\def\f{{\bf f}}

\def\FP{{\rm FP}}
\def\({{(\!(}}
\def\){{)\!)}}
\def\du#1{\underline{\underline{#1}}}
\def\ssm{{\smallsetminus}}
\def\vecm{{\stackrel{\rightarrow}{m}}}
\def\vecx{{\stackrel{\rightarrow}{x}}}
\def\vecy{{\stackrel{\rightarrow}{y}}}

\def\vect{{\stackrel{\rightarrow}{t}}}

\def\U{{\mathcal U}}
\newcommand{\mapstoself}%
  {\,\protect\rotatebox[origin=c]{90}{\large$\circlearrowleft$}\,}
\newcommand{\rmapsto}%
  {\mathrel{\raisebox{.22ex}{\protect\rotatebox[origin=c]{180}{$\mapsto$}}}}
\def\ka{{\kappa}}
\def\ml{{\mu}}

\def\k{{\bf k}}
\def\K{{\bf K}}
\def\h{{\bf h}}
\def\wf{{\widetilde f}}
\def\Per{{\bf Per}}
\newcommand\Tstrut{\rule{0pt}{3.2ex}}         
\newcommand\Bstrut{\rule[-1.2ex]{0pt}{0pt}}   
\newcommand\ssk{{\smallskip}}
\newcommand\msk{{\medskip}}
\newcommand\bsk{{\bigskip}}

\newcommand\C{{\mathbb C}}
\newcommand\bP{{\mathbb P}}
\newcommand\R{{\mathbb R}}

\newcommand\Q{{\mathbb Q}}
\DeclareMathOperator\Arg{Arg\,}

\def\Rhat{{\widehat\R}}

\def\PSL{{\rm PSL}}

\def\I{{\mathcal I}}
\def\M{{\mathcal M}}
\def\npc{{N_{\bf pc}}}
\def\F{{\mathscr F}}
\def\PCF{{critically finite~}}
\def\<{{\langle}}
\def\>{{\rangle}}
\newcommand\Z{{\mathbb Z}}

\makeatletter
\def\newaliasedtheorem#1[#2]#3{%
  \newaliascnt{#1@alt}{#2}
  \newtheorem{#1}[#1@alt]{#3}
  \expandafter\newcommand\csname #1@altname\endcsname{#3}
}
\makeatother

\def\equationautorefname~#1\null{Equation~(#1)\null} 
\numberwithin{equation}{section}
\numberwithin{table}{section}
\newaliasedtheorem{lem}[equation]{Lemma}
\newaliasedtheorem{theo}[equation]{Theorem}
\newaliasedtheorem{cla}[equation]{Claim}
\newaliasedtheorem{coro}[equation]{Corollary}
\newaliasedtheorem{prop}[equation]{Proposition}
\newaliasedtheorem{question}[equation]{Question}
\newaliasedtheorem{conj}[equation]{Conjecture}
\newaliasedtheorem{asse}[equation]{Assertion}
\newaliasedtheorem{prob}[equation]{Problem}

\newtheoremstyle{mydef}
  {\topsep}
  {.5em}
  {\normalfont}
  {\parindent}
  {\bfseries}
  {.}
  { }
  {\thmname{#1}\thmnumber{ #2}\thmnote{ (#3)}}
\theoremstyle{mydef}
\newaliasedtheorem{definition}[equation]{Definition}
\newaliasedtheorem{rem}[equation]{Remark}
\newaliasedtheorem{ex}[equation]{Example}

\makeatletter
\let\oldtable\table
\def\table{\@ifnextchar[\table@i \table@ii}   
\def\table@i[#1]{\addtocounter{equation}{1}\oldtable[#1]} 
\def\table@ii{\addtocounter{equation}{1}\oldtable} 
\let\oldlongtable\longtable
\def\longtable{\@ifnextchar[\longtable@i \longtable@ii}   %
\def\longtable@i[#1]{\addtocounter{equation}{1}\oldlongtable[#1]} %
\def\longtable@ii{\addtocounter{equation}{1}\oldlongtable} %
\makeatother

\makeatletter
\def\@seccntformat#1{\@ifundefined{#1@cntformat}%
   {\csname the#1\endcsname\quad}  
   {\csname #1@cntformat\endcsname}
}
\let\oldappendix\appendix 
\renewcommand\appendix{%
    \oldappendix
    \newcommand{\section@cntformat}{\appendixname~\thesection\quad}
}
\makeatother

\begin{document}
\thispagestyle{empty}

\title{\bf The W. Thurston Algorithm\\ for Real Quadratic Rational Maps\\}

\author{\bf Araceli Bonifant, John Milnor and Scott Sutherland}

\maketitle

\begin{abstract}
 A study of real quadratic maps with real critical points,  
  emphasizing the effective construction of critically finite maps with specified
  combinatorics. We discuss the behavior of the Thurston
  algorithm in obstructed cases, and in one exceptional badly behaved case, and
provide  a new description of the appropriate
  moduli spaces. There is also an application to topological entropy.
\end{abstract}

\vspace{.2cm}

\vspace{.2cm}

\thispagestyle{empty}
\section{Introduction.}\label{s-1}
This paper will study real quadratic maps 
with real critical points, and especially with those which are
\textbf{\textit{critically finite}}, in the sense that both
critical points have finite orbit.\ssk

\autoref{s2} provides a classification of critically
finite maps in terms of their combinatorics. By definition,  
the \textbf{\textit{combinatorics}} is an ordered
list $\(m_0,\ldots,m_n\)$ of integers describing how the union of the two
critical orbits maps to itself. This section
also provides very rough classifications, either according 
to the shape of the graph restricted to the interval $f(\Rhat)$, or else
according to dynamic behavior which may be either hyperbolic of type B,
C, or D, or half-hyperbolic or totally non-hyperbolic.\ssk

\autoref{s3} provides a convenient way of representing such maps. Every real
quadratic map can be described topologically as a map from the circle
$~\Rhat=\R\cup\{\infty\}$ 
to itself. Hence its graph can be described as a subset
of the torus~~ {\bf circle}$\times${\bf circle}\,; and this torus can be
represented as a square with opposite sides identified.\ssk

\autoref{s4} describes an effective implementation of the Thurston algorithm,
which starts from combinatorics and produces the corresponding quadratic
rational map whenever such a map exists (except in one special case as 
described in \autoref{s6}).\ssk

\autoref{s6}
proves that every conjugacy class of critically finite maps is uniquely
determined by its combinatorics. It can be constructed by the Thurston algorithm
in nearly every case. But there is one exceptional  case 
where the algorithm does not converge from a generic 
 choice of starting point.\ssk

\autoref{s7} discusses two moduli spaces for such real quadratic maps: one
in which we allow only orientation preserving changes of coordinate,
and one where we
also allow orientation reversing changes of coordinate. Both moduli 
spaces are smooth surfaces; but the first is topologically an open cylinder,
while the second is simply-connected with two boundary edges. This section also
discusses asymptotic relations between different coordinate systems, 
and discusses a rich family of critically finite maps constructed by 
Filom and Pilgrim \cite{FP}.
\ssk

\autoref{s5} concerns obstructed cases, distinguishing
between ``weak obstructions'' which are harmless, and ``strong obstructions''
which are serious. Combinatorics of bimodal shape $-+-$ are always
strongly obstructed.  Those of shape $+-+$ are often strongly 
obstructed. There is a simple criterion which applies in all $+-+$ strongly
obstructed cases that we have observed,  
using the construction of a Levy cycle to prove obstruction.\ssk 

\autoref{s8} makes a particular study of the unimodal case, depending 
 essentially on work of  Filom \cite{F}, and making use of
 kneading theory. It shows that hyperbolic or half-hyperbolic 
unimodal combinatorics is never strongly
 obstructed; and also completes a partial proof 
 by Filom concerning topological entropy in the unimodal region. (This proof
 has also been completed by Yan Gao \cite{G}.) 
\ssk

 \autoref{a2}  illustrates all possible minimal, 
 non-polynomial combinatorics $\(m_0,\cdots, m_n\)$ 
 with {$n\le 4$},  plus  a few  cases with $n\ge 5$. \ssk

 \autoref{ap-b}  provides more information about those figures in 
 this paper which illustrate some combinatorics. 
 \autoref{t-1} classifies them in terms of their 
 dynamic type and topological shape; while \autoref{t-2} gives the
 corresponding parameters  $\mu$, $\kappa$, $\Sigma$ and $\Delta$.

\bigskip

\noindent\textbf{Acknowledgment:} We are grateful to Khashayar Filom for 
extremely useful comments. 
\ssk

\section{Combinatorics} \label{s2}
The phrase \textbf{\textit{real quadratic map}} will always be used
to mean a quadratic rational map which not only has real
coefficients, but also has real critical points. 
 Let $\PSL_2(\R)$ be the group of all orientation
 preserving fractional linear transformations
 $$L:x\mapsto\frac{ax+b}{cx+d}\qquad{\rm with}\qquad ad-bc>0~.$$
 
 \begin{definition}\label{D-c} Two real quadratic maps $f$ and $g$ are
   \textbf{\textit{conjugate}} if they correspond under some orientation
    preserving change of coordinates,  or equivalently if
   $g=L\circ f\circ L^{-1}$ for some
 $L\in \PSL_2(\R)$. The notation $\langle f\rangle$ 
 will be used for the conjugacy class of $f$. If $f$ and $g$ correspond
 under some change of coordinate which may be either orientation
 preserving or orientation reversing, then we will use the term 
\textbf{\textit{$\pm$-conjugate}}. 
\end{definition}\msk

The combinatorics  for a critically finite
real quadratic map is a rough but easily understood description
which suffices to determine the map uniquely up to conjugation. (See
\autoref{T-notexc}.) It can be defined as follows.\msk

For any real quadratic map $f:\Rhat\to\Rhat$,
the image $f(\Rhat)\subset\Rhat$ is a compact interval
bounded by the two critical values. We can always
assume (after replacing $f$ by a conjugate if necessary)
that $f(\Rhat)$ is contained in the finite line $\R$.\msk

\begin{enumerate}[series=_cases, label={\textbf{Case \arabic*.}}, ref=Case \arabic*, 
                 leftmargin=*,itemindent=\parindent]
\item \label{combi-case1}
Suppose that both critical points are contained in
$f(\Rhat)\subset\R$. Let
$$ x_0<x_1<\cdots< x_n$$ 
be an ordered list of all of the critical and postcritical points.
By definition,  the \textbf{\textit{combinatorics}}
$$\vecm=\( m_0,\,m_1,\,\ldots, m_n\)$$ is the list
of integers between $0$ and $n$  such that $f(x_j)=x_{m_j}$ for each $j$.
\end{enumerate}
\msk

A good way of visualizing the combinatorics is to consider the associated
piecewise linear mapping $~\f:[0,n]\to[0,n]~$ 
which maps $j$ to $m_j$ and is linear between consecutive integers.
Evidently the combinatorics  determines this 
map $\f$, and it determines the combinatorics, so we can
pass freely from one to the other. (Those with a musical ear may
want to think of $\vecm$ as a sequence of rising and falling
musical notes.) 

As an example, \autoref{f1} shows a quadratic map with combinatorics
$\(5,6,4,1,0,2,3\)$, together with the associated piecewise linear model. 
In this example there are periodic critical orbits of period three
and four. Ben Wittner   \cite{W},  showed that up to $\pm$-conjugacy 
there is only one such real quadratic map.\footnote{The corresponding complex
  map has a Sierpinski carpet as Julia set. Compare \cite[App. F]{M}, 
  written with Tan Lei.}
\msk

\begin{figure}[!htb]\centerline{
    \includegraphics[width=2.2in]{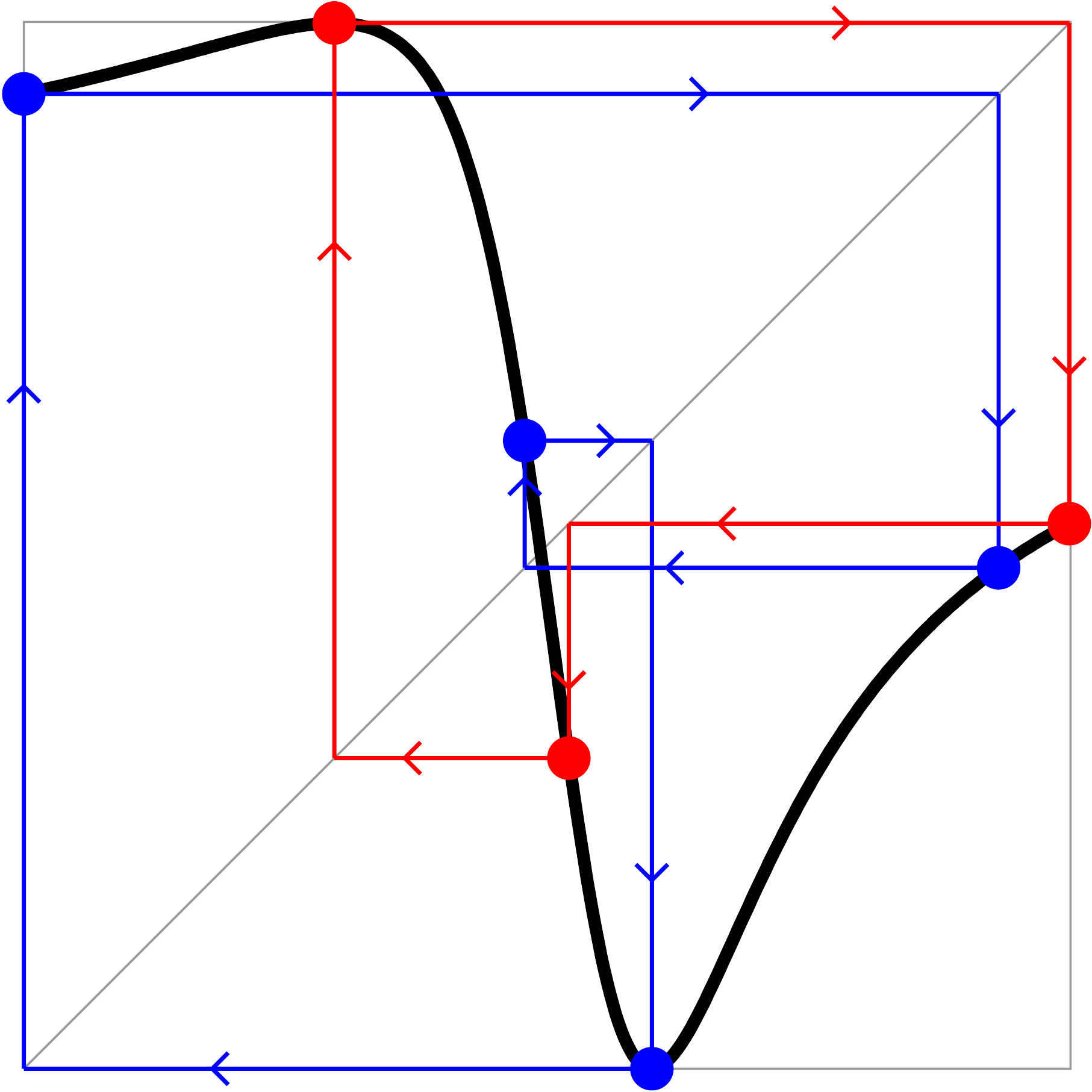}\hfil
    \includegraphics[width=2.2in]{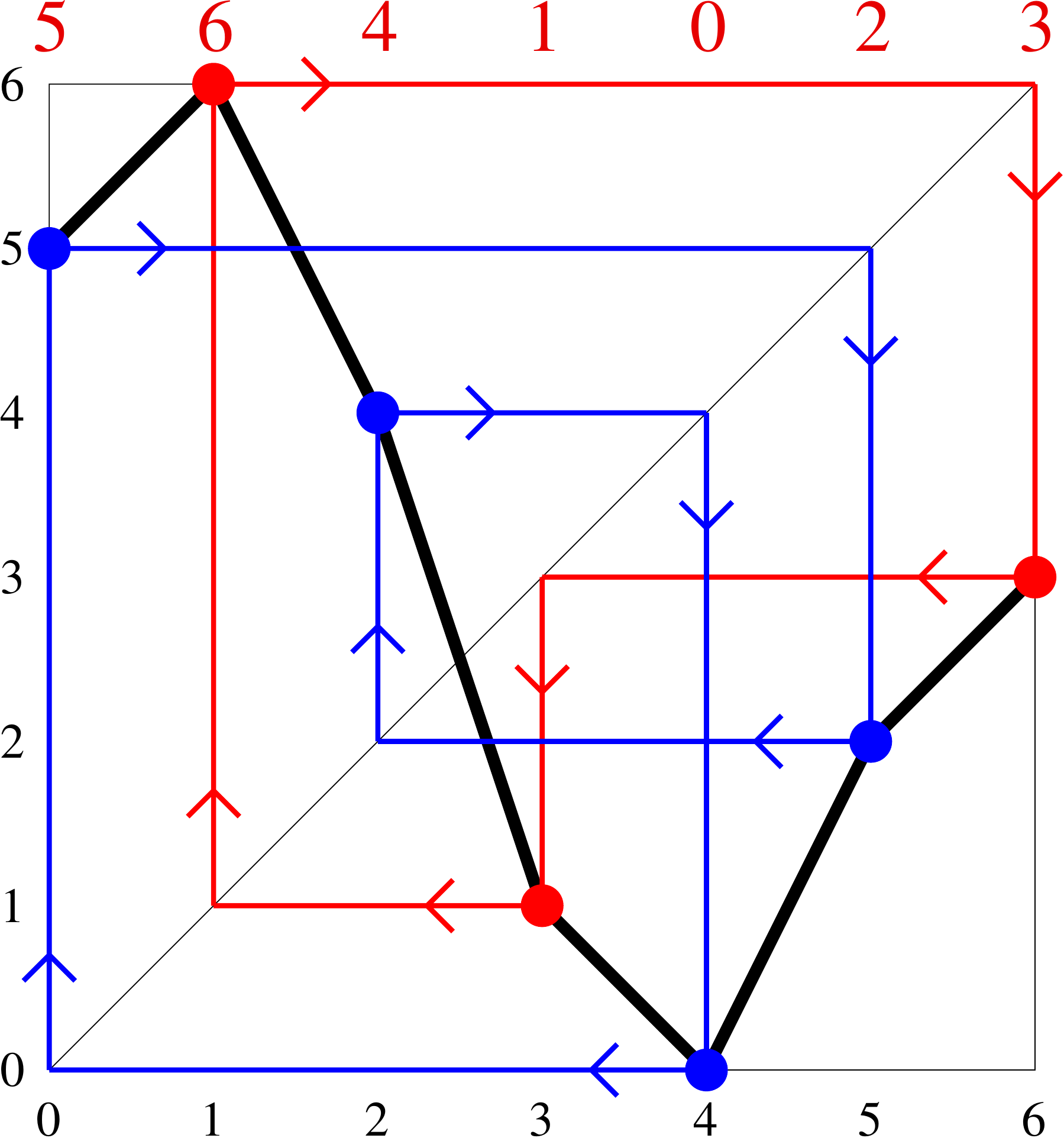}}
  \caption{\label{f1} On the left a quadratic rational map
    with {combinatorics} $\(5,\,6,\,4,\,1,\,0,\,2,\,3\)$, and with
    critical orbits of period three and four. On the right,
    the corresponding piecewise linear model, with the 
    combinatorics shown along the top.}
\end{figure}\ssk

Of course any list of $n+1$ numbers between zero and $n$ will yield a
corresponding piecewise linear graph; but most such graphs could not
possibly represent a quadratic map. 
In \autoref{D-admissible} we will specify
strong and explicit restrictions on which lists $\vecm$ are to be considered.

\begin{enumerate}[resume*=_cases]
\item \label{combi-case2}
  If there is only one\footnote{We will see in \autoref{L2} that 
    there is always at least one critical point in $f(\Rhat)$ in the
    critically finite case. Maps with only one critical 
    point in $\f(\Rhat)$ will be called ``strictly unimodal''.}
 critical point in $f(\Rhat)$, then it might seem that it doesn't 
 matter whether the other critical point is to the left or the right
 of $f(\Rhat)$ or at infinity, since we can always change this by
 replacing $f$ by a conjugate map. However, the following explicit choice
 of where to put it in the combinatorics
 will be important later:\footnote{In fact, this choice guarantees
   that the associated piecewise linear map will have a fixed 
   point in the lap between the two critical points, and this will be
   important for the implementation of the Thurston algorithm.}

 \begin{quote}
 {\sf If one critical point is outside of $f(\Rhat)$ then
  put it: 

  \begin{itemize}[beginpenalty=10000,midpenalty=9999, 
    topsep=-.2ex,itemsep=0ex,parsep=0pt]
  \item to the left  in the combinatorics
    if the associated critical value is a maximum,

  \item or to the right if it is a minimum.
\end{itemize}

\noindent  In the first case, the combinatorics will start with
$m_0=n$, with  all $m_j>0$. In the second case 
it will end with $m_n=0$,  with all $m_j<n$. (Compare \autoref{f2}.)}
\end{quote}
\ssk

Otherwise the definition  proceeds just as above.
\end{enumerate} 
\msk

\begin{figure}[!htb]
  \centerline{
    \includegraphics[width=2.6in]{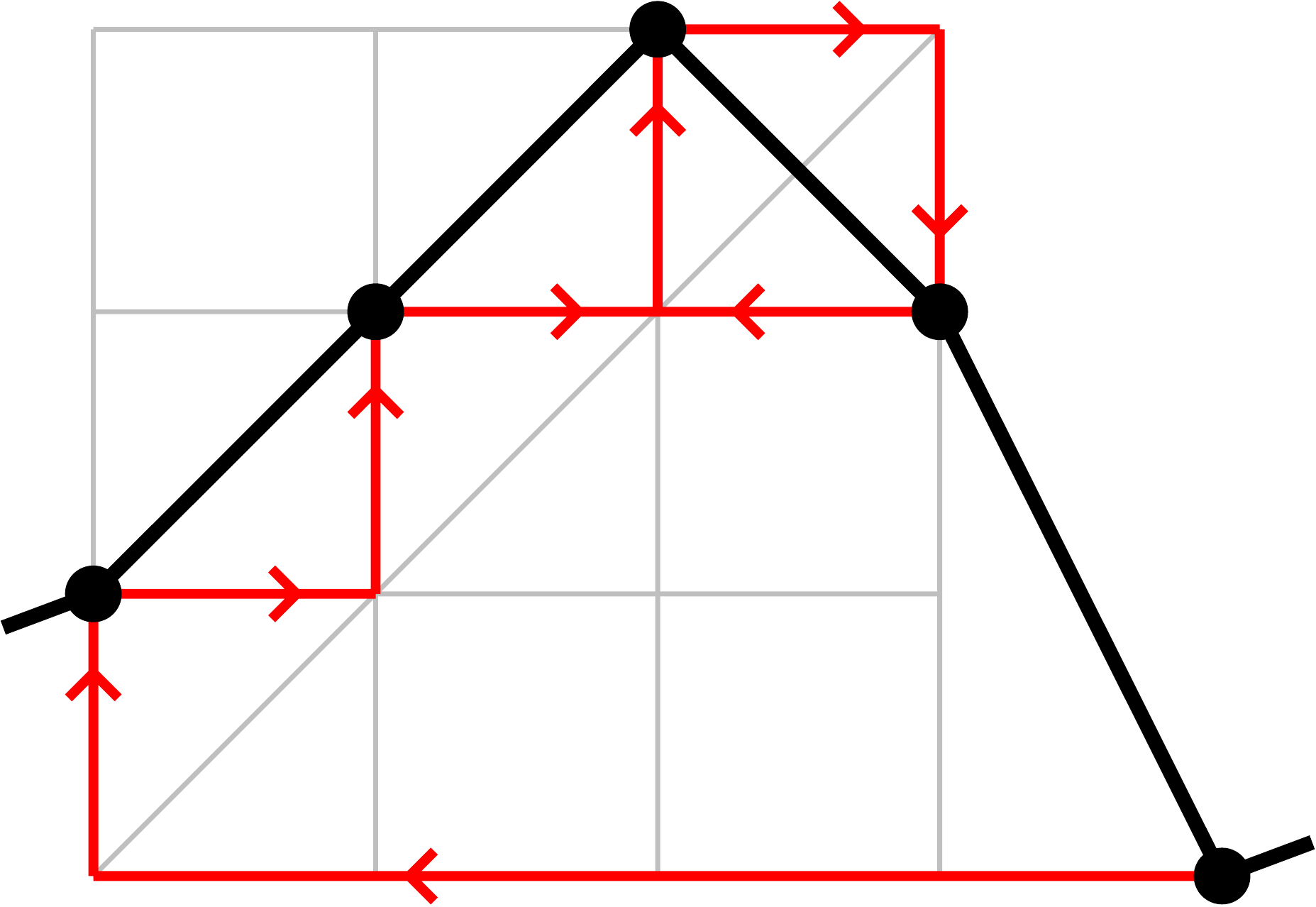}\qquad
  \includegraphics[width=3.1in]{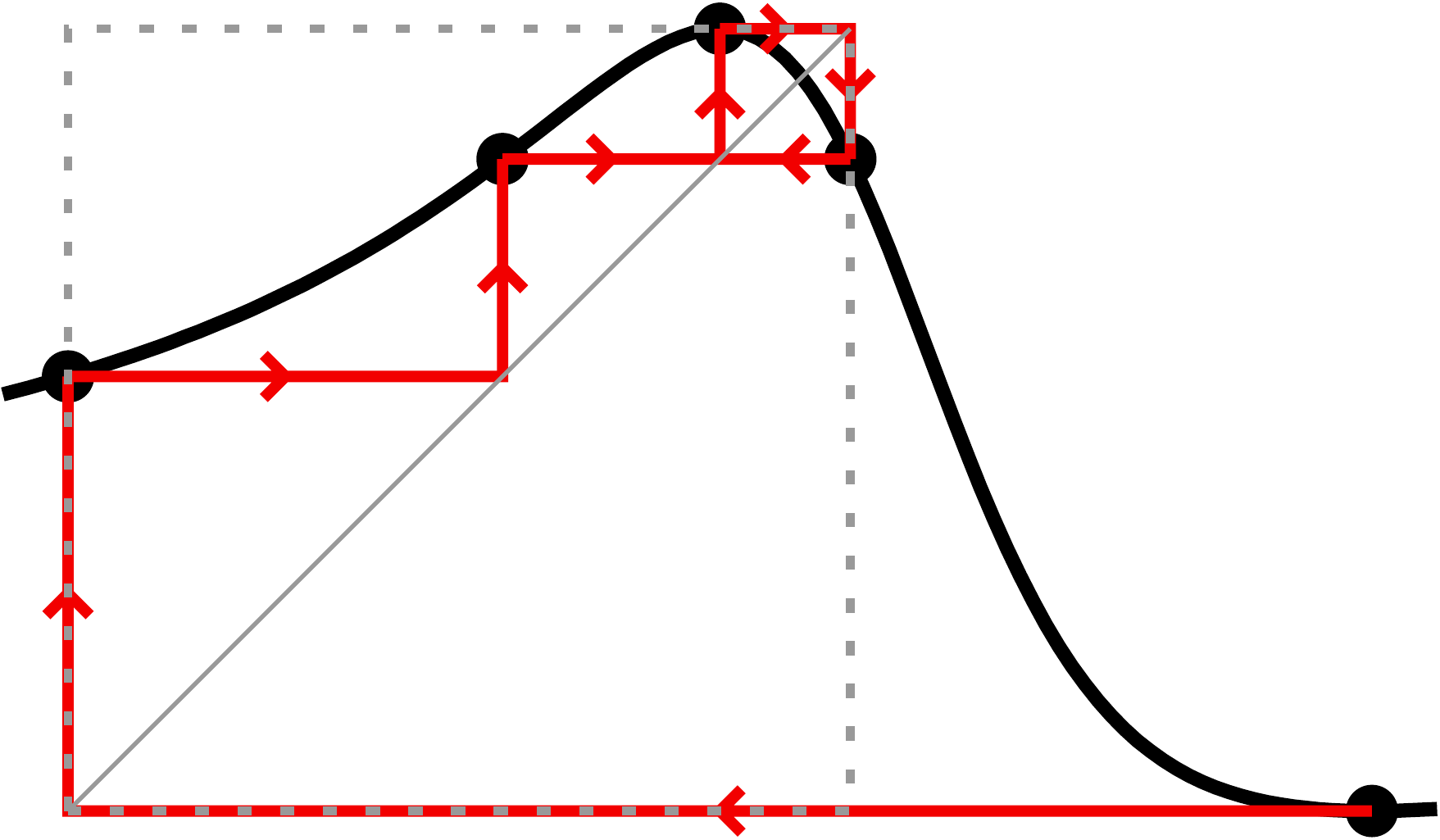}}

\caption{  \label{f2} Graphs of the piecewise linear model map with
  combinatorics $\(1,\,2,\,3,\,2,\,0\)$, and the associated rational map. The
  mapping pattern is
  {$\underline{\underline{x_4}}\mapsto x_0\mapsto x_1
  \mapsto\underline{\underline{x_2}}
    \leftrightarrow x_3 $}
  (with the critical points double underlined). 
  The square $[0,3]\times[0,3]$ on the left, with gray grid lines,
  corresponds to the square $f(\Rhat)\times f(\Rhat)$, outlined
  by dotted lines on the right.   In this example, since the fixed point
  must lie between the two critical points, the critical point
  which is outside $f(\Rhat)$ must be placed to the right.}
\end{figure}
\msk

\begin{rem}[Mapping Patterns] By an abstract \textbf{\textit{ mapping pattern}}
  we will
  mean simply a finite set $S$ of marked points, together with a function from
  $S$ to  itself, and a two element subset $S_0\subset S$ consisting of
  ``critical points''. Evidently our combinatorics  is essentially just a
  special kind of mapping pattern, together with an explicit ordering of the
  points of $S$.  Given such an abstract mapping pattern, there may be several
  compatible combinatorics, or there may be none.\footnote{See \autoref{L-para}
    for some specific examples of this problem.}

  Now suppose that  $S$  is given as a subset of $\Rhat$. Then at least
  we are given a cyclic  order of the points of $S$.
  \ssk
  
  \begin{lem} \label{L-mp} Given such an abstract mapping pattern, together
    with a cyclic ordering of $S$, there is at most one compatible combinatorics.
  \end{lem}

  \begin{proof}  In fact we can easily find the corresponding combinatorics
    (if it exists),   in two steps as follows.

\begin{enumerate}[series=_steps,
   label={\textbf{Step \arabic*.}}, ref=Step~\arabic*, 
   leftmargin=*,itemindent=\parindent]
\item List these marked points in positive cyclic order; for  example as
$$ x_0~<~x_1~<~\cdots~<~x_n~,$$
where all of the entries except possibly the last are finite.
\end{enumerate} \ssk

\noindent Then one of the following should be true:
Either the two critical values
are next to each other in cyclic order; or they are separated in
cyclic order only by a single critical point. (If neither is true,
then the mapping pattern is not compatible with any real quadratic map.)
\ssk

\begin{enumerate}[resume*=_steps]
\item Assuming that this is the case, there is a unique cyclic
permutation $ x'_0,~ x'_1,~\cdots,~x'_n$ of the $x_j$ so that either:

  \begin{itemize}
  \item the two extreme points $x'_0$ and $x'_n$ are the two critical values; or

  \item the two extreme points are a critical point and its associated
  critical value, with the other critical value  next to this critical
      point.
   \end{itemize}
 \end{enumerate} 

\noindent The required combinatorics $\vecm$ is then defined by the usual
rule: $f(x'_j)=x'_{m_j}$. \end{proof}
\msk

If we started with a mapping pattern
which is possible for a real quadratic map, then the resulting
$\vecm$ will always be admissible, as defined below.
\end{rem}\msk

\begin{definition}[Admissibility]\label{D-admissible}
There are several restrictions that one
  can put of the sequence $\vecm$ of $n+1$ integers between zero and $n$.
  We will say that $\vecm$ is \textbf{\textit{admissible}} if it satisfies
  the following two essential conditions. (Four further possible 
  restrictions will be described later.)
\ssk

  \begin{enumerate}[series=_admiss]
  \item \label{admiss_cond1}
    The difference between the largest $m_j$ and the smallest is
    either $n$ (in \ref{combi-case1}) or $n-1$ (in
    \ref{combi-case2}). Furthermore, in \ref{combi-case2}, either  
 the first entry will be $n$ or the last entry will be zero. 
\ssk

  \item \label{admiss_cond2}
  After a cyclic permutation which places the smallest $m_j$ on
  the left the resulting sequence will consist of a strictly monotone increasing
  sequence from smallest to largest, followed by a strictly decreasing sequence
  which never gets as far as the smallest.\msk
  \end{enumerate}
\end{definition}

As examples, for \autoref{f1} the cyclically permuted
sequence would be
$$0,~2,~3,~5,~6,~4,~1~ $$
while for \autoref{f2} it would be $0,~1,~2,~3,~2$.
One immediate consequence of Condition~\ref{admiss_cond2}
is that no $m_j$ can occur more than two times in the sequence.
This is clearly a necessary property for quadratic maps.\msk

To relate these properties to Thurston's ideas, we need the following.\ssk

\begin{definition} A \textbf{\textit{Thurston map}} is an orientation preserving
branched covering map from a topological 2-sphere to itself
such that the forward orbit of any branch point is finite.
\end{definition}\msk

The branch points are often referred to as ``critical points'' and their
iterated forward images as ``postcritical points''. See \cite{KL} for
a detailed classification of all Thurston maps with at most four postcritical
points.\ssk

For our purposes, we can take
the Riemann sphere as our 2-sphere, and call a Thurston
map \textbf{\textit{real}} if it commutes with complex conjugation, and has real
critical points.\ssk

\begin{lem} \label{L1} Any admissible combinatorics
  gives rise to a real Thurston  map of degree two.
\end{lem}
  \ssk

\begin{proof} We will identify $\Rhat$
  with the equator $E$ of the Riemann sphere, which divides the
  sphere into a ``northern hemisphere'' and a ``southern hemisphere''.
  (Compare  \autoref{F-RS1}.) The pure imaginary
  complex numbers correspond to an orthogonal great circle $C$,
  which divides the sphere into an ``eastern hemisphere'' to the right,
  and a ``western hemisphere'' to the left.  Each of these great
  circles  is divided into a positive and negative semicircle;
 and they divide the sphere into four quadrants, which are  labeled
 as 1 (for northeast) through 4 (for southeast).
 
\begin{figure}[!htb]
  \centerline{\includegraphics[width=2.5in]{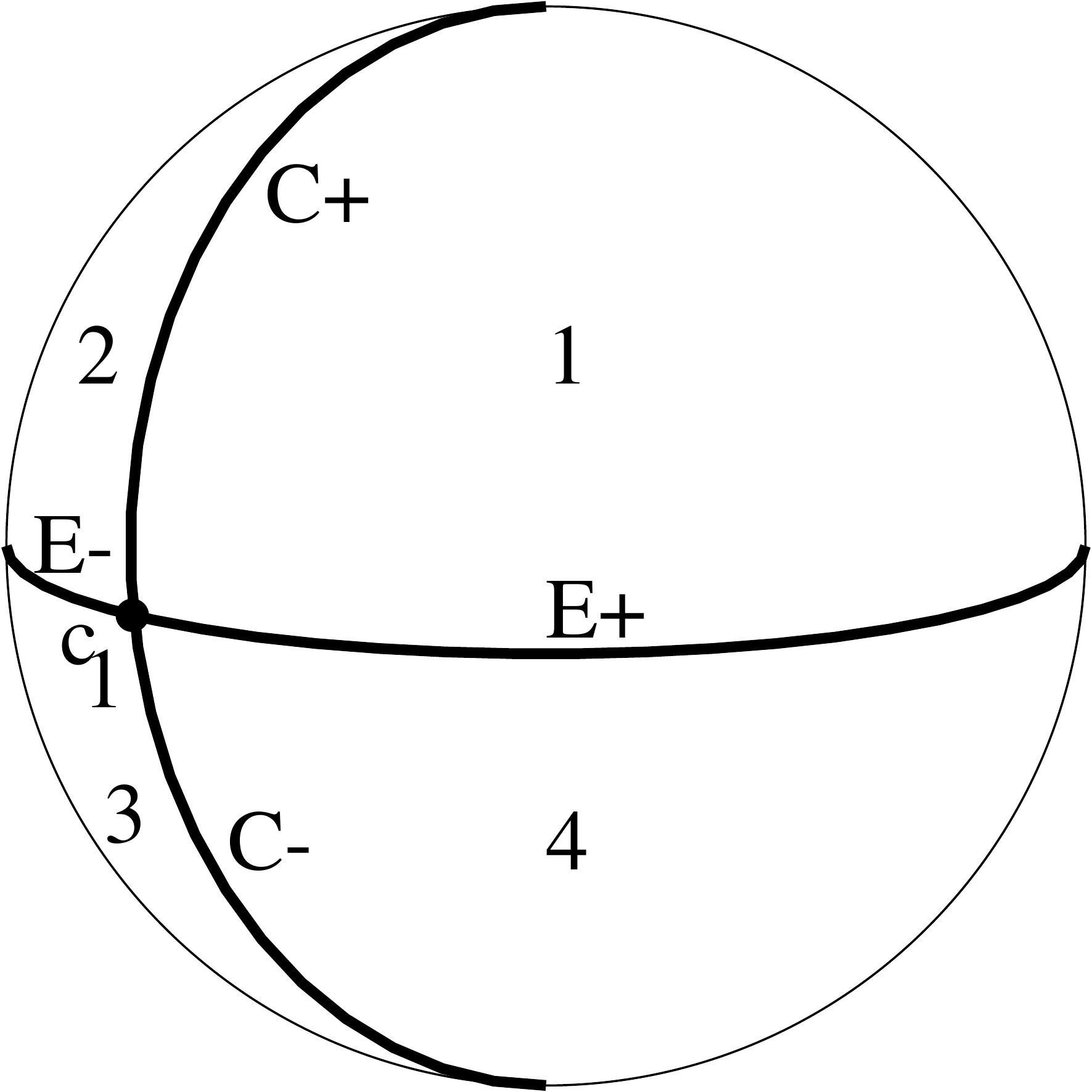}}
  \caption{\label{F-RS1}  The Riemann sphere, divided into
    four quadrants.}
\end{figure}

To begin the construction, place the marked points $x_0$ through
$x_n$ in positive cyclic order around the equator in such a way
that points in an increasing lap are placed in $E+$, those in a decreasing
lap are placed in $E-$, while the two critical points are placed
at the two points of $E\cap C$. Then choose a map $\f$ from
$E$ to itself which sends each $x_j$ to $x_{m_j}$, and maps $E+$
by an orientation preserving homeomorphism, and $E-$ by an orientation
reversing homeomorphism. Thus $E$ will be mapped two-to-one onto
the set $\f(\Rhat)\subset\Rhat=E$. Now extend to a map from $E\cup C$
onto $E$ which sends each of the two semicircles $C\pm$ homeomorphically
onto the closure of the gap $\Rhat\ssm\f(\Rhat)$.

Then it is not hard to check that the boundary of each of the four
quadrants maps homeomorphically onto the equator. It follows that we
can extend to a map $\f$ from the Riemann sphere onto itself which
sends the first and third quadrants homeomorphically onto the
northern hemisphere, and sends the second and fourth quadrants
homeomorphically onto the southern hemisphere. This is the required
two-fold branched covering map, branched only at the two points of
$E\cap C$.\end{proof}

We  will say that the combinatorics $\vecm$ is \textbf{\textit{unobstructed}}
if there exists a rational map having 
combinatorics $\vecm$. Otherwise it is  \textbf{\textit{obstructed}}. In nearly
every case, if the combinatorics is unobstructed, then Thurston's iterated
pull-back construction, as described in \S\ref{s4}, converges locally uniformly
to the required rational map. (Compare \S\ref{s6}.) 
For the unique exceptional case, see \autoref{P-exc}.
\msk

In addition to the essential requirements~\ref{admiss_cond1}
and~\ref{admiss_cond2} described 
earlier, there are four further requirements that we may want to impose 
on the combinatorics. We will refer to $\{0,\,1,\,\cdots,\, n\}$ as the 
``marked points'' for the associated 
 piecewise linear map, and the two points where this map 
 is maximized or minimized as the ``critical points''.\msk

We will say that the combinatorics is \textbf{\textit{minimal}}
if it is admissible, and also 
satisfies the following two conditions, which say roughly that
every vertex and every edge of the associated PL graph is essential.

\begin{enumerate}[resume*=_admiss] 
\item\label{admiss_minimal}
  \textbf{(All marked points are critical or postcritical)}
  (Compare \autoref{F-min}.)  Of course
  this condition is automatically satisfied for combinatorics constructed
  from a given postcritically finite map, as described in the 
  beginning of this section. 

\item\label{admiss_expansive}\textbf{(Expansiveness)} 
  For any edge of the piecewise
  linear model, some forward image contains a critical point. (Compare
  \autoref{F-nonex}.)
\end{enumerate}

For the analogous theory 
for polynomial maps of any degree, there is no Thurston obstruction
if and only if the corresponding condition is satisfied.
(Compare \cite{BS} and \cite{P}, as well as \cite{BMS}.) 
However for quadratic rational maps, although this is still a necessary
condition, it is far from sufficient: In many cases, there is an
obstruction even when the combinatorics is expansive.

\begin{rem}[\bf Lifted Graphics] 
  From now on, most  graphs of quadratic rational maps will be shown in
  \textbf{\textit{lifted normal form}}. The precise definition
  will be given  in \autoref{s3}; but roughly speaking 
this is a   form which provides  a uniform presentation
in which the point at infinity does not play any special role. The
parameters $\mu$ and $\kappa$ which uniquely determine the map will
usually be given in \autoref{t-2} of \autoref{ap-b}.
\end{rem}
  
\begin{figure}[!htb]
    \centerline{
      \includegraphics[height=1.7in]{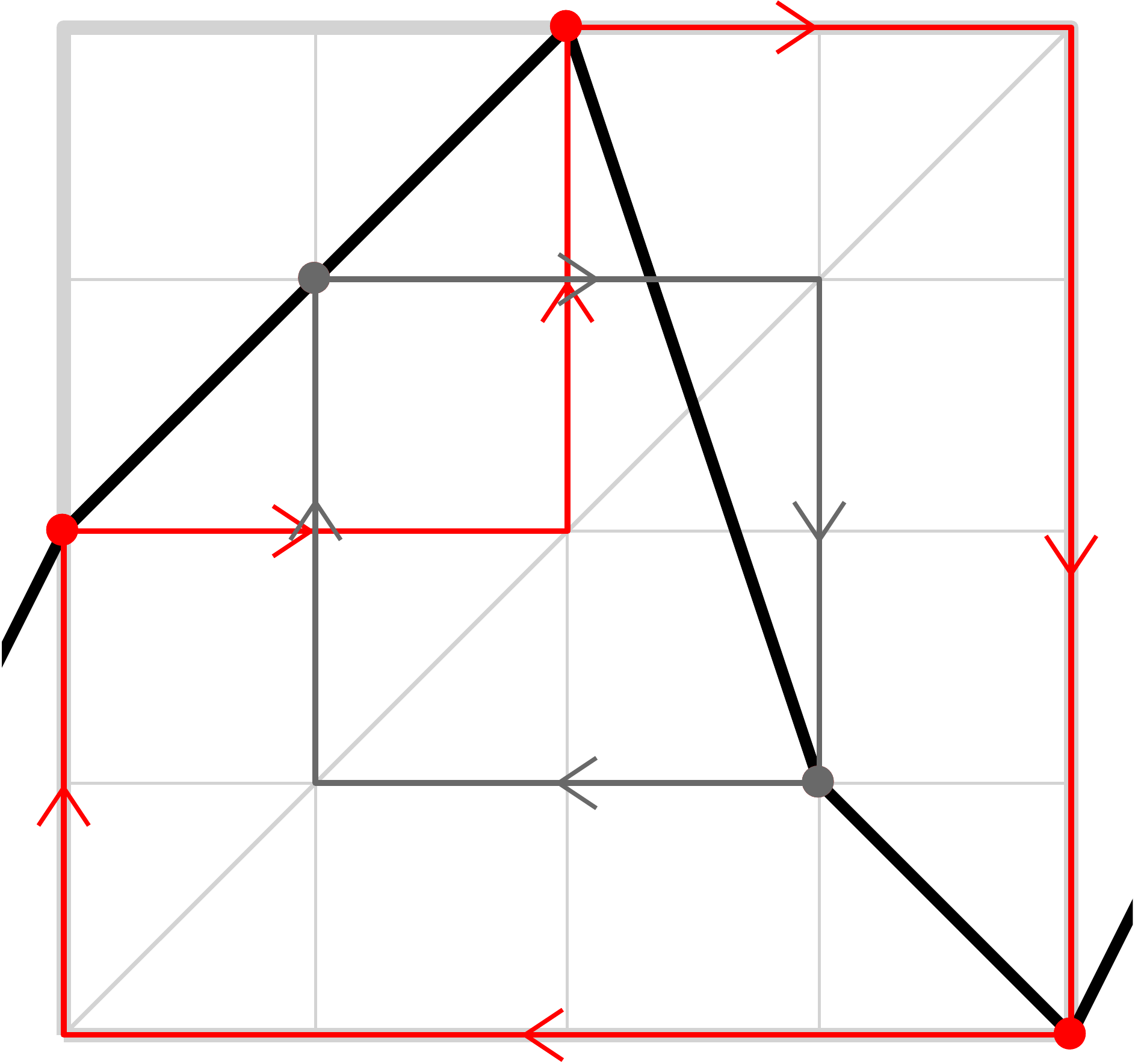}\hfill
      \includegraphics[height=1.7in]{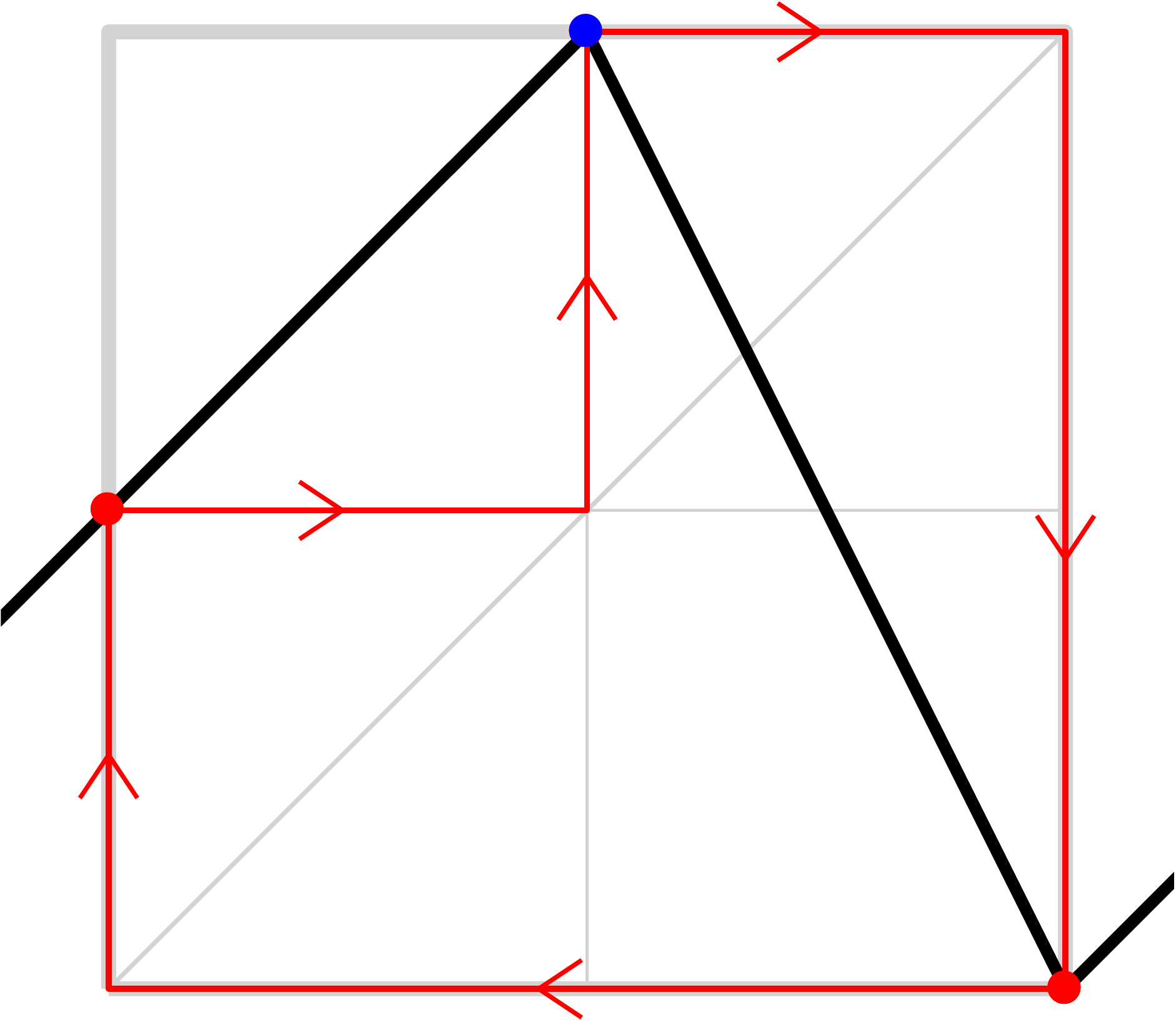}\hfill
      \includegraphics[height=1.7in]{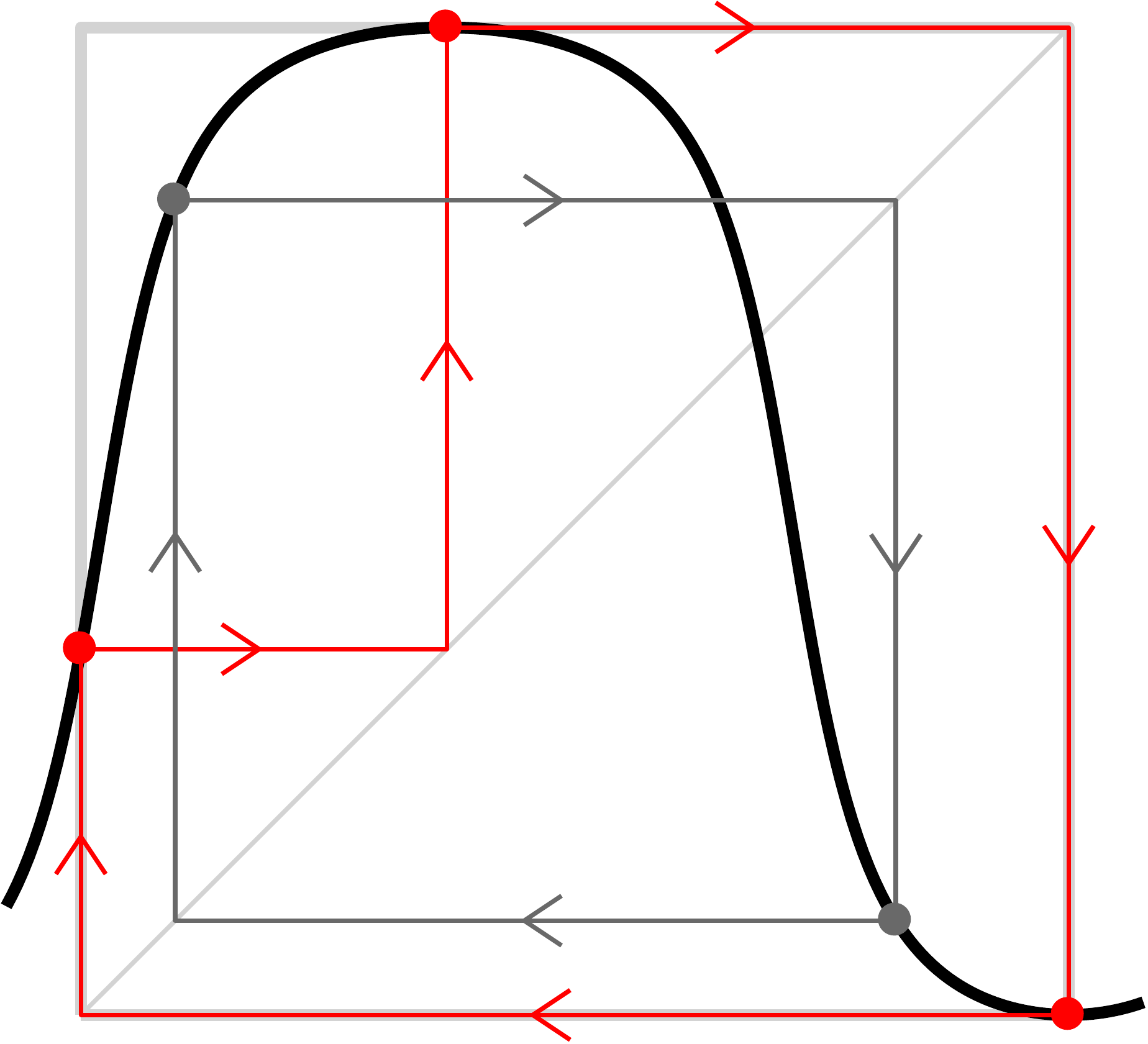}}
    \caption{\label{F-min}  The combinatorics $\(2,\,3,\,4,\,1,\,0\)$ is not
      minimal, since the points $1$ and $3$ are not critical or postcritical.
      Removing $1$ and $3$
      and renumbering, we get the simplified combinatorics $\(1,\,2,\,0\)$. The
      resulting rational map will be the same in either case. The only
      difference is that in the graph of the corresponding rational map,
      the period two orbit $x_1\leftrightarrow x_3$ will be shown in the
      first case but not in the second.}
    \end{figure}

\begin{figure}[!htb]
  \centerline{
    \includegraphics[height=1.7in]{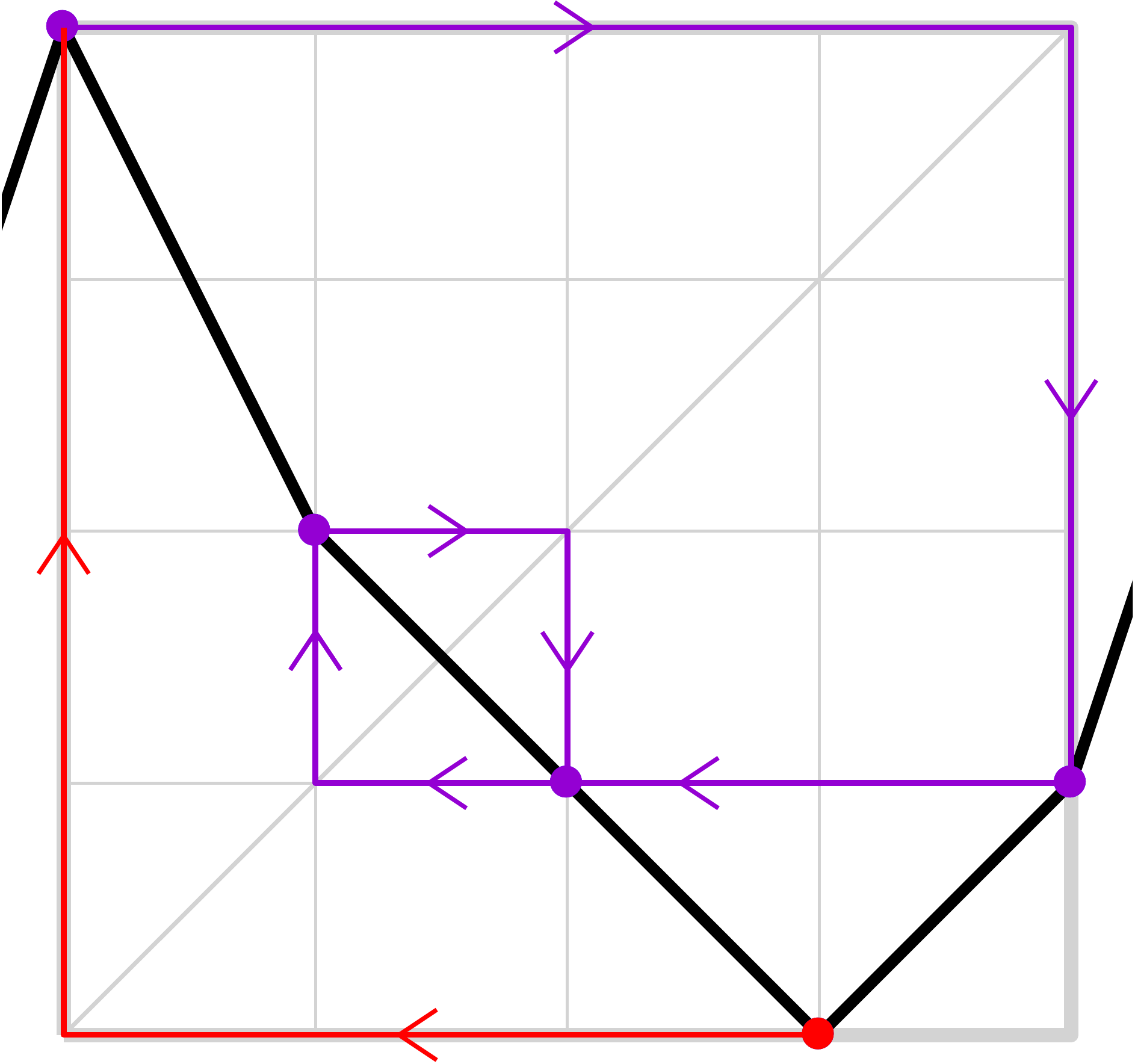}\qquad\qquad
    \includegraphics[height=1.7in]{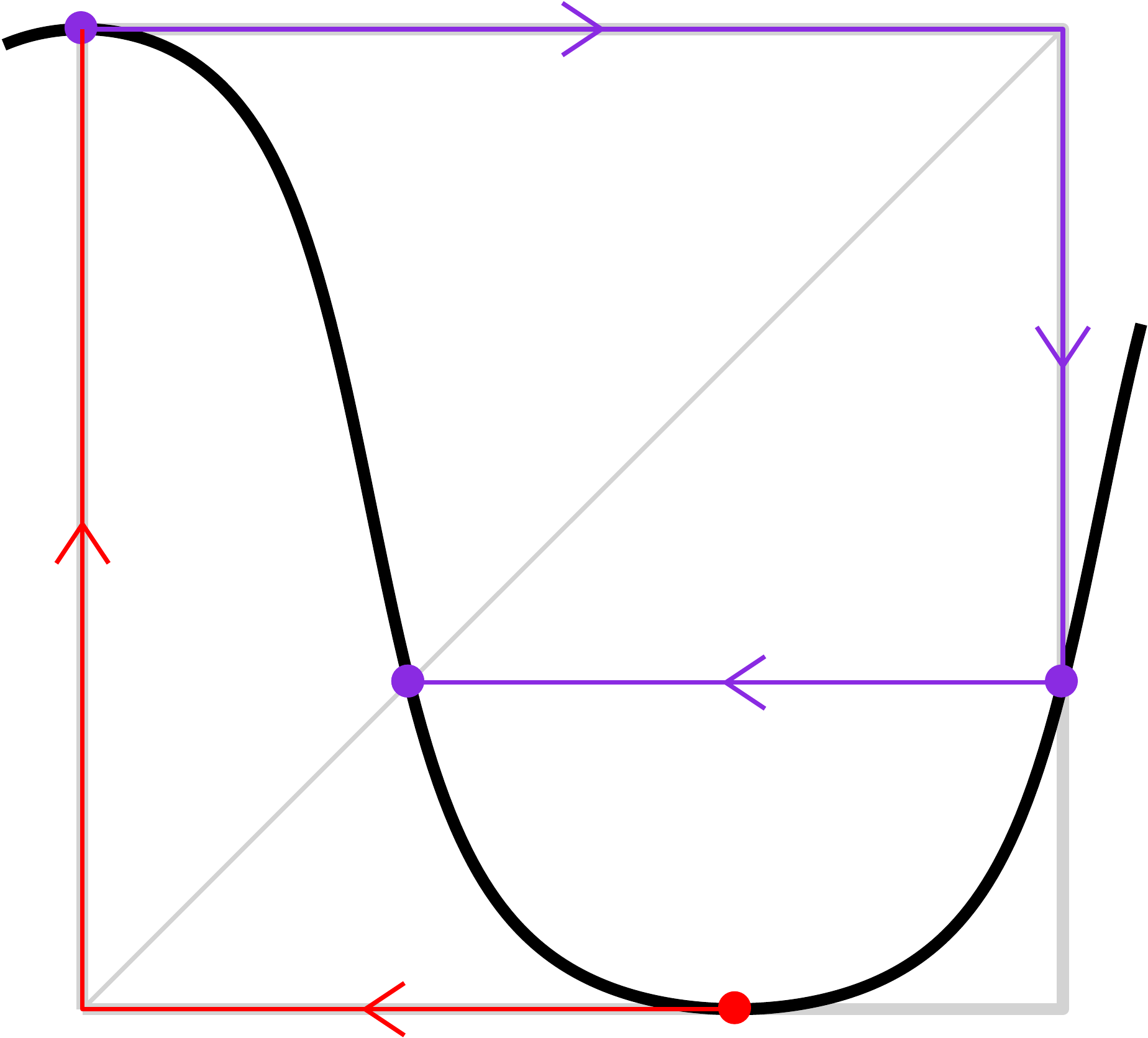}}
  \caption{\label{F-nonex}    An example of combinatorics
    which satisfies Condition~\ref{admiss_minimal} but is not expansive.
    Here $n=4$, the map has  combinatorics
    $\(4,\,2,\,1,\,0,\,1\)$. The edge $[2,\,3]$ is not expansive,
    and hence collapses to a point in the corresponding lifted rational
    map, which has simplified combinatorics $\(3,\,1,\,0,\,1\)$.} 
\end{figure}

\begin{rem}\label{R-bad}
  In most cases, any admissible combinatorics which is not minimal can be
  reduced to a minimal example with simplified combinatorics in three steps
  as follows. (Compare Figures~\ref{F-min} and \ref{F-nonex}.) However this
  is not always possible, so a fourth step is needed.

  \begin{enumerate}[resume*=_steps,start=1,
      topsep=3pt plus 1ex,midpenalty=9999,endpenalty=-9999,
      leftmargin=2\parindent,labelwidth=\parindent,itemsep=0pt plus 1ex]
\item Remove any vertices of the PL model which are not critical or
  postcritical.

\item Collapse any edge which is not expansive to a point.

\item Renumber the vertices which are left.

\item Check that the resulting combinatorics is admissible.
  (For an example where this last step fails, see \autoref{f-bad}.) 
\end{enumerate}

\end{rem}\bigskip

 \begin{figure}[!htb]
\centerline{\includegraphics[height=1.7in]{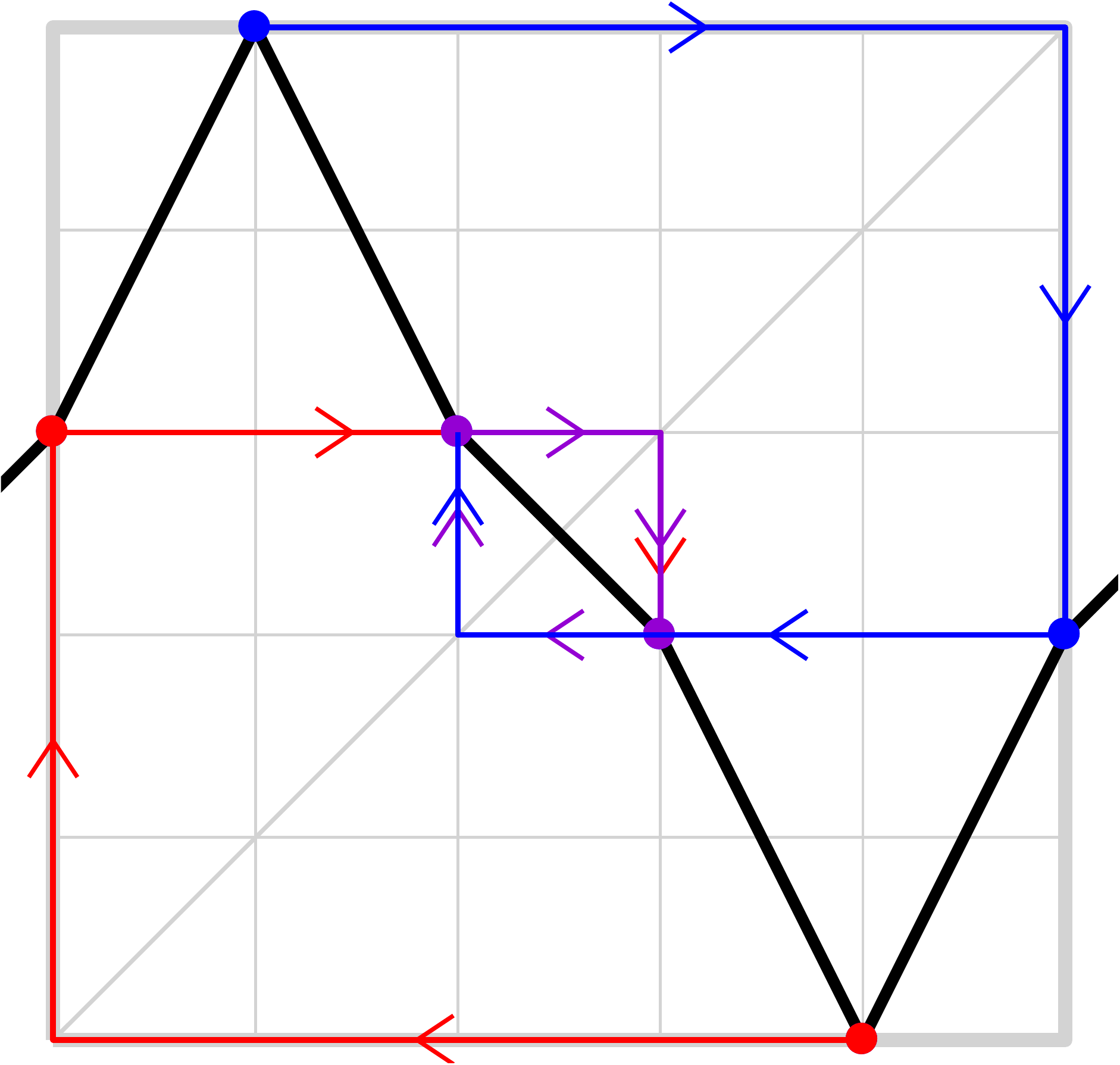}\qquad\qquad
            \includegraphics[height=1.7in]{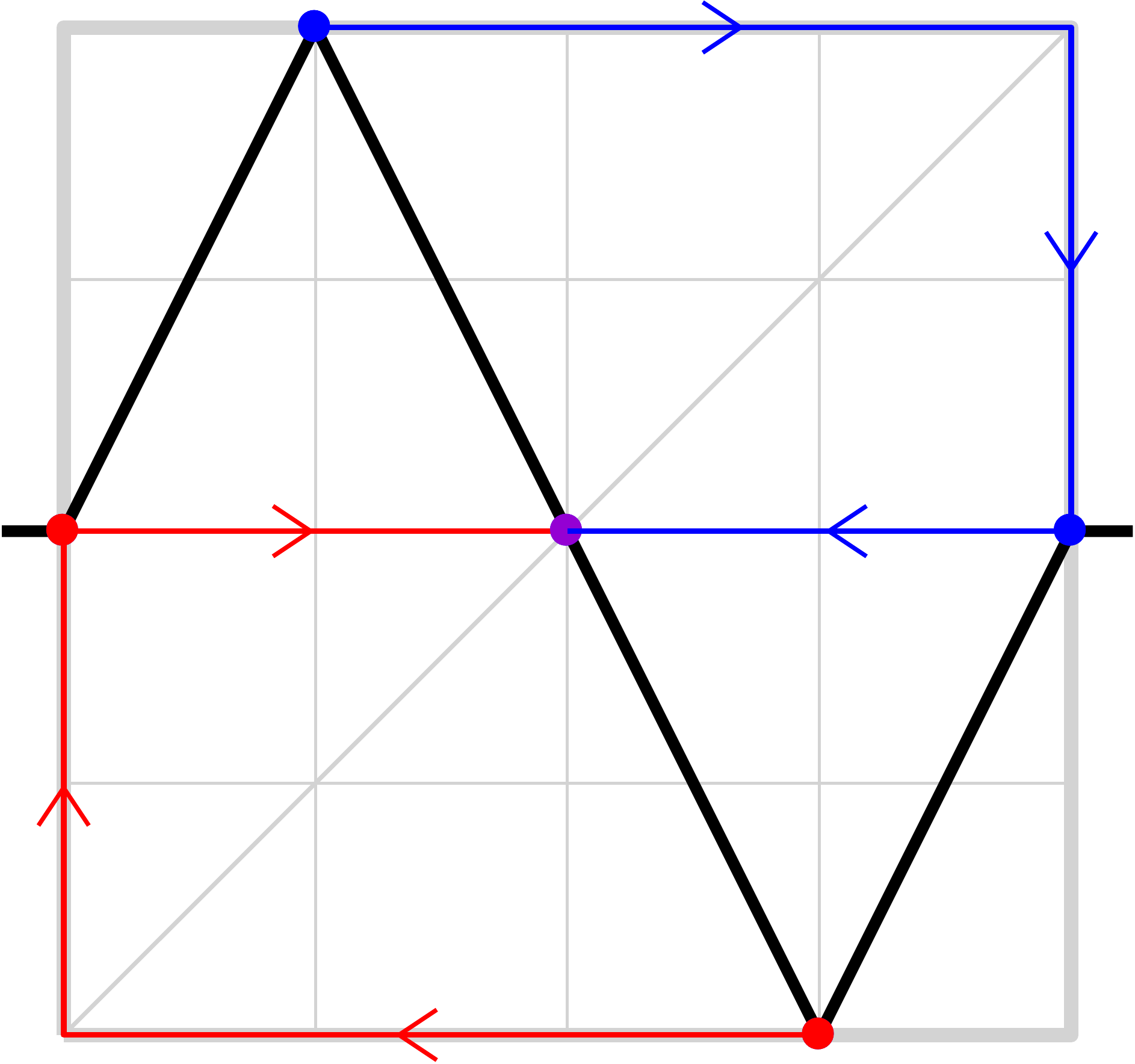}}
 \caption{\label{f-bad}  The combinatorics $\(3,\,5,\,3,\,2,\,0,\,2\)$
 on the left is admissible, although not minimal. However, if we try to
 reduce
 it to a minimal example, we obtain the combinatorics $\(2,\,4,\,2,\,0,\,2\)$
 on the right, which is not admissible and hence 
 cannot describe any quadratic example, since one
 horizontal line intersects the graph three times. (In the terminology we will introduce in
\autoref{s5}, it follows that the combinatorics on the left is ``strongly obstructed''.) 
}
   \end{figure}

      There are two further restrictions which we may sometimes
      want to impose on the  combinatorics.\smallskip

\begin{enumerate}[resume*=_admiss]
\item\label{admiss_notpoly}\textbf{(Not a polynomial).} The combinatorics
  satisfies  $m_0>0$ and $m_n<n$ so that there 
is no critical fixed point. There is nothing wrong with combinatorics
with a critical fixed point since they correspond to maps conjugate to a
polynomial. In fact they are much easier to deal with; and quadratic 
polynomials are well understood. However, we will concentrate
on non-polynomial maps in the subsequent discussion, except in
\autoref{s8} where polynomial maps will play an essential role). 
\end{enumerate}

\begin{figure}[!h] 
  \centerline{\includegraphics[height=1.6in]{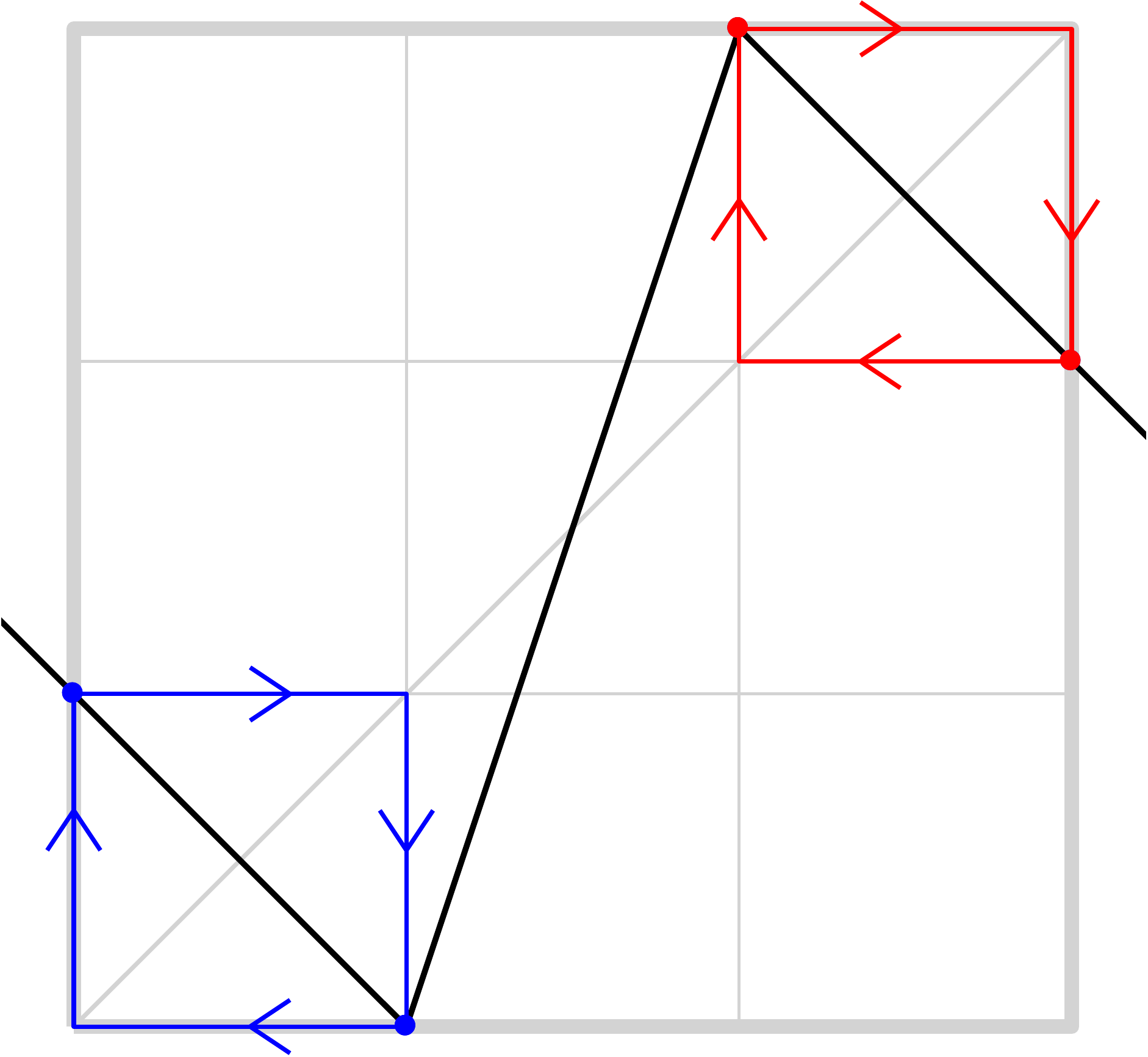}}
  \caption{\label{f4}  A piecewise linear example with combinatorics
    $\(1,\,0,\,3,\,2\)$ and with three fixed points.
   See also \autoref{f-aBn3a}R.}
  \end{figure}

\begin{enumerate}[resume*=_admiss]
\item\label{admiss_onefp}\textbf{(Only one fixed point)}
   The associated piecewise linear model has exactly 
one fixed point. This is closely related to the previous requirement,
since a critically finite quadratic polynomial always has three fixed
points. However we will see in \autoref{L2} that a critically finite
quadratic map which is not conjugate to a polynomial always has just
one fixed point. Thus for combinatorics such as $\( 1,~0,~3, ~2\)$, 
as illustrated in \autoref{f4}, there must be a Thurston obstruction.
(Compare \autoref{s5}  and \cite[Lemma 10.2]{M}. 
)
\end{enumerate}\msk

\begin{rem}[\bf Dynamic Classification]
By definition a critically finite rational map is \textbf{\textit{hyperbolic}}
if every postcritical cycle contains at least one critical point. In the
quadratic  case, every hyperbolic map  belongs to one of the following three types:

  \begin{description}[style=sameline,leftmargin=2.5\parindent,
                      rightmargin=2\parindent,labelwidth=\parindent]
  \item[\bf Type B (bicritical).] Both critical points are contained in a common
    periodic orbit. (Compare \autoref{F-min}.) 
    
  \item[\bf Type C (capture).] The orbit of one critical point lands, after one or
    more iterations, on a periodic orbit containing the other critical point.
    (Compare \autoref{f2}.)

  \item[\bf Type D (disjoint).] The orbits of the two critical points are periodic
    and disjoint. (Compare \autoref{f1}.)
  \end{description}

Similarly, each non-hyperbolic map belongs to one of two types:

  \begin{description}[resume*]
  \item[\bf Half-Hyperbolic.] One critical orbit is periodic; but the other is
    only eventually periodic. (Compare \autoref{f-semihyp1}.)

  \item[\bf Totally Non-Hyperbolic.] No critical orbit is periodic, 
    although both are eventually periodic.
    (Compare \autoref{f-nonhyp}.)
  \end{description}

  Conjugacy classes of Hyperbolic Type are always isolated, 
  since they are contained in an open hyperbolic component which contains
  no other critically finite point. However a sequence of hyperbolic
  critically finite conjugacy classes may well converge to a limit which
  is Half-Hyperbolic or Totally Non-Hyperbolic.
\ssk

This classification extends easily to our piecewise-linear model maps, and
hence to any admissible combinatorics.
\end{rem}

\begin{rem}\label{R-npc}
  One important number associated with any combinatorics  is the number of
  postcritical points, which we will denote by $\npc$. Note that $\npc=n+1$ for
  combinatorics of Type B or D, but $\npc=n$ in the  Type  C or Half-Hyperbolic
  cases; while $\npc$ may be either $n$ or $n-1$ in the Totally
  Non-Hyperbolic case. Cases with $\npc\le 4$ require special
  attention in   Thurston's theory (Compare \autoref{P-exc}.)
\end{rem}

\begin{rem}[Topological Classification]\label{R-top-class}
  Real quadratic maps can be classified topologically by the location of
  their  critical points with respect to the interval $f(\Rhat)$. For a very rough
  classification, let $\ell$ be the number of critical points in the interior
  of $f(\Rhat)$. Then the map $f$ restricted to $f(\Rhat)$ can be described as
  \textbf{\textit{bimodal}} if $\ell=2$, or \textbf{\textit{unimodal}}
  if $\ell=1$, and \textbf{\textit{monotone}} if $\ell=0$.

  The $\ell$ critical points
  divide $f(\Rhat)$ into $\ell+1$ \textbf{\textit{laps}}, or maximal intervals of
  monotonicity. On each lap, $f$ is either monotone increasing  if $f'>0$
  (indicated by a plus sign), or monotone decreasing  if $f'<0$ (indicated by a
  minus). Thus in the bimodal case
  we either have the case $+-+$ or the case $-+-$, with a similar dichotomy
  for the unimodal and monotone  cases.\msk

\begin{figure}[!t]
\newcommand{\figHT}{.18\textwidth}
\begin{tabular*}{\textwidth}{c @{\extracolsep{\fill}} c c}
 \includegraphics[height=\figHT]{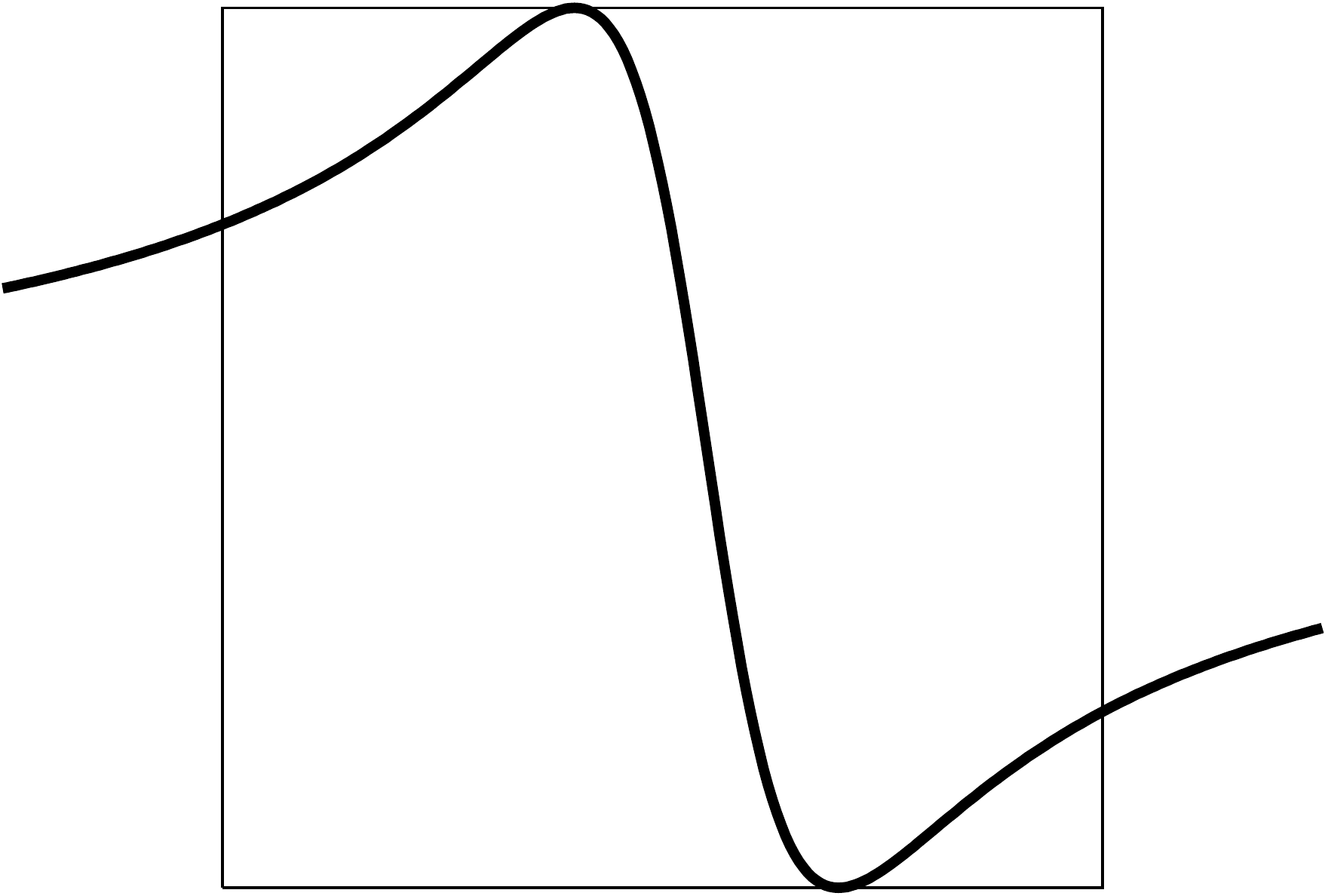}  &
 \includegraphics[height=\figHT]{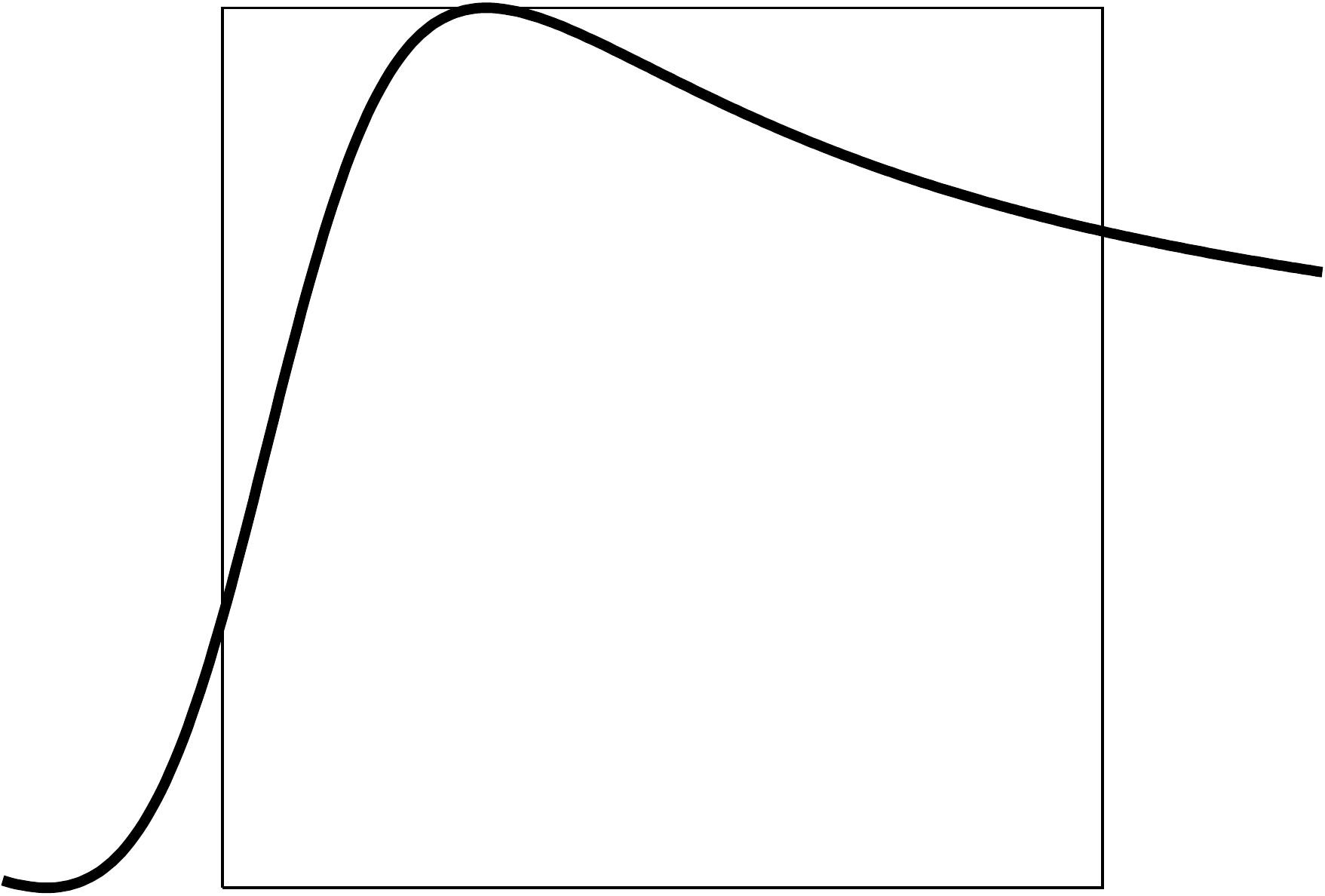}   &
 \includegraphics[height=\figHT]{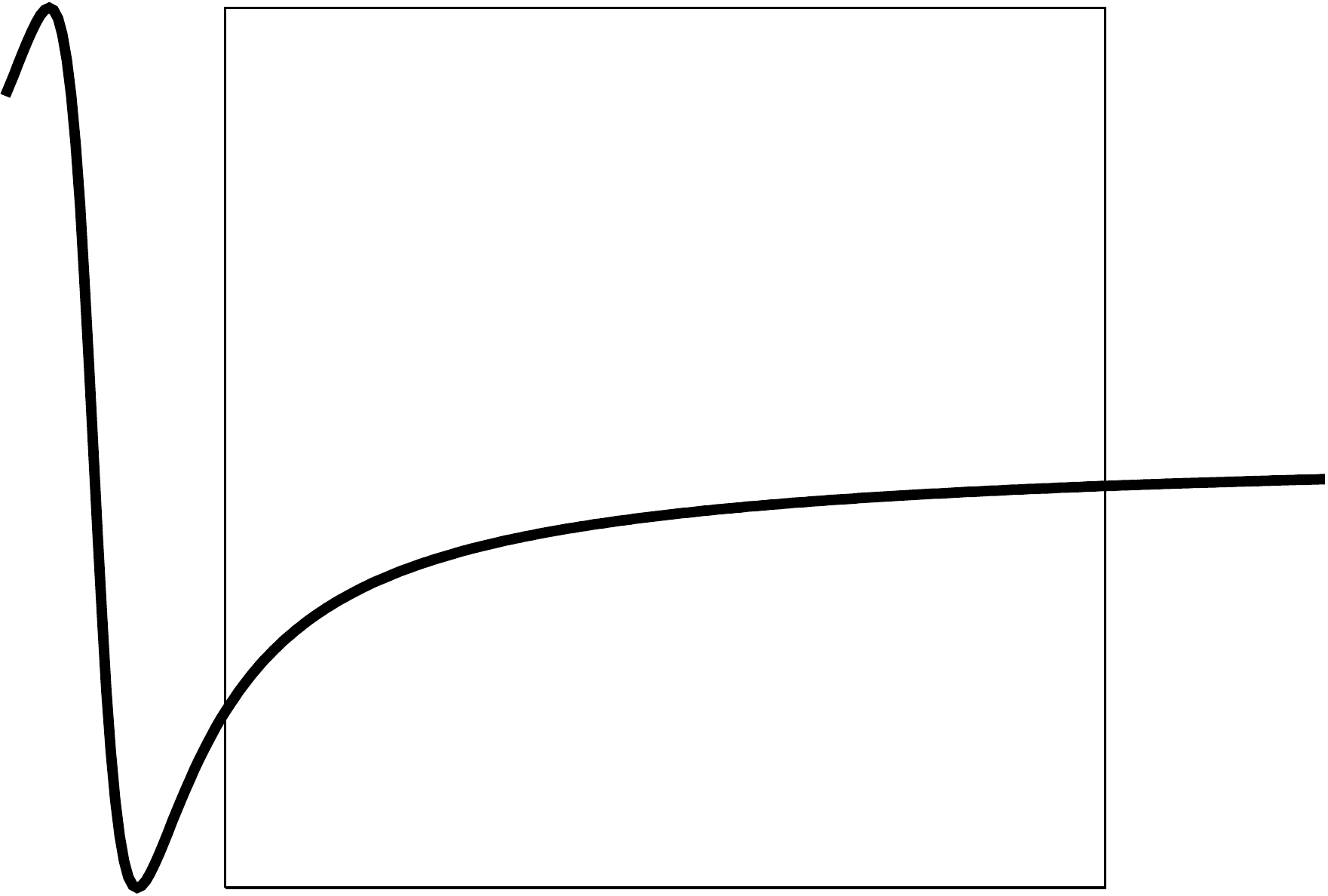}    \\
 $\displaystyle +\,-\,+$ & 
 $\displaystyle +\,-$    & 
 $\displaystyle +$ \\[2ex]
 \includegraphics[height=\figHT]{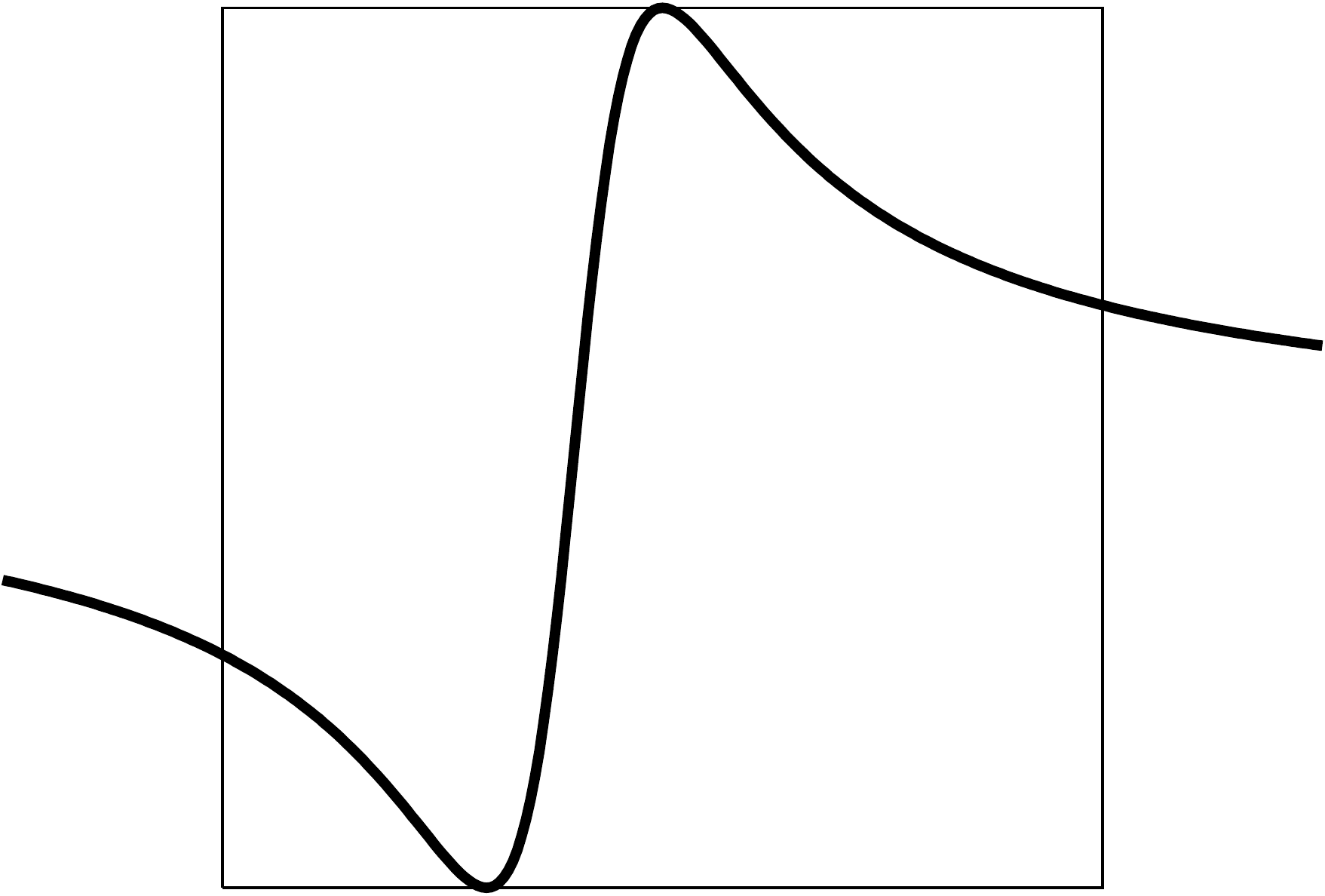}  &
 \includegraphics[height=\figHT]{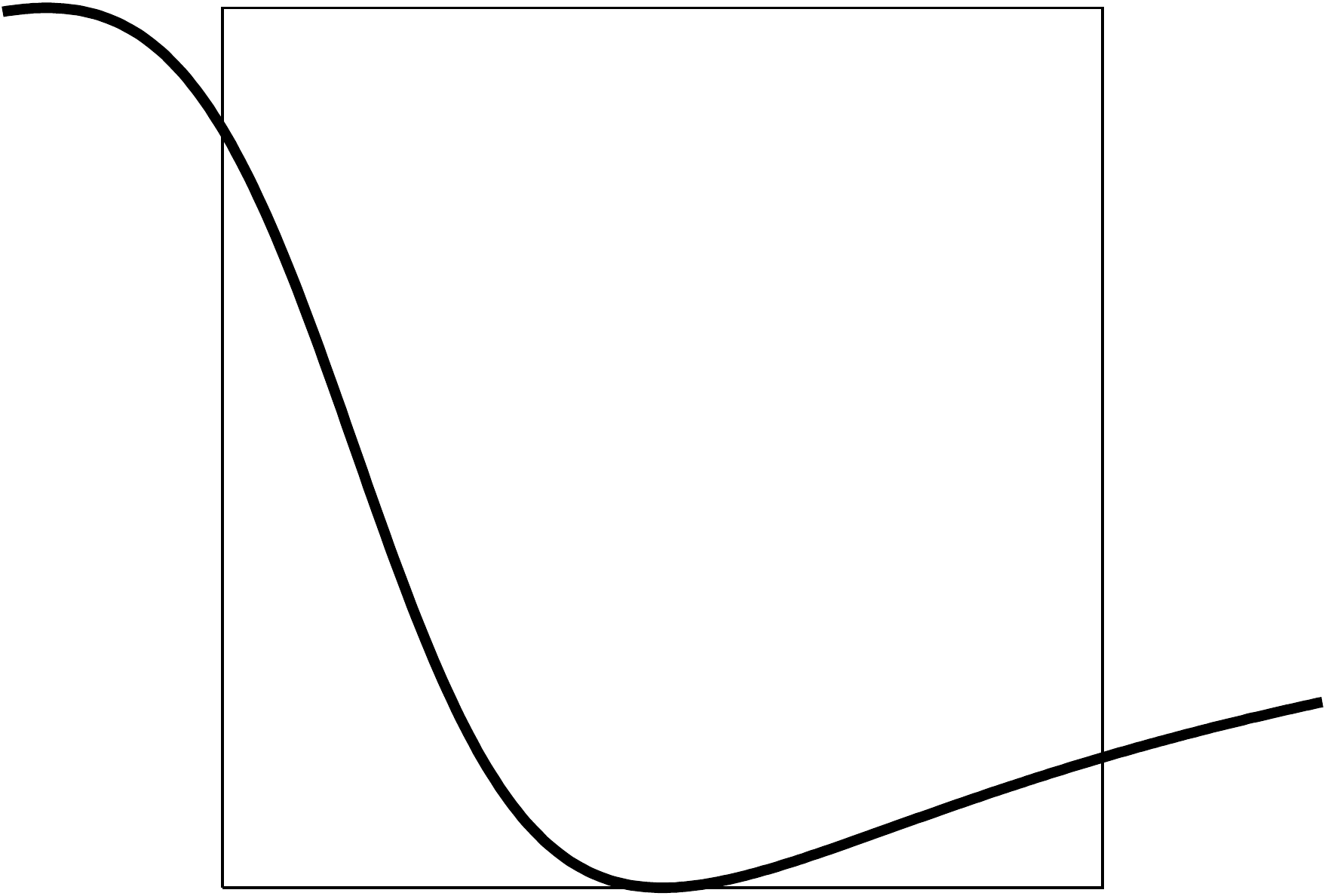}  &
 \includegraphics[height=\figHT]{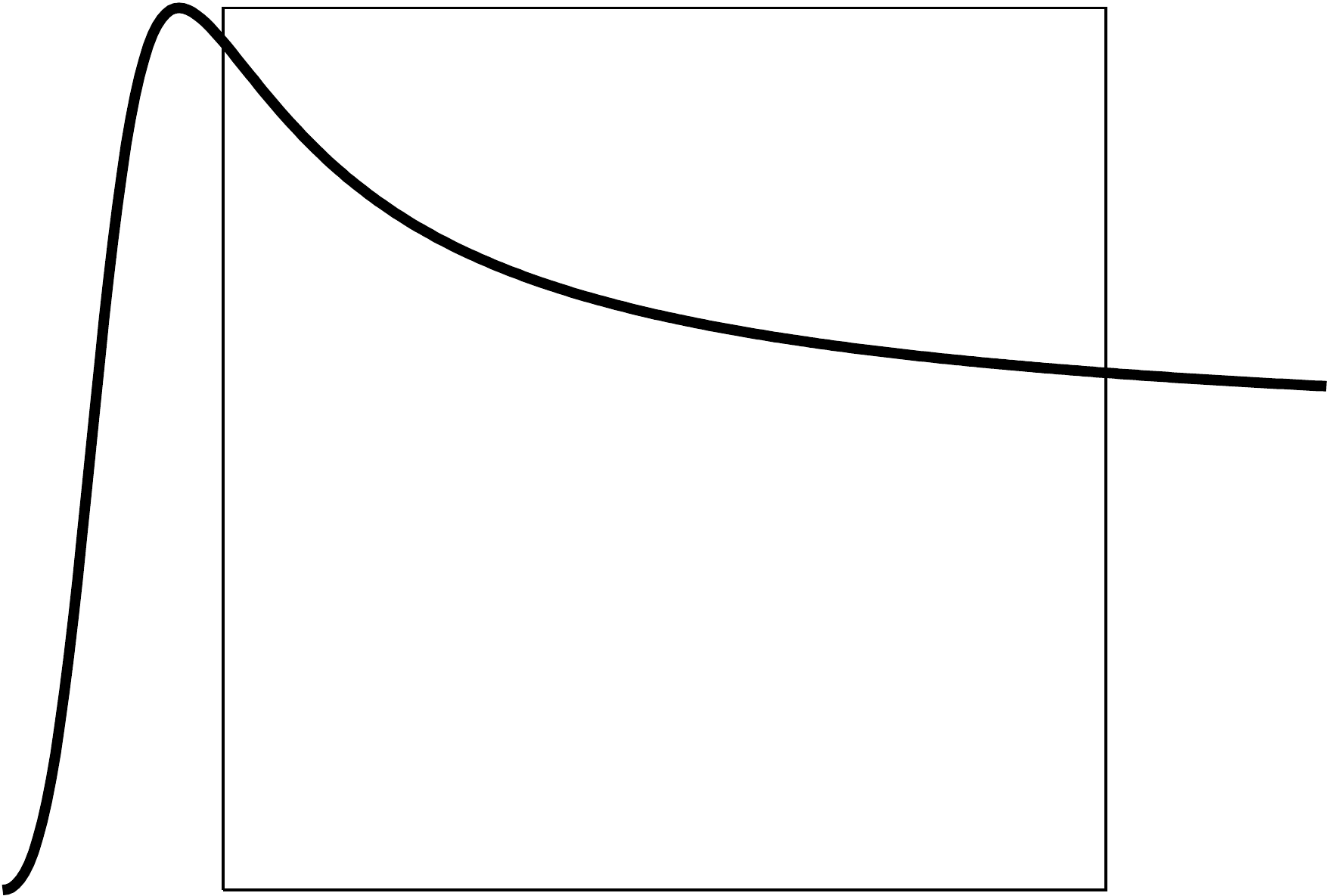}  \\
 $\displaystyle -\,+\,-$ & 
 $\displaystyle -\,+$    & 
 $\displaystyle -$ 
\end{tabular*}
\caption{ \label{af1} The six possible shapes for maps which
  have   no critical point on the boundary of $f(\Rhat)$.
 The left two figures are of bimodal shape, the middle two of unimodal
   shape, and the last two of monotone shape.
 (The $+-$ and $-+$ cases are not really different from each other, since
 one can be obtained from the other by the orientation reversing change
 of coordinate $x\leftrightarrow 1-x$.)}
\end{figure}

For a more precise classification, we must single out the cases where there is
a critical  point precisely in the boundary of $f(\Rhat)$, or in other words,
a critical point which is also a critical value. There are two possibilities:

\begin{definition}\label{d-copoly} The quadratic map $f$ is of
  \textbf{\textit{polynomial shape}}\footnote{We will reserve the word
    ``type'' for dynamic properties, which involve following critical orbits,
    and use the word ``shape'' for topological properties, which are
usually   evident from a glance   at the  graph of $f$, restricted
    to a neighborhood of $f(\Rhat)$.}
  if it has a critical  fixed point; and is
  of \textbf{\textit{co-polynomial shape}}
  if one critical point maps to the other. Note that $f$ is of polynomial
  shape  if and only if it is conjugate to a polynomial; and is of
  co-polynomial shape
    if and only if it is conjugate to a map of the form $x\mapsto 1/p(x)$ where
    $p(x)$ is a polynomial.  (See \autoref{P-copoly}.) \end{definition}

  We will say that a map is \textbf{\textit{strictly unimodal}} if its
  restriction to $f(\Rhat)$ is unimodal, and if there is no critical point of
  the boundary of $f(\Rhat)$, so that it will remain unimodal under a small
  perturbation.  Thus means that one critical point is in the interior and
  one critical point is
  strictly outside of $f(\Rhat)$. Similarly it is strictly monotone if both
  critical points are strictly outside of $f(\Rhat)$. (Compare \autoref{af1}.)

   This dynamic classification of maps gives rise to a corresponding
  partition of the  ``moduli space'', which consists 
  of all conjugacy classes, into six connected open sets, and two connected
  closed sets made  of points which are on the common boundary
  between two or more of 
  these open sets. (Compare \autoref{F-canmod} in \autoref{s7}.)\msk

  In the unimodal and bimodal cases, this classification  extends easily to our
  PL model maps,
  and hence to any admissible combinatorics. However, there is no
  such thing as strictly monotone combinatorics.
  \end{rem}

  \begin{rem}[Relations between dynamic and topological classifications]\label{R-rel}
    Although these classifications are quite different, there are some
    obvious and some not so obvious relations between them. As 
an obvious relation, for combinatorics of Type B or D, both critical orbits
are periodic, so we  cannot be  in the strictly unimodal case.
Note also that co-polynomial combinatorics can only be of
Type B, or Totally Non-Hyperbolic. This is true since any given point can have
at most two immediate pre-images, counted with multiplicity. 
If for example the critical point $c_1$
maps to $c_2$, then no other point can map to $c_2$. Therefore we must be
either in the Type B or the Totally Non-Hyperbolic case. A similar
argument shows that polynomial combinatorics can only be of Type D
or Half-Hyperbolic. Here is a less obvious example.

\begin{prop}
  No combinatorics of shape $~-+-~$ can be of   Type B.
\end{prop}

\begin{proof} The piecewise linear map $\f:[0, n]\to[0, n]$ necessarily 
  has three fixed points.  Let $\widehat x$ be the middle fixed
      point. Then either $\widehat x>f(0)$ or $\widehat x<f(n)$ or both.
      In the first case the interval $[0,~\widehat x]$ maps to itself,
      and in the second case $[\widehat x,~n]$ maps to itself. In
      either case, no periodic orbit can contain both zero and $n$.
    \end{proof}
    
  If we consider only unobstructed combinatorics, corresponding to
 actual quadratic maps, then there are further restrictions.
We will see in \autoref{s3}
that the  $(-+-)$-bimodal region does not contain any maps which are
critically finite. However, it is easy to find combinatorics which
are of shape $-+-$; and it follows that these must be strongly obstructed. 
See \autoref{ap-b} for further information.
\end{rem} 

  \begin{rem}[The Cross-Ratio Invariant]\label{R-cr}
    One simple and useful invariant is the following. 
  If $c_1$ and $c_2$ are the two  critical points of $f$
  and $v_1$, $v_2$ are the corresponding critical values,
  then the cross-ratio
  $$ \rho(f)~=~\frac{(c_1-v_1)(c_2-v_2)}{(c_1-c_2)(v_1-v_2\
)} $$
is clearly invariant under fractional linear changes of
coordinate. It is an easy exercise to check the following:

\begin{itemize}
\item[$\bullet$] $\rho=1$ if and only if $f$ is of polynomial shape.

\item[$\bullet$] $\rho=0$ if and only if $f$ is of co-polynomial shape.

\item[$\bullet$] $0<\rho<1$ if and only if $f$ is  strictly unimodal.

\item[$\bullet$] $\rho$ is finite in all cases.
\end{itemize}

\noindent However this invariant does not distinguish between
the $+-+$
bimodal case and the $-$ monotone case, both with $\rho<0$. 
Similarly it does not distinguish between the $-+-$ bimodal
case and the $+$ monotone case, both with $\rho>1$.
\end{rem}
\ssk

  \begin{rem}[Orientation Reversal]\label{R-I}
    If we reverse orientation, then any given combinatorics
    $\vecm=\(m_0,~\ldots,~m_n\)$ will be replaced by 
    $$\I(\vecm)~=~\(n-m_n,~n-m_{n-1},~\ldots,~n-m_1,~n-m_0\)~.$$
    This corresponds to 180 degree rotation of the graph.
    It does not affect the dynamical classification or the
    cross-ratio invariant. However, it replaces
 any unimodal combinatorics of shape $+-$ 
    by  unimodal combinatorics of shape $-+$ with identical dynamic
    properties. For more on this orientation reversing involution,
    see \autoref{s7} and \autoref{F-M/I}. 
    \end{rem}

    \section{The Lifted Normal Form}\label{s3}
    
The object of this section will be to introduce a
family of  real quadratic maps, parametrized by their
two critical values, in a form which is easy to understand and
which is convenient for carrying out the Thurston algorithm. The
following will help to motivate the construction.
 \ssk

\begin{lem}\label{L2} Let $f$  be a real quadratic map, not of polynomial
   shape, such that every real fixed point is strictly repelling. Then
  $f$ has precisely one real fixed point, and precisely
  one decreasing lap,  which must contain this fixed point. Furthermore,
  there must be at least one critical point in $f(\Rhat)$. 
\end{lem}
\ssk

In particular, these statements apply to any critically finite map
which is not of polynomial  shape. For such maps,
 the unique  real fixed point is always repelling, with  multiplier
 $\ml<-1$. The map may be  of  shape  $+-+$, or strictly 
 unimodal, or co-polynomial;
but it can never be  of  shape $-+-$, or strictly monotone. (These
statements do not apply to critically finite 
maps of polynomial  shape; so these may require slightly different
treatment.)
\ssk

\begin{proof}[Proof of $\autoref{L2}$]
  First note every decreasing lap must contain exactly one fixed point.
  In fact, for the graph of the given lap, the left hand endpoint must be
  above the diagonal and the right hand endpoint must be below the diagonal;
  and it is easy to see that the graph cannot cross the diagonal twice.
On the other hand, for an increasing lap there can be at most one fixed point.
In fact the orbit of any point
between two consecutive fixed points must converge to one or the other, which
would  contradict our hypothesis that there is no attracting or indifferent
fixed point.

If a real quadratic map has two real fixed points, then it must have three,
counted with multiplicity. Since we have excluded indifferent fixed points, this
means that there must be three distinct laps, each with its own repelling
fixed point. This is perfectly possible for a smooth or piecewise 
linear map. (Compare \autoref{f4}.) But it is not possible for a quadratic
map. According to \cite{M}, the multipliers of these three fixed
points must be related by the equation
  $$ \ml_3~=~\frac{2-\ml_1-\ml_2}{1-\ml_1\ml_2}~.$$
    Thus if $\ml_1>1$ and $\ml_2>1$, then it follows that $\ml_3>0$;
    while if $\ml_1<-1$ and $\ml_2<-1$, it follows that $\ml_3<0$.
    Thus all three  multipliers must have the same sign,
    which is impossible, since the sign must be
  positive in an increasing lap and negative in a decreasing lap. This contradiction
  proves that there can be only one fixed point; and hence only one decreasing lap. Finally note that every strictly monotone map must have an
  attracting or parabolic fixed point. In fact $f\circ f$ will always
  be monotone increasing on $f(\Rhat)$, hence every orbit of $f\circ f$
  must converge to an attracting or parabolic fixed point.
\end{proof}
\bsk

In particular, it follows that a bimodal
map of shape $~-+-~$ with two decreasing laps can never be critically finite.
Furthermore, it is not hard to check that a map with only one lap is
critically finite only in two very special cases, 
namely maps $\pm$-conjugate to $f(x)=x^2$ or $f(x)=1/x^2$. 
\bsk

We are finally ready to discuss normal forms. We will be primarily
interested in maps which have exactly one fixed point in the lap
between the two critical points. This will be called the 
\textbf{\textit {primary fixed point}}.\footnote{Of course it is often
  the only real fixed point. In some $-+-$ cases there will be three fixed points in the middle
  lap, so that the normal form is no longer unique. In some polynomial cases, there is
  no such fixed point.} In such cases, the map is conjugate
to a uniquely defined map with critical points at $\pm 1$ and with
primary fixed point at zero. 
We can  write the resulting map in \textbf{\textit{Epstein normal form}} as
 \begin{equation}\label{EDM}
   f(x)~=~\frac{\ml x}{x^2+2\ka x+1}~.\end{equation}
 (Compare Epstein \cite{E}, as well as DeMarco \cite{D}.)
Note the identity  $f(x)= f(1/x)$. 
 Here $x=0$ is a fixed point of multiplier $\ml\ne 0~$.  The critical points
 are $c_1=-1$ and $c_2=1$, and the associated critical values are
 $$v_1=\frac{\ml}{2(\ka -1)}\quad{\rm and}
 \quad v_2=\frac{\ml}{2(\ka +1)}~.$$
 Alternatively we can solve for the two parameters as functions of 
 the critical values, with
 \begin{equation}\label{DM2}
\ml= \frac{4v_1v_2}{v_1-v_2}\quad{\rm and}\quad
 \ka=\frac{v_1+v_2}{v_1-v_2}~.\end{equation}

\noindent This may seem ideal for the Thurston algorithm. However in practice it
seems to give very distorted pictures, and the poles at
$x=-\ka\pm\sqrt{\ka^2-1}$ are awkward. Furthermore, there can be
a drastic transition if we deform the parameters. If  $\ka$ passes
through $\pm 1$, one critical value will pass through the point
at infinity, and two poles will appear or disappear. In fact for any normal
form that we choose, it may seem that the infinite point
will cause trouble for some maps of interest.

However there is an easy way to avoid this problem. The space
$\Rhat\cong\bP^1(\R)$ can be identified with the quotient $\R/\Z$, identifying
each $t\in\R/\Z$ with $\tan(\pi t)\in\Rhat$, or with
$\big(\sin(\pi t):\cos(\pi t)\big)$ in the real projective line. 
Hence the universal covering space of $\Rhat$ can be identified with
the real line. Any real quadratic
map $f:\Rhat\to\Rhat$ (always of degree zero)
lifts to a periodic map $\wf:\R\to\R$ with $\wf(t+1)=\wf(t)$. Furthermore, $f$
restricted to a neighborhood of $f(\Rhat)$ is real analytically conjugate
to $\wf$ restricted to a corresponding neighborhood of $\wf(\R)$.

Another way of thinking of this is the following. Since the circle $\Rhat$ is
canonically isomorphic to $\R/\Z$, the graph of $f$ can be thought of
as a subset of the torus $(\R/\Z)\times(\R/\Z)$. This torus 
is conveniently represented by a square  $[0,1]\times[0,1]$ with
opposite edges identified. Compare \autoref{f5}  (where the coordinate
of $\R$ has been translated so that $\wf(\R)\subset[0,1]$).

\begin{figure}[!htb]
  \centerline{\includegraphics[width=2.5in]{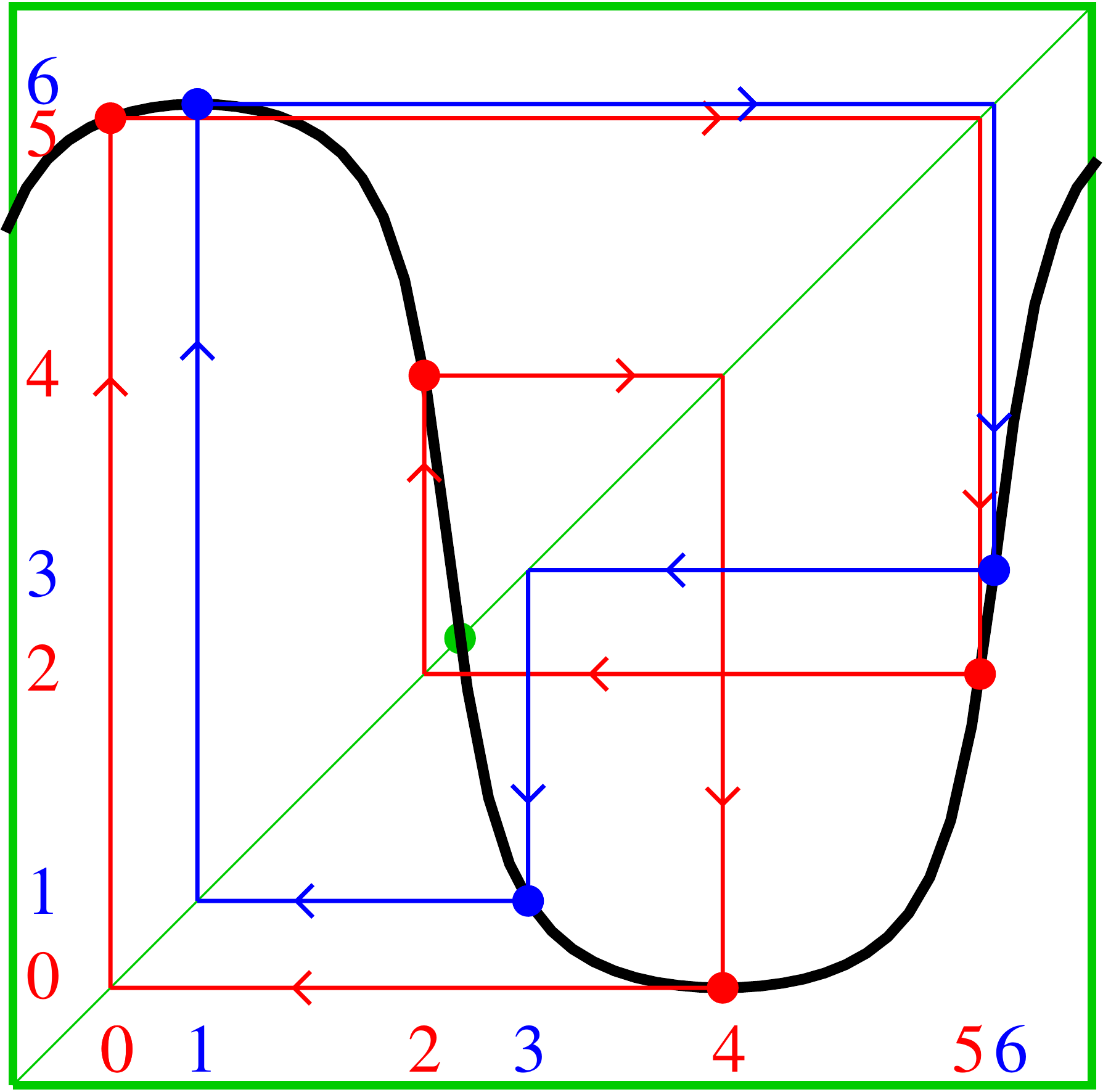}}\ssk

  \caption{\label{f5}  The Wittner map of \autoref{f1}, lifted to the
    universal covering space of $\Rhat$. Here the heavy green lines bound a
    fundamental domain. Thus the corresponding graph in the torus
    $\Rhat\times\Rhat$  can be obtained from this picture simply by
    identifying opposite green edges. Note that the product
     $f(\Rhat)\times f(\Rhat)$ corresponds to a square which is
 properly contained in this fundamental domain.}
\end{figure}
\bsk

Thus if we use the Epstein  normal form lifted to the universal 
covering space, then we will have a unique normal form such that
the graph will deform smoothly as we change the parameters. 
Note that the fixed point is half way between the two 
critical points, either in Epstein normal form or in lifted form.
\bsk

\begin{lem}\label{l-liftf}
  For the map $f$ of $\autoref{EDM}$, the corresponding lifted
map $F$ is given by
  $$  t \mapsto  F(t)~=~\frac{1}{\pi} \Arg (z)~,$$
  where
  $z = (1+ \kappa \sin(2\pi t)) + i(\mu \sin(2\pi t)/2) \,\in\, {\mathbb C}$,
    and  $\Arg (z)$ is the branch of the argument\footnote{%
   $\Arg(x+\mathrm{i}y)$ can also be expressed as ~\texttt{atan2(y,x)}~ in several
  computer languages.} with  $-\pi < \Arg (z) < \pi$~.
 \end{lem}
  \ssk
  
\begin{proof}
  Let $x = \sin(\pi t)/\cos( \pi t)\in \Rhat$  with
  $t \in {\mathbb R}/{\mathbb Z}$.  After replacing $x$ in \autoref{EDM}, a brief
  computation shows that,

  \begin{equation} \label{e-liftf}
    f(x)~ =~ \frac{\mu \sin(2\pi t)/2}{1+\kappa \sin(2 \pi t)}~.
  \end{equation}

  Notice that the right hand side of \eqref{e-liftf} is the slope of the line
  from  $0$ to $z$ in~${\mathbb C}$.
\end{proof}
\ssk

\begin{rem}
  There seems to be a problem since the function  $z \mapsto \Arg(z)$ has a jump
  discontinuity on the negative real axis. However $z$ is never negative real,
  since $\mu \neq 0$ and since if $\sin(2\pi t)=0$ then $z=1$.
  \ssk
  
  The lifted map of \autoref{l-liftf} will be used in the implementation of the
  algorithm described in the next section.
\end{rem}
\ssk

As an example, putting the Wittner map of \autoref{f1}  into lifted
normal form we obtain \autoref{f5}.
As another example, \autoref{f6} shows the lifted normal form for the map
$$ f(x)~=~\frac{x^2-1}{x^2+1}~.$$ In this case
we have $x_0=\infty,~x_1=-1,~x_2=0$, and $x_3=1$,
with combinatorics $\( 3,\, 2,\,1,\,2\)$. The mapping pattern is
$$\xymatrix{\du{x_0} \ar@{|->}[r] & x_3 \ar@{|->}[r] & \du{x_2} \ar@{<->}[r]& x_1~}.$$

\begin{figure}[!htb]
  \centerline{%
    \includegraphics[height=1.6in]{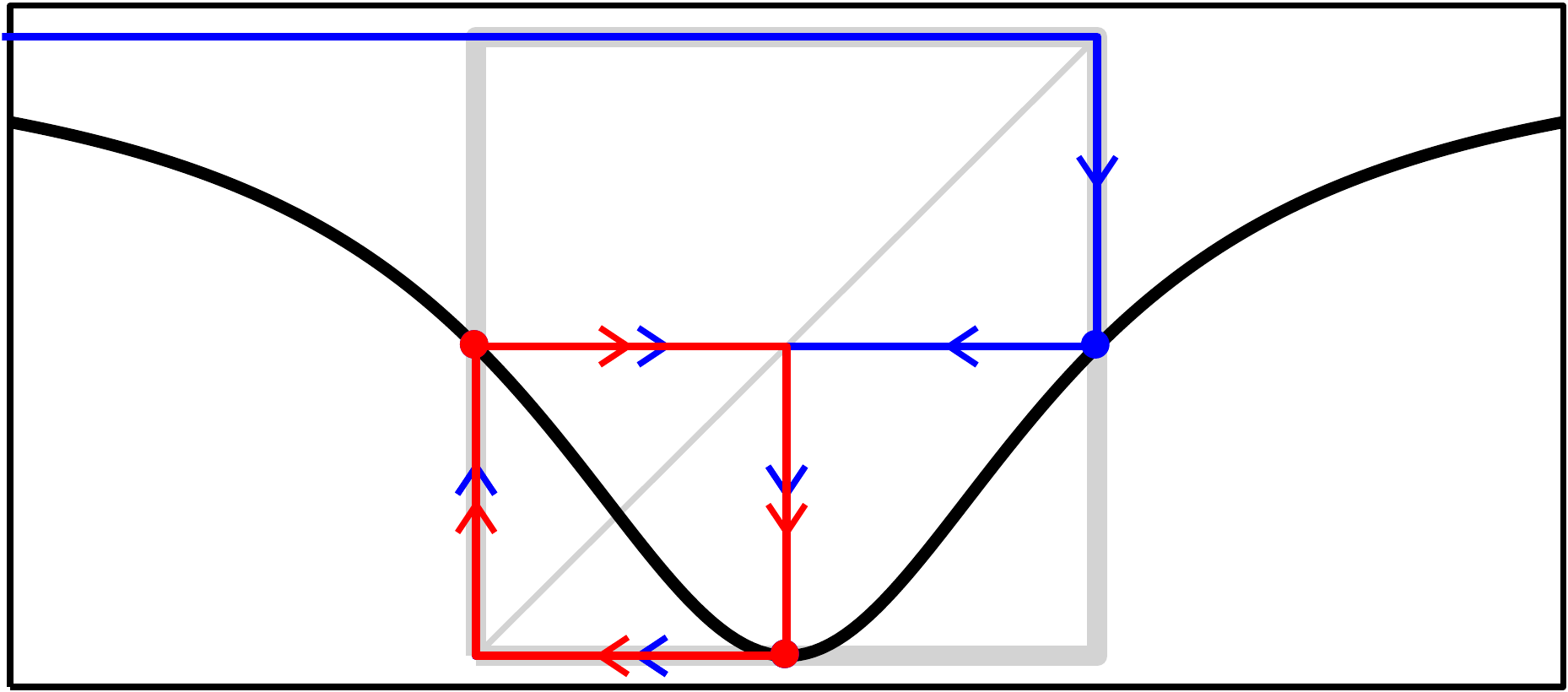} \qquad
    \includegraphics[height=1.6in]{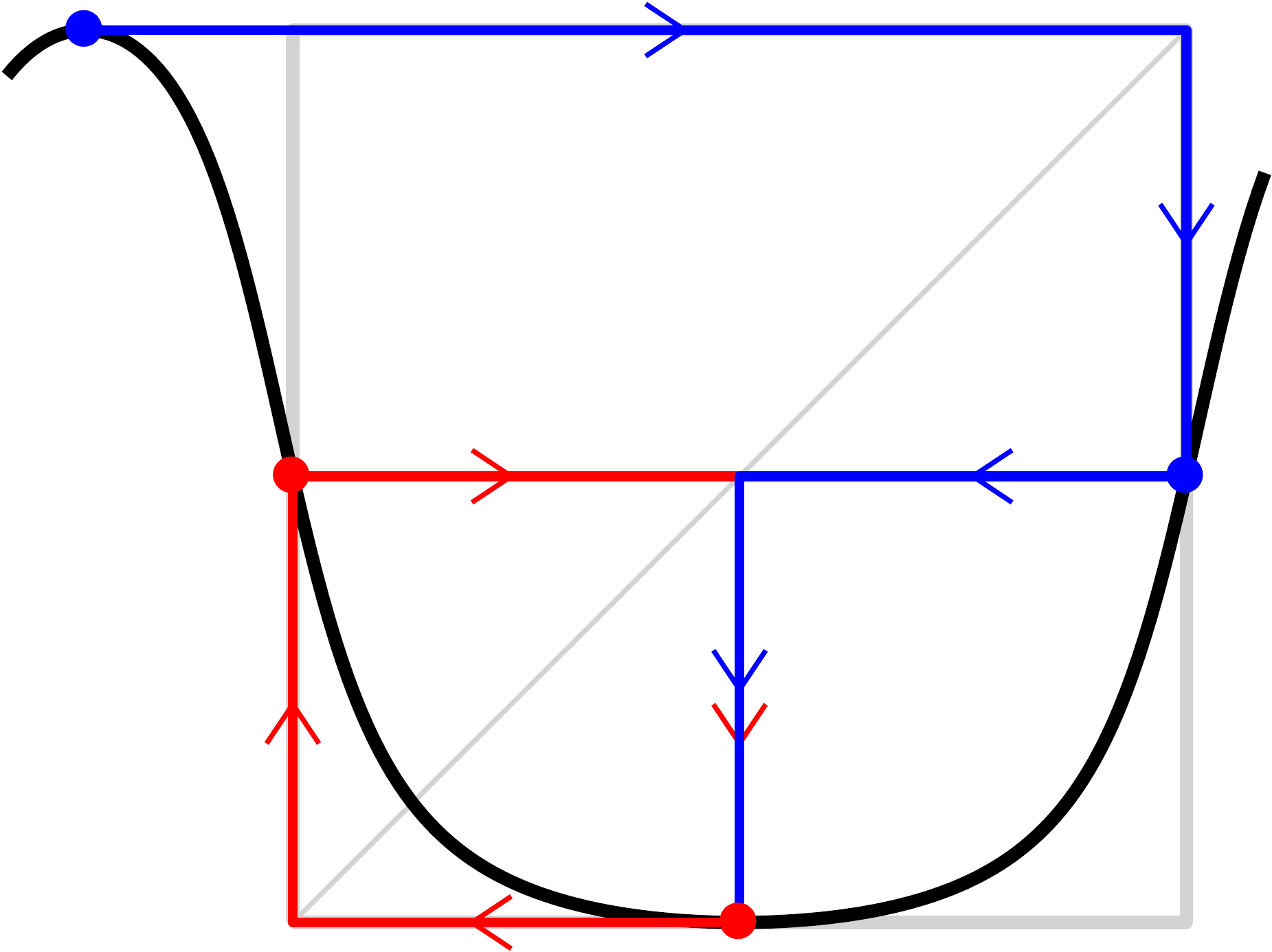}}
  \caption{ \label{f6} 
    The map $x\mapsto (x^2-1)/(x^2+1)$.  On the left is shown the rational
    map for $-2.5 < x < 2.5$ along with the forward orbits of the critical
    points $\infty$~(blue) and $0$~(red).  On the right is shown the map
    lifted to the universal covering line.  In both cases, $f(\Rhat)\times
    f(\Rhat)$ is shown as a gray box.  Note that the critical point $x_0$ is
    not within $f(\Rhat)$.}
  \end{figure}

\section{The Algorithm}\label{s4}

For any real quadratic  map,
the circle $\Rhat$ is divided by the two real critical points
into an ``increasing'' (or orientation preserving) half-circle
and a ``decreasing'' half-circle. For $x$ in the interior
of $f(\Rhat)$, there is one branch $f_+^{-1}$ of $f^{-1}$
taking values in the increasing half-circle,
and one branch $f_-^{-1}$ taking values in the decreasing
half-circle. Note that these two branches of $f^{-1}$ coincide
at the critical values, which map to critical points.

Suppose that some admissible combinatorics $\vecm$ has been specified 
(see \autoref{D-admissible}). Let $m_{j_1}$ be the smallest
$m_j$ and let $m_{j_2}$ be the largest one. The two ``critical''
indices $j_1$ and $j_2$ divide $\{0,\ldots,n\}$ into
two or three laps, each of which is either
increasing or decreasing. 
We will always assume that $n\ge 2$.\msk

\subsection*{The basic construction.}\label{ss-basic}

Let $X_n\subset\Rhat^{n+1}$ be the space consisting of all $(n+1)$-tuples
$\vecx=(x_0,\cdots,x_n)$ of distinct points of $\Rhat$ which are in positive
cyclic order. Let $f$ be a quadratic rational map such that $f(\Rhat)$ is
precisely the smallest interval containing
all of the $x_{m_j}$. In other words, $f(\Rhat)$ is the
interval consisting of all points $x$ which lie between the critical values
$x_{m_{j_1}}$ and $x_{m_{j_2}}$ in cyclic order. Then the
\textbf{\textit{pullback}} $T_f(\vecx)=\vecy$ is defined by setting
$$ y_j ~=~ f_\pm^{-1}(x_{m_j})~,$$
using either $f_+^{-1}$ or $f_-^{-1}$ according as $m_j$ is in an increasing
or decreasing lap. (In the case where $j$ is critical index, it doesn't
matter which branch we choose.)  It is not hard to check that the image
points $y_j$ are always in positive cyclic order.\msk

\begin{theo} \label{t-pullb}
   Every admissible combinatorics $\vecm$  
   gives rise to a well defined pullback map $~~ T:X_n/G\to X_n/G$.
\end{theo}

\begin{proof}
We will prove first that this construction does not depend on the choice of $f$.
Any other quadratic rational map satisfying the same conditions can be written
as a composition $f\circ L$ where $L$ belongs to the group $ G={\rm PSL}_2(\R)$ of
orientation preserving fractional linear transformations. Then
$(f\circ L)^{-1}=L^{-1}\circ f^{-1}$, and it follows that $\vecy$
will be replaced by $L^{-1}(\vecy)$, using the diagonal action of
$G$ on $X_n$.\msk

On the other hand, if we replace each 
 $\vecx$ by $L(\vecx)$, then the image
$\vecy$ will not change. In fact, we can simply 
 replace $f$ by $L\circ f$,
so that $L\circ f$ will map $\Rhat$ to  $L\big(f(\Rhat)\big)$. 
Thus each $y_j$ will be replaced by the appropriate branch of
$$ x_j~\mapsto~(L\circ f)^{-1}\big(L(x_j)\big)=f^{-1}(x_j)~,$$
which is just $y_j$ itself. \end{proof} \msk
  \bsk

  Note that this quotient space $X_n/G$  is diffeomorphic to a convex 
  open subset\footnote{Caution: The precise shape of this convex set depends
    on the following rather arbitrary choices, 
    and does not have any invariant meaning. In particular, the boundary
    of this set does not have any invariant meaning.}
 of $\R^{n-2}$.  In fact for each $\vec x$,
we can choose a uniquely defined  group element $L$ so that
$L(\vecx)=\vecy$ satisfies
$y_{n-2}=1, ~ y_{n-1}=\infty,$ and $ y_n=0$. Then the remaining $y_j$
must satisfy $$\quad 0<y_0<y_1<\cdots< y_{n-3}< 1~.$$
Thus they form an interior point of one standard model for the $(n-2)$-simplex.
\ssk

If there is no Thurston obstruction, then in nearly every case, 
the iterated pullback converges to a unique point of $X_n/G$,
and this determines a unique conjugacy class of quadratic rational maps. In the
exceptional case, it converges to a pair of points on the (non-compact)
line of fixed points of $T\circ T$; see \autoref{s6}. In the obstructed case,
the sequence of points $T^{\circ k}(\vecx)$ always leaves every compact
subset of $X_n/G$.

\subsection*{The Lifted Pullback Map.} 

Before we begin, we must choose a convenient family of quadratic rational
maps.\footnote{One benefit of working with explicit maps 
  rather than conjugacy classes is that we obtain a well defined
  sequence of approximating rational maps as we iterate the 
  pullback construction. This will be important when we study
  obstructions.}
  As a consequence of \autoref{t-pullb}, we can choose any convenient
one, such as the 
Epstein normal form (see \autoref{EDM}), but nearly any such choice will leave
us having to deal with infinity, even though all the interesting behavior
occurs in a compact subset of $\Rhat$.  As noted in \autoref{s3}, it will be
most convenient to lift to the universal covering space, that is, to work with
the family\footnote{%
  We have also implemented the algorithm (see \autoref{f-pullback-134310}) using
  the family 
  $x\mapsto (1/x -2\kappa + x)/\mu$, which has the nice property that
  $\infty$ is fixed and $0$ is its only preimage, making the poles easy to
  deal with.
  However, the lifted family of \autoref{l-liftf} is vastly preferable for
  visualization, since all  marked points must lie inside~$(-3/4,~ 3/4)$.}
\begin{equation}\label{E-LF} F(t)~=~F_{\mu,\kappa}(t)~=~
  \frac{1}{\pi} \Arg\Big(1+\kappa\,\sin(2\pi t)  + i\,\mu\,\sin(2\pi t)/2\Big)\end{equation}
 as in \autoref{l-liftf}. 
This is a periodic function of period 1, with a fixed point of
 multiplier $F'(0)=\mu\ne 0$ at $t=0$, with critical points at
 $t\equiv \pm 1/4 ~~({\rm mod}~\Z)$, and with image $F(\R)\subset(-3/4,~ 3/4)$
 bounded by the two critical values and of length strictly less than one.
 Note that $\mu$ and $\kappa$ can
be determined uniquely from the critical values of $F$. 
If $F(-1/4)=v_1$ and $F(1/4)=v_2$, then by \autoref{DM2}: 
\[
  \mu    = \frac{4 \tan (\pi v_1) \, \tan (\pi v_2)}{\tan(\pi v_1) -
    \tan(\pi v_2)}
  \quad\textrm{and}\quad
  \kappa = \frac{  \tan(\pi v_1)  + \tan(\pi v_2)}{\tan(\pi v_1) - \tan(\pi v_2)}
  \, .
\]

\begin{coro} For any admissible combinatorics $\vecm$, there is a  well-defined
  pullback map acting on elements of the lifted family $\eqref{E-LF}$.
\end{coro}\ssk

\begin{proof}
Given a map $F$ in the lifted family, we can obtain a map $f$ in Epstein form by
conjugating with $t\mapsto x= \tan(\pi t) \in \Rhat$, which is biholomorphic
for $t$ in a  complex  neighborhood of $(-1,1)$.
Hence the pullback acting on an element of the 
lifted family corresponds exactly to the pullback acting on $X_n/G$. 
\end{proof} \ssk

\begin{rem}\label{r-liftedCombi} 
  It is essential for our argument that the fixed point at $t=0$
  must lie between the two critical points within the interval $f(\Rhat)$.
  In the strictly unimodal case where there is only one critical point
  in $f(\Rhat)$, the choice described in \ref{combi-case2} of
  \autoref{s2} is needed in order to ensure this property.
As an explicit example, the map with combinatorics $\(1, 2, 1, 0\)$
(see \autoref{f-aC1}-left) could also be described as the family with
combinatorics $\(1, 2,3,2\)$. However, this would not be consistent
 with our conventions in \autoref{D-admissible}.
 The first has the mapping pattern
$\du{x_3}\mapsto x_0 \mapsto \du{x_1}\leftrightarrow x_2$, and the
second 
$\du{x_0}\mapsto x_1 \mapsto \du{x_2}\leftrightarrow x_3$.
These are the same except for choice of labeling of the marked
points.  But the implicit ordering in the second does not follow our
convention:  both critical points are on the same side of the fixed
point (which must lie between the points of the period~2 cycle 
$x_2 \leftrightarrow x_3$).
This would cause our implementation to fail. 
\end{rem} \ssk

\begin{rem}
In our implementation, we only insist on the the first two admissibility
conditions, and do not require those of minimality~\ref{admiss_minimal},
expansiveness~\ref{admiss_expansive}, non-polynomial~\ref{admiss_notpoly},
or unique fixed point~\ref{admiss_onefp}.

However, for our implementation in the polynomial case, we will require
the additional condition 
that the combinatorics must be of $+-$  shape,
 with a zero in the first entry. Thus:
$\(0,\,2,\,1\)$ is the basilica, 
$\(0,\,3,\,4,\,5,\,6,\,2,\,1\)$ the period-doubled airplane, 
$\(0,\,1,\,3,\,1\)$ the Chebyshev point, and so on.
Since $\mu$ is the value of the derivative at fixed point
between the two critical points, the
corresponding polynomial can be written as $\mu\,x(1-x)$.
Of course polynomials can be dealt with by other methods. Compare
\cite{BMS}; and see \autoref{s8} below.
\end{rem}

Just as in \autoref{t-pullb}, we will define the pullback
$T=T_F$ as a map from a space of sequences to itself.

\begin{definition}\label{D-Xm}
  Let $X(\vecm)\subset\R^{n+1}$ be the space 
  of all sequences
  $\vect=(t_0,t_1,\ldots, t_n)$ which satisfy the following three conditions.
  \begin{enumerate}[label=\textbf{\alph*)},ref=(\alph*)]
  \item We must have
    $$ -3/4 ~<~t_0~<~t_1~<~\cdots~<~t_n~<~ 3/4~.$$ 
  \item If $j_-<j_+$ are the two \textbf{\textit{critical indices}},
  defined by the requirement that one of $m_{j_-}$ and $m_{j_+}$ is the
  largest $m_j$ and the other is the smallest, then
$$ t_{j_-} = -1/4\qquad{\rm and}\qquad t_{j_+} = +1/4~.$$
\item (Locating the fixed point.) If there is an index $j_-<p<j_+$
such that $m_p=p$, then we require that
$t_p=0$. Otherwise there must be a unique
pair of consecutive indices $j_-\le p'<p''\le j_+$ such that
the differences $m_{p'}-p'$ and $m_{p''}-p''$ have opposite sign.
In this case, we require that 
$$ t_{p'} ~<~ 0 ~<~ t_{p''}~.$$
\end{enumerate}
\end{definition}

In order to define the pullback map 
$T(\vect)= \stackrel{\to}{t'}$,
we must solve the equation  $F(t'_j) = t_{m_j}$, taking care to choose
between the two possible solutions. To do this, we
divide the interval $[-3/4,\,3/4]$ into three\footnote{%
     While there are only two half-circles, when working with the
     lift it is important to treat $I_1$ and $I_3$ separately, since
     there is one branch of $F^{-1}$ taking values in $I_1$ and a
     different      branch taking values in $I_3$.
     If we worked directly with a family of rational maps, 
     the situation would be further complicated by poles.}
   subintervals: 
   \begin{equation}\label{E-Ij}
     I_1=(-3/4,\, -1/4),\quad I_2=(-1/4,\, 1/4),
   ~~{\rm and}\quad I_3=(1/4,\, 3/4)~.\end{equation}
   Here 
   $I_1$ and $I_3$ are different lifts of the same half-circle, while 
   $I_2$ corresponds to the other half-circle. Then the requirement
   is that $t_j$ and $t'_j$ must belong to the same subinterval $I_k$.
   This is always uniquely possible since $F$ maps each interval $I_k$
   bijectively onto the interval $F(\R)$.

   \subsection*{Implementation.} 
We now give an outline of the steps involved in implementing the
Thurston Algorithm.  The implementation is very similar to that in
\cite{BMS}.  When necessary, will use $t_k^{[\ell]}$ to denote the
position of the $k$th marked point at the $\ell$th step, omitting this
superscript when it is irrelevant or apparent. We will also use $F_\ell$
to indicate the map at the $\ell$th step, and $w_k^{[\ell]}$ for 
$F_\ell(t_k^{[\ell]})$.

  Begin by examining the given combinatorics $\vecm$, confirming
  admissibility and adherence to the requirements of
  \autoref{r-liftedCombi}. From the combinatorics, determine the 
  indices of the critical points and the location of the fixed point
  lying between them, as described in \autoref{D-Xm}. Finally,
  choose initial  values of $t_j^{[0]}$ satisfying the required
  equalities and inequalities; and set $w_j^{[0]}=t_{m_j}^{[0]}$.
  \ssk
  
  The inductive step in the construction can now be described as follows.

  \begin{enumerate}[label=\textbf{\roman*)},ref=(\roman*),
    itemsep=0pt plus 2ex minus .2ex]
  \item \label{step_alg_increment}
     Increment $\ell$ by $1$, and determine $F_{\ell}$ from the
     critical values $w_{j_\pm}^{[\ell-1]}$.

  \item \label{step_alg_solve}
    For each $j$ other than the critical indices $j_\pm$ and
    or the fixed point index $p$ (when there is one),
    find the value of 
     $t_j^{[\ell]}$ by numerically solving the equation
     {$F_\ell(t^{[\ell]}_j) = w^{[\ell-1]}_j$,} with
     $t_j^{[\ell]}$ in the 
     same interval  $I_k$ as $t_j^{[\ell-1]}$. 

   \item If the results are  close enough, then return
      $F_\ell$ and $\vect^{[\ell]}$.
      Otherwise, 
      repeat the inductive step starting from \ref{step_alg_increment}.
\end{enumerate}

\begin{rem}
  It might seem more natural to use the smaller intervals
 $I_1=(-\tfrac{1}{2},\,-\tfrac{1}{4})$ and
   $I_3=(\tfrac{1}{4},\, \tfrac{1}{2})$ in \autoref{E-Ij}.
   However, in some cases doing this leads to problems; and
we found it more straightforward to just use the larger intervals.
\end{rem} 

In most cases, the calculations need to be done with at least double
precision floating point arithmetic, and often require
20 or 30~decimal digits
of working precision to get reasonably close to the limit.

\bsk
We now explicitly discuss the pullback process for a specific example.
Shown in \autoref{F-algExamp} are several steps for combinatorics
$\(1,\,2,\,5,\,6,\,4,\,2,\,1,\,0\)$, with mapping pattern
\begin{equation*}
  \xymatrix @R=2ex{
    \du{t_7}\ar@{|->}[r]&t_0\ar@{|->}[r]&t_1\ar[d]&t_6\ar@{|->}[l]&\du{t_3}\ar@{|->}[l]
    &\qquad t_4\mapstoself \\
      & &t_2\ar@{<->}[r]&t_5}
\end{equation*}
This combinatorics is not minimal,
since it includes the fixed point $t_4$, not part of any critical orbit.
After omitting the fixed point and renumbering, we would obtain
$\(1,\,2,\,4,\,5,\,2,\,1,\,0\)$. 
\ssk

\begin{figure}[!h]
\newcommand{\figHT}{.2\textwidth}
\begin{tabular*}{\textwidth}{c @{\extracolsep{\fill}} c c}
  \includegraphics[height=\figHT]{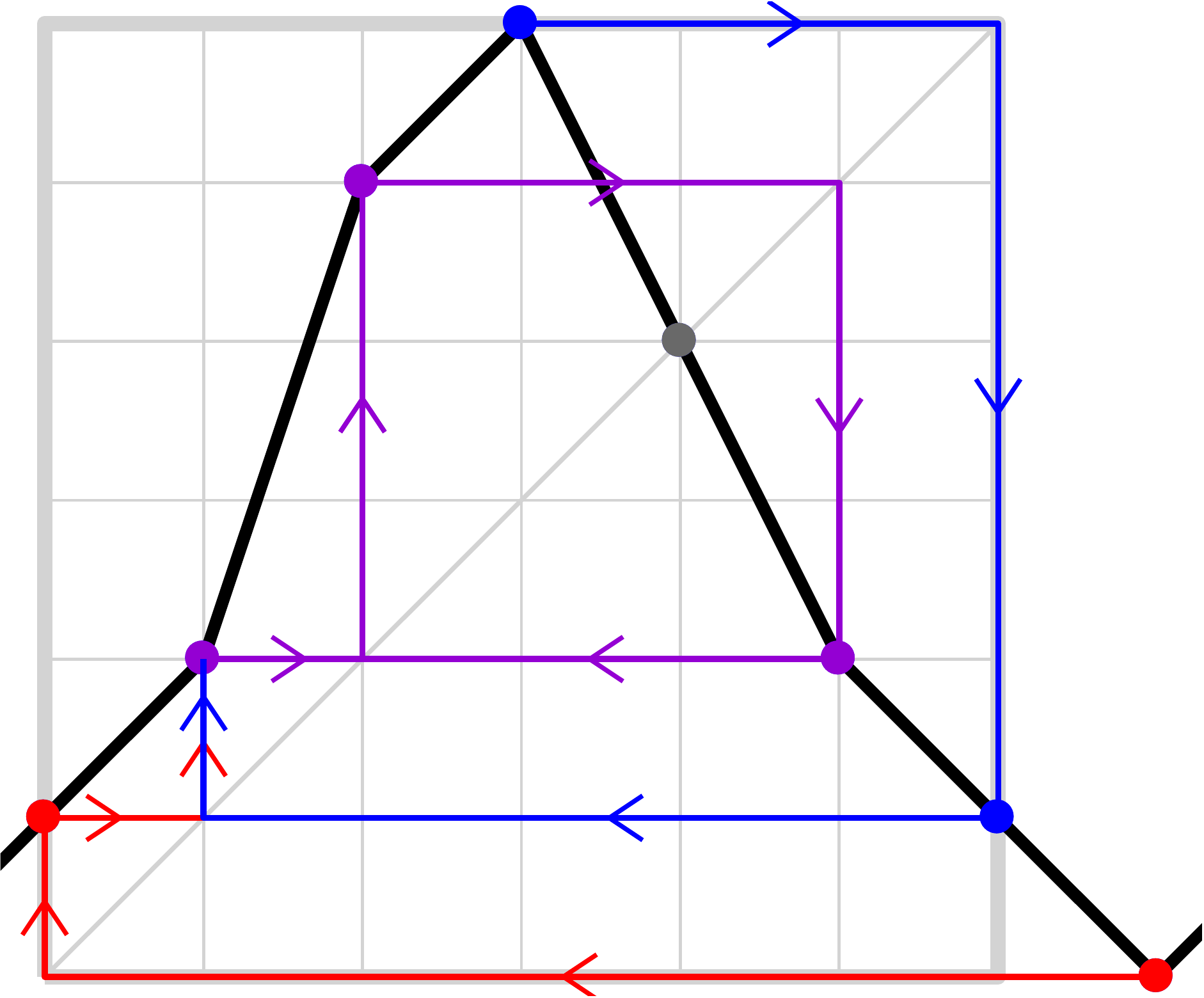} & 
  \includegraphics[height=\figHT]{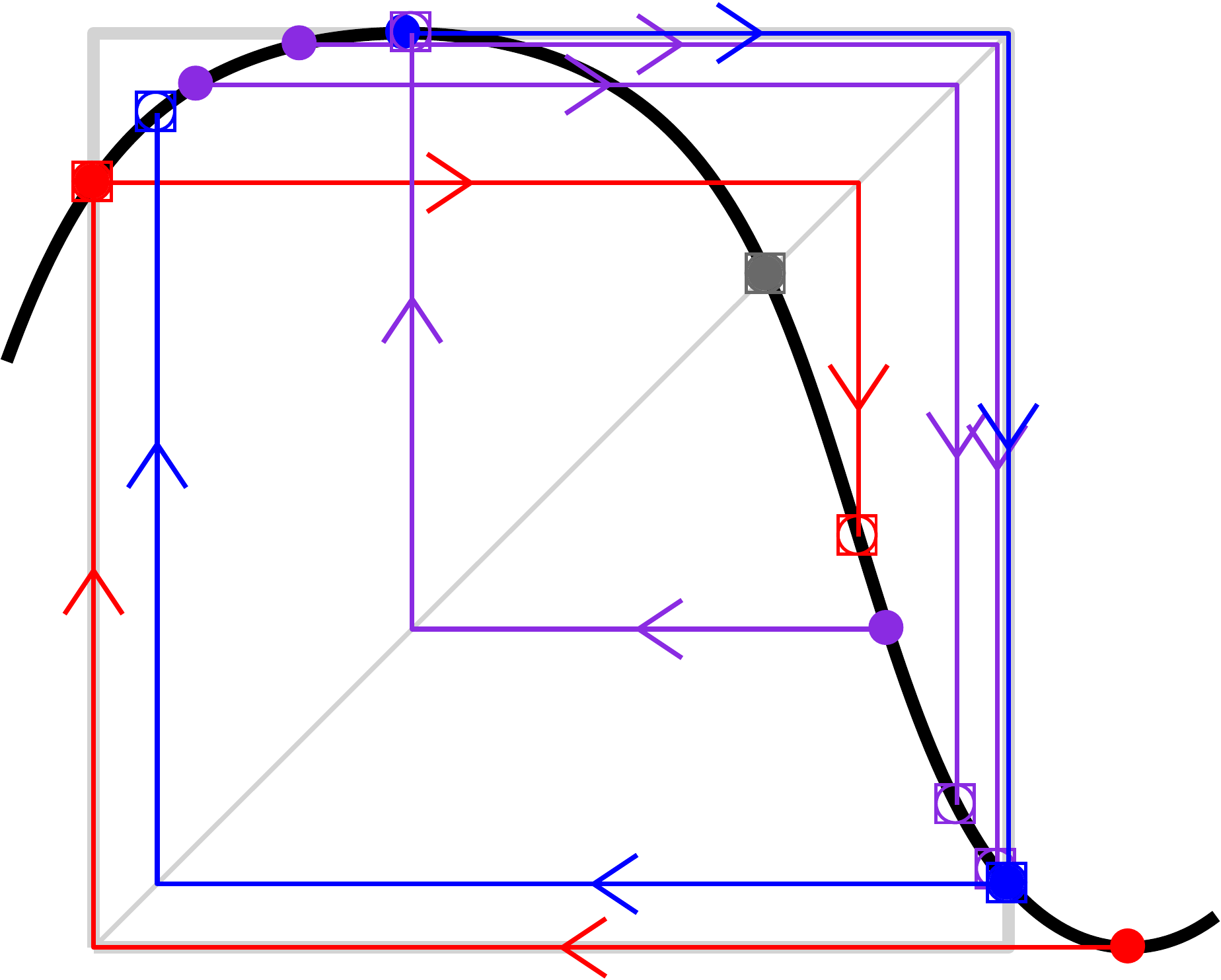} &
  \includegraphics[height=\figHT]{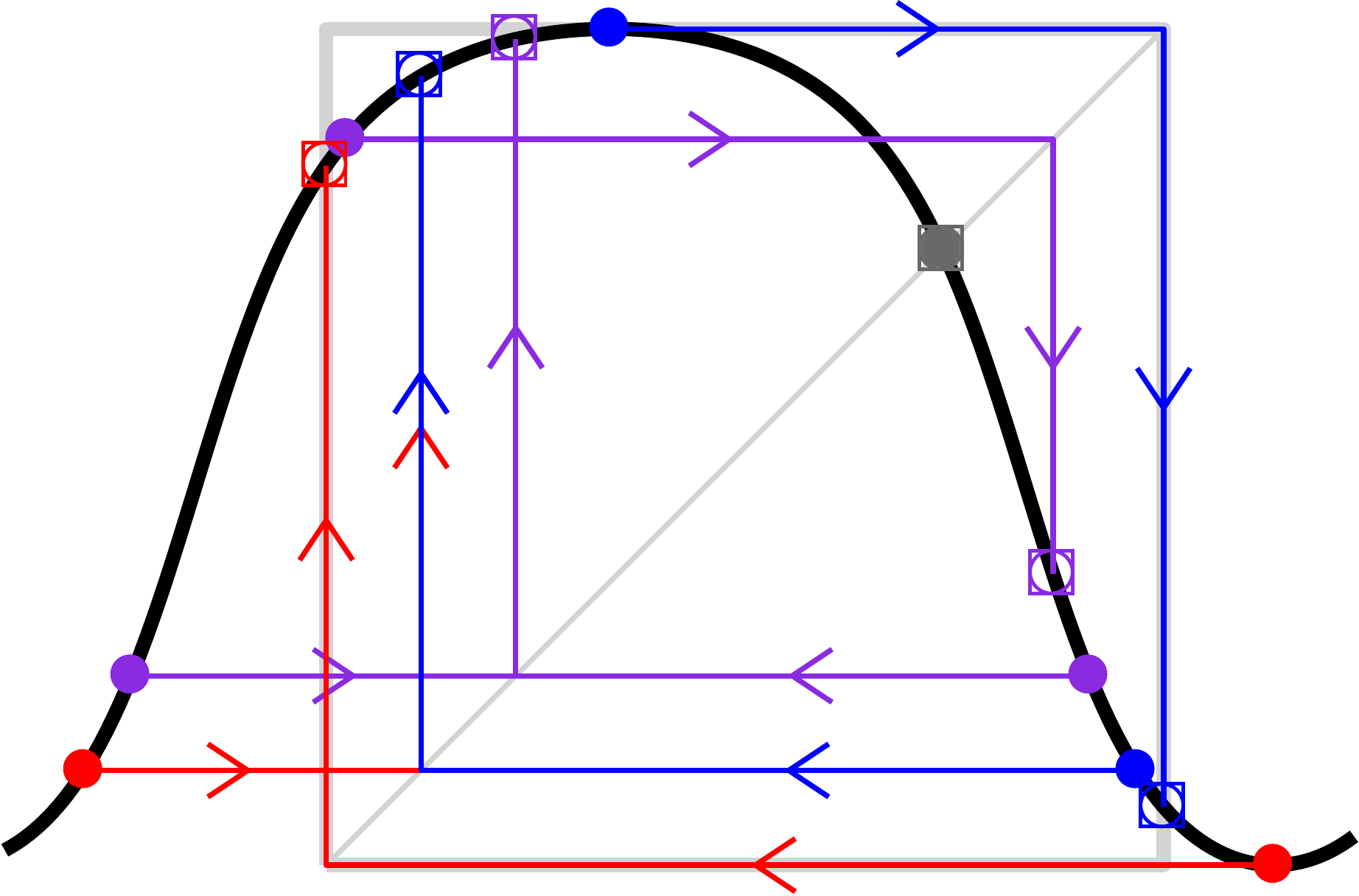} \\[-1ex]
  {\footnotesize PL-model} &
  {\footnotesize step 0: $\mu\approx-2.168,\,\kappa\approx-.8778$} &
  {\footnotesize step 1: $\mu\approx-2.168,\,\kappa\approx-.8778$} \\[2ex]
  \includegraphics[height=\figHT]{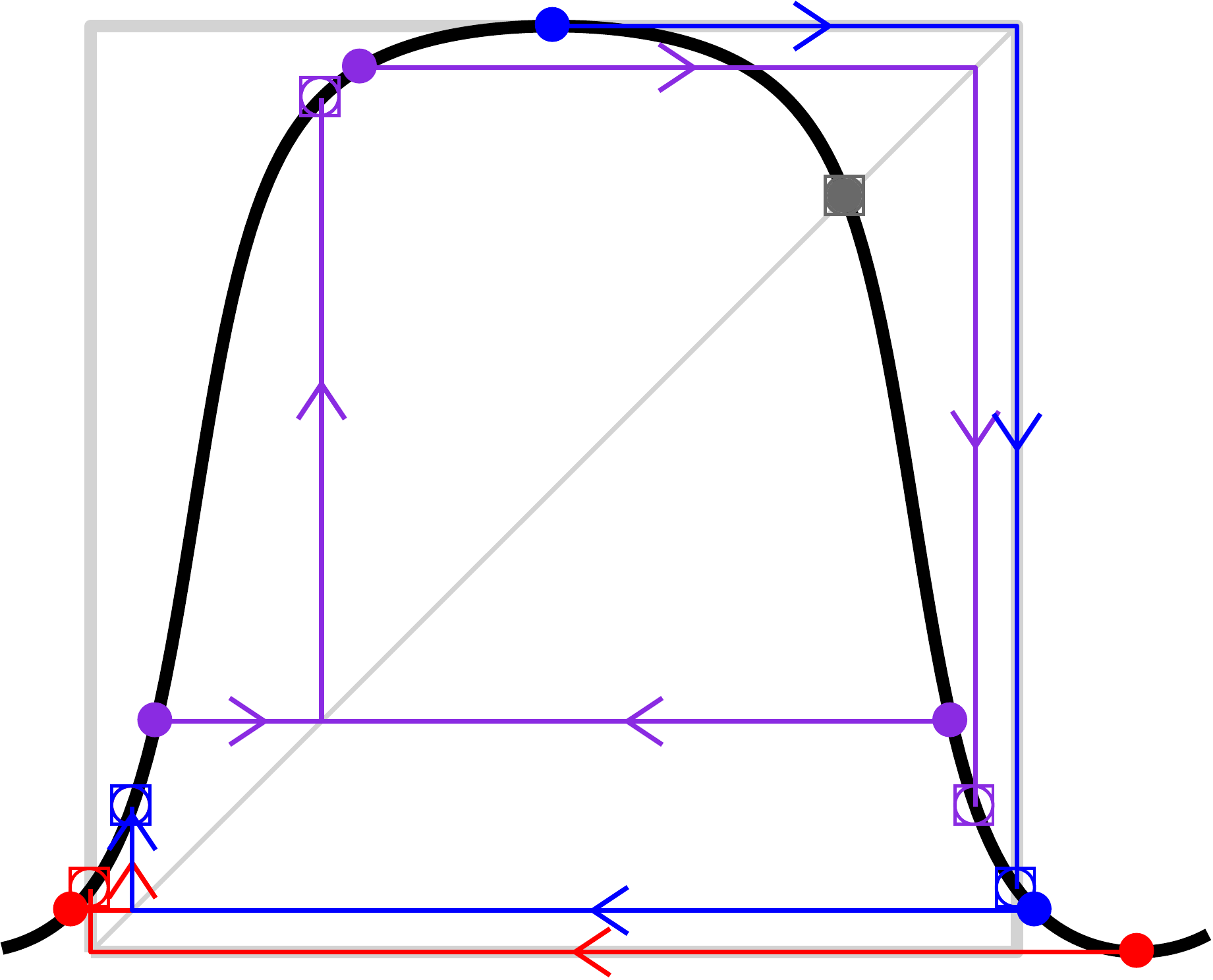} & 
  \includegraphics[height=\figHT]{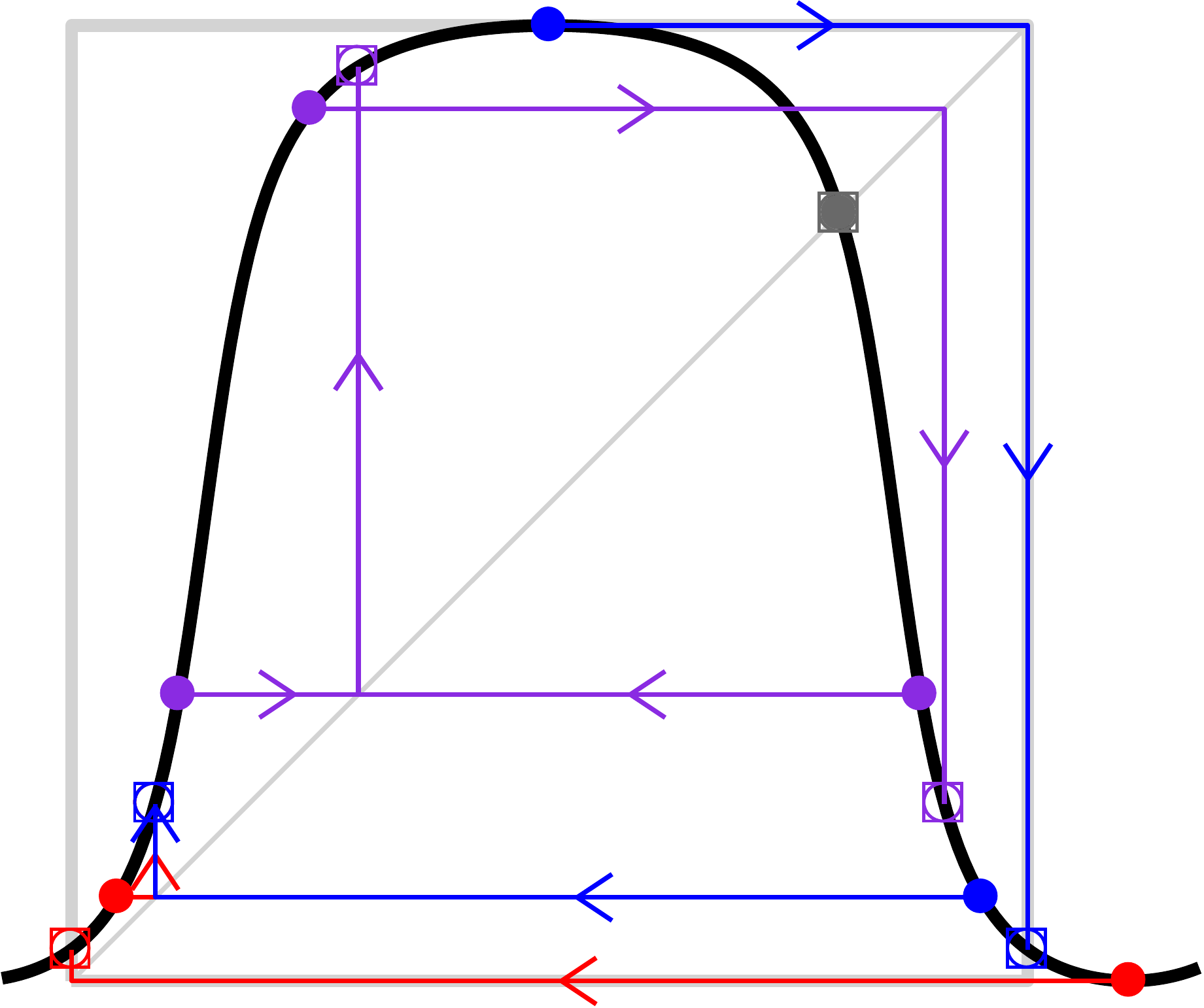} &
  \includegraphics[height=\figHT]{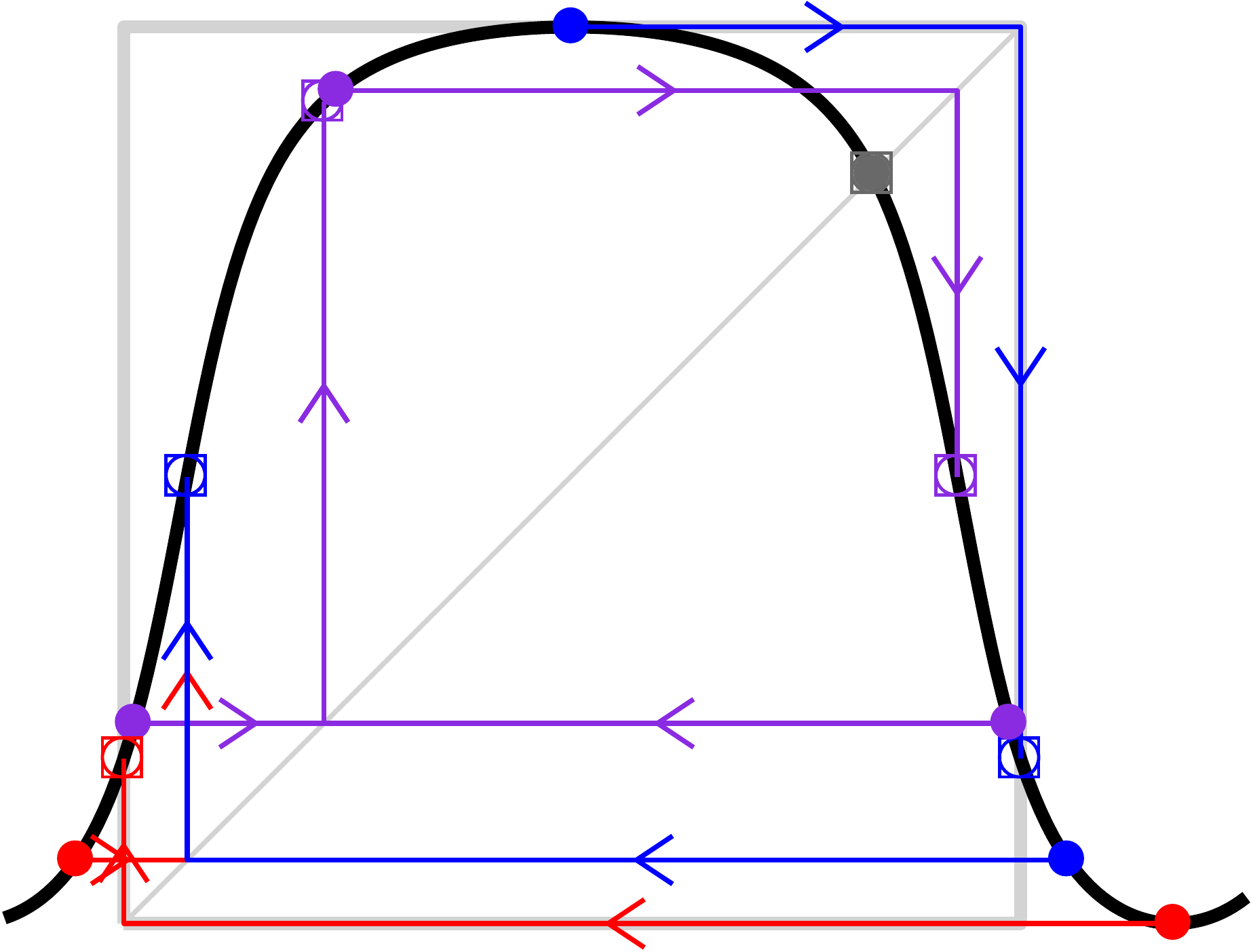} \\[-1ex]
  {\footnotesize step 2: $\mu\approx-2.625,\,\kappa\approx-1.645$} &
  {\footnotesize step 3: $\mu\approx-3.255,\,\kappa\approx-1.649$} &
  {\footnotesize step 4: $\mu\approx-1.939,\,\kappa\approx-1.910$} \\[2ex]
  \includegraphics[height=\figHT]{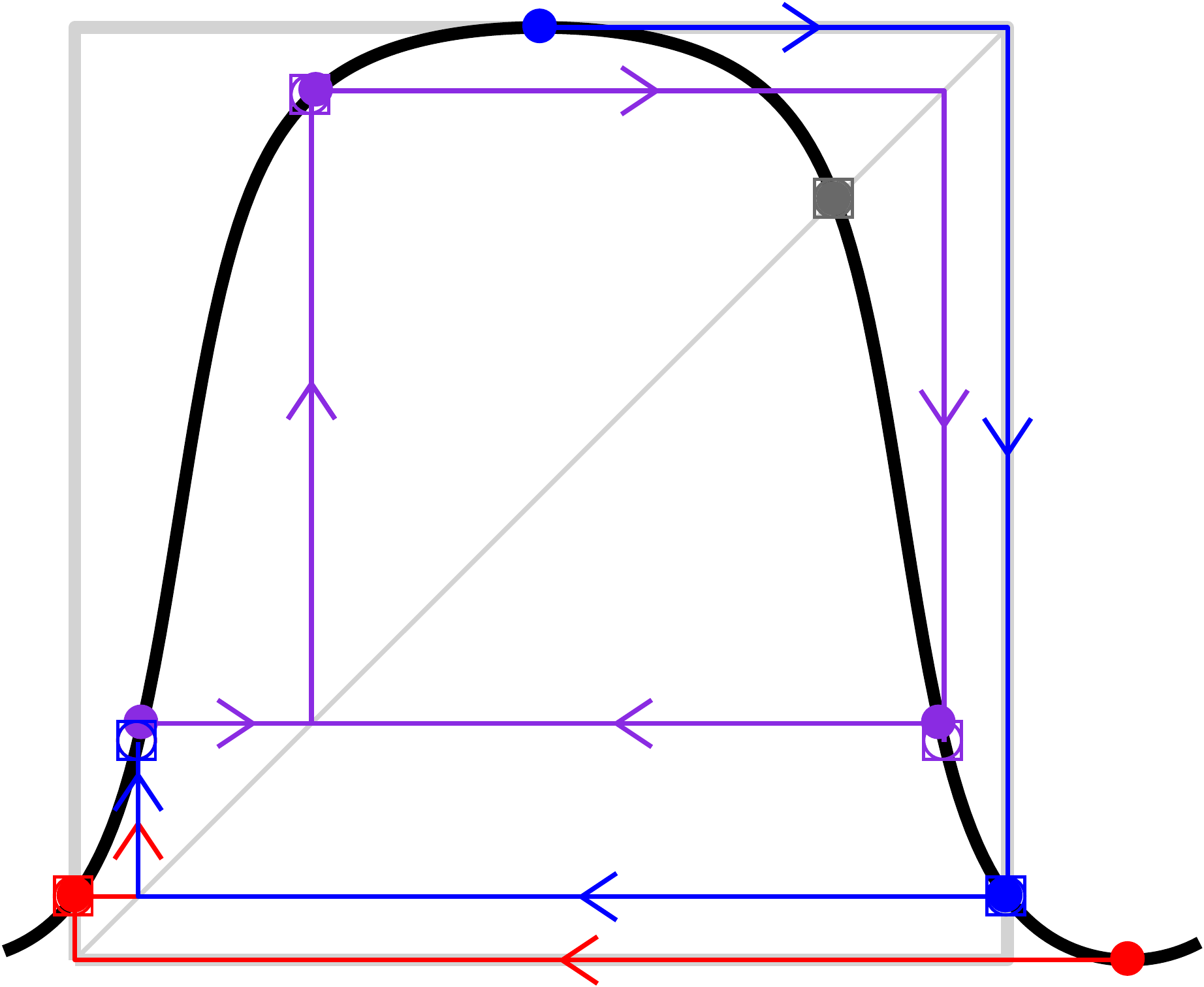} & 
  \includegraphics[height=\figHT]{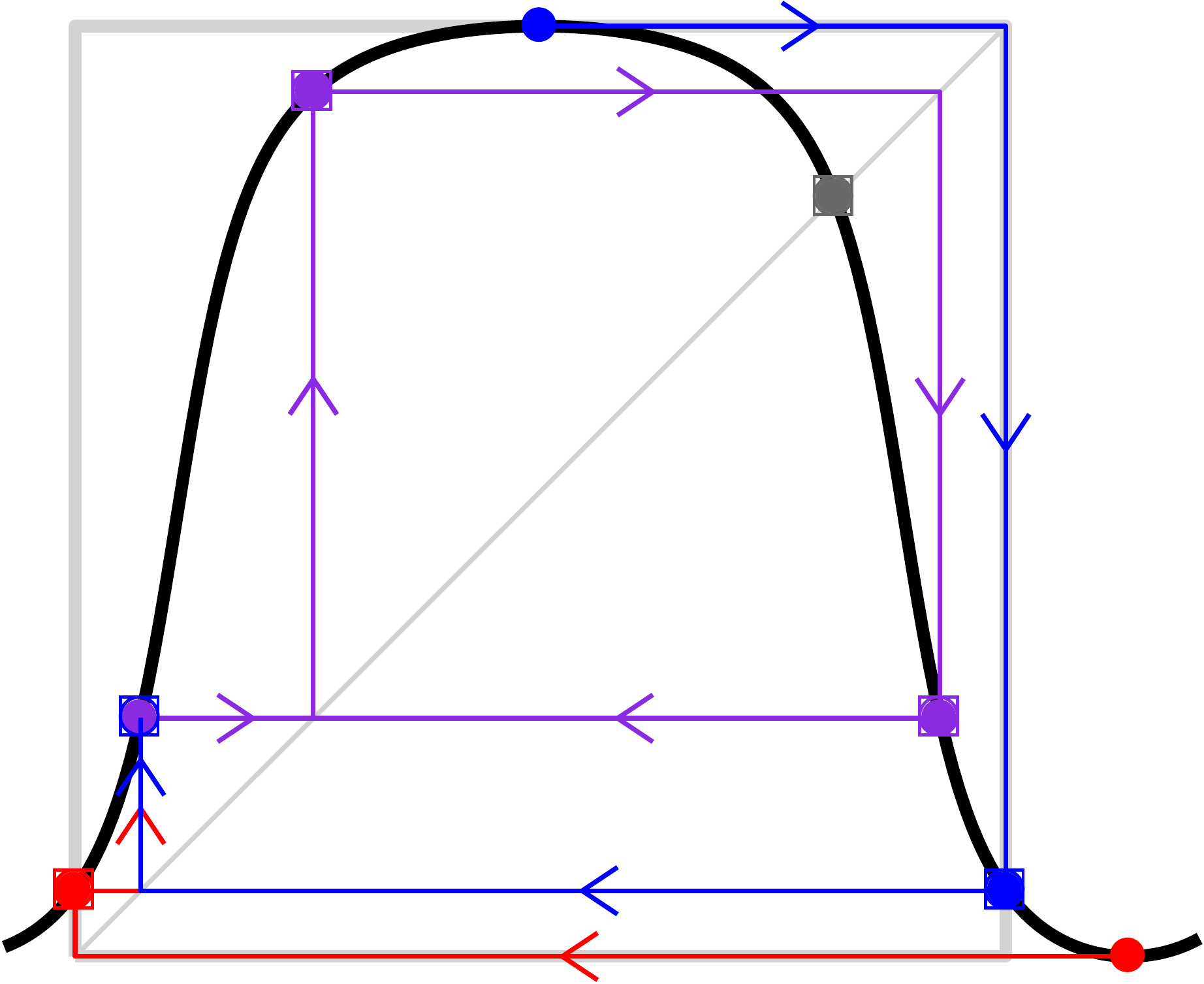} &
  \includegraphics[height=\figHT]{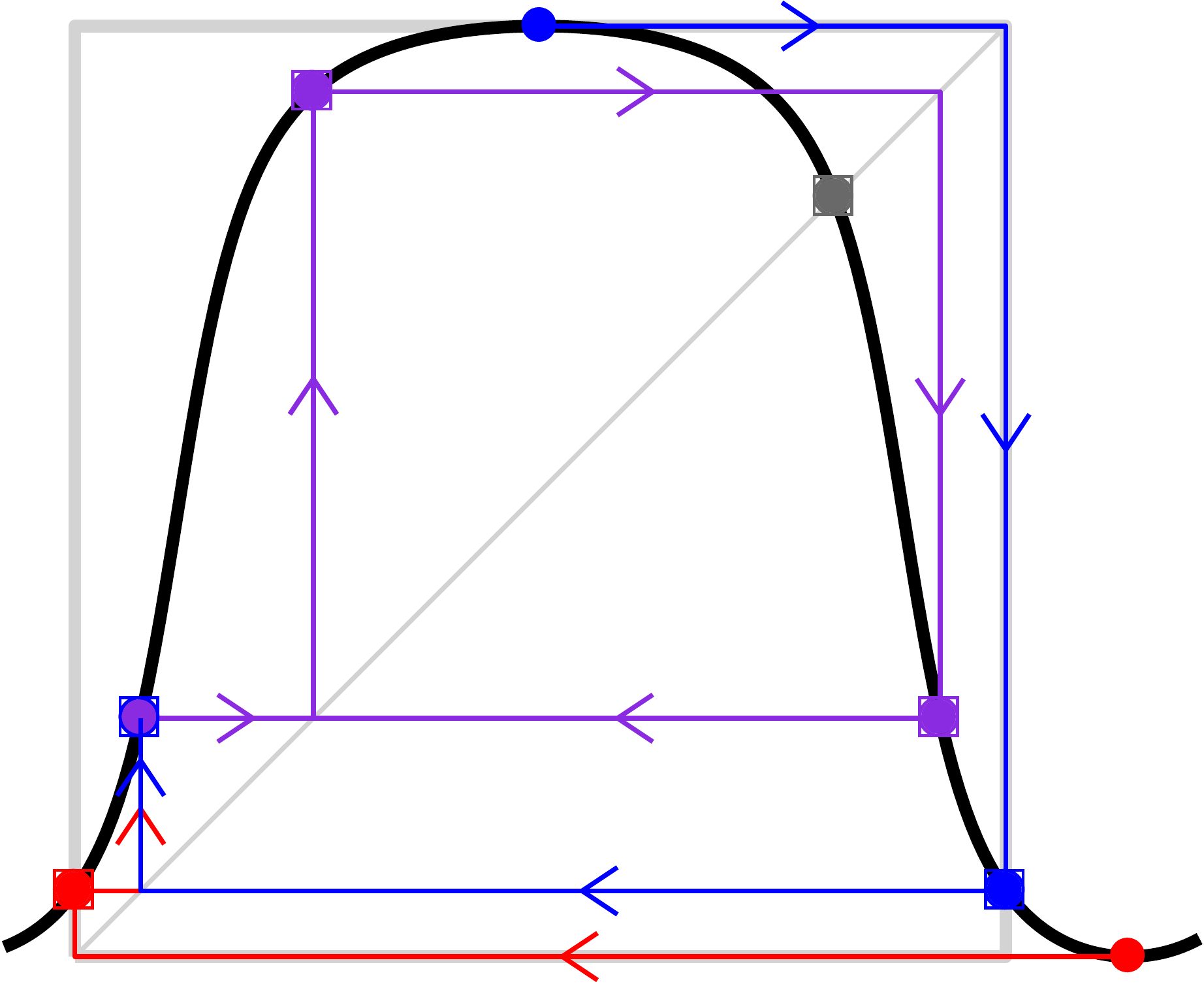} \\[-1ex]
  {\footnotesize step 10: $\mu\approx-2.635,\,\kappa\approx-1.654$} &
  {\footnotesize step 20: $\mu\approx-2.592,\,\kappa\approx-1.637$} &
  {\footnotesize step 40: $\mu\approx-2.594,\,\kappa\approx-1.638$} \\
\end{tabular*}
  \caption{\label{F-algExamp}
    The PL-model and several steps of the pullback process 
    for the combinatorics $\(1,\,2,\,5,\,6,\,4,\,2,\,1,\,0\)$.
    Each of the marked points $t_j^{[\ell]}$ is indicated by a filled disk
    (either in blue, red, or purple) with attached arrows connecting to its
    image under $F_\ell$ (indicated by an open box).
    The critical point $t_3=-1/4$ and its forward images are drawn in blue,
    the critical point $t_6=1/4$ and its images are in red, and in 
    purple are those which are a forward image of both critical points.  The
    fixed  point $t_4$ is in gray --- it is not part of a critical
    orbit. (This convention will be used in most of the figures henceforth).
    Observe (in the bottom row) that by step~20, the graph has converged
    visually, but $\mu$ and $\kappa$ still change in the third decimal
    place. 
    }
\end{figure}

In the set-up phase, we choose initial points $t_0^{[0]}$ through
$t_7^{[0]}$, with $t_3=-1/4$, $t_4=0$, and $t_7=1/4$. The critical values
$t_0$ and $t_6$ determine the map $F_0$ as in \autoref{E-LF}.
This is shown in \autoref{F-algExamp}, top-center.
Note that we have quite a lot of freedom in choosing the $t_j$, only subject
to the restrictions of \autoref{D-Xm}. In \autoref{F-algExamp}, we have
chosen the initial $t_j$ (other than the critical and fixed points) to be
equally spaced within\footnote{The other intervals have no $t_j$ for this
  combinatorics.} the intervals $(-1/2, -1/4)$ and $(0,1/4)$.

The values of $t_j^{[1]}$ are then computed by solving
$F_0(t_j^{[1]})=t_{m_j}^{[0]}$.  Since the map $F_0$ was determined by
$t_0^{[0]}$ and $t_6^{[0]}$, we will have $F_1=F_0$, but the values of $t_j$
will change, in some cases dramatically.

We repeat the process again to obtain $F_2$ and $t_j^{[2]}$, and the
dynamical behavior becomes roughly apparent (see the middle row of
\autoref{F-algExamp}), although none of the marked points actually map to
each other.  It takes 28~steps to get $\mu$ and $\kappa$ correct to three
decimal places, and after 54~pullbacks $\mu$ and $\kappa$ are good to eight
places.

\section{The Unique Exceptional Case $\(1,\,3,\,4,\,3,\,1,\,0\)$}\label{s6}

In this section we will prove the following.

\begin{theo}\label{T-notexc}
  Any critically finite real quadratic map $f$ is uniquely determined up to
  conjugacy by its combinatorics. In fact the conjugacy class
  $\langle f\rangle$
  is the unique fixed point of the associated pull-back transformation.
  For any minimal 
  combinatorics other than $\(1,\,3,\,4,\,3,\,1,\,0\)$ or its
  image under orientation reversal, this fixed point is a global attractor,
  so that the iterated   pull-back transformation will always converge to a
  map in the required conjugacy class.
\end{theo}\msk

On the other hand:

\begin{prop}\label{P-exc}
  In the exceptional case $\vecm = \(1,\,3,\,4,\,3,\,1,\,0\)$, 
 the iterated pull-back $f\mapsto T(f)$ does not usually converge
  to a single map. Instead, from a generic starting point, it
converges to a pair of maps $(f,g)$ for which $T(f)=g$ and
$T(g)=f$.  In fact the composition $T\circ T$  has a global attractor
consisting of a one-parameter family of conjugacy classes, and the fixed
point of $T$ is just  one point in this one-parameter family.
\end{prop}
\msk

\begin{figure}[!htb] 
  \centerline{%
    \includegraphics[height=1.75in]{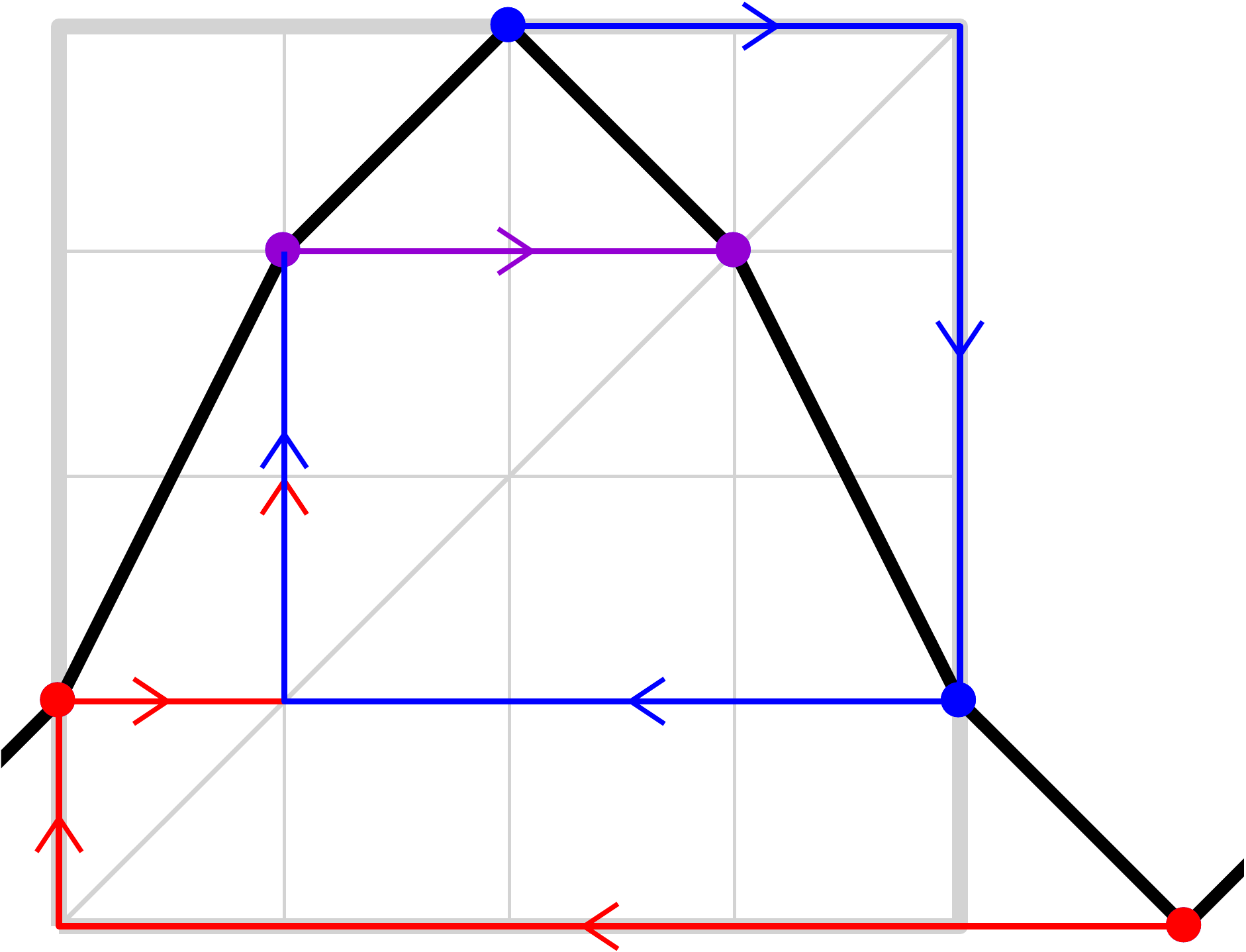}\qquad\qquad
    \includegraphics[height=1.75in]{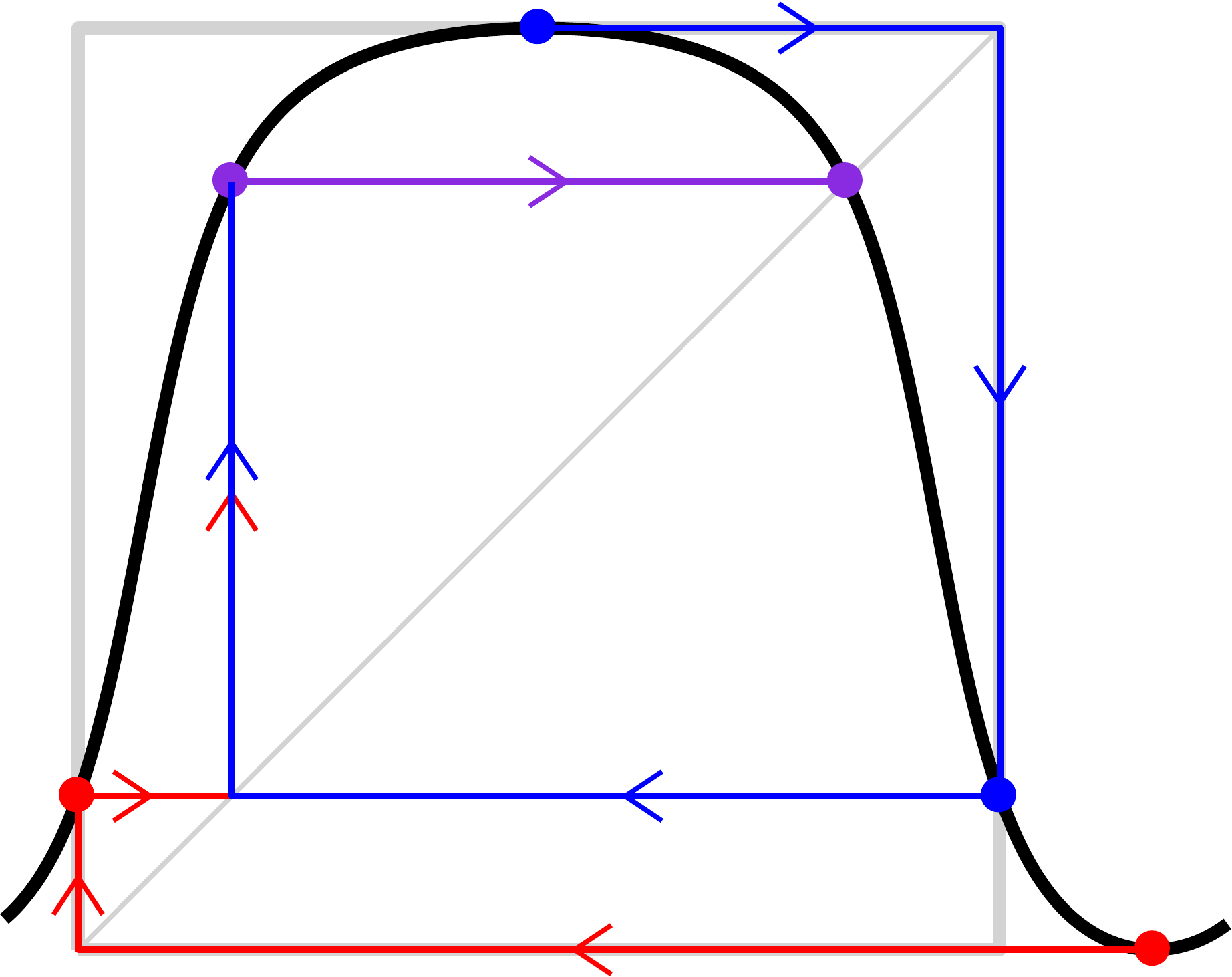}%
  }
  \caption{\label{F-exc}  On the left, the PL model for combinatorics
    $\(1,\,3,\,4,\,3,\,1,\,0\)$.  
    On the right is shown its realization in the lifted model, with 
    $\mu=-2$ and $\kappa=-\sqrt{2}$. Note that the two critical
    values 
map to a common point, which maps to a fixed point.}
\end{figure}

\begin{proof}[Proof of $\autoref{P-exc}$]
  Any quadratic map which has a non-critical fixed point can be put
  into the normal form
    \begin{equation}\label{e-ie134310}
    f_{v,w}(x) = \frac{(w-v)(x+x^{-1})}{4} + \frac{v+w}{2}~.
  \end{equation}
with critical points at $\pm 1$, fixed point at $\infty$,
 and critical values $~~f(1)=w \ne f(-1)=v$.\break
In order for such a map to be compatible with the $\(1,3,4,3,1,0\)$
mapping pattern
\begin{equation*}
  \xymatrix @R=2ex{
    \du{x_2}\ar@{|->}[r] & x_4\ar@{|->}[r]& x_1\ar[d]& x_0\ar@{|->}[l]& \du{x_5}\ar@{|->}[l]\\
  &  &x_3\ar@{->}@(ur,dr)}~~, 
\end{equation*}
we must identify this with the pattern\footnote{Alternatively
  we could identify $x_2$ with $\,-1$ and $x_5$ with $1$;
  but this would correspond to the orientation reversed combinatorics
  $\(5,4,2,1,2,4\)$.} 
\begin{equation*}
  \xymatrix @R=2ex{
    \du{1}~\ar@{|->}[r] &~ w~\ar@{|->}[r]& ~x_1~\ar[d]&~ v~
   \ar@{|->}[l]&~ \du{-1}\ar@{|->}[l]\\  &  &\infty\ar@{->}@(ur,dr)}~~. 
\end{equation*}
In particular, we must assume that the points
$(x_0,x_1,x_2,x_3,x_4, x_5)=(v,~x_1,~1,~\infty,~w, -1)$
are in positive cyclic order, or in other words that
\begin{equation}\label{E-order}
 w~<~-1~<~v~<~x_1~<~1~. \end{equation}
Now suppose that we start with any map in the form
\eqref{e-ie134310} satisfying \eqref{E-order}, and apply
the $\(1,3,4,3,1,0\)$-pullback transformation.
Then we obtain a new rational map $f$ for which the two
  critical values must map to a common point. Putting $f$
  into the form \eqref{e-ie134310} with the postcritical fixed point at
  infinity, the equation $f(v)=f(w)$ implies
  that   $v+v^{-1}=w+w^{-1}$. Since $v$ can never be equal to $w$,  
  this implies   that $w=v^{-1}$. Similarly, since $x_1$ can never be
  equal to $\infty$, the equation $f(x_1)=f(\infty)=\infty$
  implies that $x_1=0$.
  In other words, as we iterate,
  after the first   step we will always
  have a rational map of the form \eqref{e-ie134310} with $w=v^{-1}$
  and with $x_1=0$
 More explicitly, we will show that the action of $T$ then
  corresponds to the  transformation
  $$(v, 1/v)~ \longleftrightarrow~ (v', 1/v')\qquad{\rm with}\qquad
  v'=\frac{v+1}{v-1}~.$$
(See \autoref{f-pullback-134310} for an example with $v=-1/2 \leftrightarrow
  v'=-1/3$.) 

\begin{figure}[!htb]
  \centerline{%
    \includegraphics[height=2.5in]{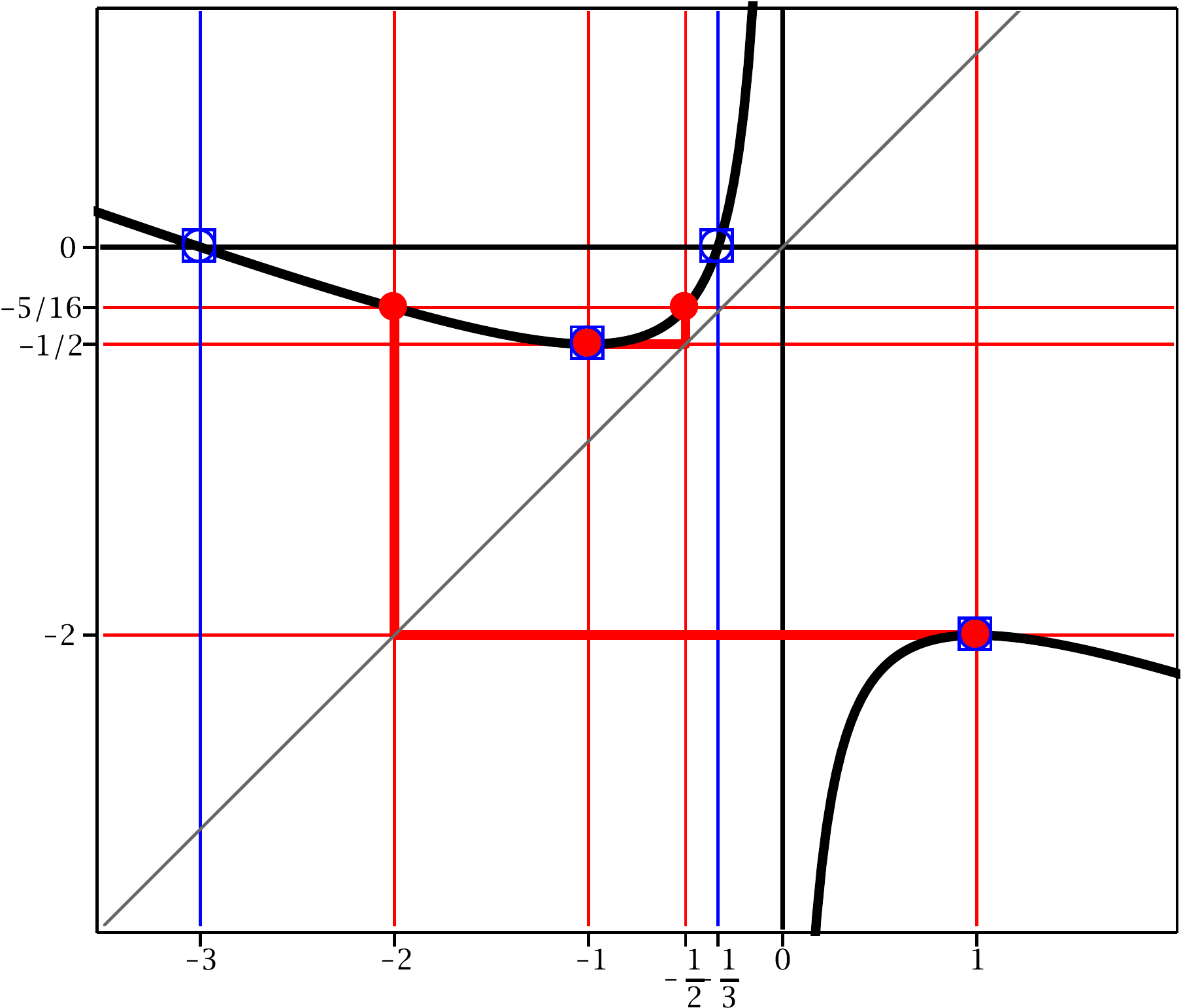}\qquad
    \includegraphics[height=2.5in]{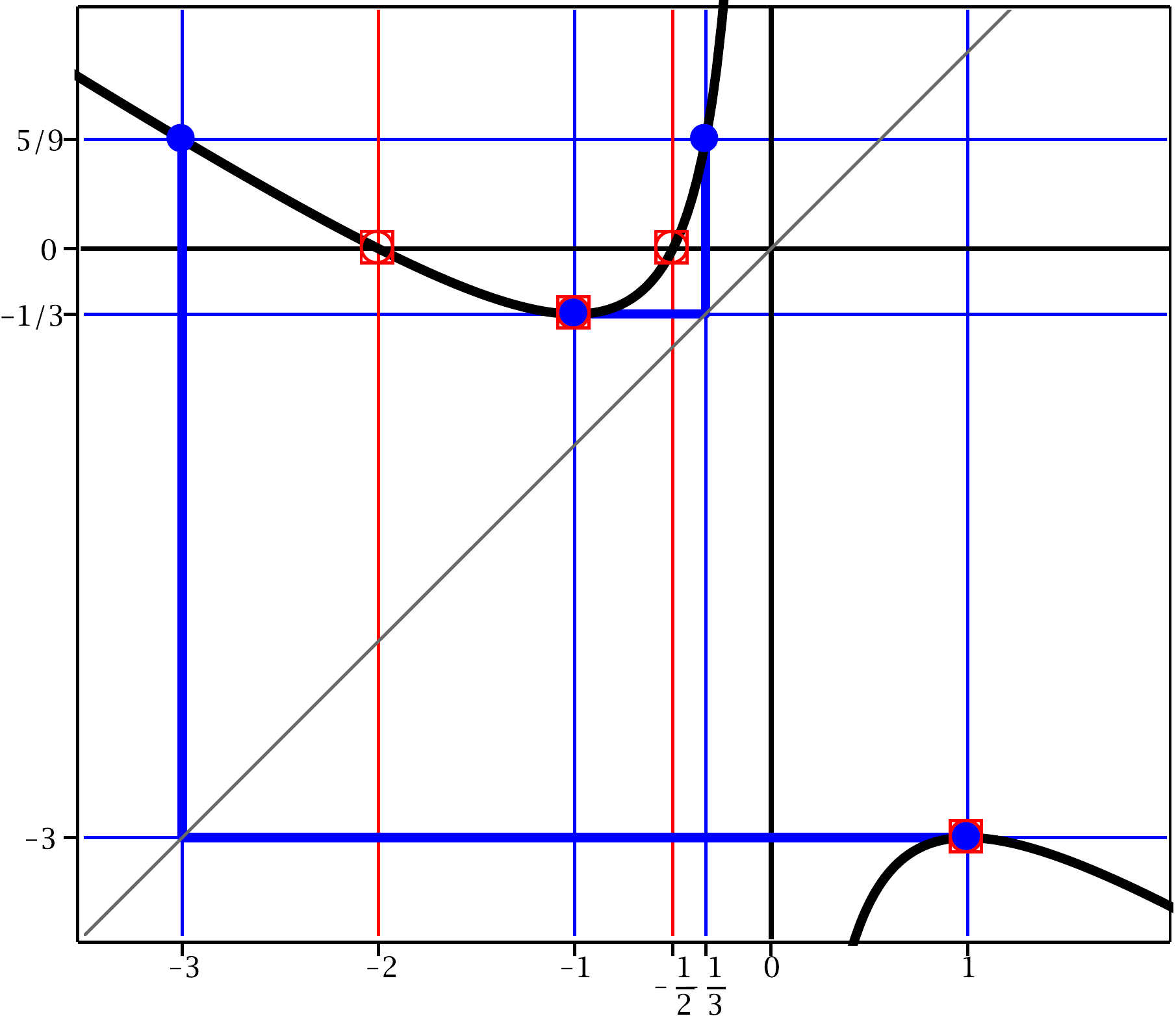}\qquad
  }
  \caption{\label{f-pullback-134310} 
     On the left is the rational map as in \autoref{e-ie134310}
    with critical values $v=-1/2$, $w=-2$ and on the right is its pullback
    with $v=-1/3$, $w=-3$. Conversely the left hand figure represents the
    pullback of the right hand one, so that this pair represents a
    two cycle for the pullback map. 
    For each graph, the marked points are indicated by a solid disk (either
    in red or blue) with corresponding vertical and horizontal lines, and the
    pullbacks of these are also indicated with open boxes of the other color.
    Thus, in the left-hand figure, the points at $(-1/2, -5/16)$ and 
    $(-2, -5/16)$ have only a red disk because they are marked points, but not
    the pullback of any marked point. By contrast, $(-1/3, 0)$ and $(-3, 0)$ are
    indicated by a blue square since these points are the pullback of $0$; the
    critical points at $x=\pm1$ have both symbols, since they are both marked
    points and the pullback of a marked point.
  }
\end{figure}

  Suppose that $(x_0,\ldots, x_5)$ are the marked points
  for $f=f_{v,w}$. To find the corresponding marked points
  $(x'_0,\ldots,x'_5)$ for the image under $T$ we must solve the equations 
  $$ f(x'_j)~= ~x_{m_j}~,$$
  taking care to choose the solution which belongs to the correct half-circles.
  In particular, to compute $v'$ we must solve the quadratic equation
  $f(v')=0$,   choosing
  the solution which belongs to the lap $-1<v'<0$. 
  A brief computation shows that $v'=(v+1)/(v-1)$ is the correct
  solution. 
  
  Since the fractional linear transformation $v\mapsto(v+1)/(v-1)$
has period two, it
  follows that $T\circ T$ is the identity for maps of this form.
  Furthermore,  since the equation $v=v' =(v+1)/(v-1)$ with $v<0$
  implies that $v=1-\sqrt 2$,
  it follows that $T$ has only one fixed point.
   This completes the proof of \autoref{P-exc}.
  \end{proof}

\bigskip
\begin{proof}[Proof of $\autoref{T-notexc}$]  
  To explain this behavior, we must go back to Douady and Hubbard.
  To every Thurston map (and therefore to every combinatorics),
  they assign an orbifold structure on the 2-sphere. It can be defined as follows
  (compare \cite[Page 2]{DH} ). Each point $x$ of the sphere is assigned a
  ramification
  index $\nu(x)\in\{1,2,3,\cdots, \infty\}$ which is greater than one if and
  only if $x$ is a postcritical point. More precisely, it can be defined as the
  supremum over all iterated preimages $f^{\circ k}(y)=x$ of the local degree of
  $f^{\circ k}$ at $y$. Thus $\nu(x)=\infty$ if and only if $x$ belongs to a
  periodic critical orbit; but $\nu(x)$ is bounded by the product of the local
  degrees of $f$ at its
  critical points otherwise. This orbifold structure has a well defined orbifold 
  Euler characteristic   $\chi\in\Q$ defined by the formula
  \begin{equation}\label{euler1}  \chi~=~ 2~-~ \sum_x \Big(1-\frac{1}{\nu(x)}\Big)~,
  \end{equation}
  to be summed over all postcritical points. They show that $\chi\le 0$ in all
  cases.
  By definition, the orbifold is Euclidean if $\chi=0$; and non-Euclidean 
  if $\chi<0$. In particular, it is clearly non-Euclidean whenever the number of
  postcritical points satisfies $\npc>4$.
\msk

As an example, if $~~\vecm=\(1,\,3,\,4,\,3,\,1,\,0\)~~$ with
mapping pattern  as described earlier,
  then all four postcritical points $x_4,~x_1,~ x_0,~x_3$ have ramification index
  $\nu=2$, so  $\chi=0$. 
  For a more typical example, consider  $\(1,\,2,\,1,\,0\)$ with mapping pattern
  $$\xymatrix{\du{x_3}\ar@{|->}[r] & x_0\ar@{|->}[r]& \du{x_1}\ar@{<->}[r]& x_2}~.$$ (Compare
  \autoref{f-aBn4a}-left.)
  In this case, we have $\nu(x_0)=2$ but $\nu(x_1)=\nu(x_2)=\infty$, and it follows
 that $\chi=-1/2<0$.
\bsk

Douady-Hubbard (\cite[Theorem 1]{DH}) asserts that: 

\begin{quote} \sf If the orbifold is non-Euclidean, then either:
  \begin{itemize}[beginpenalty=10000,midpenalty=9999, 
    topsep=-.2ex]
  \item[$\bullet$] the iterated pull-back transformation converges, yielding  a
    rational map with the specified combinatorics; or

\item[$\bullet$] there is no such rational map. \end{itemize}\end{quote}

\noindent However they make no such assertion in the Euclidean case where $\chi=0$.
We will prove the following for real quadratic combinatorics. 
\msk

\begin{lem} The only admissible, minimal and non-polynomial combinatorics with
  $\chi=0$ are $\(1,\,0\)$, $\(2,\,3,\,2,\,0\)$, and $\(1,\,3,\,4,\,3,\,1,\,0\)$;
  together with the corresponding cases with reversed orientation. 
\end{lem}

(Compare Figures~\ref{f-aBn2}-left, \ref{f-nonhyp}-left, and \ref{F-exc}.)\msk

\begin{proof} Since our maps are quadratic, the ramification index of a
  postcritical point can only take the values 2, 4, or $\infty$. It will be
  convenient to set

  \begin{itemize}[itemsep=.2ex]
  \item[\textbf{s}] equal to the   number of ``simply postcritical'' points with
    $\nu=2$,

\item[\textbf{d}] the number of ``doubly postcritical'' points with $\nu=4$, and

\item[\textbf{i}] the number of ``infinitely postcritical'' points with $\nu=\infty$.
  \end{itemize}\ssk

  \noindent Then the formula \eqref{euler1}  becomes
  $$\chi ~=~ 2-\frac{1}{2}~\textbf{s}  -\frac{3}{4}~\textbf{d} -\textbf{i}~.$$
  \ssk

  First consider combinatorics of Type B or D. Then all $n+1$ of the marked points
  are infinitely postcritical, so that $\chi=1-n$. Thus $\chi=0$ only for $n=1$,
  with combinatorics $\(1,\,0\)$.\msk

  Next consider Type C.~~ Since critical fixed points have been excluded, there
  must be a periodic critical orbit with period at least two; thus $~{\bf i}\ge 2$.
  Furthermore, the  other critical point   can't map directly to this periodic
  orbit, so that $~{\bf s}\ge 1$; and it follows that   $\chi\le -1/2$.\msk

  Similarly in the Half-Hyperbolic case we have $~~{\bf i}\ge 2$ and
  $~{\bf s}\ge 2$, so $\chi\le -1$.\msk

  There remains the Totally Non-Hyperbolic case. Since neither critical
  point can 
  map directly to a periodic cycle, and since at most one point can map to a
  fixed 
  point, the only possible  mapping patterns with $\chi=0$ are the following:

  \begin{equation}\label{mp1}
    \xymatrix{\du{c_1}\ar@{|->}[r] & \du{c_2}\ar@{|->}[r] & v_2\ar@{|->}[r]& x
      \ar@{->}@(ur,dr)}\end{equation}
  or
  \begin{equation}\label{mp2}
    \xymatrix @R=2ex{
      \du{c_1}\ar@{|->}[r]& v_1\ar@{|->}[r]& x \ar@{|->}[d]& v_2\ar@{|->}[l]
      & \du{c_2}\ar@{|->}[l] ~,\\
    &  & y\ar@{->}@(ur,dr)& &\\}~\end{equation}
or
  \begin{equation}\label{mp3}
\xymatrix{\du{c_1}\ar@{|->}[r]& v_1\ar@{|->}[r]& x \ar@{<->}[r]& y&v_2\ar@{|->}[l]& \du{c_2}\ar@{|->}[l]}~.\end{equation}
\end{proof}

  \begin{figure}[h!]
    \centerline{\includegraphics[width=2in]{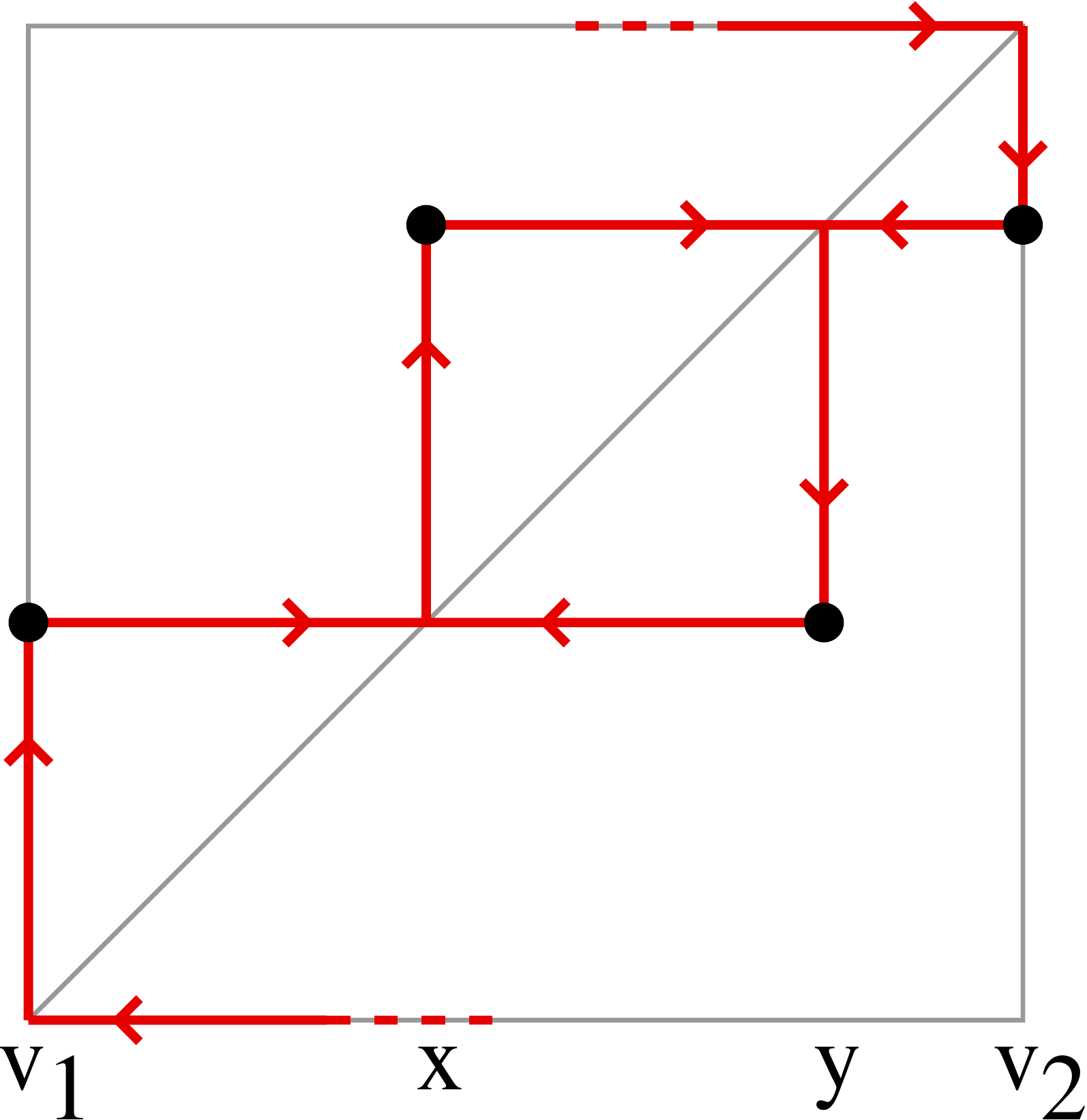}\qquad\qquad\qquad
      \includegraphics[width=2in]{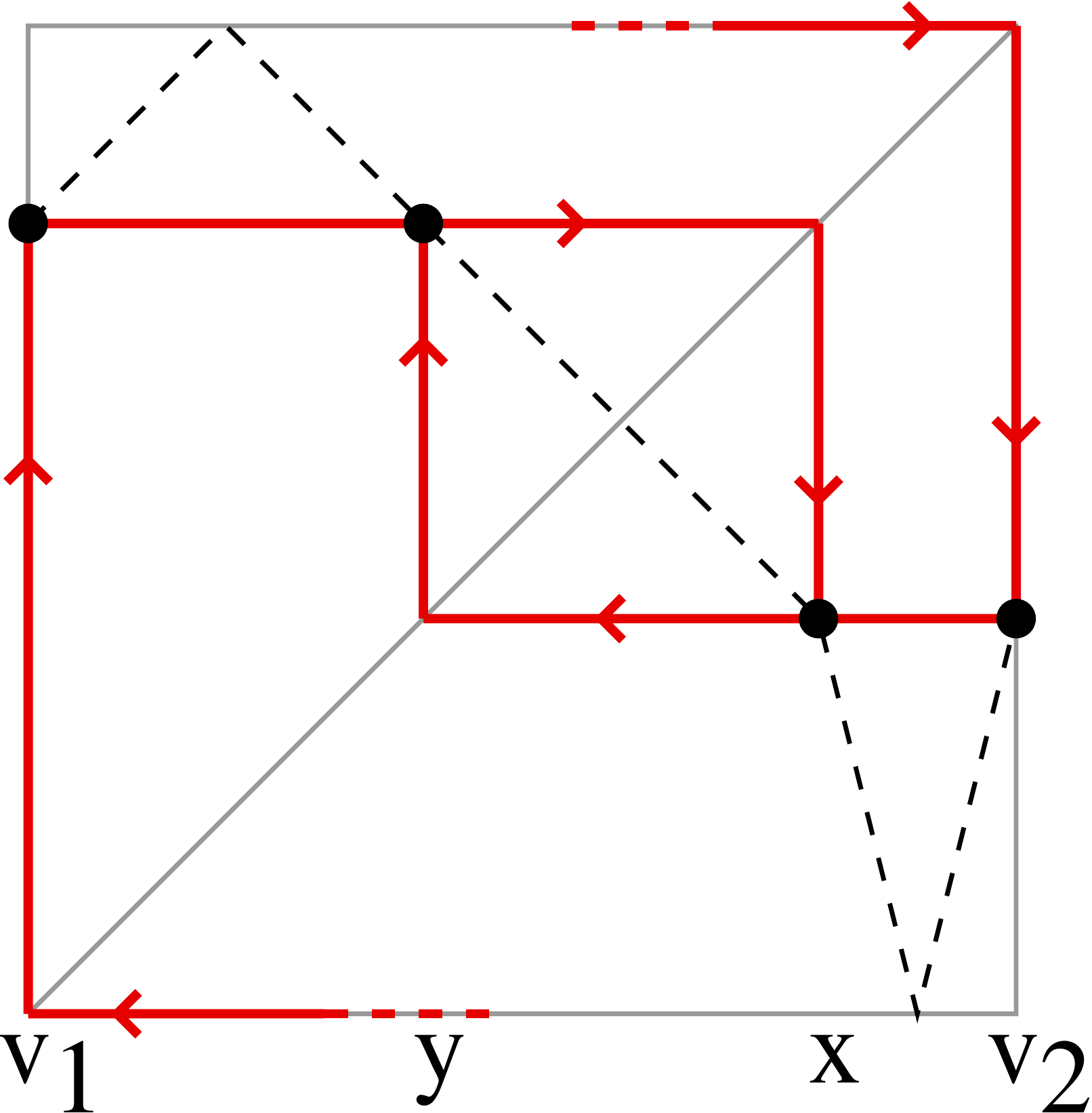}}
    \caption{\label{f-notpos}  The topologically possible arrangements
    for a map with mapping pattern~\eqref{mp3}.}
    \end{figure}

\begin{lem}\label{L-para} There is only one admissible and minimal
  combinatorics compatible with the   mapping pattern
  \eqref{mp1} or with \eqref{mp2}; but there is none compatible
  with \eqref{mp3}.
\end{lem}

\begin{proof} First consider the mapping pattern \eqref{mp3}.
Clearly both $x$ and $y$ must belong to the interval $f(\Rhat)=[v_1,~v_2]$ and
cannot be equal to $v_1$ or $v_2$. Furthermore, either $x<y$ as in
\autoref{f-notpos}-left,  or $x>y$ as in \autoref{f-notpos}-right.
The left hand picture is impossible since any compatible graph would
have to pass through the four points marked by black dots. But then a
 horizontal line between levels $x$ and $y$ would have to 
 intersect this graph at least three times: once between the first
 two points,  once between the two middle points, and once between
 the last two points. This is impossible for any graph which is
 supposed to imitate a function of   degree two.

  The right hand figure is possible; but only if 
 the graph is topologically as described by the dotted lines. But then the
  corresponding combinatorics is not minimal, since the interval $[y,~x]$ maps
  to itself and contains no critical point. Thus there is no minimal
  solution in this case. (See \autoref{f-bad} for further discussion of this
  combinatorics.)
\msk

The proof for \eqref{mp2} is similar. If $x<y$, then we see easily that the
topological arrangement shown in \autoref{F-exc} is the only one possible; while
if $x>y$ then we get the   $180^\circ$ rotation of that figure.
This case has been thoroughly discussed in \autoref{P-exc}. 
\msk

  Finally, for \eqref{mp1}, if we write the mapping pattern as
$$~\xymatrix{\du{c_2}\ar@{|->}[r]&(v_2=\du{c_1}) \ar@{|->}[r]& v_1\ar@{|->}[r]& x\ar@{->}@(ur,dr)}~~~~~,$$
    then we are quickly led to the combinatorics shown in
  \autoref{f-nonhyp}-left. This completes the proof of \autoref{L-para}.
    \end{proof}

    Thus we have explicitly described what happens in the three cases
    with Euclidean orbifold. This also completes the proof of
    \autoref{T-notexc}, since in all other cases
    the statement follows from Thurston's
    arguments, as described in \cite{DH}.
  \end{proof}\bsk

\vspace*{-.5\baselineskip}  
\section{The Moduli Spaces $\M$ and $\M/\I$}\label{s7}

Let $\M$  be the moduli space for real quadratic maps with real
critical points up to orientation preserving change of coordinates.
This section will show that $\M$ 
is a smooth\footnote{The corresponding complex
  moduli space has one singular point. Compare Rees \cite{R}.}
 manifold with the topology of a cylinder or annulus. On the
 other hand, if we allow orientation reversing changes of coordinate, then
 we obtain the quotient manifold $\M/\I$, which is a simply-connected
 smooth manifold with
boundary. Compare Figures \ref{F-canmod} and \ref{F-M/I}.

We will first prove the following.

\begin{theo}\label{t1} 
  Every quadratic map with real coefficients and real critical points
  is conjugate, under an orientation preserving change of coordinates,
  to one and only one map in the \textbf{\textit{canonical form}}
    $$ f(x)=\frac{A\,x^2+B}{C\,x^2+D}\quad{\it satisfying}
\quad A^2+C^2\,=\,B^2+D^2\,,\quad
 {\sf and}\quad AD-BC\,>\,0~.$$
    Here $A,~B,~ C,~D$ are uniquely determined up to
    multiplication by a common non-zero constant.
  \end{theo}

\begin{proof} Every conjugacy class can be put into the form
$$ g(x)~=~ \frac{ax^2+b}{cx^2+d} $$
by placing its two critical points at zero and infinity. We can always
assume that {$ad-bc>0$,} conjugating $g$ if necessary by the
orientation preserving transformation $x\mapsto -1/x$ which interchanges
the two critical points and changes the sign of $ad-bc$. (It then
follows that $g'(x)>0\Leftrightarrow x>0$; assuming that the denominator
is not zero.)

Now consider a scale change, replacing $g(x)$
by the map {$f(x)=g(\lambda x)/\lambda$} with $\lambda>0$. Then
$$ f( x)~=~\frac{Ax^2+B}{Cx^2+D} \quad{\rm with}\quad
A=\lambda^2 a,~~B=b,~~ C=\lambda^3 c,~~ D=\lambda d~.$$
We must choose $\lambda$ so that
$$ A^2+C^2=B^2+D^2\quad{\rm or~equivalently}\quad
\lambda^6 c^2 +\lambda^4 a^2-\lambda^2 d^2- b^2 ~=~0\,.$$
Dividing the last equation by $\lambda^3$, we get
$$ \lambda^3 c^2+\lambda\, a^2 - d^2/\lambda - b^2/\lambda^3~=~0~.$$ The
left side of this equation, considered as a function of $\lambda$, is clearly
monotone, mapping the half-line $\lambda>0$ diffeomorphically onto the
entire real line. Therefore there is one and only one choice of $\lambda$
which satisfies the equation. This proves that every such map is
conjugate to one in canonical form. Since each step of the argument
is uniquely determined, uniqueness follows easily. This proves
\autoref{t1}.\end{proof}

\begin{figure}[!htb]
  \centerline{\includegraphics[width=4in]{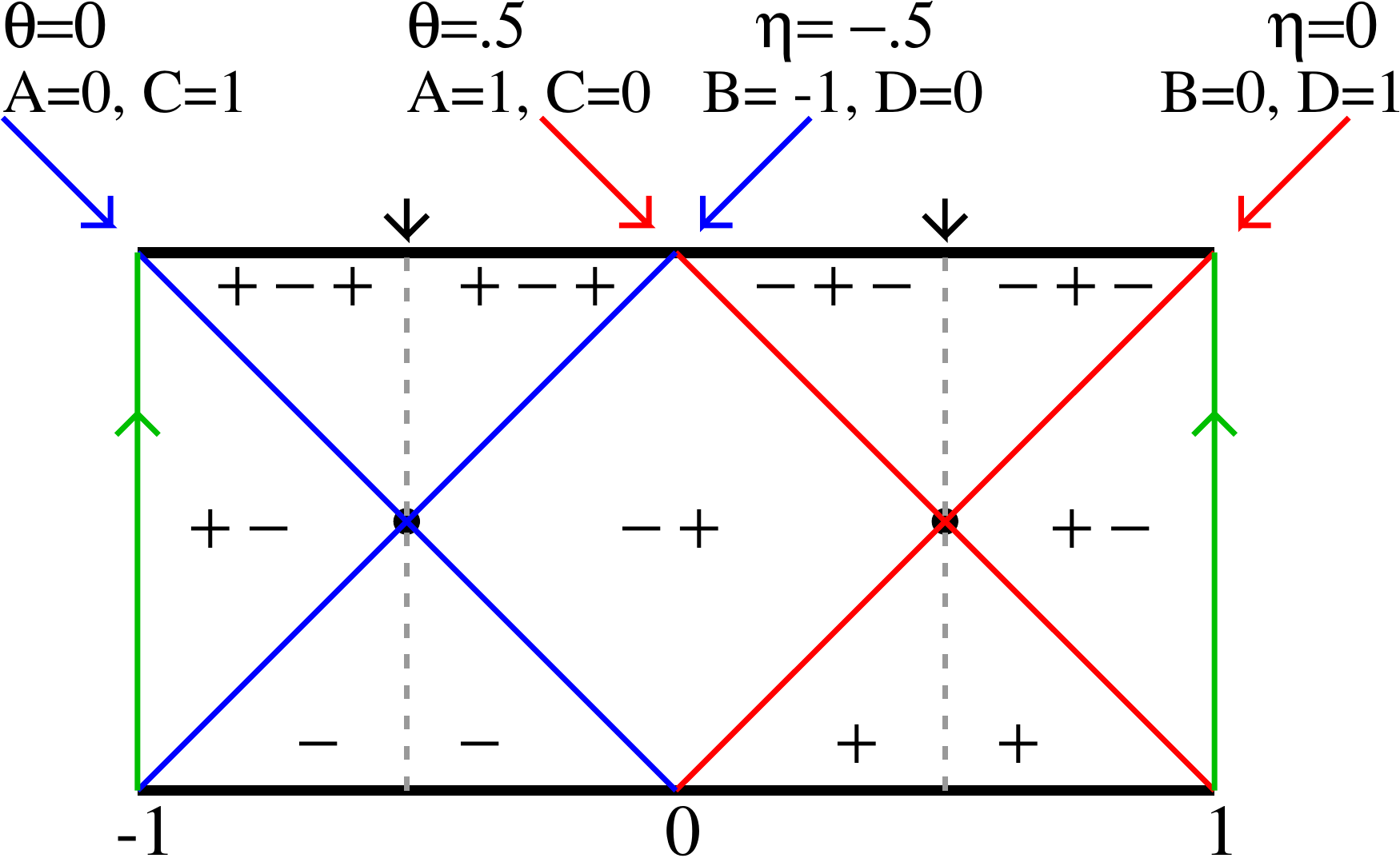}}
\caption{ \label{F-canmod} The cylindrical moduli
  space $\M\cong(\R/2\Z)\times(0,1)$ can be   obtained from the
  rectangle $[-1,1]\times[0,1]$ in the $(\Sigma,\,\Delta)$-plane
 by identifying the left and right edges  $\Sigma=\pm 1$ so that the arrows match.
Here the top and bottom edges  represent points in the ideal
boundary of $\M$. The red lines represent maps of polynomial  shape.
They cross each other at the point $\langle x\mapsto x^2\rangle$.
The blue lines represent maps of co-polynomial  shape, with one critical
point mapping to the  other. These lines  cross at the point
$\langle x\mapsto -1/x^2\rangle$. The
four red and blue lines divide $\M$ into six complementary regions,
each either monotone, unimodal, or bimodal of the  shape indicated.
The dotted lines form the symmetry  locus, consisting of
conjugacy classes  which are invariant under
 reflection in either of these lines (or under the 
 orientation reversing transformation
 {$(\Sigma,\,\Delta)\leftrightarrow(\pm 1-\Sigma,\,\Delta)$}).
 The small black arrows point to what we believe are 
 the only possible limit points of the iterated
pullback in the strongly obstructed case. (See \autoref{C-2lims}.)
}
\end{figure}

Using this result, we can provide an explicit description for the moduli space $\M$
consisting of all conjugacy classes of real quadratic maps. First note that
we can  always normalize so that $A^2+C^2=B^2+D^2=1$ by multiplying $A,B,C, D$
by a suitable common constant. It is then natural to choose angles $\theta$ and $\eta$
so that
$$ A=\sin(\pi\theta),~~C=\cos(\pi\theta)~,\qquad{\rm and}\qquad
 B=\sin(\pi\eta),~~D=\cos(\pi\eta)~.$$
 Here it is necessary to be careful. If we add one to both $\theta$
 and $\eta$, then the constants $A,~B,~C,~D$ will all be multiplied by
 $-1$, and the map $f$ will not change. However, if we replace
 $(\theta,~\eta)$ by $(\theta+1, ~\eta)$, then we will get a quite
 different map.

 Note also that the determinant can be written as
 $$ AD-BC~=~\sin(\pi\theta)\cos(\pi\eta)-\cos(\pi\theta)\sin(\pi\eta)~=~
 \sin\big(\pi(\theta-\eta)\big)~.$$
 Since we require that $AD-BC>0$, it will be convenient to assume that
 $$ \theta-\eta~=~ \Delta\qquad{\rm with}\qquad 0<\Delta<1~.$$
 On the other hand, since we can't add one to $\theta$ without also adding
 one to $\eta$, it follows that the sum $\Sigma=\theta+\eta$ is actually
 well defined modulo two. This proves the following
 
 \begin{coro}\label{c1} 
   A map in the normal form of $\autoref{t1}$ is uniquely
   determined by the two invariants
   $$ \Sigma ~=~\theta+\eta ~\in~\R/(2\Z)\qquad{\rm and}\qquad \Delta~=~\theta-\eta
   ~\in~(0,1)~.$$ Therefore the moduli space $\M$, consisting of all conjugacy classes
   of quadratic maps with real coefficients and real critical points, is diffeomorphic
   to the cylinder {$(\R/2\Z)\times(0,1)~$.}
˜\end{coro}

We can provide a more geometric interpretation of the invariants
$\theta$, $\eta$, $\Sigma$, and
$\Delta$ as follows. We make use of three different closely related models for 
$\Rhat$. By definition  $\Rhat=\R\cup\{\infty\}$. However, $\Rhat$ can be
identified
with the standard circle $\R/\Z$ by letting $t\in\R/\Z$ correspond to
$\tan(\pi\,t)\in\Rhat$;
and can also be identified with the real projective line by letting $t$
correspond
to the ratio  
$$\big(\sin(\pi\,t)~:~\cos(\pi\,t)\big)~\in ~\bP^1(\R)~ .$$

\begin{figure} [t!]
  \centerline{\includegraphics[height=2.7in]{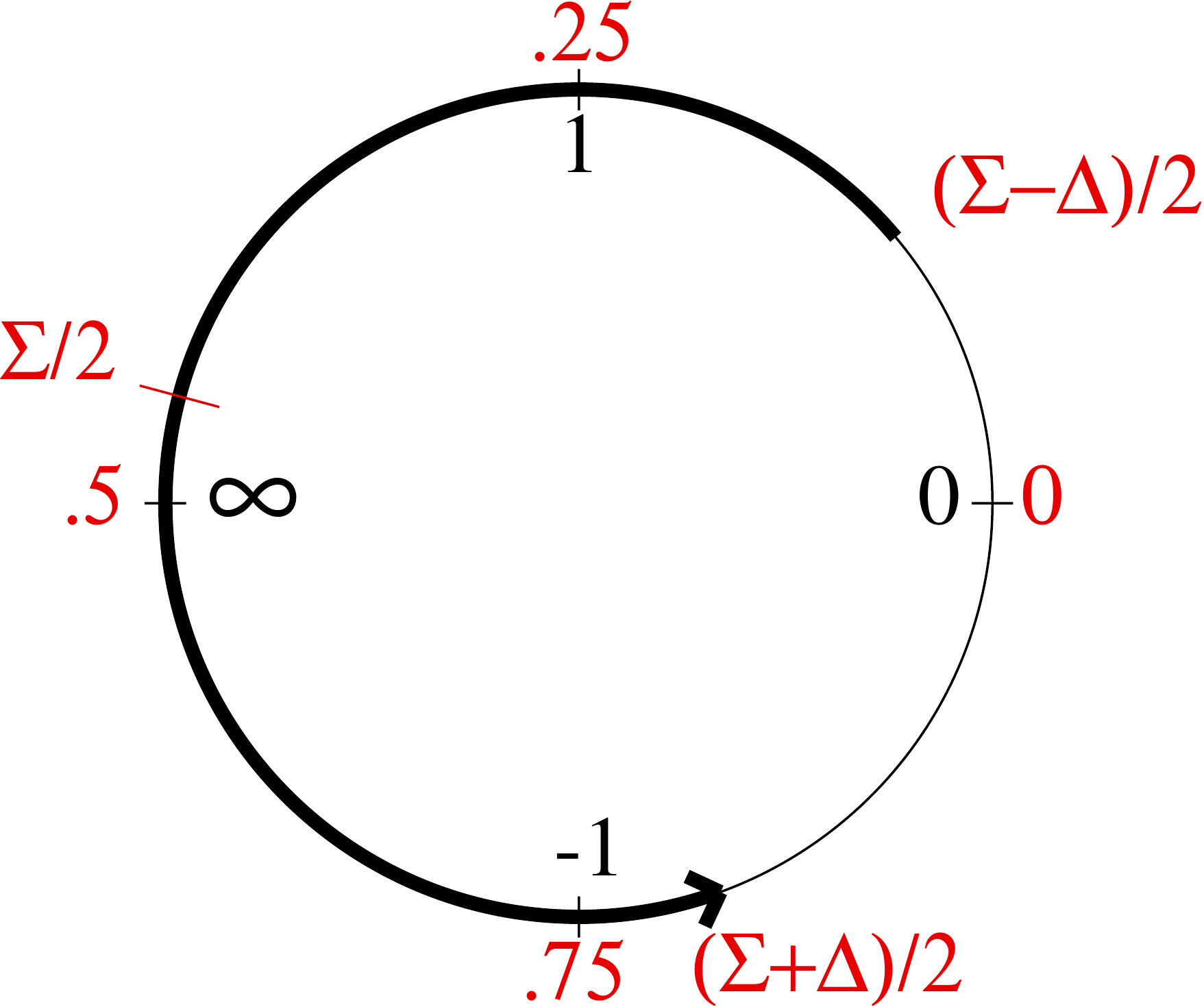}}
  \caption{\label{f2circ} Identifying $\widehat\R$ with the unit circle, or
    with $\R/\Z$.
    Several points $x\in\widehat\R$ are indicated in black, inside the circle,
    while the corresponding coordinates  $t\in \R/\Z$, with $\tan(\pi t)=x$,
    are indicated in red, outside.     The basic
    invariant for a map in canonical form is the image $f(\Rhat)$,
    indicated here (for a typical example of  shape $+-+$) by a
    heavy black arc. The length of this arc in $\R/\Z$ coordinates is equal
    to $\Delta\in(0,1)$, and its midpoint is  $\Sigma/2$.}
\end{figure}

\begin{coro}\label{c2}
  If $f$ is in canonical form, then the image $f(\Rhat)
  \subset \Rhat\cong\bP^1$ is the circle arc of length $\Delta$ in $\R/\Z$
  coordinates, with end points $\eta$ and $\theta$, and with mid point $\Sigma/2$.
\end{coro}

Compare \autoref{f2circ}. Since $\Sigma$ is well defined mod $2\Z$, it follows
that $\Sigma/2$ is well defined mod $\Z$.

\begin{proof}[Proof of $\autoref{c2}$] The end points of $f(\Rhat)$ are the critical
  values
  $$f(0)~=~B/D~=~\big(\sin(\pi\eta):\cos(\pi\eta)\big) \quad{\rm and}\quad
  f(\infty)~=~A/C~=~\big(\sin(\pi\theta):\cos(\pi\theta)\big)~,$$
  corresponding to the points $t=\eta$ and $t=\theta$ in $\R/\Z$.
  These two points divide the circle into two arcs. We must check that $\Sigma/2$ is
  the center point of the arc corresponding to $f(\Rhat)$, which has length $\Delta$.
  It is enough to check this for a single example, since all other cases will then follow
  by continuity.   As our example, for $f(x)=x^2$ with $A=D=1$ and $B=C=0$, it is not
  hard to check that
  $$ \eta=0\,,\quad \theta=1/2\,,\quad {\rm and}\quad \Sigma=\Delta=1/2~.$$
On the other hand, $f(\Rhat)=[0,\,+\infty]$  corresponds to the interval $0\le t\le 1/2$,
with center point at $\Sigma/2=1/4$, as required. This completes the proof.  \end{proof}

\subsection*{\bf Orientation Reversal: The Canonical Involution $\I$.}
\ssk

Let $\langle f\rangle\in\M$ denote the  conjugacy class of $f$,
 and let 
$J:\Rhat\to\Rhat$ be any orientation reversing fractional linear
transformation.
There is a canonical involution $\I$ of $\M$
defined by the equation
$$ \I\big(\langle f\rangle\big)~=~\langle J\circ f\circ J\rangle ~.$$
This conjugacy class  
does not depend on the choice of $J$. If $J'$ is another involution
and if  $L=J\circ J'$ with $L^{-1}=J'\circ J$, then evidently
$$L\circ(J'\circ f\circ J') \circ L^{-1}~=~J\circ f \circ J~, $$
so the two are conjugate.
For example we could take $J(x)= -x$ or $1/x$. 
If we consider maps normalized so that $f(\Rhat)=[0,1]$,
then the most convenient choice is $J(x)=1-x$, corresponding to a
$180^\circ$ rotation of the graph of $f$.

The fixed points of $\I$ form the  \textbf{\textit{symmetry locus}}.
The  conjugacy class  $\langle f\rangle$ 
belongs to this symmetry locus if and only if $f$  commutes
with some orientation reversing fractional linear transformation, which
necessarily interchanges the two critical points of $f$.

\begin{rem}\label{R-I2}
  This orientation reversing involution also reverses combinatorics,
{replacing} $\(m_0, \ldots, m_n\)$ by the sequence
$$ \(n-m_n\,, ~~n-m_{n-1}\,, ~~ \ldots\,,~~ n-m_1\,,~~ n-m_0\)~.$$
It acts on the Epstein parameters by sending $(\mu,\,\kappa)$ to 
$(\mu,\,-\kappa)$.
\end{rem}
\msk

\begin{definition}\label{D-M/I} Let $\M/\I$ be the quotient
 in which each  conjugacy class $\langle f \rangle$ is identified
 with $\I(\langle f\rangle)$. This is the appropriate moduli space to
work with when studying properties which do not depend on orientation.
Since the involution $\I$ acts
on the cylinder by mapping each pair
$(\Sigma,\,\Delta)\in (\R/2\Z)\times(0,1)$ to the  
pair {$(1-\Sigma,\,\Delta)$}. It follows that each pair
$(\langle f\rangle,~\I\langle f\rangle)$ has a unique representative
for which $|\Sigma|\le 0.5$. {\sf In other words, 
the middle half of \autoref{F-canmod}, consisting of pairs
$(\Sigma,~\Delta)$ with $|\Sigma|\le 0.5$, maps bijectively onto
$\M/\I$.}
\end{definition}\msk

\begin{figure} [t!]
  \begin{center}
    \begin{overpic}[width=4in]{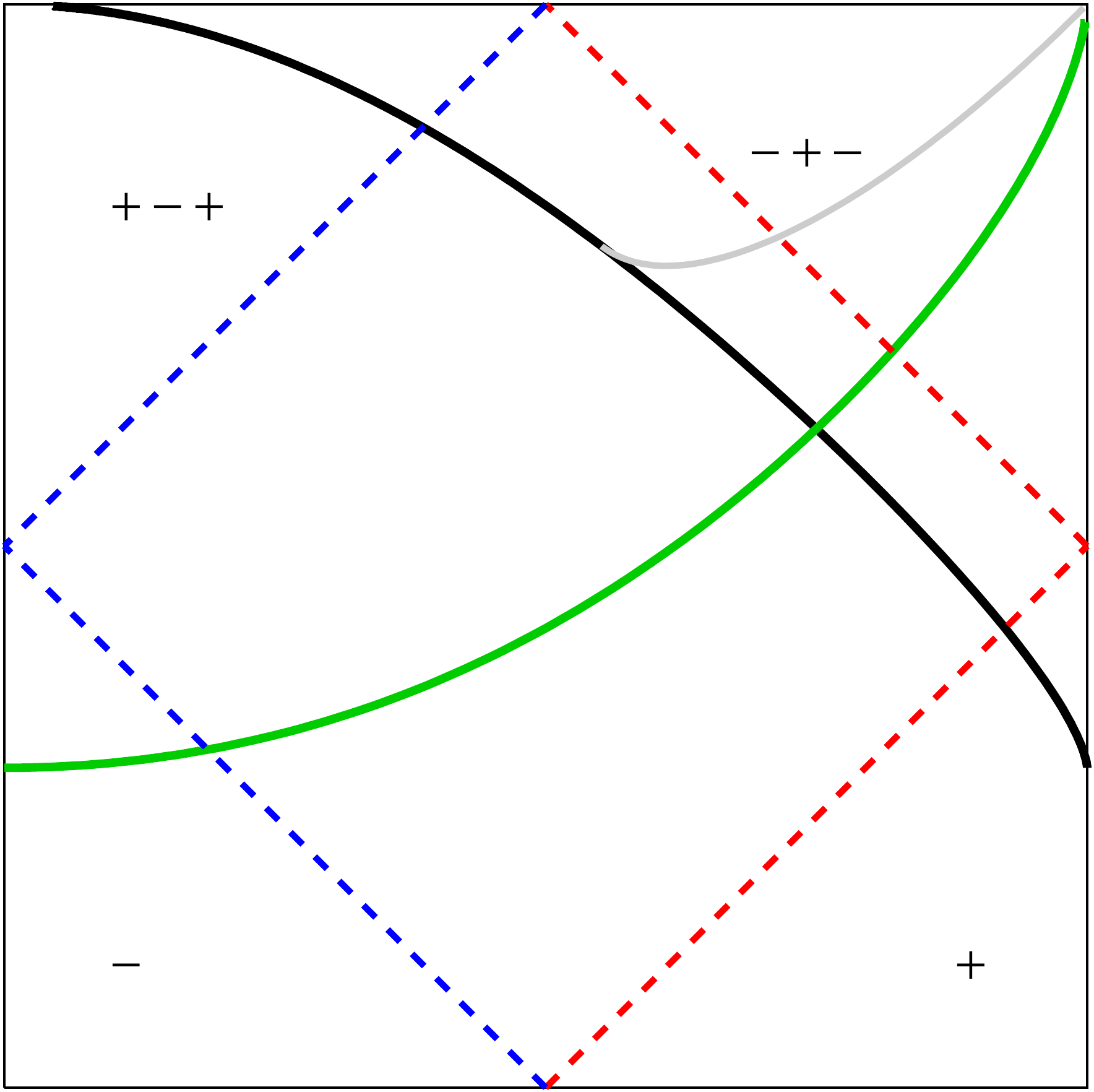}
      \put(180,265){$\rm SL$}
      \put(100,265){$\rm SL$}
      \put(140,265){$\rm SL$}
      \put(213,220){$\F(1,2)$}
      \put(245,190){$\F(2,1)$}
      \put(125,150){\rm unimodal}
      \put(80,190){$\F(0,1)$}
      \put(125,90){$\F(1,0)$}
      \put(128,50){$J_\R=\emptyset$}
    \end{overpic}
    \caption{\label{F-M/I}   The quotient space $\M/\I$ can be
    identified with the square
  region $[-.5,~.5]\times(0,1)$ between the two dotted green lines of
  \autoref{F-canmod}. In the figure above, the black curve $\Per_1(1)$ 
  and the green curve $\Per_1(-1)$  divide $\M/\I$ into four regions
  $\F(\a,\,\r)$, distinguished by the number of attracting and
  repelling fixed points; while the dotted
  red and blue lines divide it independently into five regions,
  with the unimodal region in the center. The left and
  right boundaries form the symmetry locus and are
  part of $\M/\I$; but the top and bottom boundaries
  represent ideal limit points. The shift locus $\rm SL$ is the subset 
  of $\F(1,2)$ which lies above both $\Per_1(1)$ and the gray Chebyshev
  curve. It intersects the $+-+$, unimodal, and $-+-$ regions.
  (See \autoref{R-M/I}.)}
\end{center}
\end{figure}
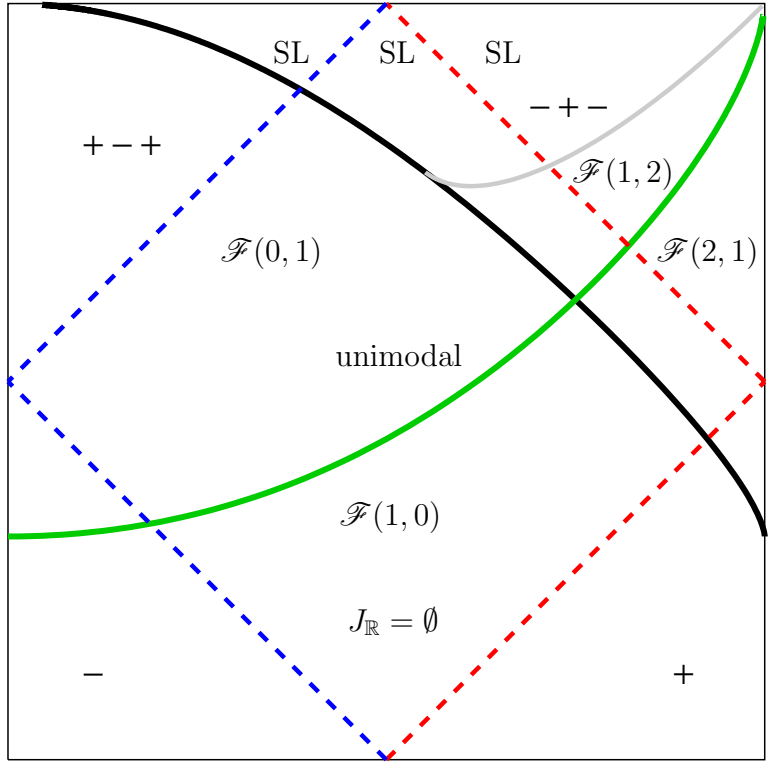 

\begin{rem}[Important subsets of $\M/\I$]\label{R-M/I}
  For a picture of $\M/\I$ see \autoref{F-M/I}.
(Compare the older picture 
  of $\M/\I$ in \autoref{F-old}; but be 
 warned that comparison of the two  pictures can be very confusing.)
  Just as in \autoref{F-canmod},
  the diagonal lines representing conjugacy classes of polynomial
  and co-polynomial  shape divide the figure into bimodal, unimodal
  and monotone regions. However in this case there are only five such
  regions since we no longer distinguish between $+ -$ unimodal
  and $- +$ unimodal.

  There is a quite different subdivision of $\M/\I$  as follows.
  The black curve in \autoref{F-M/I}
    represents $\Per_1(+1)$, the set of all conjugacy classes with
  a fixed point of multiplier $+1$; and similarly the green curve
  represents $\Per_1(-1)$. These curves divide $\M/\I$ into four
   ``fixed point regions'', which we denote by $ \F(\a,~\r)$,  
  where $\a$ is the number of attracting fixed points of $f$
  for $\langle f\rangle$ in this region, and $\r$ is the number
  of repelling fixed points. Here the lower region  $\F(1,0)$ is
  dynamically rather boring. It consists of maps for which the topological
  entropy is zero and the real Julia set $J_\R$ is empty. (All real orbits 
  converge to the unique attracting fixed point.) The right hand region
  $\F(2,1)$ is somewhat more interesting. Here $J_\R$ consists of one repelling
  fixed point in $f(\Rhat)$ and its one preimage. The orbit of every
  point in  $\Rhat\ssm J_\R$ converges to one of the two attracting fixed
  points.

  The two regions  $\F(0,1)$ and $\F(1,2)$ above the green curve
  are much more interesting. In particular, every critically
  finite class which is not of polynomial  shape must be contained
  in  $\F(0,1)$; while every one which is of polynomial  shape
  must be contained in  $\F(1,2)$ (except in the special case of
  $\langle x\mapsto x^2\rangle$).\msk

  There is an important sub-region of $\F(1,2)$. The
  \textbf{\textit{hyperbolic shift locus}}, labeled as ${\rm SL}$,
  is the open subset consisting all conjugacy classes $\<f\>$ for which:
  
  \begin{quote}
    \begin{itemize}
    \item[$\bullet$] all orbits in $\Rhat\ssm J_\R$ converge to the unique
      attracting fixed point; and

    \item[$\bullet$] if we put the critical points at zero and infinity,
      then every orbit \hbox{$x_0\mapsto  x_1\mapsto\cdots$} in $J_\R$
      is uniquely determined by the sequence of signs
  $({\rm sgn}(x_0),\,{\rm sgn}(x_1),\,\cdots)$, where any such sequence can
  occur.
\end{itemize}
\end{quote}
  
\noindent  The closure $\overline{SL}$ is the set of pf points with
maximal topological entropy $\log(2)$. (See \cite[Prop. 3.6]{F}.) 
The boundary of  the hyperbolic shift locus within $\M/\I$ is
  a piecewise  analytic curve.   The left part of the boundary is a
  subset of $\Per_1(1)$ called the
  \textbf{\textit{parabolic shift locus}}. The right part of the
  boundary   (the gray curve in our figure) will be called the
  \textbf{\textit{Chebyshev curve}}. It consists
  of all   $\langle f\rangle$ which have one critical orbit of the form
  $$~\xymatrix{\du{c_1}\ar@{|->}[r]& v_1 \ar@{|->}[r]&  x\ar@{->}@(ur,dr)}~~~~~,$$
  where $x$ is a fixed point with multiplier $\mu>1$. (This curve
 intersects the locus of polynomial  shape maps precisely in the
  class of the Chebyshev map $x\mapsto x^2-2$.)
\end{rem}

\begin{rem}[\bf Computing with Epstein coordinates]\label{R-ep} For computational purposes, Epstein coordinates,
  with $$f(x)=\mu x/(1+2\kappa x +x^2)~,$$ are often convenient. One 
  useful quantity is the \textbf{\textit{discriminant}}
  $$ D~=~\mu+\kappa^2-1$$
  which is positive if there are three real fixed points, negative if there
  is only one, and zero along the curve $\Per_1(1)$. (Caution: $D$ can be
  very large for points which are very close to $\Per_1(1)$.)
  The \textbf{\textit{hyperbolic shift locus}}
  is the region defined by the inequalities
$\mu>1$ and $|\kappa|>1$. The \textbf{\textit{parabolic shift locus}} is
the part of the boundary of this region with     $\mu=1$ and $|\kappa|>1$,
while the \textbf{\textit{Chebyshev curve}} is the rest of the boundary with
$|\kappa|=1$ and $\mu\ge 1$. (It extends analytically into the shift locus,
but with $\mu<1$).
  \end{rem}
\msk

\begin{figure}
  \centerline{\includegraphics[width=4.5in]{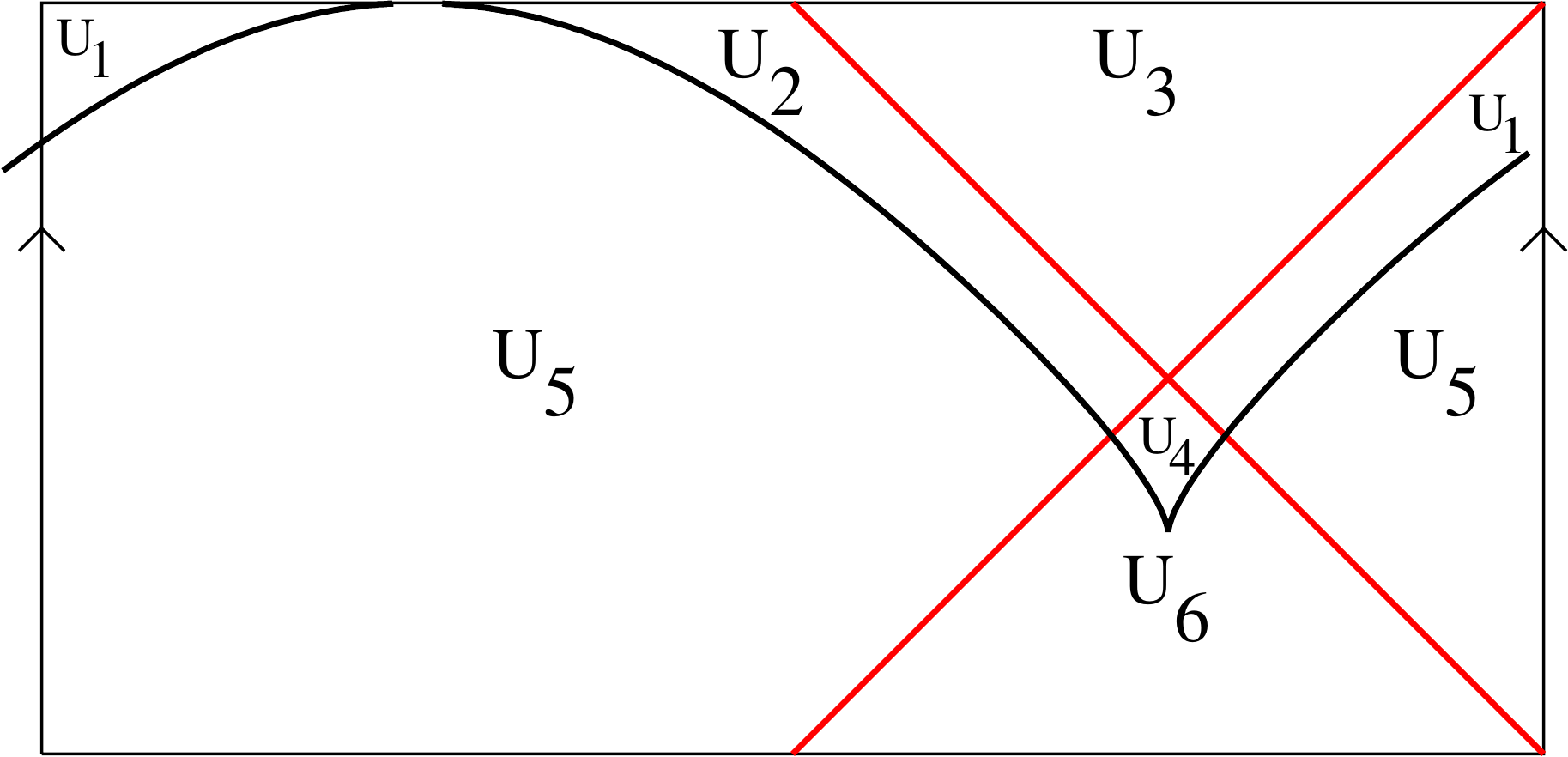}}
  \caption{\label{F-psiinv} This figure shows a partition of the moduli space
    $\M$ into six regions with the following property: If $\Psi$ is
    the natural transformation from Epstein coordinates to canonical coordinates,
    then each branch of $\Psi^{-1}$ is smooth and real analytic within each
    region.}\msk
  \end{figure}

  \begin{rem}[\bf Epstein Coordinates $\leftrightarrow$ Canonical Coordinates] 
    \label{R-ep2can} The natural transformation
    $$\Psi:(\mu,\,\kappa)~\mapsto~(\Sigma,\,\Delta)$$
 from Epstein coordinates  to canonical coordinates 
    is  smooth, real analytic, and not too hard to compute. (Compare the
    proof of \autoref{t1}.) But this does not mean that it is easy to understand.
    Given a real quadratic map $f$ and a non-critical fixed point $x_0=f(x_0)$,
    we can choose coordinates which place the critical points at $\pm 1$ and
    place $x_0$ at the origin, and then compute the corresponding Epstein
    coordinates. If $\<f\>$ lies below the curve $\Per_1(1)$ in moduli space,
    then there is only one real fixed point, so that $\Psi^{-1}\<f\>$ 
    is uniquely defined.\footnote{Here we are identifying a conjugacy class
      $\<f\>\in\M$ with its coordinate pair $(\Sigma,\,\Delta)$.} 
But if $\<f\>$ is above 
    $\Per_1(1)$ then there are three distinct real fixed points, so that
    generically there are three different branches of $\Psi^{-1}$. Points
    on the polynomial locus provide an additional complication, since one of
    their fixed points is critical, and hence doesn't correspond to any
    choice of Epstein coordinates.
    
      However, if we remove both the (red) polynomial locus  and  the
      (black) curve  $\Per_1(1)$ from $\M$, as illustrated in \autoref{F-psiinv},
      then we are left with the complementary open set which has six
      connected components $U_j$. In some sense, the inverse map is well
      behaved and real analytic on each of these components. More
    precisely, $\Psi^{-1}$ is uniquely defined on the lower regions $U_5$ and
    $U_6$, and has three distinct well defined branches on each of the upper
    regions $U_1$, $U_2$, $U_3$, $U_4$. Here the image $\Psi^{-1}(U_5)$ is
    contained in the left half-plane $\mu<0$; while $\Psi^{-1}(U_6)$ is
    contained in the right half-plane $\mu>0$.
 For $U_3$ (which is precisely the $-+-$ region),
    two of the branches of $\Psi^{-1}$ map to the left half-plane, and one
    maps to the right half-plane. Similarly, for $U_1$ and $U_2$,
    which intersect the $+-+$ region,  two branches map to the right and
    one to the left. On the other hand,  for $U_4$ which lies in the $+$
    monotone region, all three branches map to the right half-plane.\msk

    We will be particularly interested in asymptotic behavior as
    $\mu\to\pm\infty$.

\begin{theo}\label{T-asymp}
  As $\mu\to\pm\infty$ with $\kappa$ fixed:

  \begin{itemize}[beginpenalty=10000,midpenalty=9999, 
                 itemsep=0pt plus 1ex,topsep=0pt plus 1ex]
 \item[$\bullet$] the pair $(\Sigma,~\Delta)$ converges to $(\pm .5,~1)$,

 \item[$\bullet$] the difference ratio
   $\displaystyle{\frac{\pm .5-\Sigma}{1-\Delta}}$ converges to  $-\kappa$,
   and

 \item[$\bullet$] the product $|\mu|(1-\Delta)$ converges to $4/\pi$.
 \end{itemize}
 \end{theo}

\begin{proof}
Start with $~f(x)=\mu\,x/(1+2\,\kappa\,x + x^2)~$. Let $~p=2(1+\kappa)~$
and {$~q= 2(1-\kappa)$,} so that $p+q=4$ and $p-q=4\,\kappa$.
The orientation preserving automorphism $$L(x) ~=~(1+x)/(1-x)\quad{\rm
satisfies}\quad L~:~0\mapsto 1\mapsto\infty\mapsto -1\mapsto 0~.$$
Furthermore
\begin{equation}\label{e0}
 L\circ f \circ L^{-1}(x)~=~ \frac{a\,x^2+b}{c\,x^2+d}~,\qquad{\rm with}
\qquad\begin{matrix} a = p + \mu, & b = q -\mu ,\\  c = p - \mu ,&   d = q + \mu .
  \end{matrix}\end{equation}
Following the proof of \autoref{t1}, we must now solve the equation
\begin{equation}\label{e1}
 u^3\, c^2~+~ u\,a^2~-~d^2/u~-~b^2/u^3~=~ 0~.\end{equation}
It will be convenient to make the substitutions $ \mu=1/s$ and $u=e^t$.
Multiplying equation \eqref{e1} by $s^2=1/\mu^2$, it takes the form
$$ e^{3\,t}(1-s\,p)^2 + e^t (1+s\,p)^2 -e^{-t}(1+s\,q)^2 - e^{-3\,t}(1-s\,q)^2
~=~ 0~.$$
For each fixed value of $\kappa$,
the left side of this equation is clearly
a real analytic function $\Omega(s,t)$
which can be expanded in an everywhere convergent power series
$$\Omega(s,t)~=~ \sum_{i,j=0}^\infty \omega_{i,j}s^i t^j~.$$
It is not difficult to compute the first few coefficients:

\begin{center}
\begin{tabular}{lll}
$\omega_{0,0}=0$ & $\omega_{0,1}=8$ & $\omega_{0,2}= 0$\Bstrut \\
$\omega_{1,0}=0$ & $\omega_{1,1}= -16$& \Bstrut\\
 $\omega_{2,0}=32\,\kappa$ & & \\
\end{tabular}~,
\end{center}
so that
  $$ \Omega(s,t)/8 = t\,(1-2\,\,s) + 4\,\kappa\, s^2
  + ({\rm higher~order~terms})~=~0~.$$
This implies the asymptotic equality 
  $$ t~\simeq~ -4\kappa\,s^2/(1-2s)~ \simeq~ -4\,\kappa\,s^2
  \quad{\rm as}\quad t,~ s~\to 0~.$$
  Thus, if we ignore terms of order $s^2$ then we can just take $t=0$
  hence $u=1$. This means that we can just use the original values of
  $a,~b,~c,~d$ as given in \autoref{e0}. 
  The angle $\theta$ can now be computed,
     modulo $1/2$,  by the equation
     $$\tan(\pi\,\theta)~=~a/c~=~(p+\mu)/(p-\mu)~=~(p\,s+1)/(p\,s-1)
      ~=~ -1 -2\,p\,s +O(s^2)~,$$
  or equivalently
  $$ \theta ~\equiv~ \frac{1}{\pi}\arctan(-1-2\,p\,s) + O(s^2)
  \quad ({\rm mod}~ \scriptsize{\frac{1}{2}}\,\Z)~.$$

  {\bf Note.} For a direct computation of $\theta\in\R/\Z$ we would have to
  work with both the sine and cosine functions; but the computation mod
  1/2 using the tangent is easier, 
 and will be enough for the proof.\ssk

 Since $\arctan(-1)=-\pi/4$ and  the derivative of $\arctan(x)$
  evaluated at $x=-1$ is $1/\big(1+(-1)^2\big)=1/2$,
  this yields
  $$\theta~\equiv~ - .25~-~ p\,s/\pi + O(s^2)\quad ({\rm mod}~ 1/2)~~.$$
  Similarly, since $\tan(\pi\eta)\equiv~b/d~\equiv~(q-\mu)/(q+\mu)$, we get
  $$\eta~\equiv~ -.25~+~q\,s/\pi +O(s^2).$$ Therefore
  $$\Sigma~=~\theta+\eta~\equiv~  (q-p)s/\pi \equiv  +4\,\kappa\,s/\pi
  +O(s^2)\qquad ({\rm mod}~ 1/2)~~.$$
  Similarly
  $$ \Delta ~=~\theta-\eta~\equiv~-4\,s/\pi +O(s^2)\qquad ({\rm mod}~ 1/2)~~.$$
  What we want is the value of $\Sigma\in \R/2\Z$ modulo two, and the actual
  value of $\Delta\in(0,1)$; but each of these formulas  may be wrong by an
  integer or half-integer additive constant. However this constant cannot
  change as  as we vary $s>0$ or as we vary $s<0$. This means that
  to get the required formulas, we need only choose the right additive
  constants for   any one case  
  with $\mu\to  +\infty$ and any one case with $\mu\to -\infty$. The general
  cases will then follow by continuity. The correct formulas, obtained 
  in this way, are
  $$ \Sigma ~=~ \pm .5+ 4\,\kappa\,s/\pi +O(s^2) $$
  and $$ \Delta= 1- 4\,s/\pi + O(s^2)~.$$
Replacing $s$ with $1/\mu$, the theorem as stated follows easily. 
\end{proof}
\end{rem}\msk

\begin{rem}[{\bf The Filom-Pilgrim Maps}]\label{R-FP}
  These form a rich family of critically finite maps of Type B. 
  Given relatively prime numbers {$0<p<q$,} consider the combinatorics
  $$ \(p,\,p+1,\,p+2,\,\ldots,\,q-1,\,0,\,1,\,\ldots,\,p-1\)~,$$
  corresponding to a cyclic permutation of the integers between zero and
  $q-1$. Filom and Pilgrim \cite{FP} show that this combinatorics
  is unobstructed in all cases, yielding  maps of Type B which they
  denote  by $f_{p/q}$. Note that $f_{p/q}$ is of shape $+-+$ except in the
  two extreme cases $p=1$ and $p=q-1$ where it is co-polynomial.
  Under orientation reversal, we have $\I\<f_{p/q}\>= \<f_{(q-p)/q}\>$.

  \begin{figure}[ht!]
  \begin{center}
    \begin{overpic}[width=5.5in]{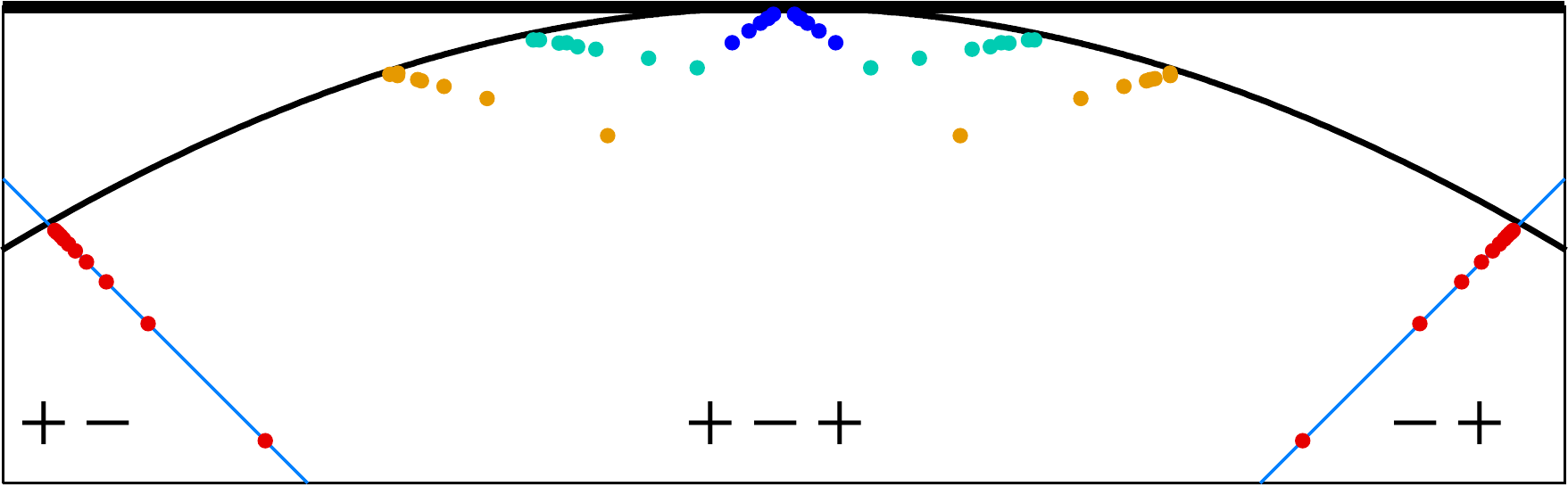}
      \put(350,100){\rm SL}
      \put(45, 100){\rm SL}
   \end{overpic}
  \caption{\label{F-FP} Part of the $+-+$ region in $\M$,
    with $\Delta\ge .75$.
    The red dots represent those Filom-Pilgrim conjugacy classes
    $\FP(p,\,p')$ for which the minimum of $p$ and $p'$ is $1$; 
    the orange dots represent points with ${\rm min}(p,\,p')=2$, and the
    green points have minimum $3$. On the other hand,
    the blue dots represent points with $p'=p\pm 1$ and with $p,\,p'\ge 4$.
    The  black curve is $\Per_1(1)$.}
\end{center}
\end{figure}

  For our purposes 
it will be convenient to introduce the notation $\FP(p,\,q-p)$
  for the conjugacy class $\<f_{p/q}\>$ in moduli space.
  Thus the point $\FP(p,\,p')\in\M$ is well defined for every pair of
  strictly positive coprime integers $p$ and $p'$. Here $p'$ is the number of
  iterations needed to map the first critical point to the second, and $p$
  is the number needed to map the second critical point back to the first.
  The orientation reversing involution satisfies $\I:\FP(p,p')\leftrightarrow
  \FP(p',\,p)$

  \msk
  
  \begin{conj} As $p'\to\infty$ with fixed $p$, the conjugacy classes
  $\FP(p,p')$ tend to a well defined limit $\FP(p,\,\infty)$
  which belongs to the
  curve $\Per_1(1)$ in moduli space; and similarly, as $p$ tends to infinity
  with fixed $p'$ there is a well defined limit $\FP(\infty,\,p')\in\Per_1(1)$.
  On the other hand, if both $p$ and $p'$ tend to infinity, then the limit
  is the ideal point with coordinates $(\Sigma,\,\Delta)=(-.5,\,1)$.
  \end{conj}
    \msk
    
    We don't know why these statements should be true;
    but empirical evidence   certainly suggests them. (See \autoref{F-FP}.)
    Furthermore the following is known: 

\begin{prop}[{\bf Filom and Pilgrim}] \label{P-FP} The topological entropy of
      $F(p,p')$ depends only on
      the sum $p+p'$, and is an explicitly computable number which converges 
      monotonically to $\log(2)$ as $p+p'\to\infty$.
    \end{prop}

    See \cite[Proposition 3.2 and Lemma 4.1]{FP}. This clearly implies at
    least that   $F(p,p')$ converges towards the set $\Per_1(1)$ as
    $p+p'\to\infty$, since the topological entropy is a continuous function
    on $\M$ which takes the  value $\log(2)$
    only on the closure of the shift locus. (See \cite[Prop. 3.6]{F}.) 
    For graphs of individual maps see Figures \ref{f-aBn2}R, \ref{f-aBn3}L,
    \ref{f-aBn4}R, \ref{f-aBn4a}R, \ref{f-aBn5a} as well as \ref{F-FP2}.
\end{rem}
    \msk
    
  \begin{rem}\label{R-info} 
  Here are descriptions of some special points of $\M/\I$.\vspace{-.1cm}

  \begin{itemize}
\item[$\bullet$] The lower end 
  points of the two curves $\Per_1(\pm 1)$ occur at the
  conjugacy classes of $\displaystyle{f(x)=\frac{\pm x}{x^2+1}}$, with coordinates
$(\Sigma,\,\Delta)= (\pm 0.5,\,0.295167)$.\smallskip

\item[$\bullet$] The center point of the
figure, with $(\Sigma,\,\Delta)=(0,\,0.5)$, is represented by the
critically finite map $\displaystyle{x\mapsto\frac{x^2-1}{x^2+1}}$ of Type C.
(Compare \autoref{f-aC1}(left).) \smallskip

\item[$\bullet$] The crossing point between the two curves $\Per_1(\pm 1)$
is represented by the maps $\displaystyle{f(x)=\frac{x}{x^2\pm x\sqrt{2}+1}}$,
with fixed points of multiplier  $1$ at the origin and multiplier~$-1$ at
{$x=\pm\sqrt{2}$}. Here $\Sigma$ is 0.25 or 0.75, and $\Delta=0.60817$.

\item[$\bullet$] The crossing point between $\Per_1(1)$ and the locus of
  co-polynomial points is represented by the map
  $\displaystyle{x\mapsto \frac{x}{1-3x+x^2}}$, with a fixed point
of multiplier $1$ at $x=0$. Here $(\Sigma,\,\Delta)=(-0.11353,\, 0.88647)$.

\item[$\bullet$] The end point of the Chebyshev curve on $\Per_1(1)$
  has coordinates $(0.051576,~ 0.776737)$. This is the only conjugacy
  class for which the unique fixed point has multiplier one. 

\end{itemize}
\end{rem}

\begin{figure}[!htb]
\centerline{\includegraphics[width=6in]{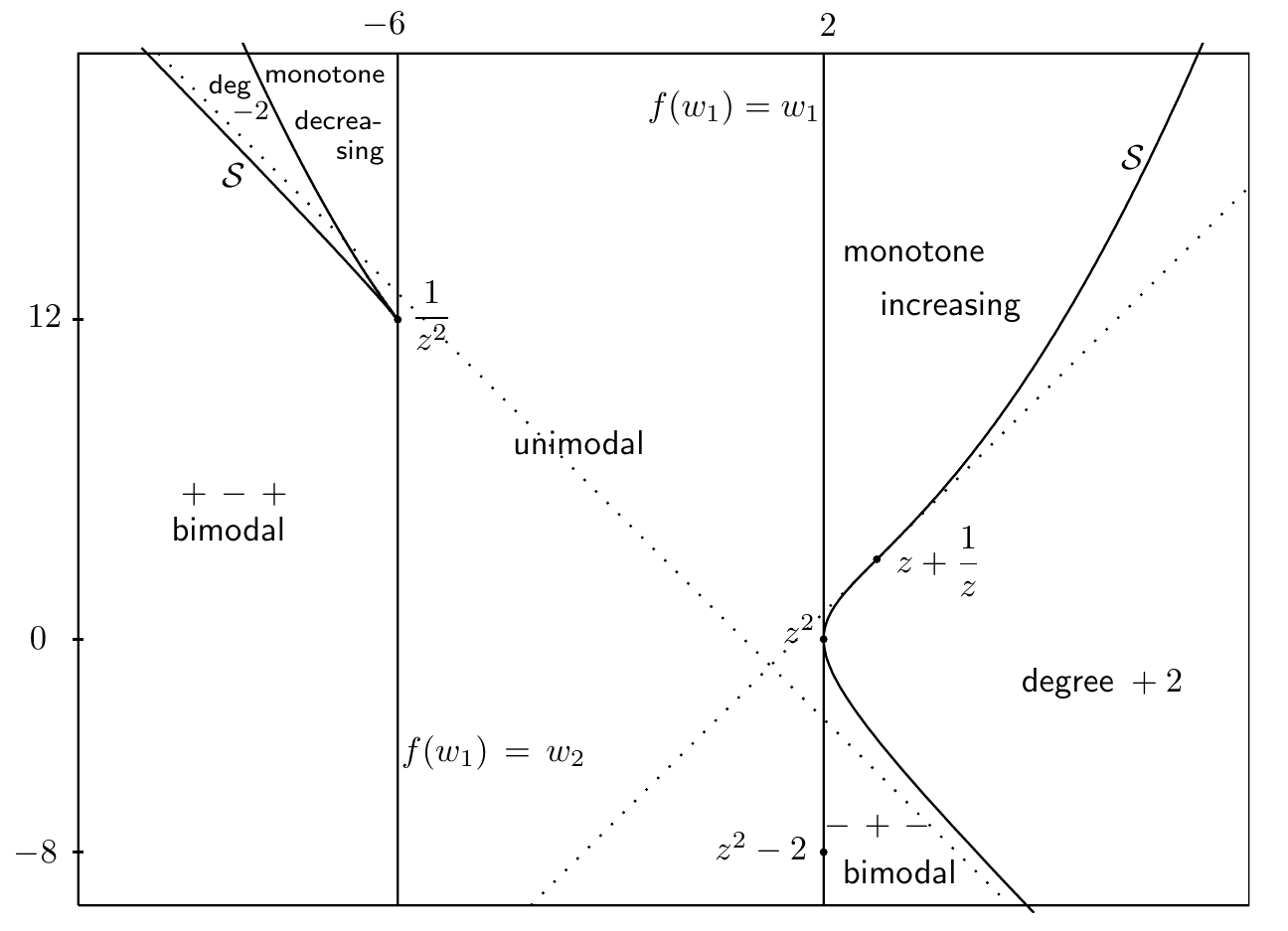}}
\caption{\label{F-old}  A composite moduli space including also maps
of degree $\pm 2$. Here the coordinates are the first two symmetric
functions of the fixed point multipliers. The vertical line through
the class of $z\mapsto z^2$ is the locus of polynomial  shape maps,
while the line through $z\mapsto 1/z^2$ is the locus of co-polynomials.
The dotted lines of slope $\pm 1$ represent maps with a fixed point
of multiplier $\pm 1$. The
figure should be cut open along each component of the symmetry
locus $\mathcal S$: There is a well defined limit as we approach $\mathcal S$
from either side; but the two limits are related only by a complex
change of coordinate. }
\end{figure} 

\begin{rem}[Comparing the two pictures of $\M/\I$] 
  \label{R-compare}
\autoref{F-M/I} can be compared
with \autoref{F-old}  (taken from \cite{M}), which shows not only
$\M/\I$, but also the corresponding moduli spaces for maps of
degree $\pm 2$, all in one figure.\footnote{For a 
  more colorful version, see \cite[Figure 1]{F} or \cite[Figure 2]{FP}.}
 Notice that \autoref{F-old}  is upside down in comparison to
\autoref{F-M/I}, so that the top of one figure corresponds to the
bottom of the other. Furthermore the change of coordinates is not
at all linear, so that small features in one figure can be quite
large in the other. Note also that \autoref{F-M/I} shows all of
$M/\I$, while \autoref{F-old} shows only a central region of
$\M/\I$.
\end{rem}

\section{Obstructions}\label{s5}
  
  The combinatorics will be called \textbf{\textit{obstructed}} if there
  is no corresponding rational map, and \textbf{\textit{unobstructed}}
  otherwise. In the unobstructed case, the corresponding rational map
  is always unique up to conjugacy. Furthermore, in most cases the
  iterated Thurston pull-back map will converge to the required rational map.
  (For the essentially 
  unique exceptional case, see \autoref{s6}.)

  By a theorem of Rees, Tan Lei, and Shishikura,
  any quadratic Thurston  map is obstructed if and only if
  it has a Levy cycle, which is a particularly simple form of
  Thurston obstruction. (See \cite{T}.)
  \bigskip
  
\begin{definition}\label{D-Levy} A \textbf{\textit{Levy Cycle}}
    of period $p$  for a Thurston map
 $\f:\widehat\C\to\widehat\C$   with postcritical set $P_\f$,
 is a list of disjoint simple closed curves 
  $\Gamma_j\subset \widehat\C\ssm P_\f$, indexed by integers $j$
  modulo $p$, with the following two properties:
\begin{itemize}
\item[$\bullet$] Each component of the complement of $\Gamma_j$
  contains at least two points of $P_\f$.

\item[$\bullet$] For each $j$ there
  is a connected component of $\f^{-1}(\Gamma_j)$ which maps
  bijectively onto $\Gamma_j$, and which is homotopic within
  $\widehat\C\ssm P_\f$ to $\Gamma_{j+1}$.
  \end{itemize}
  \end{definition}
  \bigskip
  
  Given some arbitrary combinatorics of shape $+-+$, 
  we do not know any general procedure for deciding whether or not there
  is a Levy cycle.  In practice  we will 
  proceed simply by carrying out the Thurston algorithm to see whether
  it converges.  However, for all of the minimal obstructed cases that we have
  found, 
  it is not too difficult to construct a corresponding Levy cycle, 
using \autoref{T-ara} below. 
  \ssk

   \begin{figure}[!hbt] 
  \centerline{%
    \includegraphics[height=1.25in]{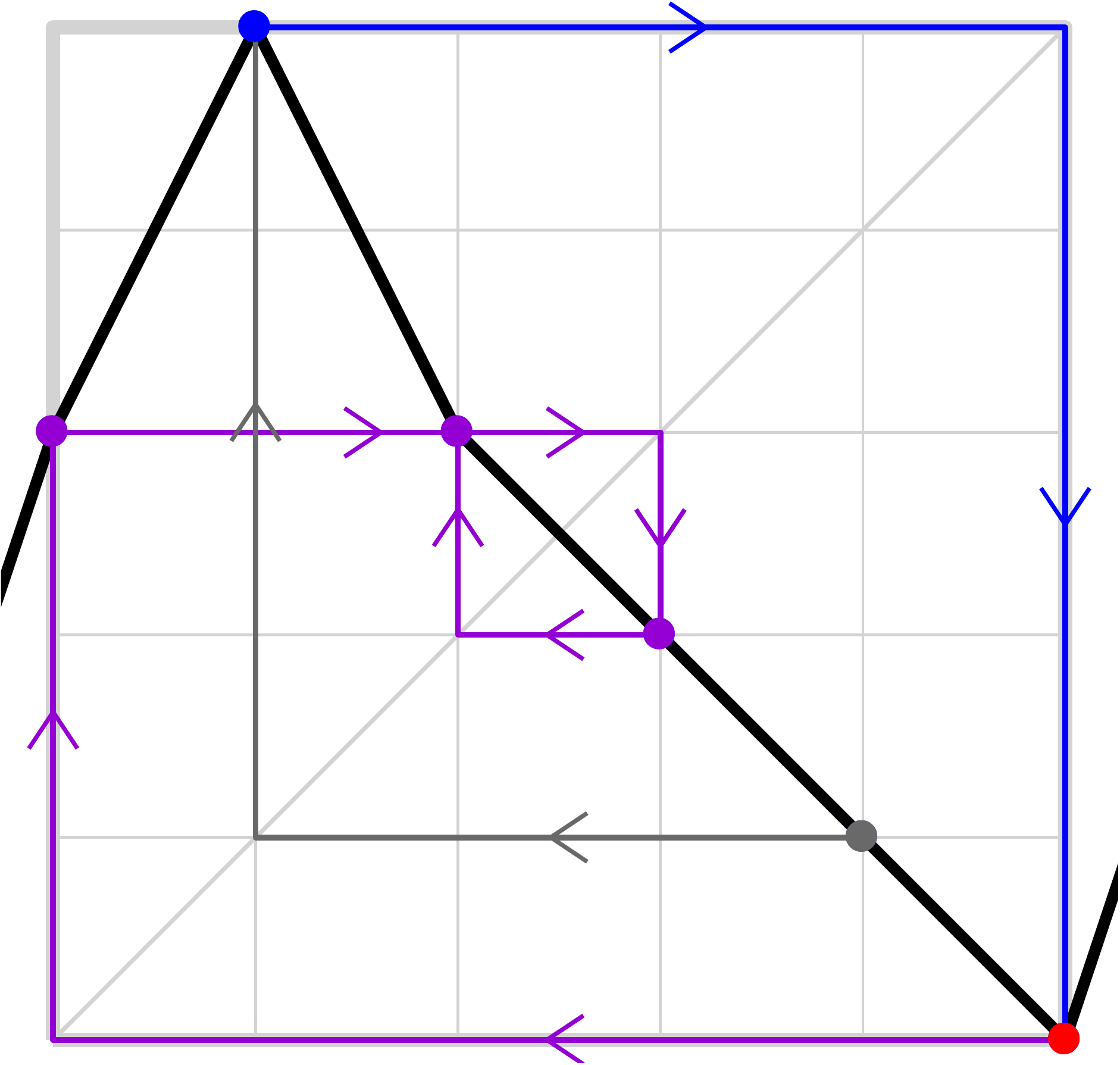}
    \includegraphics[height=1.25in]{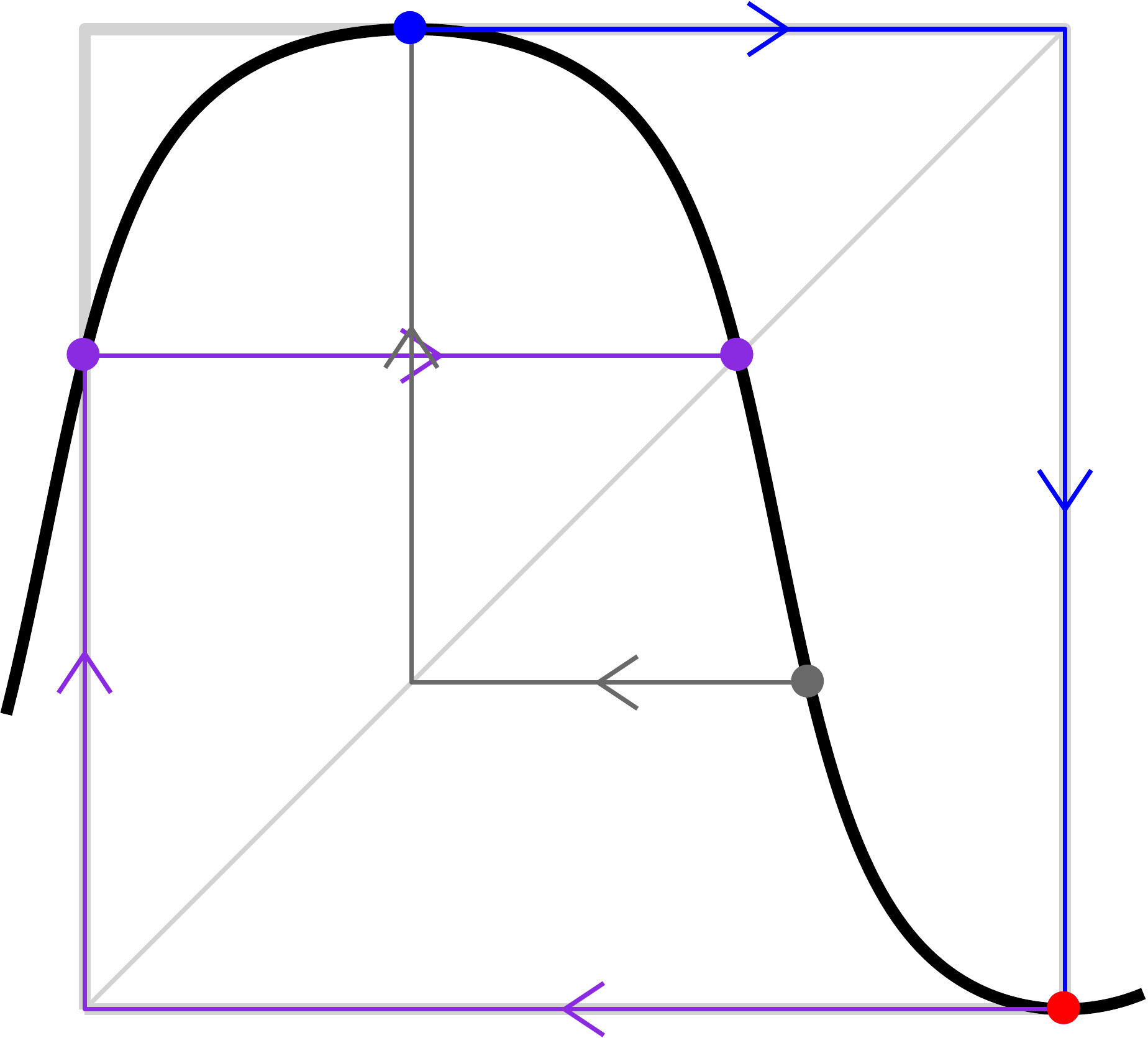} \hfill
    \includegraphics[height=1.25in]{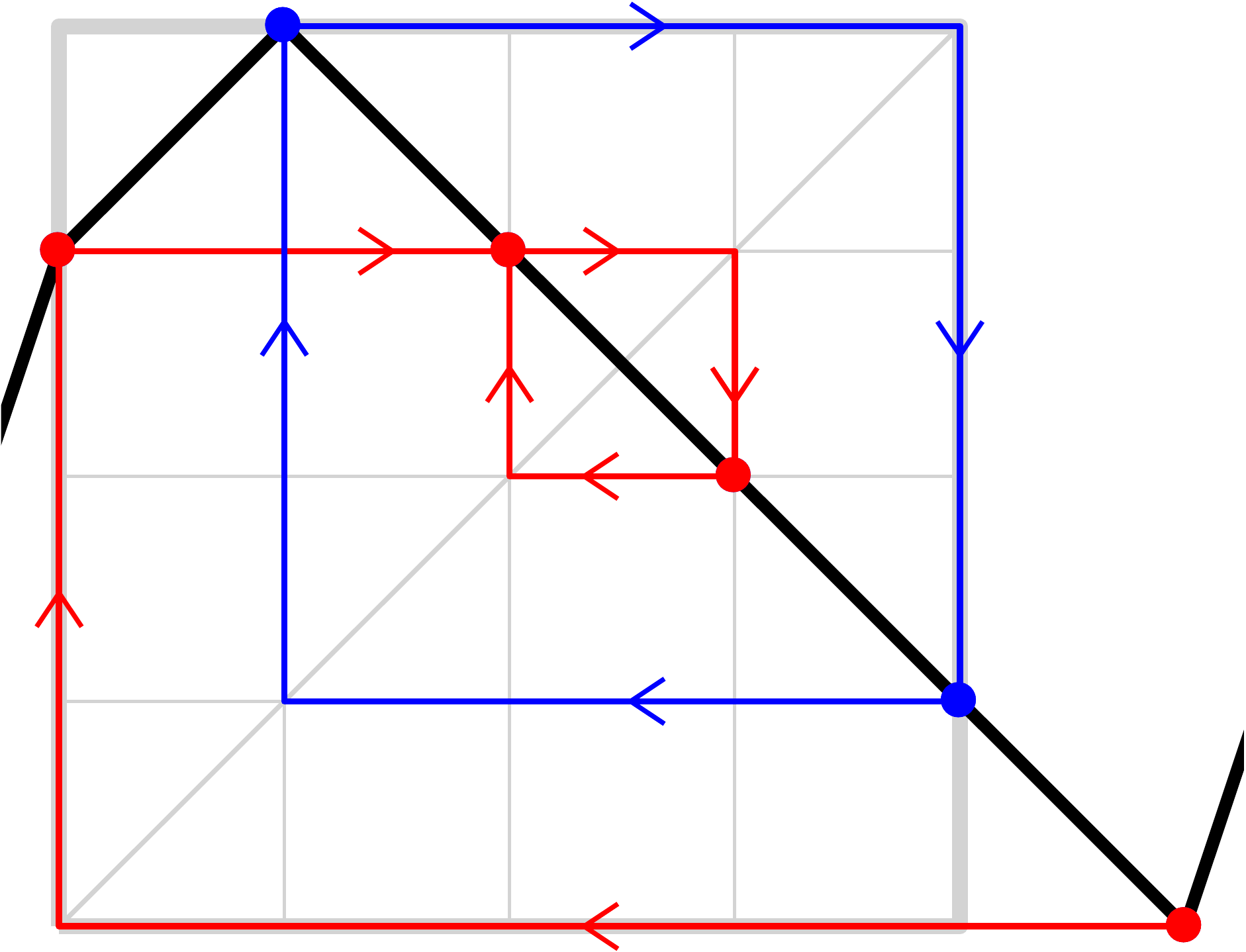}
    \includegraphics[height=1.25in]{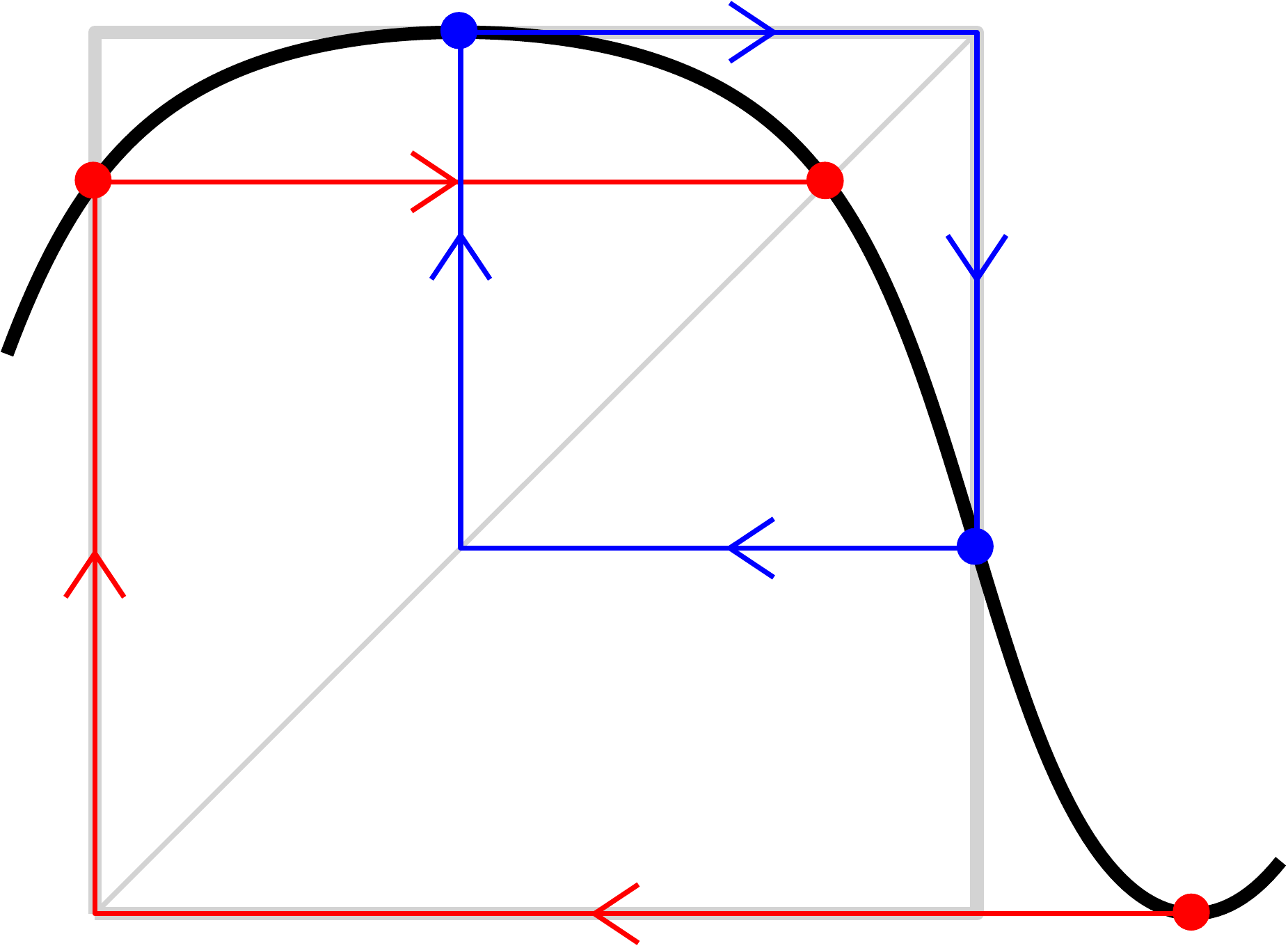}}
  \caption{\label{f-weakobs}  Piecewise linear model and the
    limiting lifted map for two  weakly obstructed maps.
    On the left is shown a map with combinatorics $\(3,\,5,\,3,\,2,\,1,\,0\)$
    and with the mapping pattern 
    $x_4\mapsto \du{x_1}\mapsto \du{x_5}\mapsto x_0
        \mapsto x_3\leftrightarrow x_2~.$  On the right, a strictly unimodal
    map with combinatorics $\(3,\,4,\,3,\,2,\,1,\,0\)$ and mapping pattern
    $\du{x_5}\mapsto x_0\mapsto x_3\leftrightarrow x_2$ and
    $\du{x_1}\leftrightarrow x_4$.}
\end{figure}
\bigskip

\begin{definition} An obstruction will be called \textbf{\textit{weak}}
if the rational maps $f_j$ constructed during the iterated pull-back
transformation converge locally uniformly to a 
critically finite   map (but with simplified combinatorics,
 as defined in \autoref{R-bad}).
Compare Figures~\ref{F-nonex} and \ref{f-weakobs}.
This can never happen if we start out with minimal combinatorics.
(In the analogous case of polynomial maps, it follows from
Selinger \cite[Prop. 6.2]{S} that weak obstructions are the only
kind which can occur; but this is far from true for quadratic
rational maps.)
\msk

 The obstruction will be called  \textbf{\textit{strong}}
 if the fixed point multipliers $\mu_j$ converge
 to $\pm\infty$. Every combinatorics of $-+-$ bimodal  shape is strongly
 obstructed (compare \autoref{L2}); and
 there are many strongly obstructed examples of $+-+$  shape
 (see \autoref{a2}). On the other  hand, for unimodal
 combinatorics we will show in \autoref{s8} that there  cannot
 be any strong obstruction except possibly in the totally
 non-hyperbolic case.
\end{definition}
\msk

 Note that $\mu_j$ always converges to $+\infty$ in the $-+-$
 case, and to $-\infty$ in the $+-+$ obstructed case.\msk

\begin{figure} [!bht]
  \centerline{%
    \includegraphics[height=1.27in]{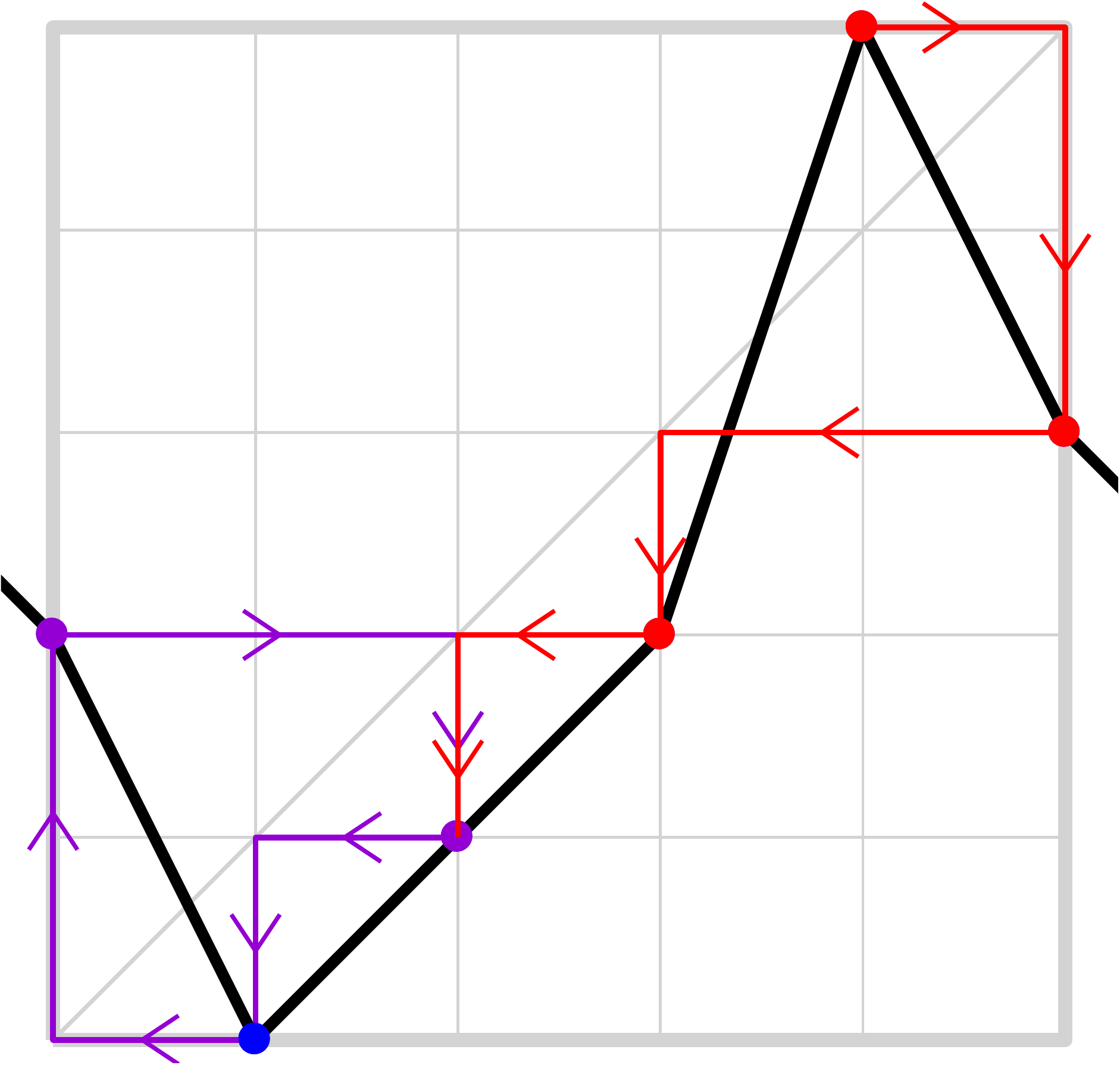} \hfil
    \includegraphics[height=1.27in]{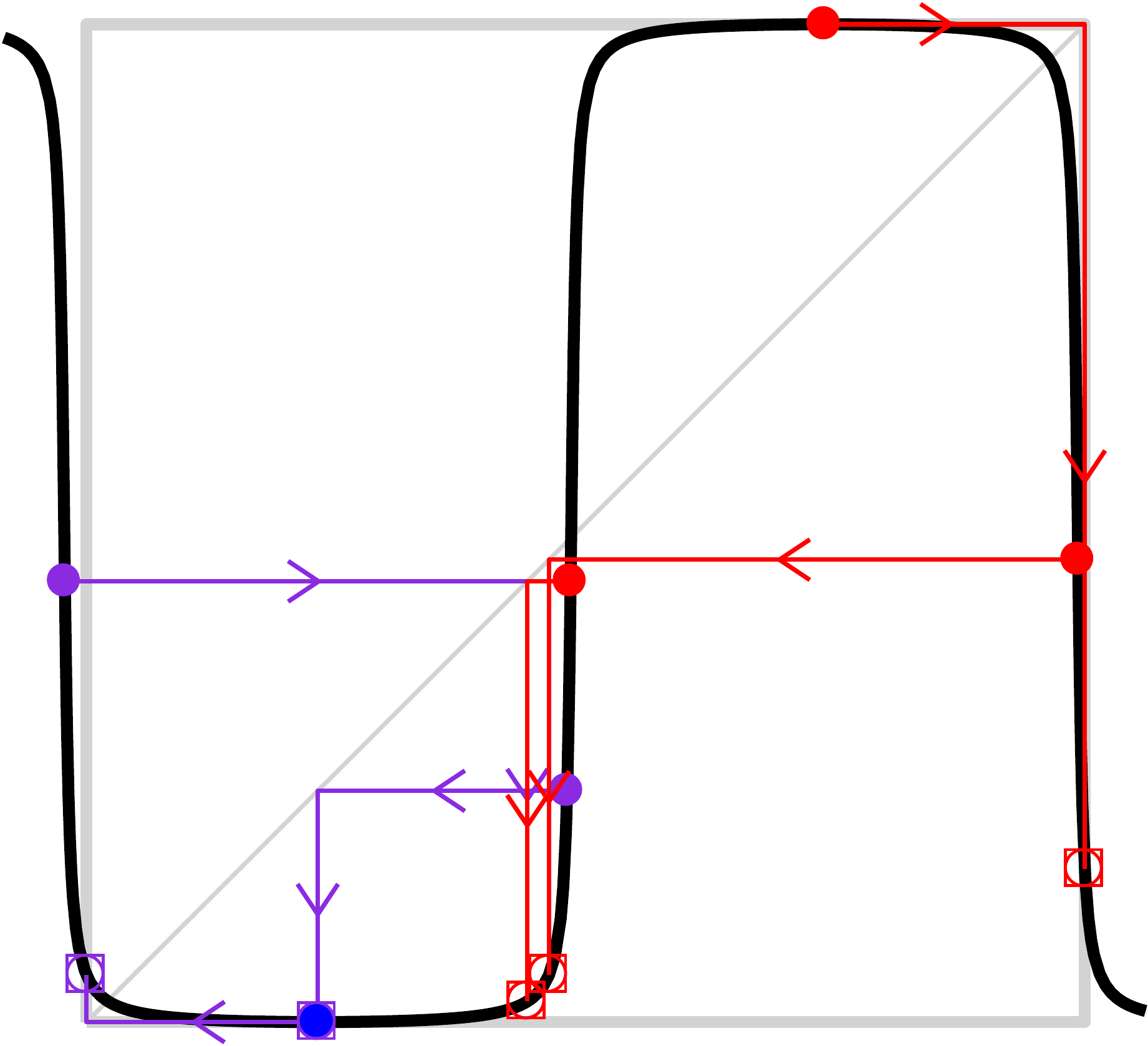} \hfil
    \includegraphics[height=1.27in]{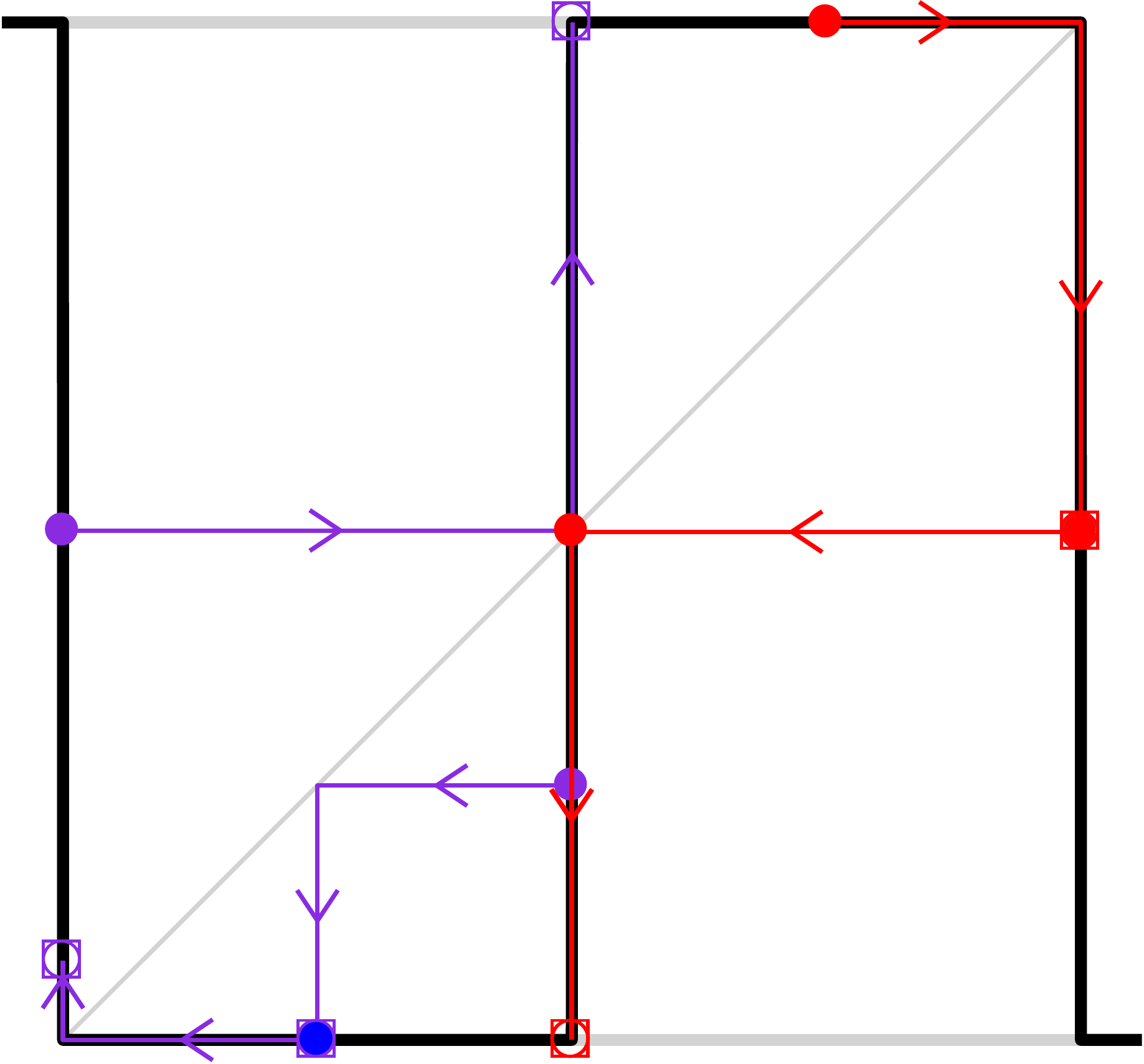}}
  \caption{\label{f-sq}   Piecewise linear model (left) and
    limiting lifted map (right) for the strongly
    obstructed combinatorics $\(2,\,0,\,1,\,2,\,5,\,3\)$.
    In the center is shown the result after two pullbacks.
    Of course what is actually shown on the right represents the
    quadratic map for some high
    iterate of the Thurston pull-back. Since these iterates are
    ``trying'' to duplicate  combinatorics which cannot be realized,
    the pullback can never approximate the required dynamics.
    As in \autoref{F-algExamp} of \autoref{s4}, the image of
    a marked point $t_j$ under $F_\ell$ is indicated by an open square.  In
    this example, the marked points $t_2$ and $t_3$ become arbitrarily close
    together while their images under the combinatorics ($t_1$ and $t_2$)
    remain distinct.
  }\end{figure}

 If we consider the ideal boundary of moduli space, as
 described in \autoref{s7}, then empirically,  in all strongly  
 obstructed cases, the following seems to be true. Either: 
 
    \begin{enumerate}[beginpenalty=10000,midpenalty=9999] 
   \item the combinatorics is $~+-+~$ bimodal, and the conjugacy
 classes  $\langle f_j\rangle$ converge to the center  point 
 $(\Sigma,\,\Delta)=(-0.5,\,1)$ of  ideal boundary of the $~+-+~$ region,
 or

   \item the combinatorics is $~-+-~$ bimodal, and the conjugacy classes
 converge to the corresponding center point $(+0.5,\, 1)$ for the
 $~-+-~$ region.
 \end{enumerate}

\begin{figure}[!htb]
   \begin{center}
     \begin{overpic}[height=.4\textwidth]{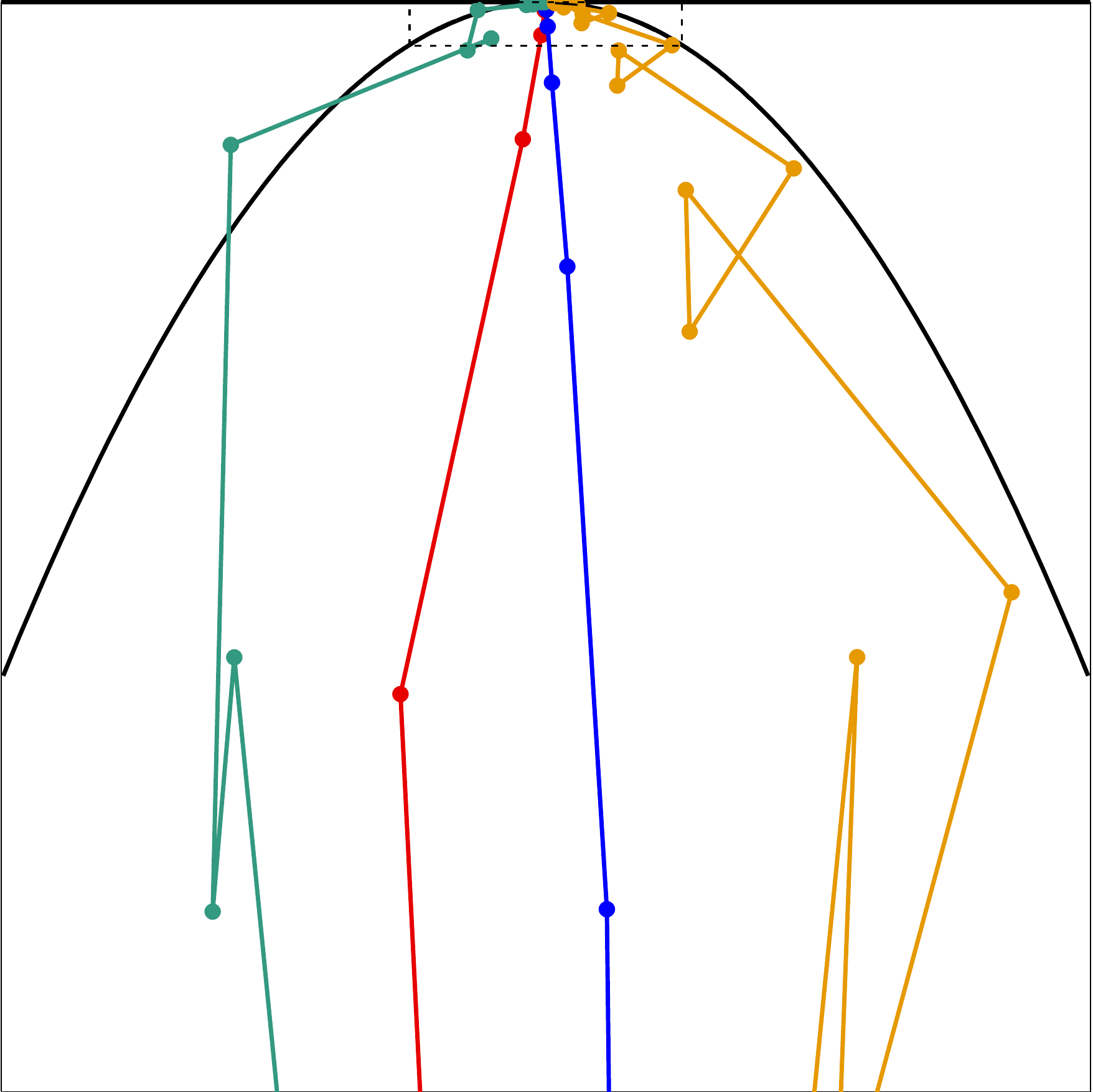}
       \put(10,165){\rm SL}
       \put(160,165){\rm SL}
       \end{overpic}\qquad
       \begin{overpic}[height=.4\textwidth]{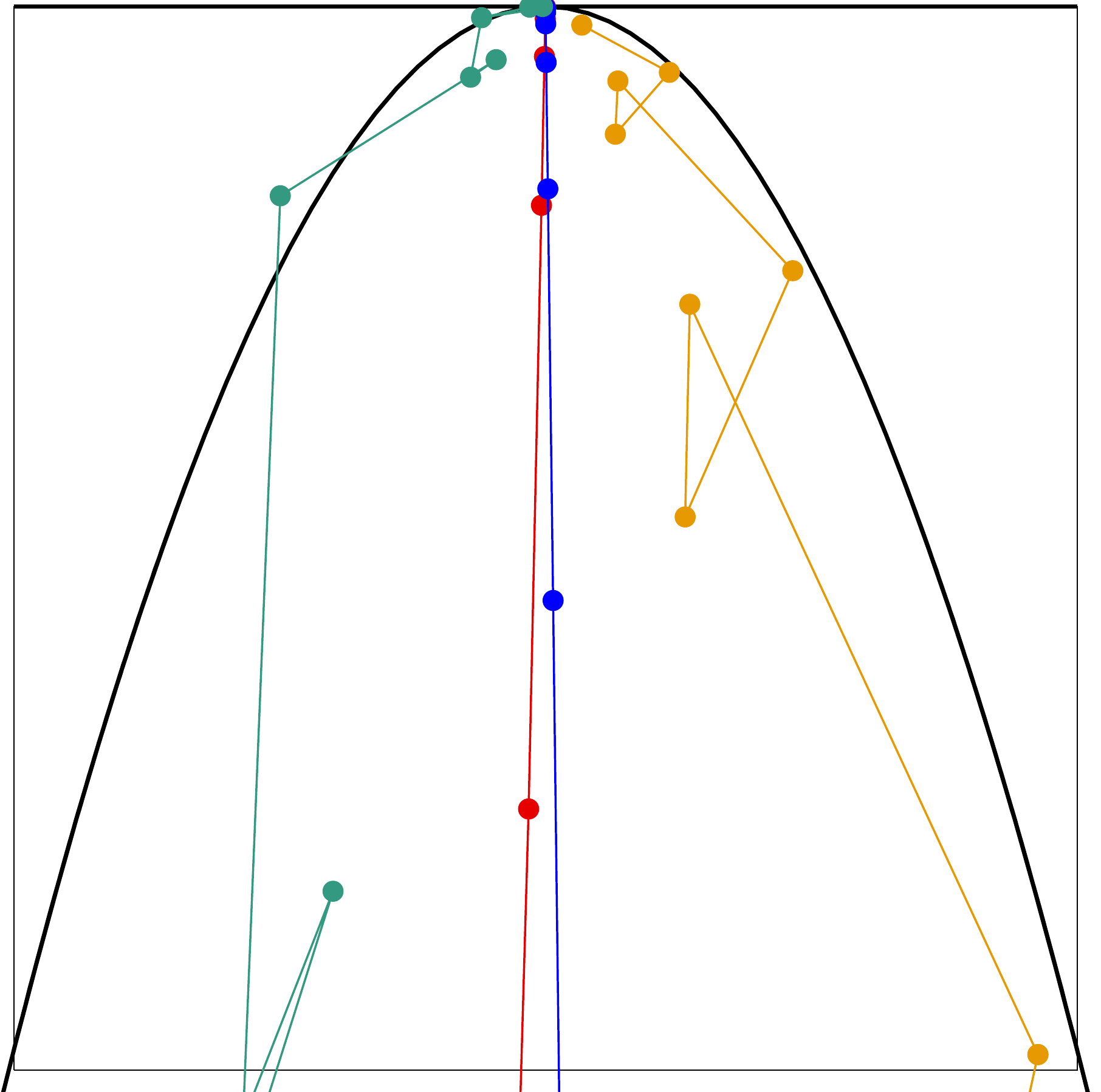}
         \put(10,165){\rm SL}
         \put(160,165){\rm SL}
       \end{overpic}
     \vspace{2ex}
     \begin{center}{\footnotesize
   \newcommand{\clne}[1]{\includegraphics[width=5em]{#1}}
   \begin{tabular}{c c c c}
$\(2,3,4,6,4,0,1\)$&$\(2,4,3,0,1\)$   &$\(3,5,4,0,1,2\)$ &$\(4,5,0,1,2,3\)$\\
\clne{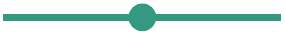}   &\clne{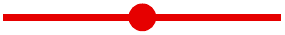}&\clne{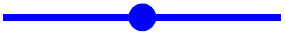}   &\clne{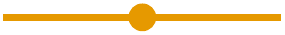}\\
$\kappa=-\infty$   &$\kappa=-1$    &$\kappa=(\sqrt{5}-1)/2$&$\kappa=+\infty$\\
                   &(\autoref{f-aDn4a}L)&               &(\autoref{f-aDn4b}R)
     \\[-2ex]  
   \end{tabular}
 }\end{center}
 \caption{\label{f-sigdelgr}
These images illustrate the convergence of four obstructed maps of
topological shape $+-+$ toward the point $(\Sigma,\Delta)=(-0.5,\,1)$ in
moduli space (see \autoref{s7}). Each colored broken line
indicates $(\Sigma,\,\Delta)$ for consecutive 
steps of the pull-back; while the black curve is $\Per_1(1)$ which is the
lower boundary of the shift locus (abbreviated as SL) in this region.
The left half of this figure
shows the window  $(\Sigma,\,\Delta)\in [.8, 1.2]\times [.95, 1)$
with the vertical scale exaggerated by a factor of 8. The right half
corresponds to the much smaller
window $[.95,\,1.05]\times[.998,\, 1)$ 
(indicated in dotted lines on the left), with the vertical scale exaggerated
by a factor of 50.
Observe that the green line enters the
shift locus several times  before converging to the point $(-0.5,\, 1)$
in moduli space.  The last line lists the conjectured limiting value of
$\kappa$ as we converge to this ideal point.
}
\end{center}
\end{figure}

 \begin{figure}[!htb] 
    \begin{center}
     \begin{overpic}[height=.4\textwidth]{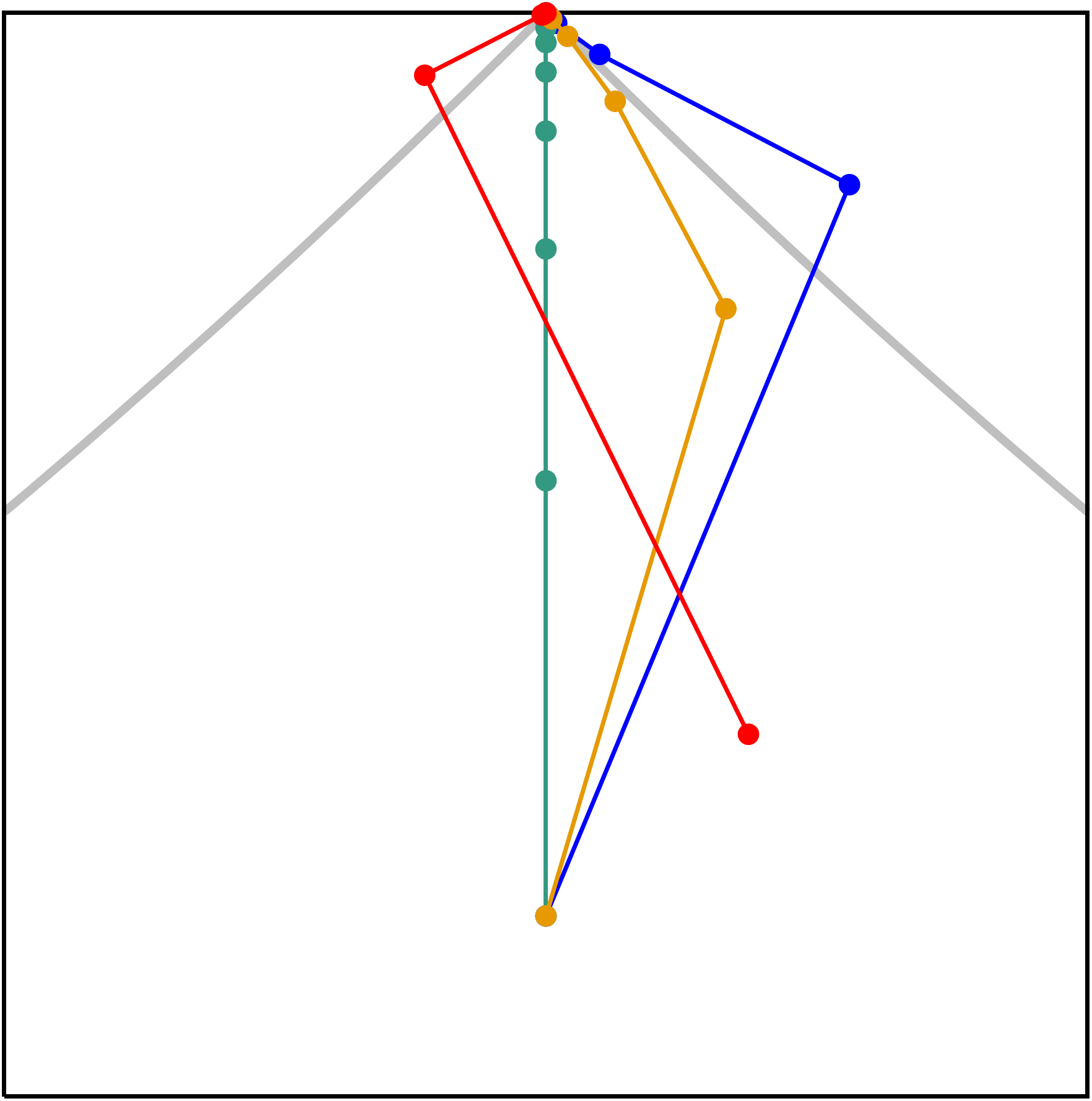}
       \put(20,165){\rm SL}
       \put(160,165){\rm SL}
     \end{overpic}\qquad
    \begin{overpic}[height=.4\textwidth]{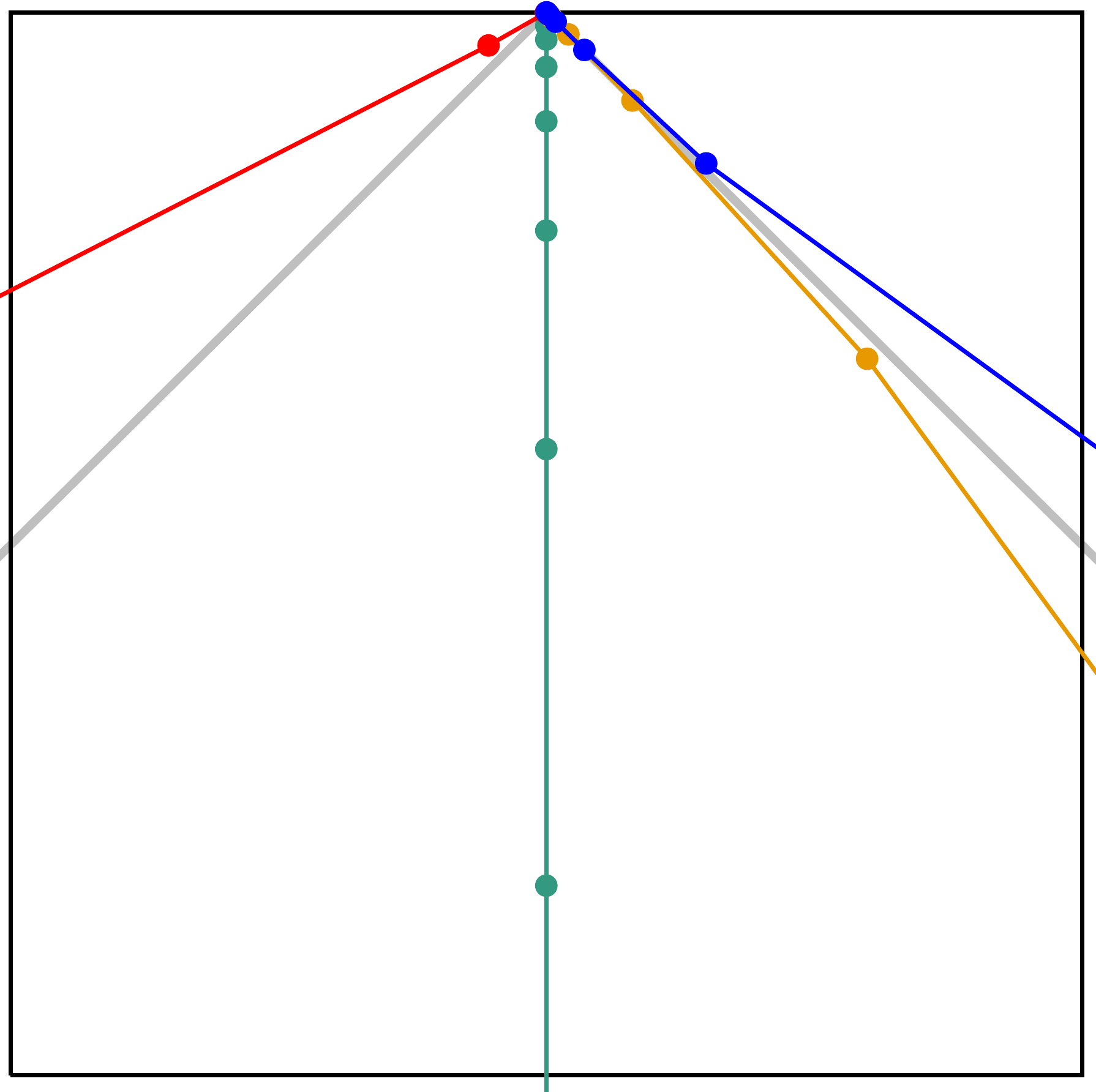}
       \put(10,165){\rm SL}
       \put(160,165){\rm SL}
     \end{overpic}\qquad
  \begin{center}{\footnotesize
  \vspace{2ex}

   \newcommand{\clne}[1]{\includegraphics[width=5em]{#1}}
   \begin{tabular}{c c c c}
$\(2,0,3,5,6,5,3\)$     &$\(1,0,3,2\)$    &$\(1,0,3,4,2\)$  &$\(1,0,1,4,2\)$ \\
\clne{redline.pdf}      &\clne{greenline} &\clne{orangeline}&\clne{blueline} \\
$\kappa=+1$             &$\kappa=0$       &$\kappa=-1$     &$\kappa=-1$\\
(\autoref{f-nonhyp-obst1}L)               &          &       &  (\autoref{f-aC3}R)   \\[-2ex]
 \end{tabular}
}\end{center}
 \caption{\label{f-sigdelgrobs}
This image illustrates the convergence of four obstructed maps of
topological shape $-+-$ toward the point $(\Sigma,\Delta)=(0.5,\,1)$ in
moduli space (see \autoref{s7}). Each colored broken line
indicates $(\Sigma,\,\Delta)$ for successive steps of the pull-back. The
Chebyshev curve $|\kappa|=1$, shown in gray, is the lower boundary of the
shift locus $|\kappa|>1$ in this region.   It satisfies
\hbox{$\Delta\,=\, 1-|\Sigma-.5|+{\mathcal O}\big((\Sigma-.5)^2\big)$;} 
see \autoref{T-asymp}. The  figure on the left
shows the window \hbox{$(\Sigma,\,\Delta)\in [.9,\,1.1]\times [.8,\, 1)$}. The
figure on the right is a zoom of the upper left by a factor of 15. Note
that both the red and the blue curves seem to  stay inside the shift locus
while converging to the ideal point  $(\Sigma,\Delta)=(0.5,\,1)$;
but that the yellow and  green curves remain outside the shift locus
while converging to the same point. (In fact the green curve is contained
in the line $\Sigma=.5$.)  Near 
  $(0.5,1)$, the blue and the yellow curves get extremely close to the
boundary of the shift locus. Again the conjectured limiting values of $\kappa$
are indicated.}
\end{center}
\end{figure}

\noindent It is interesting that  $(\Sigma,\,\Delta)=(\pm .5,~1)$ 
 are the only two points in the upper
ideal boundary which can be approximated by conjugacy classes
which are not in the shift locus. (Compare \autoref{F-M/I}.)
\bigskip

\begin{conj}\label{C-noSL} 
  For any admissible combinatorics, either one or both of the following
  two conditions must hold:\ssk
  \begin{enumerate}[topsep=-.5ex,itemsep=0ex,label=\textbf{\upshape{(\alph*)}},
    ref=(\alph*)]
  \item\label{C-noSL-finite}
    at most finitely many of the successive conjugacy classes
  generated by the Thurston algorithm belong to the shift locus, and/or

  \item\label{C-noSL-ideal} the successive  conjugacy classes converge to
    one of the two ideal points with coordinates 
    $(\Sigma,\,\Delta) =(\pm .5,~1)$.
  \end{enumerate}
\end{conj}

Clearly \ref{C-noSL-finite} must be satisfied in the unobstructed
case. However, Figures~\ref{f-aC3}R and \ref{f-nonhyp-obst1}L 
illustrate an obstructed case of shape $-+-$ where the successive conjugacy
classes appear to converge to $(.5,\,1)$ through the shift locus.
Figure~\ref{f-sigdelgr} illustrates an obstructed case of shape $+-+$ where the
successive conjugacy classes enter the shift locus at least three times.\msk

In Case~\ref{C-noSL-ideal}, the direction of approach to the ideal point
depends of the limiting behavior of  the Epstein parameter $\kappa$.
(Compare \autoref{T-asymp}.) 
In many strongly obstructed cases, $\kappa$ tends to a finite limit, 
and the convergence seems quite orderly. However there are also
cases where the convergence is much wilder, and this is particularly
true in cases where $\kappa$ tends to infinity. Various possibilities are
illustrated in Figure~\ref{f-sigdelgr}. 
\msk

Clearly \autoref{C-noSL} would have the following consequence.

\begin{conj}\label{C-2lims} For strongly obstructed combinatorics,
  the only possible
  limit is $(-.5,\,1)$ in the $+-+$ case, or $(+.5,\,1)$ in the $-+-$ case.
But in the unimodal case, every minimal 
 combinatorics is unobstructed.
  \end{conj}

For a strongly obstructed case the lifted form of the limit map
will be a square wave; as illustrated  in \autoref{f-sq} in the
$~-+-~$ case, or as illustrated in \autoref{f-aBobst} in the $~+-+~$ case.
(See also Figures~\ref{f-aC3}, \ref{f-aDn4a}, \ref{f-aDn4b}.)

\begin{figure}[htb!]
  \centerline{\includegraphics[height=2.2in]{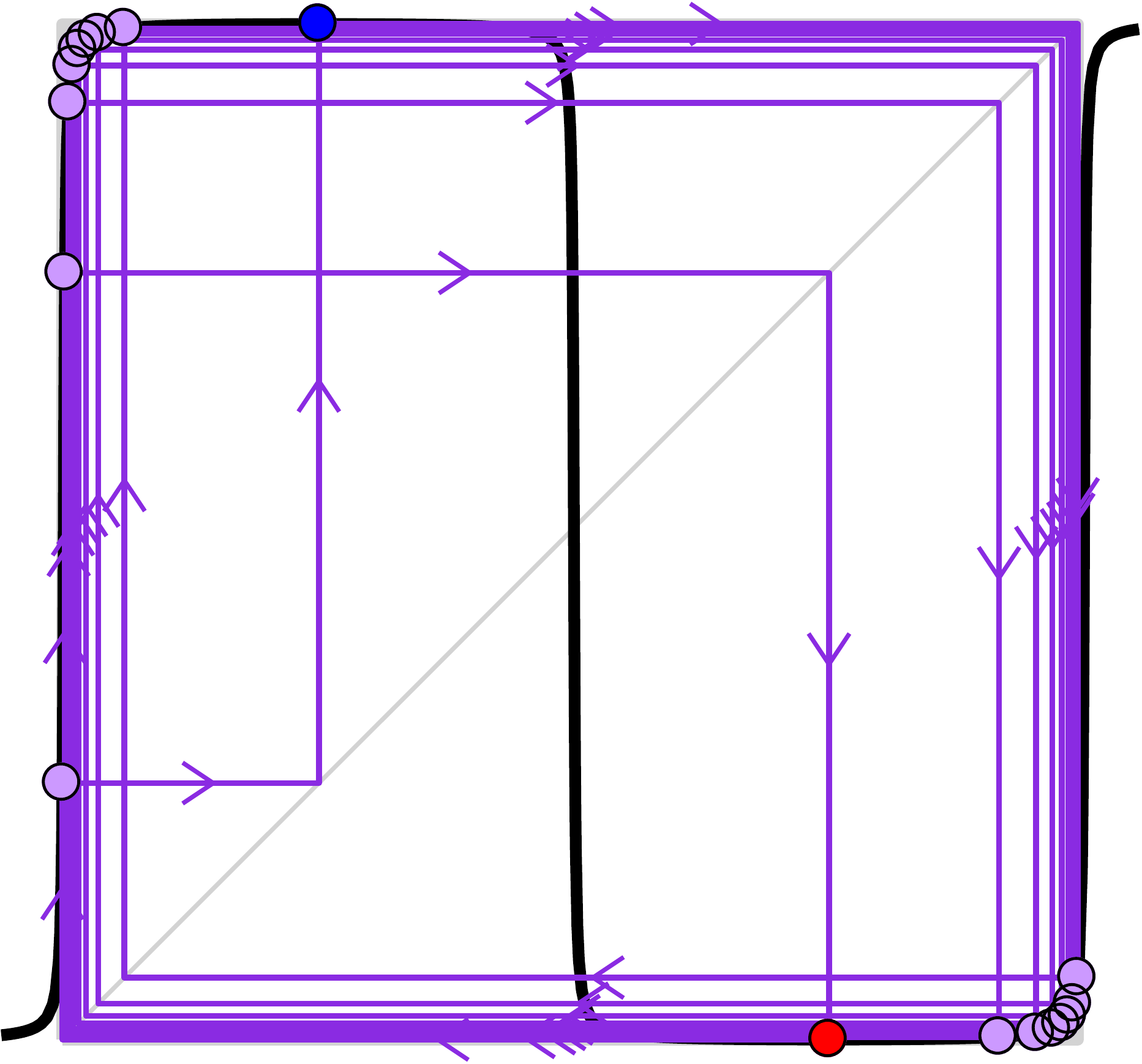}}
  \caption{\label{F-FP2}   This figure is the graph of a Filom-Pilgrim
    map of class $\FP(8,\,9)$. (Compare \autoref{R-FP}.) This is an
    honest smooth critically finite rational map of Type B and shape
    $+-+$. However it is so close to the ideal point $(\Sigma,\Delta)=
    (-.5,\,,1)$ of moduli space, that its graph looks very much like the graph
    of an obstructed map.}    \end{figure}

\msk
\textbf{Caution.} One can't be sure by looking at a graph that the
combinatorics is obstructed. In fact any
critically finite conjugacy class which is sufficiently close
to the ideal boundary will have a graph which cannot be distinguished
from a square wave. See \autoref{F-FP2}.
\msk

\subsection*{Constructing Levy Cycles}
Every combinatorics $\vecm=\(m_0,\cdots, m_n\)$ gives
rise to a corresponding
Markov partition of the interval $[0,n]$ into $n$ subintervals
$I(j)=[j-1,~j]$. The associated piecewise linear map~$\f$  sends
each $I(j)$ linearly onto some union of one or more consecutive
subintervals. By a \textbf{\textit{periodic orbit}}
of period $p$ for the associated
Markov shift we will mean a list of $p$  of these $I(j)$, indexed by
integers $k$ modulo $p$, and satisfying the condition that
$$ \f\big(I(j_k)\big)~\supset~I(j_{k+1}) $$
for each $k$.
\ssk

\begin{theo}\label{T-ara}
  Let $\vecm$ be an admissible combinatorics
  of  shape $+-+$. If there exists a
  periodic orbit $\{I(j_k)\}$ of period $p\ge 2$ for
  the associated Markov shift which
  satisfies the following three conditions, then there is a Levy
  cycle, and hence the combinatorics is strongly obstructed.

  \begin{enumerate}[label=\textbf{\upshape Condition \arabic*.},
    leftmargin=*, itemindent=2.5\parindent]
 \item Every $I(j_k)$ is contained in an increasing lap of $\f$.

 \item These intervals $I(j_k)$ are disjoint: In particular,
no two have a common end point.

\item If we think of $[0,~n]$ as a subset of the circle $\Rhat$, then the
  correspondence $I(j_k)\to I(j_{k+1})$ has a well defined rotation number.
  In particular, this correspondence preserves the cyclic order of the
  intervals $I(j_k)$   within $\Rhat$.
  \end{enumerate}
\end{theo}

\begin{rem}\label{R-shift} Here any period $p\ge 2$ can actually occur.
  Consider the combinatorics
$\(2,\,3,\,4,\,5,\, \ldots,\, n-1,\, n,\, 0,\,1\)$. If $n$ is odd, 
then this combinatorics has Type D, 
and it is not hard to check that
$$\f : I(1) \stackrel{\cong}{\longrightarrow}
I(3) \stackrel{\cong}{\longrightarrow}
I(5) \stackrel{\cong}{\longrightarrow}\cdots \stackrel{\cong}{\longrightarrow}
I(n) \stackrel{\cong}{\longrightarrow}  I(1) ~,$$
with period $p=(n+1)/2$.
Thus there is a Levy cycle, and the combinatorics is
strongly obstructed. See Figures \ref{f-aBn3a}L and
\ref{f-aDn4b}R for the cases $\(2,3,0,1\)$ and \(4,5,0,1,2,3\).
\end{rem}

\begin{rem}\label{R-evencase}
On the other hand, for $n$ even the combinatorics is of  Type B, and
is  always unobstructed. In fact, in the notation of \autoref{R-FP},
these are just the Filom-Pilgrim conjugacy classes of the form $\<f_{2,\,q}\>$,
  with $q>2$ odd. In \autoref{F-FP}, these are the orange
dots in the left half of the figure (for $q\ge 5$), together with one red dot
on the right corresponding to $\<f_{2,3}\>$.
  (See \autoref{f-aBn4a}R for the case $q=5$  and \autoref{f-aBn5a} for
the case $q=7$~.)
\end{rem}
\bsk 

We will first prove $\autoref{T-ara}$ for the case $p=2$, 
and then for $p>2$.
  The first step for any period $p$ is to
  consider the interval $[0,n]$ as a subset of $\Rhat$, which we think
  of as the equator of the Riemann sphere $\widehat\C$. (Compare
 \autoref{F-RS1}.)  Then extend
  the piecewise linear map on $[0,\,n]$ to a Thurston map $\f:\widehat\C\to
  \widehat\C$ as described in \autoref{L1}.
  \msk
  
\begin{figure}[t!]
  \centerline{ \includegraphics[width=2.8in]{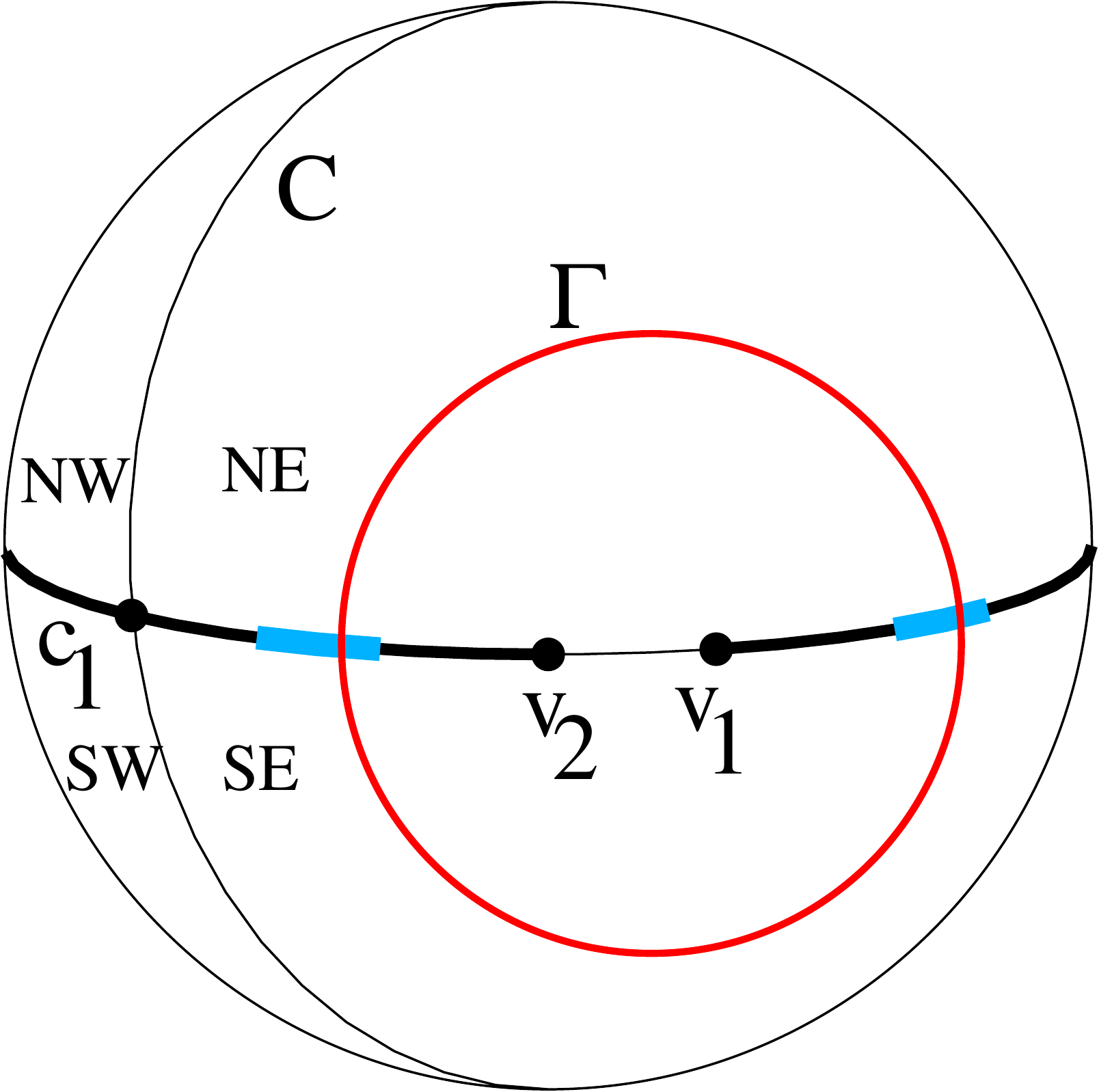}}
  \caption{\label{F-RS}  A picture of the Riemann sphere, illustrating
    the case $p=2$. Here
    $\Rhat$ is represented as the equator, and the pure imaginary axis
    is represented as an orthogonal great circle which intersects the
    equator at the two critical points. The Thurston map $\f$ sends
    $\Rhat$ two-to-one onto $\f(\Rhat)$, which is colored  heavy black,
    except for the two intervals $I(j_0)$ and $I(j_1)$ which are blue.
    The great circle $C$ maps
    two-to-one onto the gap $\Rhat\ssm\f(\Rhat)$ between the two
    critical values. The Levy cycle $\Gamma$, which is colored red,
    passes through $I(j_0)$ and $I(j_1)$. It is important that $\Gamma$
    does not separate the two critical values, or the two critical points.
    Note that  the northeast and
    southwest quadrants both map bijectively onto the northern
    hemisphere under $\f$, while the northwest and southeast
    quadrants both
map bijectively onto the southern hemisphere.}
\end{figure}
\msk

\begin{proof}[{\bf Proof of \autoref{T-ara} for the case $p=2$}]
  In this case, we have two intervals with
  $$\f\big(I(j_0)\big) \supset I(j_1)\quad{\rm and}\quad
  \f\big(I(j_1)\big) \supset I(j_0)~.$$
  (Compare  \autoref{F-RS}.) 
  Therefore there exist interior points $r\in I(j_0)$ and
  $s\in I(j_1)$   which map respectively to interior points
  $r'\in I(j_1)$ and $s'\in I(j_0)$. Our Levy cycle will consist
  of a single simple closed curve $\Gamma$ which is the union of:
  \begin{enumerate}[beginpenalty=10000,midpenalty=9999] 
  \item a path from $r'$ to $s'$ lying in the northeast
    quadrant, above the equator and to the right of $C$,
    
  \item the image of this path under complex conjugation,
    which lies in the southeast quadrant.
  \end{enumerate}\ssk

  \noindent It is not hard to check that the set $\f^{-1}(\Gamma)$ has
  two connected components. One also lies in the eastern
  hemisphere, and is
  homotopic\footnote{This homotopy sends 
 the component onto $\Gamma$
    with degree $-1$ but that is not a problem: The definition of
    Levy cycle makes no reference to orientation.}
  to $\Gamma$ within the complement of the postcritical set,
  as required. The other component lies in the western hemisphere,
  and can be ignored. This completes the proof in the period two
  case.
  \end{proof}\msk

Here is a corollary to the period two case.

\begin{coro}\label{C-per2obs} Every admissible
  combinatorics of  shape $+ - +$
  with a period two critical orbit is strongly obstructed.
\end{coro}
\ssk

\begin{proof}
Let $v_1<c_2<c_1<v_2$ be the critical points and corresponding critical
  values; and suppose for example that $c_2\leftrightarrow v_2$ is
  the period two critical orbit. Let  $I_j$ be the last interval in the
  first lap, with right hand endpoint $c_2$;  and let  $I_n$ be the
  last subinterval, with right hand endpoint $v_2$. Since
  $c_2\leftrightarrow v_2$ it is easy to check that $F(I_j)\subset I_n$
  and that $F(I_n)\supset I_j$. The conclusion follows.   \end{proof}\ssk

As examples, see Figures~\ref{f-aC3}(left), \ref{f-aBn3a}(left),
\ref{f-aDn4a}(left),   \ref{f-aDn4b}(left), and \ref{f-semihyp}(left).
There are also examples with a period two orbit for the Markov shift  but with
no period two critical orbit.
Consider \autoref{f-aBobst},  with combinatorics $\(3,5,4,1,0,2\)$ of
Type B. In this case it is not hard to check that  $F(I_1)\supset I_5$
and that $F(I_5)\supset I_1$. Hence again it follows
that the combinatorics is strongly obstructed. There are even examples
with no periodic critical orbit of any period. 
See \autoref{f-nonhyp-obst}, a  totally non-hyperbolic map with 
combinatorics $\(3,5,4,3,0,2\)$, also satisfying $F(I_1)\supset I_5$
and  $F(I_5)\supset I_1$, and hence also strongly obstructed.
\msk

\begin{figure}[t!]
  \centerline{\includegraphics[width=4in]{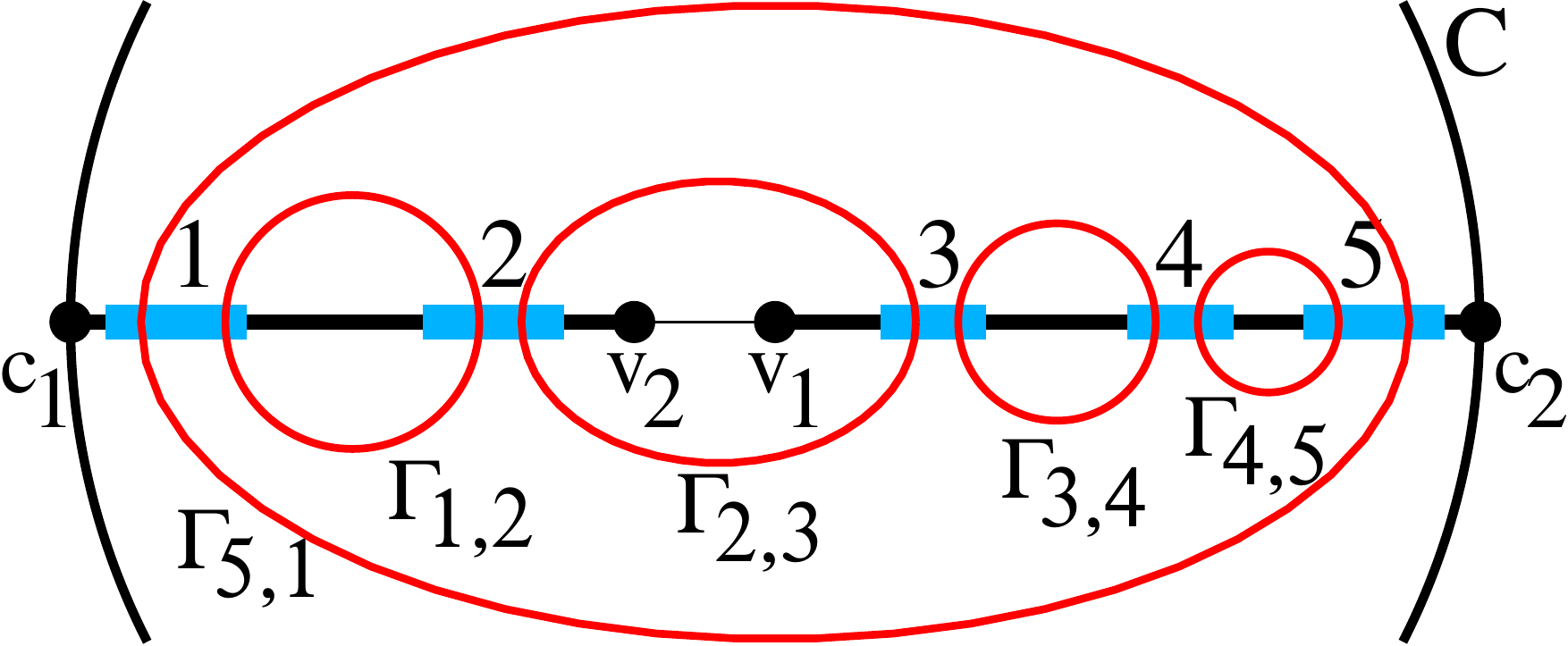}}
  \caption{\label{F-RSp}  A  neighborhood of the
    increasing half-circle within the eastern hemisphere
    of the Riemann sphere, illustrating a Levy cycle of period
    $p=5$. 
    The blue line segments represent
    the  periodic intervals, numbered from left to right by
    the numbers one through $p$. For any $p$ there must be at least
    one and at most $p-1$ such intervals in the increasing lap
    $[c_1,~v_2]$, with the remaining periodic intervals in the
other    increasing lap $[v_1,~c_2]$. The loops $\Gamma_{i,j}$ of the
    Levy cycle are numbered according to the intervals where they
    cross the real axis (= equator).}
  \end{figure}

  \begin{proof}[{\bf Proof of \autoref{T-ara} for the case $p>2$}]
    To simplify the notation, we will simply label the $p$ periodic
    intervals $I(j_k)$ by integers from one to $p$ in positive cyclic
    order within the increasing half-circle. If the rotation number
    is $m/p$, then the image of interval $j$ under the Thurston map
    $\f$ will contain interval with number $j+m~~({\rm mod}~p)$.
    Thus we can choose two points $r_j<s_j$ in each interval $j$ with images
    $r'_{j+m}<s'_{j+m}$ in the interior of interval $j+m$. Now join
    each $s'_j$ to $r'_{j+1}$ by a path within the northern hemisphere
    and also by the conjugate path within the southern hemisphere. The
    result will be a loop $\Gamma_{j,\,j+1}$. It is always possible
    to do this so that the loops are disjoint, and contained in the
    eastern hemisphere, as illustrated in \autoref{F-RSp}. Furthermore,
    it is not hard to see that there is a branch of $\f^{-1}$ which
    carries each $\Gamma_{j,\,j+1}$ to a loop which is homotopic
    to $\Gamma_{j-m,\,j+1-m}$ within $\widehat\C\ssm P$, where $P$ is
    the postcritical set. This completes the proof of \autoref{T-ara}.
    \end{proof}

  \section{Unimodal Maps}\label{s8}
  Let $\U$ be the open unimodal region of $\M/\I$  (see \autoref{F-M/I}),
  and let  $\overline \U$
  be its topological closure, consisting not only of unimodal classes
  but also of polynomial and co-polynomial classes. In this section, we
  will always choose the orientation so the the maps have shape $-+$.
  Thus for all $\<f\>$ in $\overline\U$ there is a
  \textbf{\textit{primary critical point}} $c_1$
  where $f$ takes its minimum value, and a \textbf{\textit{secondary
      critical point}} $c_2$ where $f$ takes its maximum value. If $\<f\>$
  belongs to the open set
  $\U$, then the critical point $c_1$ is in the interior of $f(\Rhat)$
  while $c_2$ is in the complement of $f(\Rhat)$.

  \subsection*{Bones}
  By definition, a \textbf{\textit{bone}} in $\overline\U$ is a connected
  component of the locus of $\<f\>$ 
for which the primary critical point $c_1$
  is periodic, with some specified period $p\ge 2$ and specified
 order type.\footnote{Compare \cite{DGMT} and \cite{MTr}. The
    \textbf{\textit{order type}} is the order of the $p$ successive images
    $f^{\circ j}(c_1)$ within the interval $f(\Rhat)$.} Filom \cite{F},
  making use of Kiwi and Rees \cite{KR}, 
  shows that every bone in $\overline\U$ is a smooth manifold, which is
  either a \textbf{\textit{bone-arc}},  diffeomorphic to a closed interval,
  or a \textbf{\textit{bone-loop}}, diffeomorphic to
  a circle.\footnote{Using Filom's work, we will prove in
    \autoref{C-no-loop} that there are no bone-loops in $\overline \U$.
  (See also Gao \cite{G}.)}
He proves the following. (See \cite[6.2]{F}.)

\begin{lem}\label{L-Fil} Every bone-arc in $\overline\U$ has one
  endpoint in the polynomial boundary and one endpoint in the
  co-polynomial boundary. Furthermore, for every polynomial or
  co-polynomial class for which $c_1$
    is periodic of period $p\ge 2$, there is a corresponding bone-arc.
    \end{lem}

    We can understand this statement on a purely combinatorial level as follows.

  \begin{prop} \label{P-copoly} There is a natural one-to-one correspondence
    between combinatorics \break
    {$\(m_0,\ldots, m_n\)$}  of polynomial  shape
    and combinatorics {$\(m_0,\ldots,m_{n-1}\)$} of 
    {co-polynomial}  shape, except in two extreme cases:
    For the polynomial combinatorics $\(2,0,2\)$, corresponding
    to the Chebyshev map $f(x)=x^2-2$, and the polynomial 
    combinatorics $\(0,1\)$ corresponding to $f(x)=x^2$,
there is no corresponding co-polynomial combinatorics.\end{prop}

\begin{proof} After reversing orientation if necessary, we may assume
that both combinatorics are of $-+$ shape. Thus
  the critical fixed point in the polynomial case (corresponding to
  the point at infinity for an actual polynomial) will be
to the right. Then the
  polynomial combinatorics takes the form $\(m_0,m_1,\ldots, m_{n-1}, n\)$.
  By definition, the associated co-polynomial combinatorics
  $\vecm=\(m_0,m_1,\ldots, m_{n-1}\)$ is obtained simply by deleting the last
  entry $m_{n}=n$. If we exclude the Chebyshev case and the $\(0,1\)$ case,
  then  it is not hard to check that this resulting $\vecm$
does indeed have co-polynomial  shape. Similarly it is not hard
to check that every $\vecm$ of co-polynomial  shape
can be uniquely  augmented to obtain a combinatorics of polynomial 
shape.
\end{proof}
\msk

For a typical hyperbolic example see Figure~\ref{f-3201}, while for a
typical non-hyperbolic example see Figure~\ref{f-43013}.

\begin{figure}[!htb]
    \centerline{
      \includegraphics[height=1.3in]{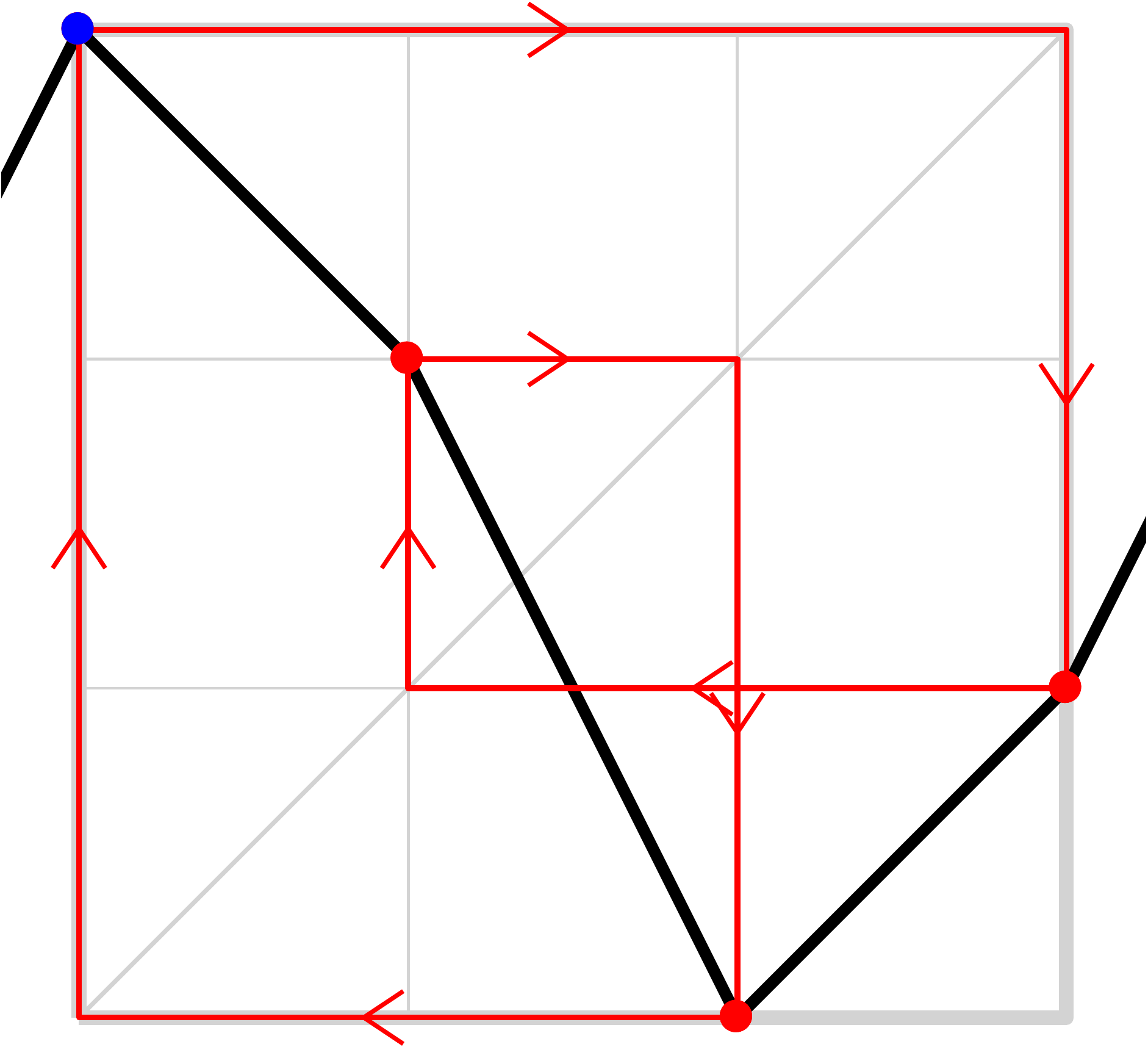} \quad
      \includegraphics[height=1.3in]{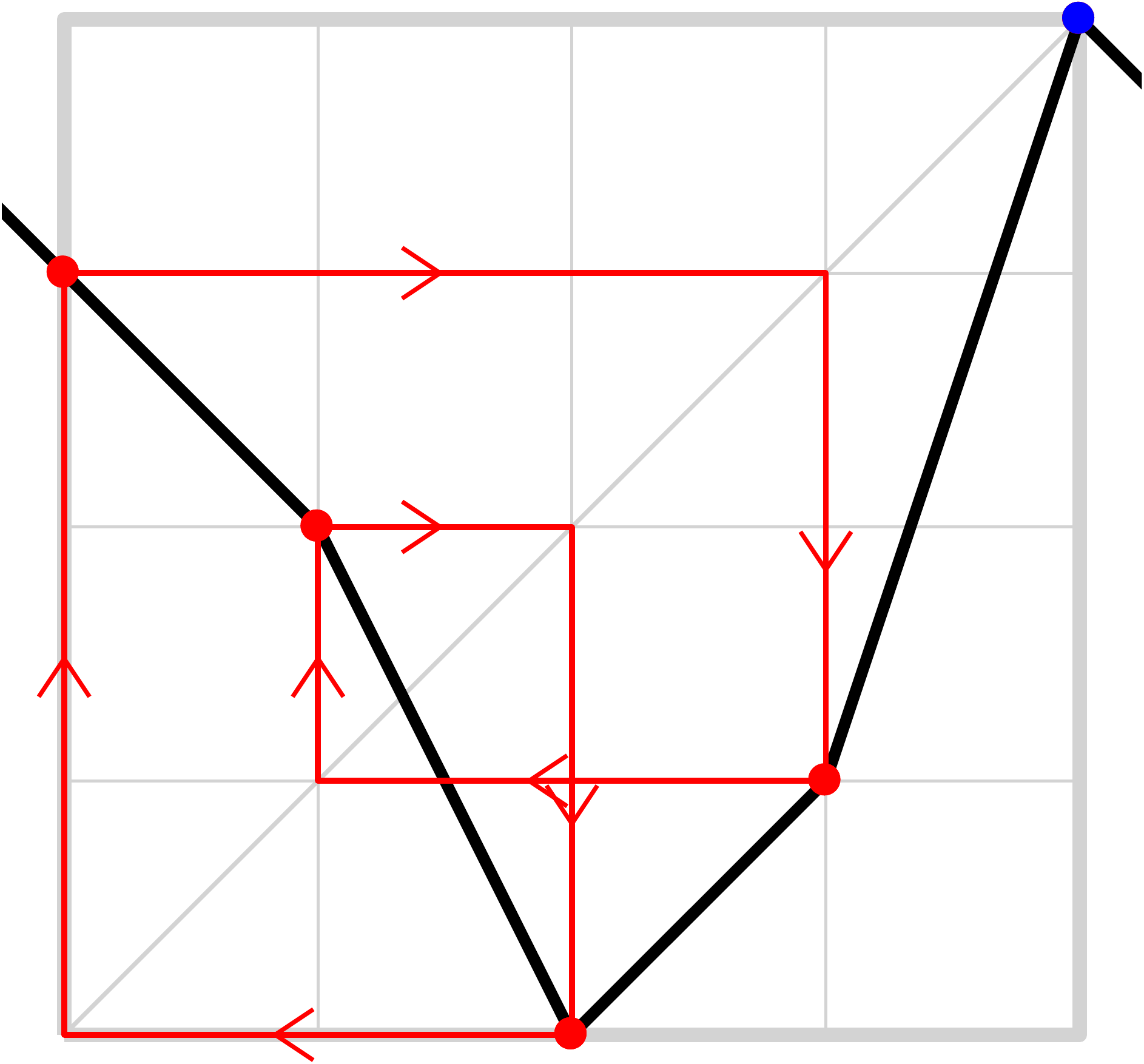} \hfill
      \includegraphics[height=1.3in]{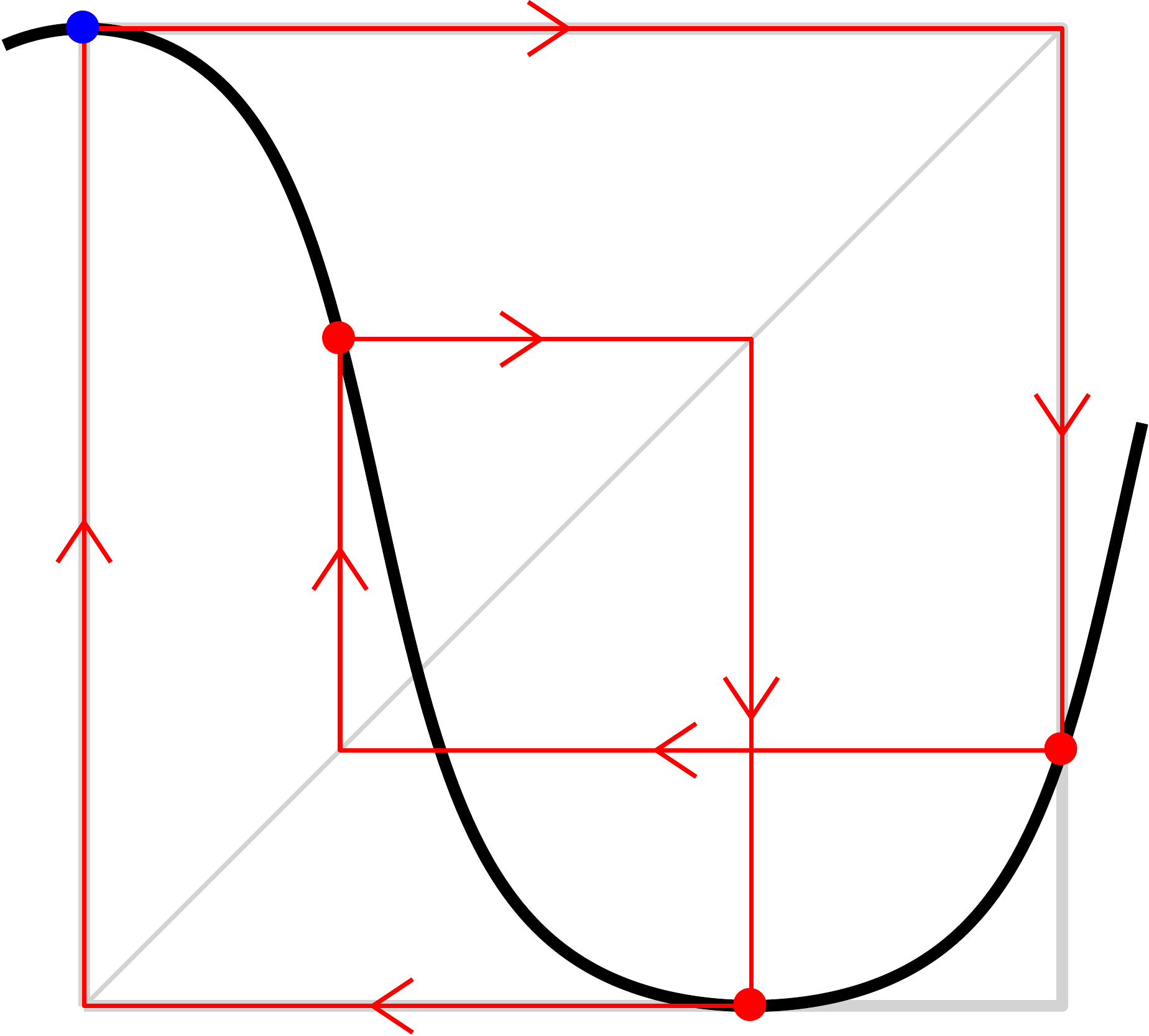} \quad
      \includegraphics[height=1.3in]{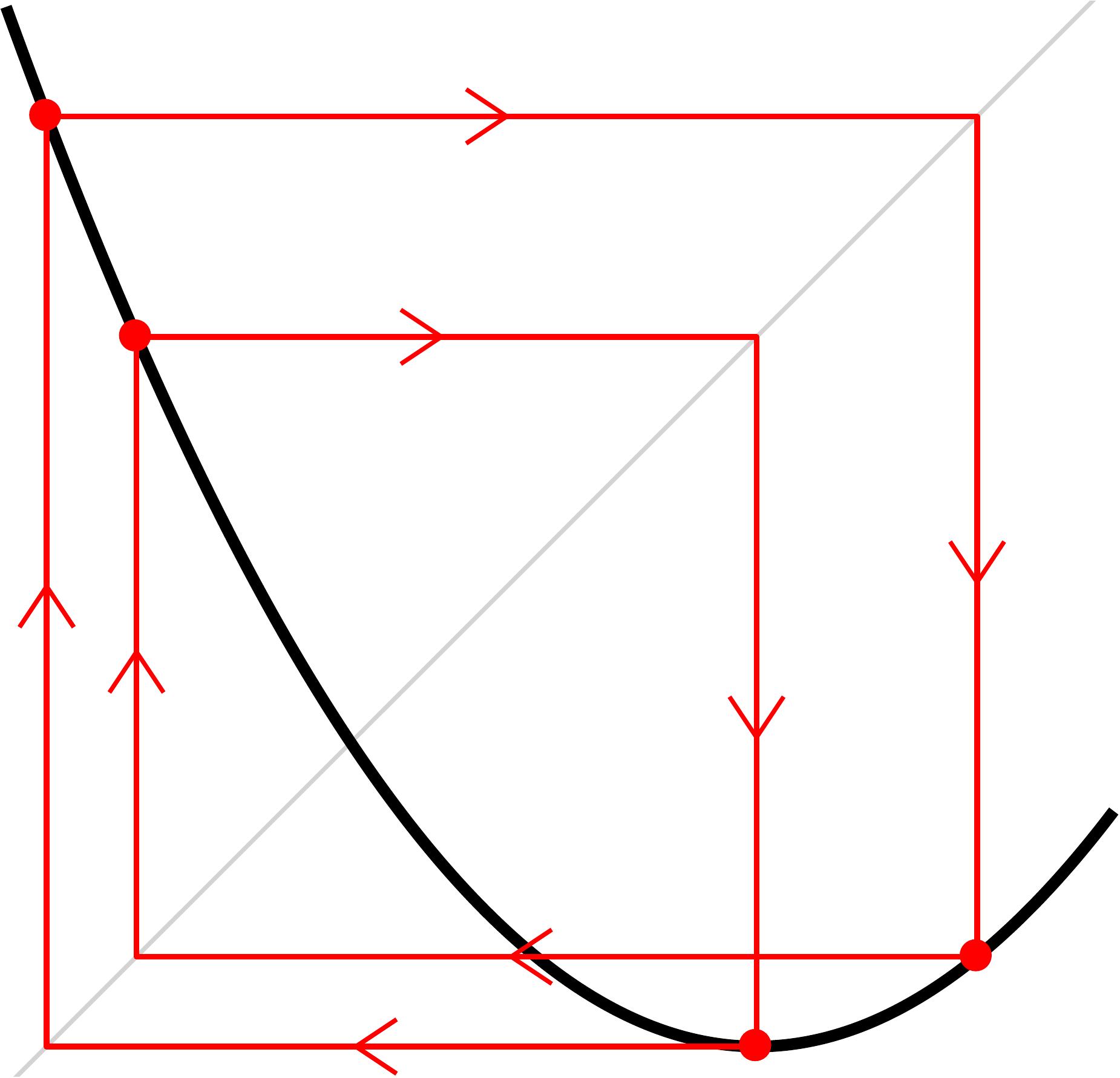}
    }
    \caption{ On the left, the  combinatorics $\(3,2,0,1\)$ has
      co-polynomial shape.
      with mapping pattern $\du{x_2}\mapsto \du{x_0}\mapsto x_3\mapsto
       x_1\mapsto\du{x_2}$.  Next to it, the corresponding polynomial
 combinatorics  $\(3,2,0,1,4\)$,  with mapping pattern
  $\du{x_2}\mapsto x_0\mapsto x_3\mapsto
  x_1\mapsto\du{x_2}$ together with $\du{x_4}\mapstoself$. Note that $x_0$ is a
  critical  point for the co-polynomial shape, but not for the polynomial
  shape.
  On the right center, the lifted rational map for the co-polynomial
      combinatorics  with  $\mu=-3.79681$ and $\kappa=0.898403$.
      To the right of it, the corresponding polynomial map
      {$f(x)=x^2-1.3107\cdots$}
  with combinatorics  $\(3,2,0,1,4\)$.
  \label{f-3201}}
    \end{figure}

  \begin{figure}[!htb]
    \centerline{
      \includegraphics[height=1.3in]{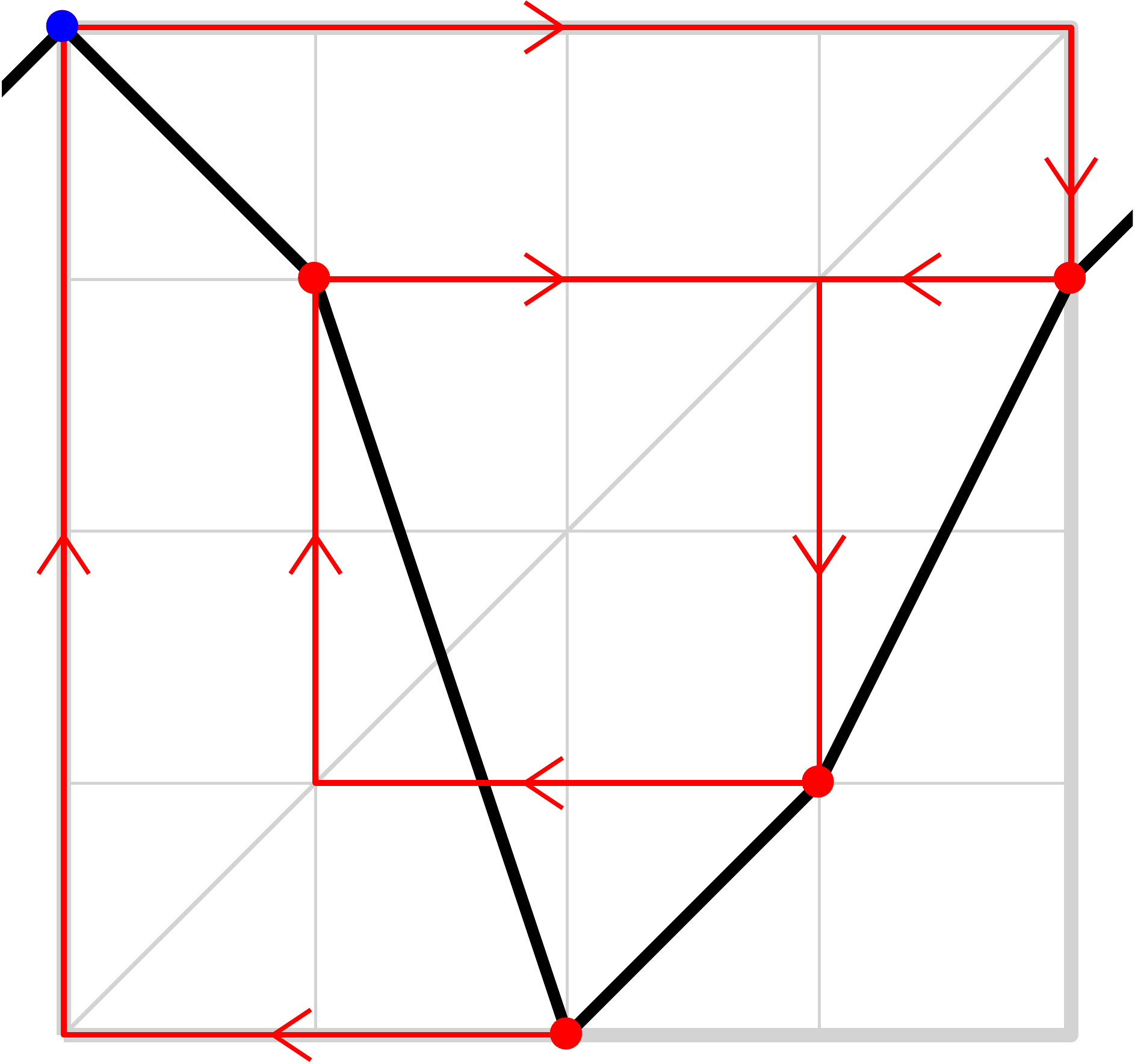} \quad
      \includegraphics[height=1.3in]{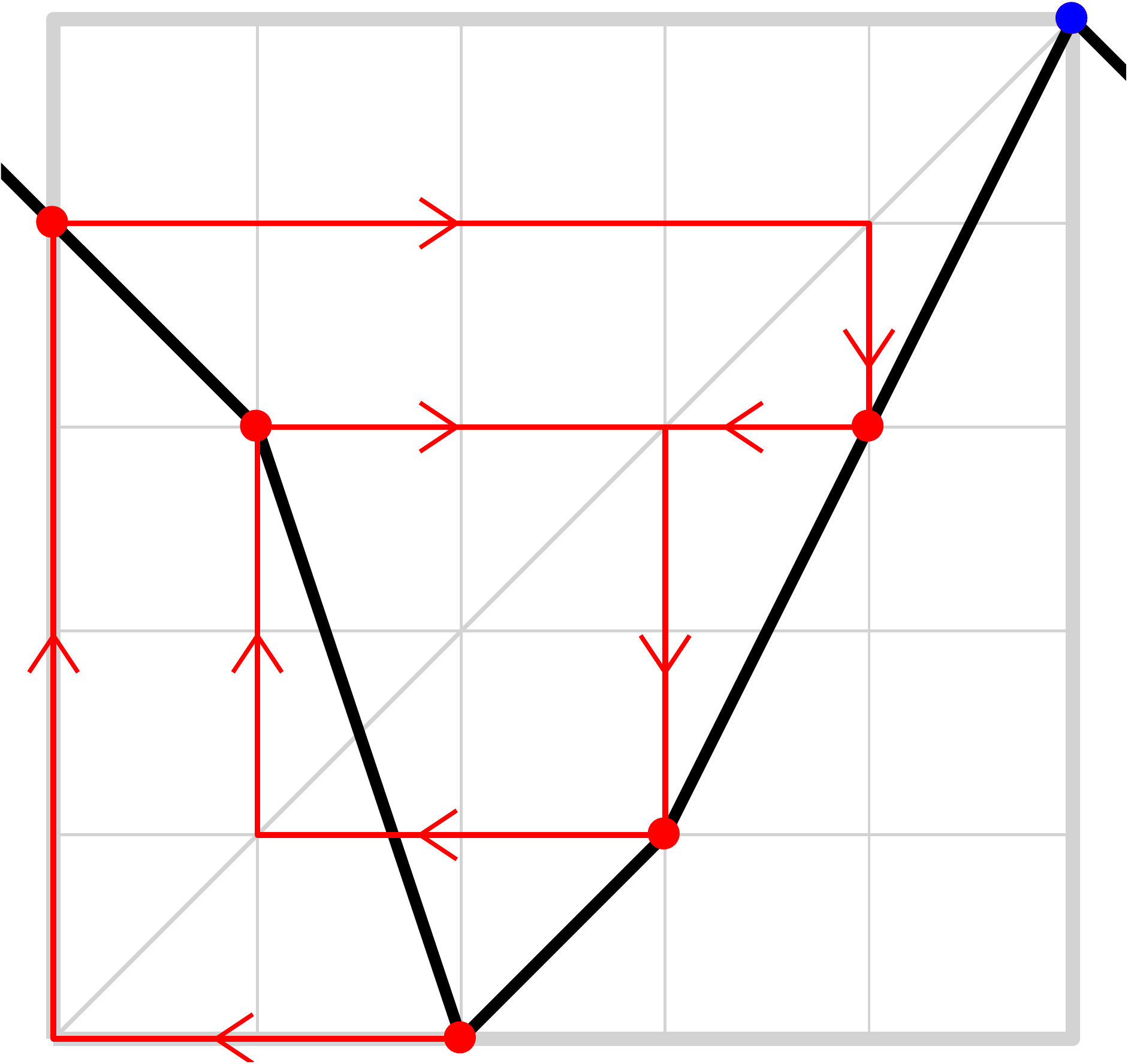} \hfill
      \includegraphics[height=1.3in]{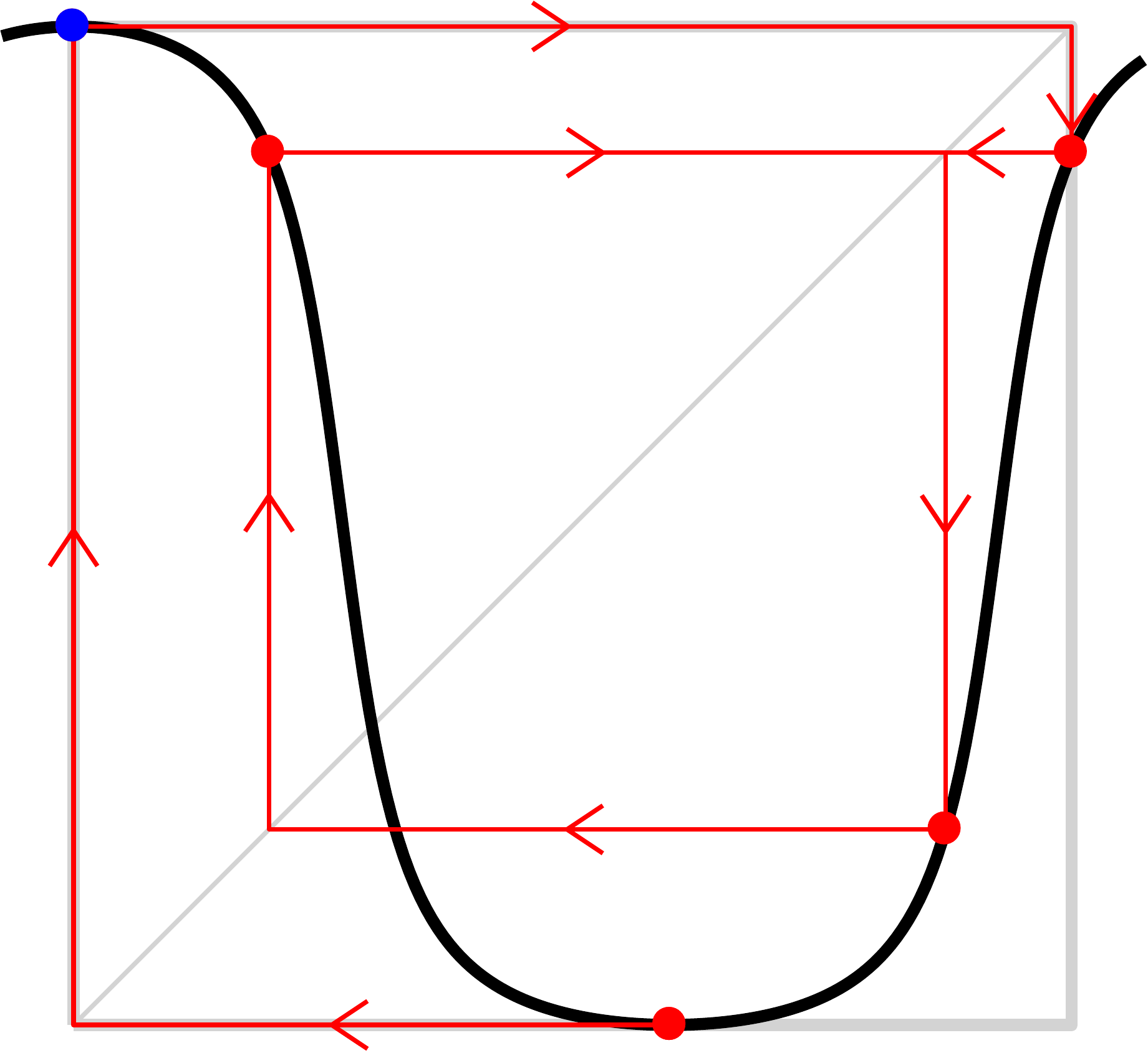} \quad
      \includegraphics[height=1.3in]{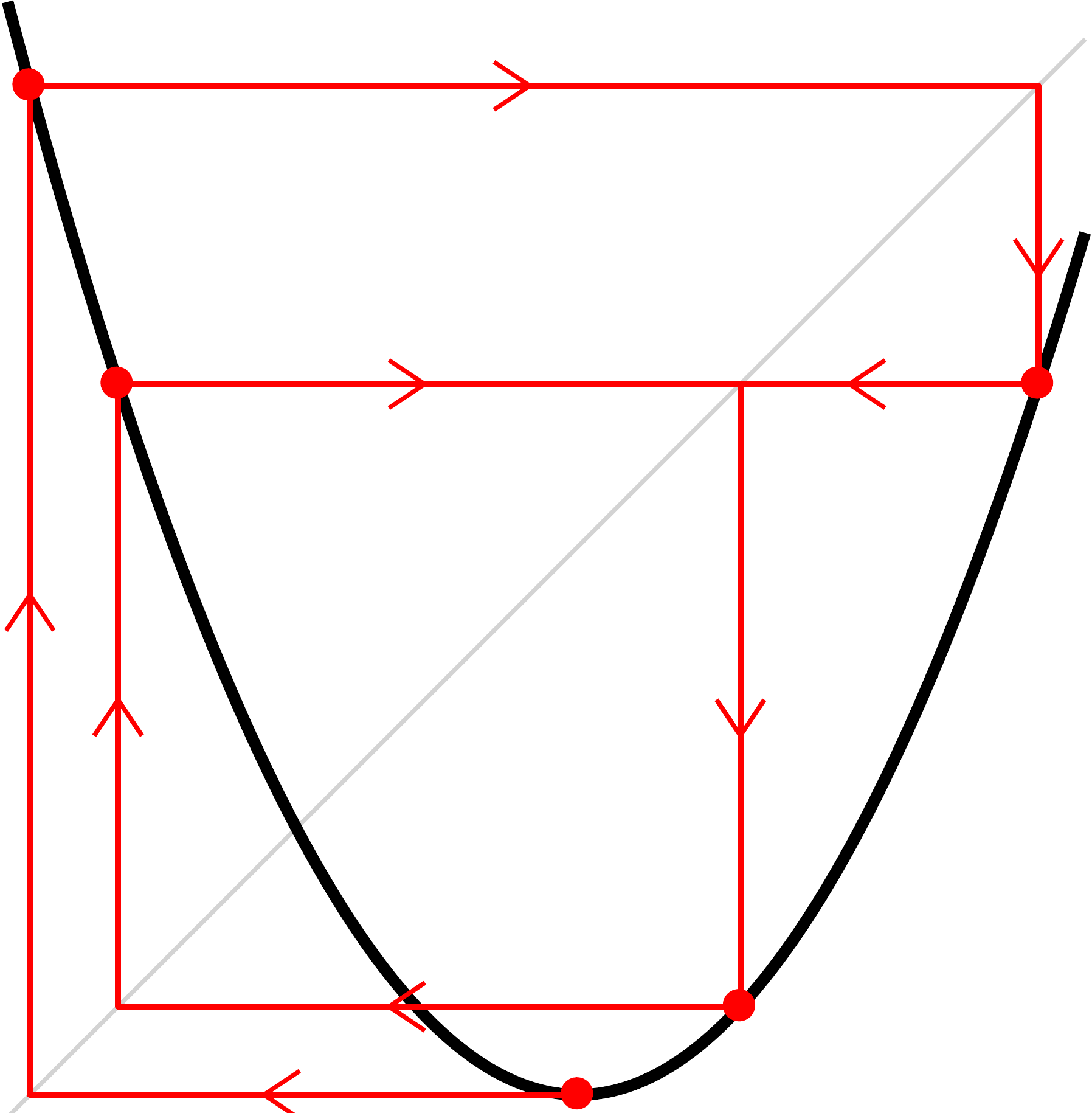}
         }
    \caption{ A non-hyperbolic example. First on the left, co-polynomial
 combinatorics $\(4,3,0,1,3\)$ with mapping pattern
 {$\du{x_2}\mapsto \du{x_0}\mapsto x_4 \mapsto x_3 \leftrightarrow x_1$}.
 Next to it, the corresponding polynomial combinatorics $\(4,3,0,1,3,5\)$ with
 mapping pattern {$\du{x_2}\mapsto x_0\mapsto x_4 \mapsto x_3
 \leftrightarrow x_1$} and $\du{x_5} \mapstoself$. Note that $\du{x_0}$ is a
      critical point on the left but not on the right.  The right center
      shows  the  lifted map for the co-polynomial combinatorics
      $\(4,3,0,1,3\)$  with $\mu=-5.53846$ and $\kappa=1.76923$. To the
       right of it is the  graph of the corresponding polynomial
      $f(x)=x^2-1.839287$.  \label{f-43013}} 
  \end{figure}
  
\msk

Thus, the correspondence between critically finite polynomial dynamics
and co-polynomial dynamics works just as well in the 
non-hyperbolic case. This suggests the  following {definition}
and conjecture. \msk

 \begin{definition}\label{D-NHb} By a ``non-hyperbolic bone'', or briefly
   \textbf{\textit{NH-bone}}, in $\overline\U$ we will mean a connected
   component of the locus of points for which $c_1$  is eventually
   periodic
   repelling,\footnote{It is necessary to be careful, since such an 
     NH-bone may terminate at a point where the repelling orbit 
     becomes  parabolic.}
   with specified eventual period $p\ge 1$ and pre-period $q\ge 2$. (Here
   by the ``pre-period'' we mean the smallest $q$ for which $f^{\circ q}(c_1)$
   is periodic.)
 \end{definition}

 \begin{conj}\label{C-NH}
   \sf Such NH-bones behave very much like the usual bones.   In particular
     they are smooth manifolds and (with one exception) every NH-bone-arc joins
     a point on the polynomial locus to a point on the co-polynomial locus.
     The unique exception is the ``Chebyshev'' NH-bone, which starts on the
     polynomial locus at $\langle x\mapsto x^2-2\rangle$, and  forms
     part of the boundary of the shift locus until it hits $\Per_1(1)$
     tangentially. (Compare \autoref{F-M/I}.) At this point the repelling
     fixed point becomes attracting. The analytically continued curve
is contained in the shift  locus and diverges towards the 
     ideal boundary point without ever reaching the co-polynomial locus.
   \end{conj}

   The ten simplest examples of the  correspondence between
   co-polynomial combinatorics and polynomial combinatorics are shown in
   \autoref{t-kninv}, and the ten corresponding bone-arcs or NH-bone-arcs
   are shown in \autoref{F-bonearc}.

\subsection*{Kneading Numbers} 
    Kneading theory is a useful tool for all piecewise monotone maps
 of the interval (compare \cite{MTh}). However, in the unimodal case it
 can be described in a particularly simple and easy to use form. 
    
 For any $f$ with $\langle f\rangle\in\overline\U$ 
there is a \textbf{\textit{dynamic kneading function}}
$$ \k~=~\k_f:f(\Rhat) \to [-1,\,1]~,$$
defined in two steps 
as follows. The image $\k(x)$  can be thought
of as an invariantly defined coordinate for the point $x\in f(\Rhat)$.

\begin{definition}\label{D-kn} 
  For each $x\in f(\Rhat)$ let: \hspace{-.4cm}
  $$\sigma(x)= \begin{cases}\mbox{\quad} 0\quad {\rm if}\quad x\quad{\rm is\;\, equal
      \;\, to\;\, the\,\; primary \;\, critical \;\, point}\;\, c_1~, \\
    \mbox{\quad} 1\quad  {\rm if}\quad x \quad {\rm is\;\, in \;\, the\;\,
      increasing\;\,lap}~,\\
  -1\quad {\rm if}\quad x\quad {\rm is \,\; in \;\, the\;\, decreasing\;\, lap}~.
  \end{cases}$$

  \noindent  Note that $\sigma(x)$ can be identified with the sign of the
 derivative $f'(x)$, except in the case of a pole, with $f(x)=\infty$.
 (Of course if $f(\Rhat)\subset\R$, then there are no poles.)
 Now for any $x_0\in f(\Rhat)$ with orbit
$~~f:x_0\mapsto x_1\mapsto\cdots$,~~ define
$$\k(x_0)~=~\sum_{h\ge 0}
\frac{\sigma(x_0)\sigma(x_1)\cdots\sigma(x_h)}{2^{h+1}}~.$$
\end{definition}

\begin{table}[!ht]
  \caption{\label{t-kninv} Examples of co-polynomial maps, showing the
    kneading invariant 
    as defined below, and listing the corresponding $c$ values
    for $f_c(x) = x^2 -c$~. The notation $B(p)$ means that both critical
    points are contained in a common orbit of period $p$. 
    }
{\newcommand{\Tsk}{3pt}
\newcommand{\lb}[1]{\raisebox{-1.7ex}[0pt][0pt]{#1}}
\newcommand{\rb}[1]{\raisebox{1.8ex}[0pt][0pt]{#1}}
\newcommand{\hd}[3]{\multicolumn{1}{#1}%
   {$\stackrel{\hbox{\textsf{#2}}\Tstrut}{\hbox{\textsf{#3}}\Bstrut}$}}
\begin{center}
\begin{tabular}{|c|c:c:c:l:c|} 
\hline
  \hd{|c}{figure}{number} & \hd{|c}{topological}{shape}
& \hd{|c}{co-polynomial}{combinatorics} & \hd{|c}{dynamic}{type} 
& \hd{|c}{kneading}{invariant} & \hd{|c|}{corresponding}{polynomial} \\
  \hline \hline
{\ref{f-aBn4}L} & $ +\, -$& $\(1,\,2,\,3,\,4,\,0\)$&\lb{$B(5)$} && \Tstrut\\[\Tsk]
    & $-\, +$& $\(4,\, 0,\, 1,\, 2,\, 3\)$&&\rb{0.9375}&\rb{$c=1.985424253$}\Bstrut\\
  \hline
{\ref{f-aBn3}L}& $ +\, -$& $\(1,\, 2,\, 3,\, 0\)$&\lb{$B(4)$}& &\Tstrut \\[\Tsk]
   & $-\, +$& $\(3,\, 0,\, 1,\, 2\)$& &\rb{0.875}&\rb{$c=1.940799807$}\Bstrut\\
  \hline
{\ref{f-nonhyp1}}& $+\, -$& $\(1,\, 3,\, 4,\, 3,\, 0\)$&\lb{$NH$}& & \Tstrut\\[\Tsk]
 & $-\, +$& $\(4,\, 1,\, 0,\, 1,\, 3\)$& &\rb{0.833333}&\rb{$c=1.892910988$}\Bstrut\\
  \hline
{\ref{f-aBn4a}L}&$+\, -$& $\(1,\, 3,\, 4,\, 2,\, 0\)$&\lb{$B(5)$}&  & \Tstrut\\[\Tsk]
 & $-\, +$& $\(4,\, 2,\, 0,\, 1,\, 3\)$& &\rb{0.8125}&\rb{$c=1.860782522$}\Bstrut\\
  \hline
\ref{f-nonhyp}R&$+\, -$& $\(1,\, 3,\, 4,\, 1,\, 0\)$&\lb{$NH$}&   &\Tstrut\\[\Tsk]
\ref{f-43013} & $-\, +$& $\(4,\, 3,\, 0,\, 1,\, 3\)$& &\rb{0.8}&\rb{$c=1.839286755$}\Bstrut\\
  \hline
{\ref{f-aBn2}R}&$+\, -$& $\(1,\, 2,\, 0\)$&\lb{$B(3)$}&  & \Tstrut\\[\Tsk]
           &$-\, +$& $\(2,\, 0,\, 1\)$& &\rb{0.75}&\rb{$c=1.754877666$}\Bstrut\\
  \hline
{\ref{f-aBn4}R}&$+\, -$& $\(2,\, 4,\, 3,\, 1,\, 0\)$&\lb{$B(5)$}&   &\Tstrut\\[\Tsk]
              &$-\, +$& $\(4,\, 3,\, 1,\, 0,\, 2\)$& &\rb{0.6875}&\rb{$c=1.625413725$}\Bstrut\\
  \hline
{\ref{f-nonhyp}L}&$+\, -$& $\(2,\, 3,\, 2,\, 0\)$&\lb{$NH$}& &\Tstrut \\[\Tsk]
  &$-\, +$& $\(3,\, 1,\, 0,\, 1\)$& &\rb{0.666667}&\rb{$c=1.543689013$}\Bstrut\\
  \hline
  \ref{f-aBn3}R&$+\, -$& $\(2,\, 3,\, 1,\, 0\)$&\lb{$B(4)$}&   &\Tstrut \\[\Tsk]
\ref{f-3201} &$-\, +$& $\(3,\, 2,\, 0,\, 1\)$& &\rb{0.625}&\rb{$c=1.310702641$}\Bstrut\\
  \hline
  \ref{f-aBn2}L&$ - $&$\(1,\,0\)$&$B(2)$& 0.5  &   $c=1$\Tstrut\Bstrut\\
  \hline 
  \hline
\end{tabular}
\end{center}
}
\end{table}

\begin{lem}\label{L-kn1}   If $f$ has shape $-+$, then $\k(x)$ 
 is a monotone increasing function. That is,
  $$~~x<y~\Rightarrow~ \k(x)\le\k(y)~. $$
Similarly, in the $+-$ case it is monotone decreasing.
This function $\k(x)$   is not continuous: It 
 has a  jump discontinuity at $c_1$ and at every iterated
pre-image of $c_1$; but is continuous everywhere else.
\end{lem}

\begin{proof}[Proof by contradiction] First consider the $-+$ case so
  that
  $$ \sigma(x) < \sigma(y)\quad\longrightarrow\quad x<y~.$$
Let $x=x_0\mapsto x_1\cdots$ and $y=y_0\mapsto y_1\cdots$ be the orbits.
If {$\k(x)> \k(y)$,} let $n$ be the smallest integer
with $\sigma(x_n)\ne \sigma(y_n)$, and let
$$ s=\prod_{0\le j<n}\sigma(x_j)=\prod_{0\le j<n}\sigma(y_j)~.$$
Then  $s\,\sigma(x_n)> s\,\sigma(y_n)$, hence $s\, x_n>s\, y_n$. But $s$ is the
sign of the derivative  of $f^{\circ n}$ on the interval between $x$ and $y$,
so it follows that $x>y$, contradicting our hypothesis. The $+-$ case
is similar, and the rest of the proof is straightforward.
\end{proof}
\msk

\begin{lem}\label{L-kn2} If we are in the $\F(0,1)$ region of $\U$,  with
  no attracting   or parabolic fixed points, then the real number
  $\k(x_0)$  completely determines the sequence of signs $\sigma(x_j)$. 
  In particular, it follows that the sign of $\k(x)$ is equal to
   $~\sigma(x)$.\end{lem}
 \ssk
 
 \begin{proof}
   This could fail only if the sequence of signs $\sigma(x_j)$
      had the form  $(\pm 1, \mp 1, 1, 1, 1, \cdots)$ for some $x$,
      so that $\k(x) =\pm (1/2-1/4-1/8-\cdots)$ would be zero
      although the signs are all non-zero. But this would imply that
      the orbit of $x$ converges to an attracting or parabolic fixed
      point.\footnote{This argument
    is needed only because of our definition of $\k$. If we chose
    to divide by $3^{h+1}$ instead of $2^{h+1}$ in the defining  
formula,  then attracting fixed points would not cause any problem.}
\end{proof}

For the two dimensional set $\overline \U$ we can define two 
invariant coordinates
$\K_1$ and $\K_2$, each sending $\overline \U$ to $[-1,\,1]$. These
\textbf{\textit{kneading functions on moduli space}} 
are defined by the formula
$$\K_j(f)~=~-\k_f(v_j)\qquad{\rm where}\quad v_j\quad{\rm is~the~critical~
  value}\quad f(c_j)~.   $$
Here the minus sign is inserted for convenience, so that
we will have $\K_1(f)\ge 0$ 
in dynamically interesting cases (although $\K_2(f)\le 0$).
 It is not hard to check that
$\K_1(f)$ is constant along each bone, and serves to distinguish one bone
from another (provided that there are no bone-loops). 
On the other hand, $\K_2(f)$ can be thought of as an invariantly defined
coordinate along each bone. Note that the definition of $\K_1(f)$ ignores
the critical points $c_2$, while definition of $\K_2(f)$ involves the interplay
between the orbits of $c_1$ and $c_2$. 
\msk

  Evidently  $\K_1(f)$ is a sum of only finitely many terms
  if and only if the critical point $c_1$ is periodic, so that $f$ belongs to
  a bone.  In this case,  $\K_1(f)$ will jump discontinuously under perturbation
  of $f$. However, whenever the critical point $c_1$ 
  is not periodic,   $\K_1(f)$ will vary continuously.
  Note also that   $\K_1(f)$ is a rational number  whenever $f$
  belongs either to a bone or to an NH-bone. 
  (Compare  \autoref{t-kninv}.)
 \msk 
 
  \begin{lem}\label{L-cfK} Every critically finite map $f$ is uniquely
    determined up to $\pm$-conjugacy by the pair of rational numbers 
    $\big(\K_1(f),\,\K_2(f)\big)$.
    \end{lem}
    \ssk
    
    \begin{proof} It follows from \autoref{L-kn2} that
      the order of the points in the orbit of $c_1$ is determined
   by $\K_1(f)$, and the relative position of the orbit of $v_2$ is then
   determined by $\K_2(f)$. Thus the combinatorics is uniquely determined;
   and the conclusion follows from \autoref{T-notexc}.\end{proof}
  \bigskip

Some typical examples of level sets of $\K_1$ are shown in
\autoref{F-bonearc}. The level sets of $\K_2$ are more complicated. 
(See \autoref{f-K2curves}.)\bigskip

\begin{figure}[!h]
      \begin{center}
        \begin{overpic}[width=3in]{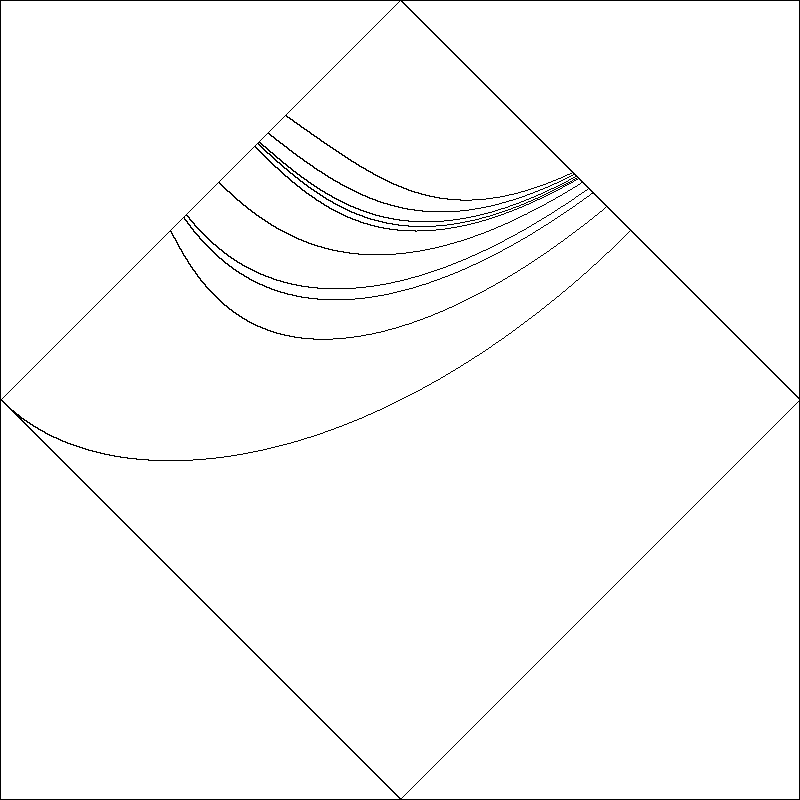}
          \put(100,75){$\U$}
       \end{overpic}
      \caption{\label{F-bonearc}  Illustrating ten
 bone-arcs or NH-bone-arcs
        joining a co-polynomial on the left to a polynomial
        on the right, corresponding to the ten cases listed
        in \autoref{t-kninv}. Here, as in Figures~\ref{F-M/I}
        and \ref{F-entpic}, the
        space $\U$ of unimodal conjugacy classes is represented
        by a right-angled rhombus. This figure is closely related
        to the plot of topological entropy in \autoref{F-entpic}.
}
      \end{center}
      \end{figure}

\begin{figure}[!htb]
\centerline{\includegraphics[width=2.6in]{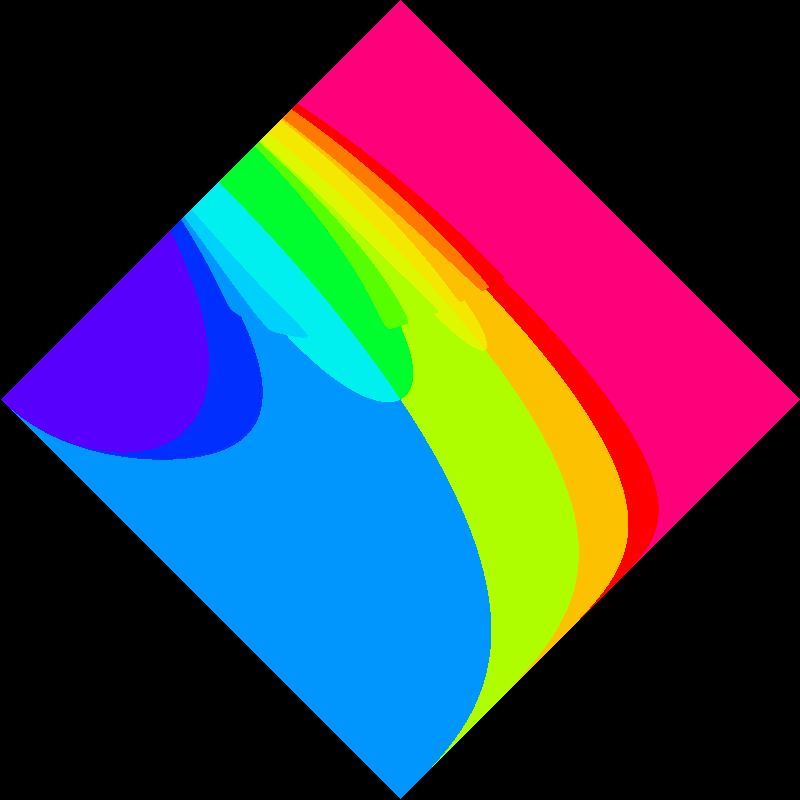}\qquad
\includegraphics[width=2.6in]{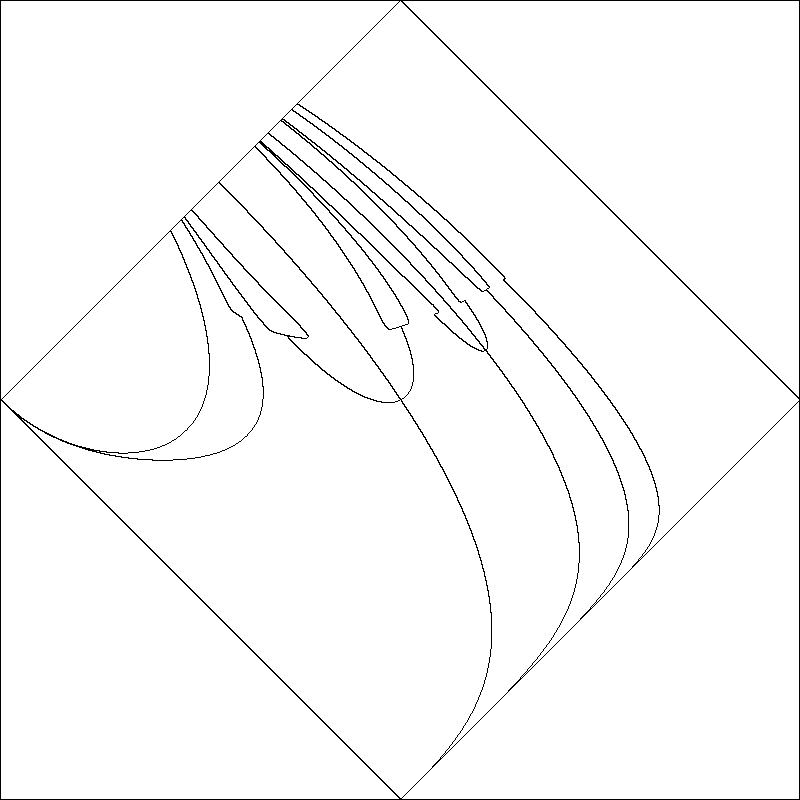}}
\caption{ The region $\U$ divided into fourteen colored regions 
  according to the value of $\K_2(f)$. On the right, the boundaries between
  these regions. \label{f-K2curves}}
\end{figure}\msk
      
Here is simple application of kneading.

\begin{lem}\label{L-co-p} Minimal co-polynomial  combinatorics
  are unobstructed in all cases.\end{lem} 
\ssk

\begin{proof} As we traverse the one-parameter family of polynomials,
  we encounter every possible critically finite kneading sequence.
  (See \cite[Thm. 12.2]{MTh}.) But as we follow the one parameter family
  of co-polynomials 
  $$f(x)=1/(x^2+c)\,,$$ it follows from \autoref{P-copoly} that we encounter
  exactly the same  critically finite kneading sequences. Thus every minimal
  admissible co-polynomial combinatorics is realized by an actual
  co-polynomial, and there is no obstruction.
  \end{proof}
\msk

Making use of Filom's proof that each bone-arc in $\overline \U$
terminates at a co-polynomial at one end and a polynomial
at the other end, we have the following.\msk

\begin{theo} The function $\K_1(f)$ is constant
  along each bone-arc; while the function $\K_2(f)$ is monotone
  increasing as we follow each bone-arc from its polynomial
  endpoint to its co-polynomial endpoint.
\end{theo}
\ssk

\begin{proof} The first statement is clear. For the second statement,
note that we need only consider the upper half of the polynomial
boundary (corresponding to conjugacy classes of the form
$\langle x\mapsto x^2+v_1\rangle$ with $v_1<0$), since the lower
half does not have any \PCF points. For $\<f\>$ in the
upper half of the polynomial boundary, the critical point $c_2$ is fixed,
so that $\K_2(f)=-1$.
As $f$ varies along a bone, note that the correspondence $f\mapsto
\K_2(f)$ can change its value only when the orbit of $c_2$ passes
through $c_1$, which means that we pass through a \PCF capture configuration.
Furthermore each such capture combinatorics can occur only once, since
a \PCF map is determined by its combinatorics. This proves 
that the map which sends each $f$ on the bone to $\K_2(f)\in[-1,1]$
is monotonic.
\end{proof}
\msk

\begin{table}[!h]
  \begin{center}
    \caption{ On the left, seven critically finite points along the
  bone-arc for which the critical point $c_1$ has period two.
  On the right,  seven points   along the bone-arc for which
  $c_1$ has period three. Note that the values of $\K_2$
are rational, with small denominators since  $n$ is small.}
\small

  \begin{tabular}{|lll||lll|}
    \hline
    \multicolumn{3}{|c||}{$\K_1=0.5$}
    &\multicolumn{3}{|c|}{$\K_1=0.75$}\Tstrut\Bstrut\\\hline \hline
{\sf shape}& {\sf combinatorics} &$ \K_2$&{\sf shape}& {\sf combinatorics} &$ \K_2$\Tstrut\Bstrut\\\hline
{\rm polynomial (D)} &$\(1,0,2\)$ & $-1$ & {\rm polynomial (D)} &$ \(2,0,1,3\)$ & $-1$ \Tstrut\\
{\rm unimodal (C)}&$\(7,2,1,2,3,4,5,6\)$&$-0.96875$& {\rm unimodal (C)} & $\(7,3,1,2,3,4,5, 6 \)$&$-0.96875$\Tstrut\\
{\rm unimodal (C)}&$\(6,2,1,2,3,4,5\)$&$-0.9375$& {\rm unimodal (C)} & $\(6,3,1,2,3,4,5\)$& $-0.9375$\Tstrut\\
    {\rm unimodal (C)}& $\(5,2,1,2,3,4\)$&$-0.875$& {\rm unimodal (C)} &$\(5,3,1,2,3,4\)$ & $-0.875$\Tstrut\\
{\rm unimodal (HH)}&$\(6,3,2,1,2,4,5\)$&$-0.83333$&{\rm unimodal (HH)}&$\(7,5,2,1,2,3,4,6\)$ &$-0.8333$ \Tstrut\\
{\rm unimodal (C)}&$\(4,2,1,2,3\)$&$-0.75$& {\rm unimodal (C)} & $\(7,5,3,1,2,3,4,6\)$ & $-0.8125$\Tstrut\\
{\rm unimodal (HH)}&$\(5,3,2,1,2,4\)$&$-0.66667$& {\rm unimodal (C)}   &$\(4,3,1,2,3\)$&$-0.75$ \Tstrut \\
{\rm co-poly (B)}&$\(1,0\)$&\quad$0$&{\rm co-poly (B)} &$\(2,0,1\)$ &$-0.5$ \Tstrut\Bstrut\\ \hline  
\end{tabular}
\end{center}
\end{table}
\msk

\begin{theo}\label{P-FIE} For any hyperbolic unimodal combinatorics 
  $\vecm$, let $\<f\>$ be the unique \hbox{co-polynomial}
  conjugacy class  for which the orbit of the critical point $c_1$
  has the same order type.   Then $~~ \K_2(\vecm) ~<~\K_2(f)$. 
 \end{theo}
 \ssk
 
 \begin{proof}
   Recall that $\K_2$ measures the location of the critical value $v_2$
   with respect to the grand orbit of $c_1$. Since $c_1$ has the same periodic
   order type for $f$ and for $\vecm$, we can make a direct comparison.
   For the co-polynomial $f$, the critical value is the rightmost point
   of the orbit of $c_1$. On the other hand, for any strictly unimodal
   combinatorics, the critical value will be strictly to the right of the
   entire periodic orbit. This means that $\k(v_1)$ for the unimodal
   case $\vecm$ will be strictly larger that $\k(v_1)$ for the
   co-polynomial case $f$.
   Therefore $\K_2(\vecm)<\K_2(f)$, as required.
 \end{proof}
 \msk
 
 \begin{coro}\label{C-no-obs} Minimal unimodal combinatorics of
   hyperbolic or half-hyperbolic type are never obstructed. In
   fact every such combinatorics is represented by an actual
   quadratic map lying on the associated bone-arc. 
 \end{coro}
 \ssk
 
 \begin{proof} For such combinatorics, the primary critical point will
   be periodic of some order type. As we follow the
   corresponding bone-arc from the polynomial
   end to the co-polynomial end, the value of $\K_2(f)$ can change only
   as we pass through a critically finite point. In fact every possible
   critically finite combinatorics for this critical order type must
 occur at some point. Thus every possible combinatorics is realized
   by some critically finite quadratic map, and there can be
   no obstruction.
 \end{proof}
 \msk

 \begin{conj}\label{RC-TNH}
   Every minimal unimodal combinatorics
   is unobstructed, even in the totally non-hyperbolic case.
 \msk
 
 {\rm This would certainly follow from \autoref{C-noSL}, since
   it is    impossible to reach the ideal boundary within the
   unimodal region without passing through the shift locus.
It would probably follow also from \autoref{C-NH}, using
an appropriate modification of the argument above.}
 \end{conj}
\msk

Making use of  Filom  \cite[Prop. 6.8]{F}, we have the following further
conclusion.
\msk

\begin{coro}\label{C-no-loop} There are no bone-loops in the 
  unimodal   region $\U$. \end{coro}
\ssk

\begin{proof} Filom shows that if there exists 
  a bone-loop in the unimodal region, then there must be one
  containing a critically finite point. But it follows from
  \autoref{C-no-obs} that any such critically finite 
  point is already contained in a bone-arc, and hence cannot be
  contained in a bone-loop.\end{proof}
\msk

\begin{figure}[t!]
  \centerline{\includegraphics[width=3in]{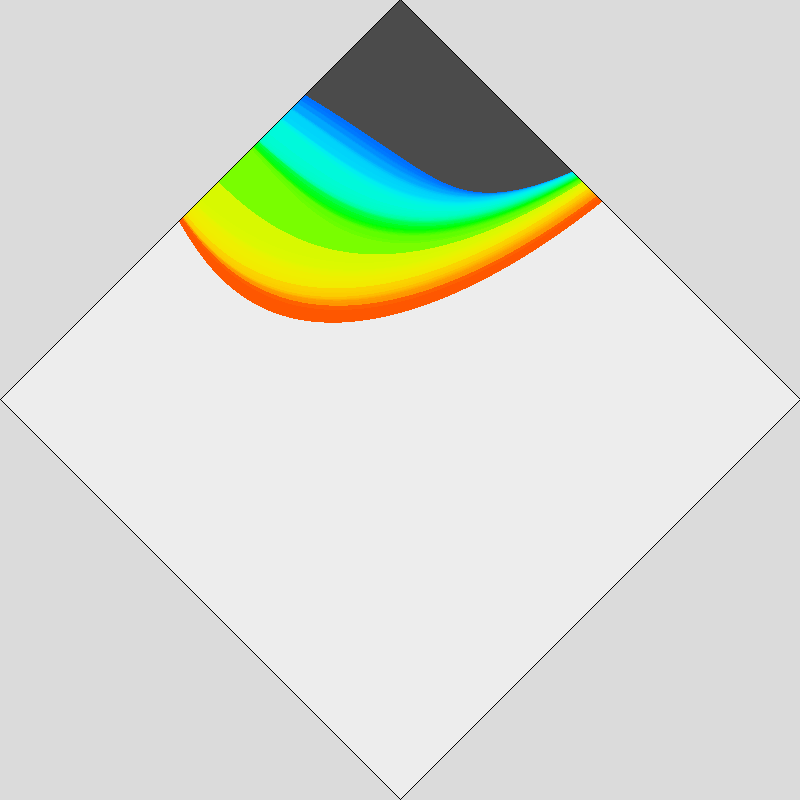}}
    \caption{\label{F-entpic}  A rough plot of topological entropy
      throughout the unimodal region. Here the black area is the
   real shift locus with $\h=\log(2)$, the white area is the 
      locus $\h=0$,  and the  colors represent intermediate values.}
\end{figure}

This has important consequences for topological entropy.
\msk

  \begin{rem}[Filom's work on Topological Entropy]\label{R-TENT}
    The topological entropy $\h(f)\ge 0$  is defined for much more
    general dynamical systems. For quadratic maps, it depends
    continuously on $f$, taking values in the interval $[0,\,\log(2)]$.
For the relationship between entropy and kneading, see 
 Remark~\ref{R-K1&h} below. \ssk

  Filom \cite{F},  showed that each  locus of constant 
  topological entropy is connected  within the $~- + -~$ bimodal region, and 
  also within  the part of the unimodal region for which $f$ has an attracting
  periodic orbit. 
(This is the region {\it above} the curve $\Per_1(1)$ in
  \autoref{F-M/I}, or {\it below} the line $\Per_1(1)$ in  \autoref{F-old}.)
 He conjectured that this is true throughout the unimodal region; and was
 able to prove this under the hypothesis that there are no bone-loops.
 Thus, using his work together with \autoref{C-no-loop},
 we have the following.
 \msk
 
 \begin{coro}\label{C-entropy-mono} Each locus of constant entropy
   in $\overline\U$ is connected.\end{coro}\ssk

  For a different proof of this statement, see Yan Gao (\cite{G}).
 His proof is
based on the similar  idea of showing that there is no bone-loop, but
utilizes a completely  different technique developed by Levin, Shen and
van Strien in \cite{LSS}.
 \ssk
 
In  the $~+ - +~$ bimodal region Filom conjectured that the corresponding
loci can be badly disconnected;
and this was later proved by Filom and Pilgrim  (see \cite{FP}).\end{rem}
\ssk

\begin{rem}[Entropy computed from Kneading]\label{R-K1&h} 
  In the unimodal case, the topological entropy $\h(f)$ is
  uniquely determined by the kneading invariant $\K_1(f)$;
  and hence is constant along each bone.
In fact $\h(f)$ can easily be computed as a continuous monotone function 
of $\K_1(f)$: Simply note that a constant 
slope map, such as $F_s(x)=s|x|-1$ has topological entropy
$\h(F_s)=\log(s)$ for $1\le s\le 2$  (see \cite{MS}).
Since the correspondence
$s\mapsto\K_1(F_s)$ is strictly monotone and easily
computable, it is not hard to locate the unique $s_0$ such that
$$ \sup_{s<s_0}\,\K_1(F_s)~\le ~\K_1(f)~\le~ \inf_{s>s_0}\,\K_1(F_s)$$
to any required degree of accuracy; and it follows that the
topological entropy is just \hbox{$\h(f)=\log(s_0)$.}
For a rough plot of entropy see \autoref{F-entpic}; which can be compared
with \autoref{F-bonearc}.)

The entropy $0\le \h\le\log(2)$ takes 
its maximal value of $\log(2)$ if and only
if $\K_1=1$; 
and $\h=0$ if and only if $\K_1\le ~0.649816\ldots$.

Kneading theory is more complicated in the bimodal case, and
computation of entropy from kneading theory is more difficult; 
but still quite feasible. (See Block and Keesling \cite{BK}.)
\end{rem}

\vspace*{1ex plus 1fil}
\appendix
\section{The Simplest Examples.}\label{a2}
This appendix will illustrate all possible combinatorics with
$n\le 4$,  plus  a few  cases with $n=5$  and $n=6$,
subject to a few restrictions. Specifically, we consider only
those cases which are minimal and nonpolynomial (compare
\autoref{s2}), and for which $\kappa\le 0$. Cases with $\kappa>0$ can be
obtained from these by orientation reversal (i.e.,\ by a   $180^\circ$ rotation
of the graph.) Compare \autoref{s3}, together with \autoref{R-I}.
\msk

This appendix is subdivided into  two main sections:
\textbf{\textit{unobstructed}} and \textbf{\textit{obstructed maps}}. Each
section is divided into five parts: first the hyperbolic 
combinatorics of types B, C, and  D, and then the half-hyperbolic
and totally non-hyperbolic cases.  The parameters $\mu$, $\kappa$,
$\Sigma$ and $\Delta$ for all figures in the appendix are in Table~\ref{t-2}.

\phantom{menace}
\ifthenelse{\IsThereSpaceOnPage{.3\textheight}}{\clearpage}{\relax} 
\subsection*{Unobstructed Combinatorics}\label{s-unobst} 

\subsubsection*{Type B: Both critical points in a common periodic orbit.}
\newcommand{\figHt}{.21\textwidth}
\newcommand{\bigfigHt}{1.7in}

  \begin{figure}[!htb]
  \centerline{%
    \includegraphics[height=\figHt]{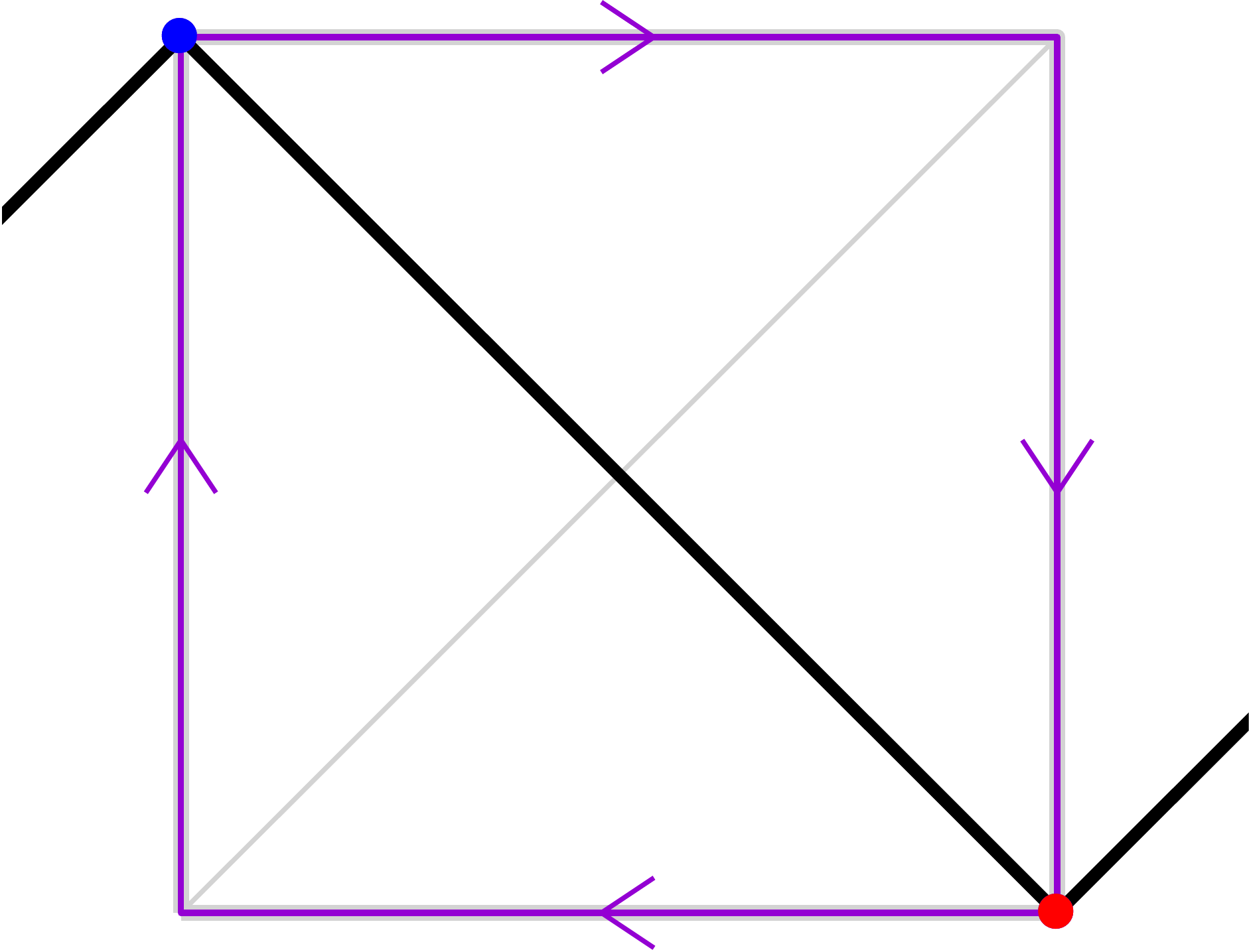}
    \includegraphics[height=\figHt]{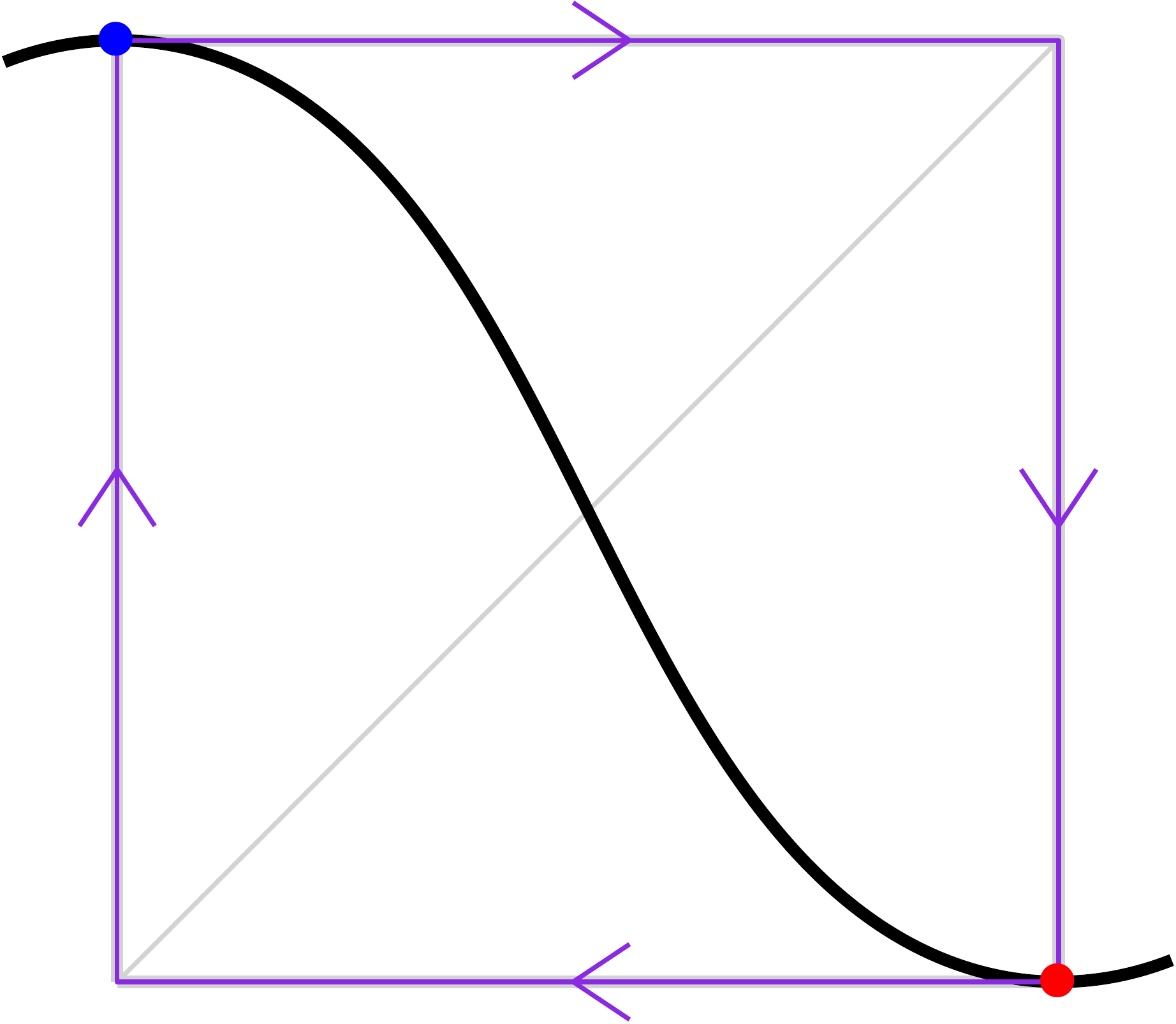}\hfill
    \includegraphics[height=\figHt]{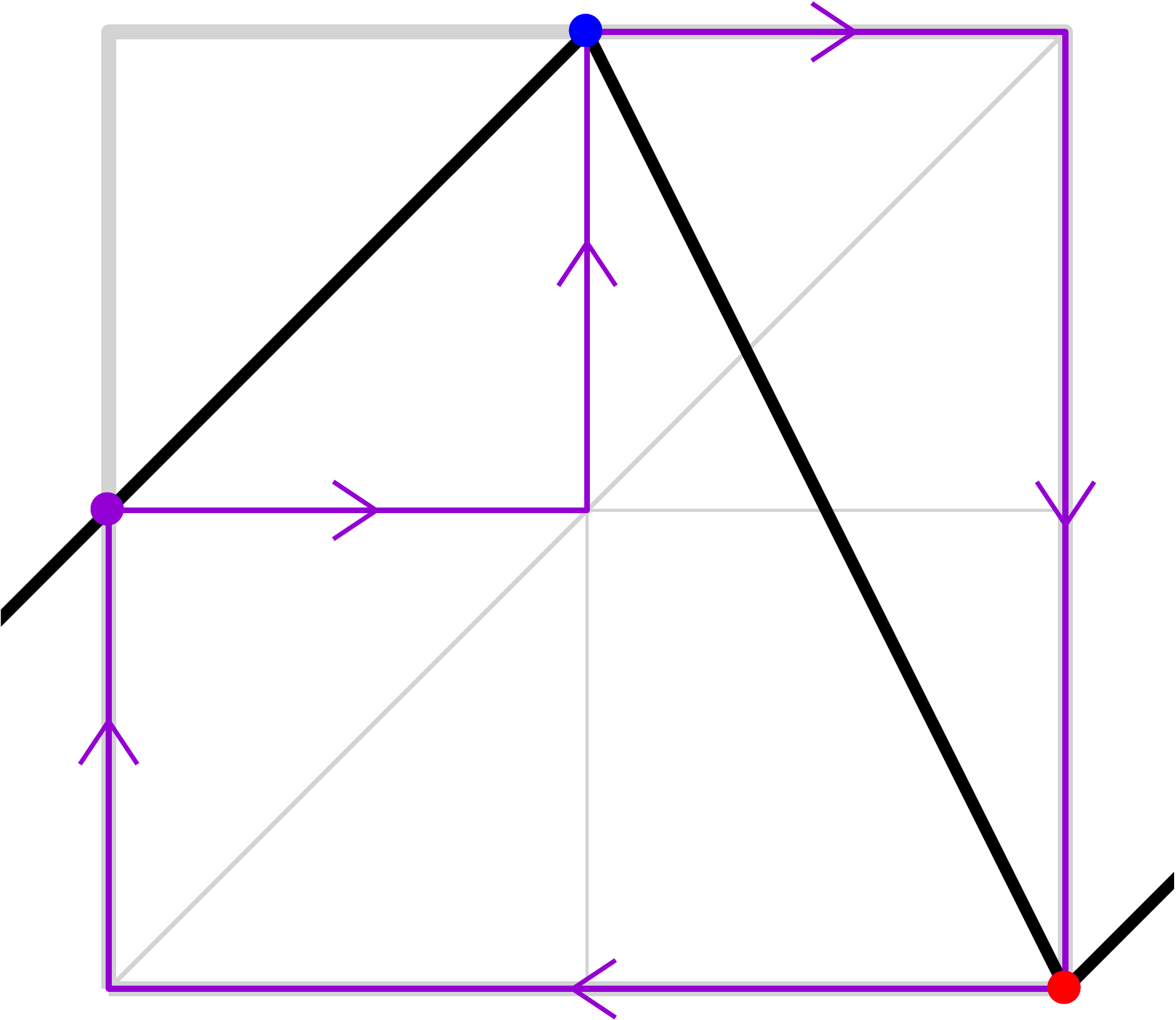}
    \includegraphics[height=\figHt]{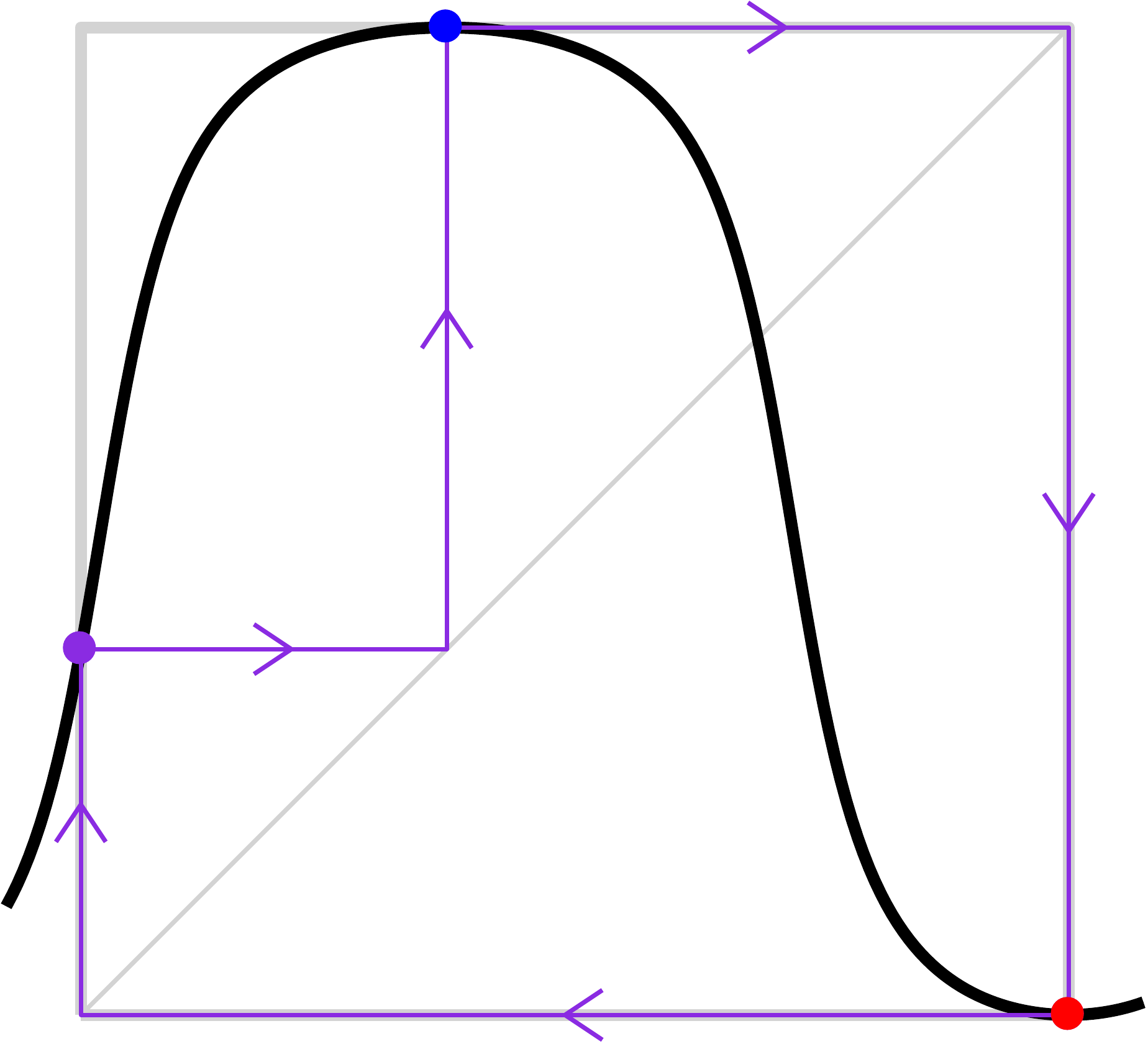}}
 \caption{ \label{f-aBn2}   Co-polynomial maps. On the left are illustrations
   of a piece-wise linear map and its corresponding rational map for the
   combinatorics $\(1,\,0\)$  of topological shape $-$~and mapping pattern
   $\du{x_0}\leftrightarrow \du{x_1}$~. On the right are illustrations of
   a map of topological  shape
   $+-$~ for   combinatorics $\(1,\,2,\,0\)$ with  mapping pattern
   $\du{x_2} \mapsto x_0\mapsto \du{x_1} \mapsto \du{x_2}$~.}
\end{figure} 

\begin{figure}[!htb]
  \centerline{%
    \includegraphics[height=\figHt]{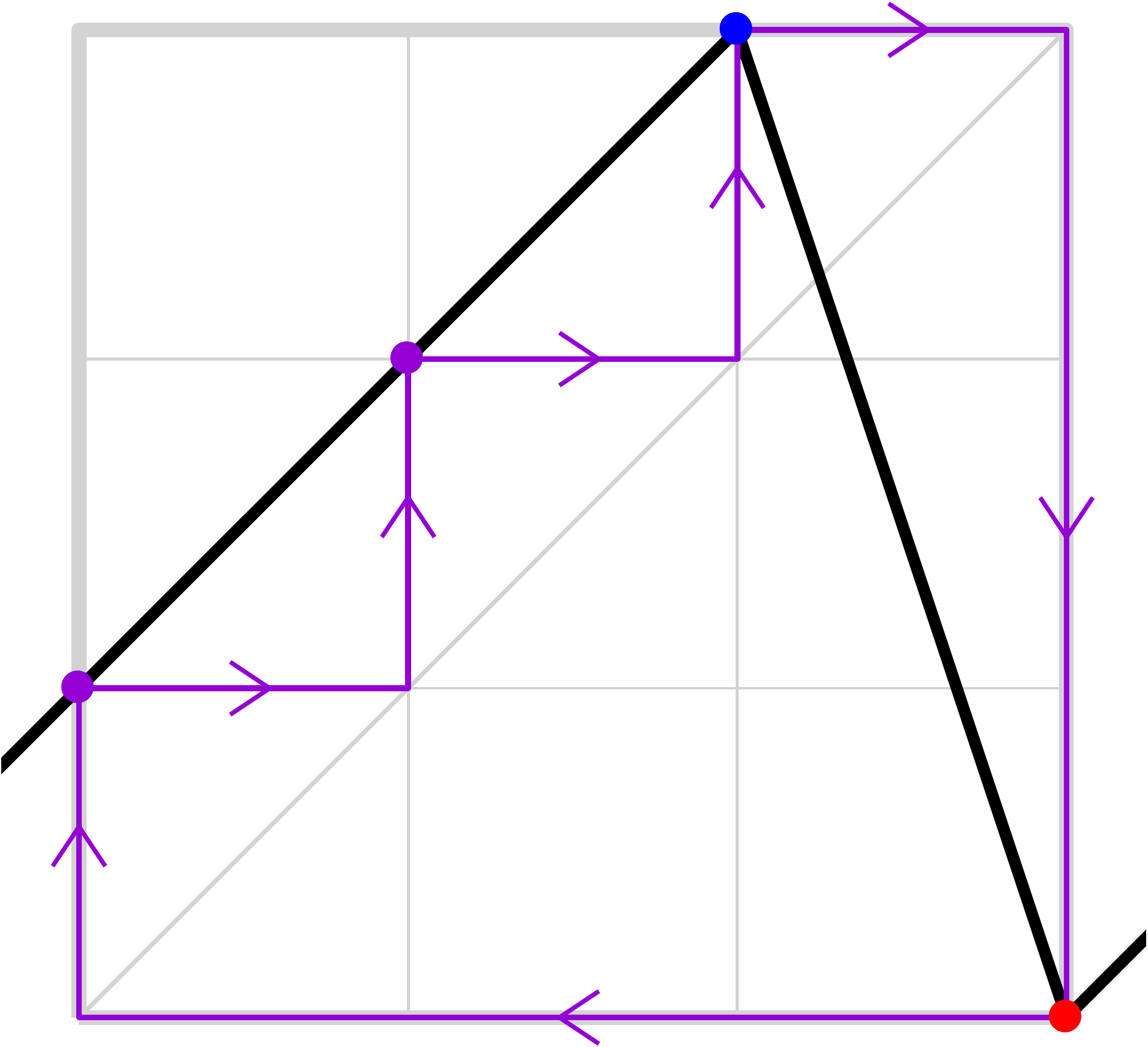} \;
    \includegraphics[height=\figHt]{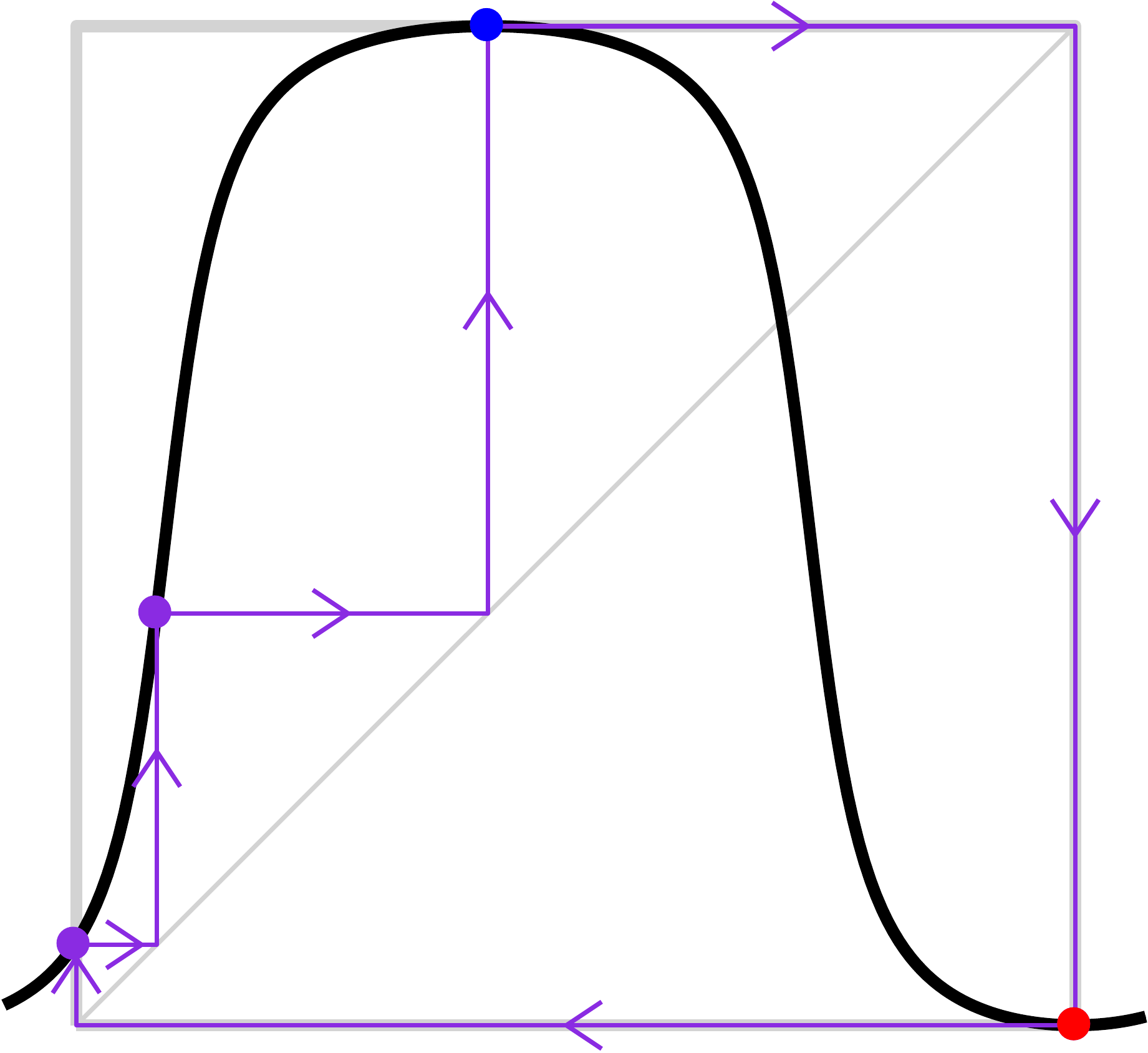}\hfill
    \includegraphics[height=\figHt]{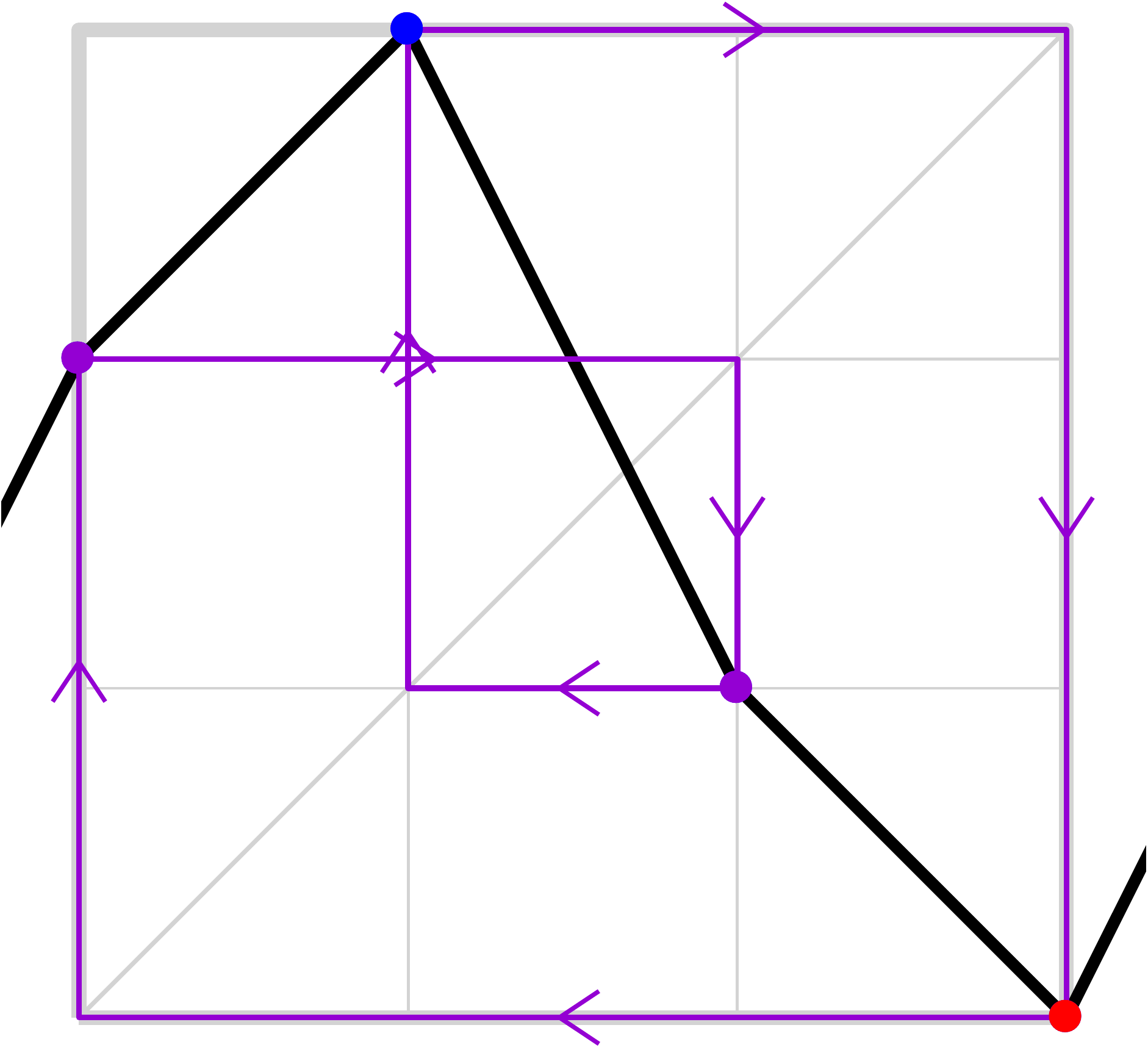} \;
    \includegraphics[height=\figHt]{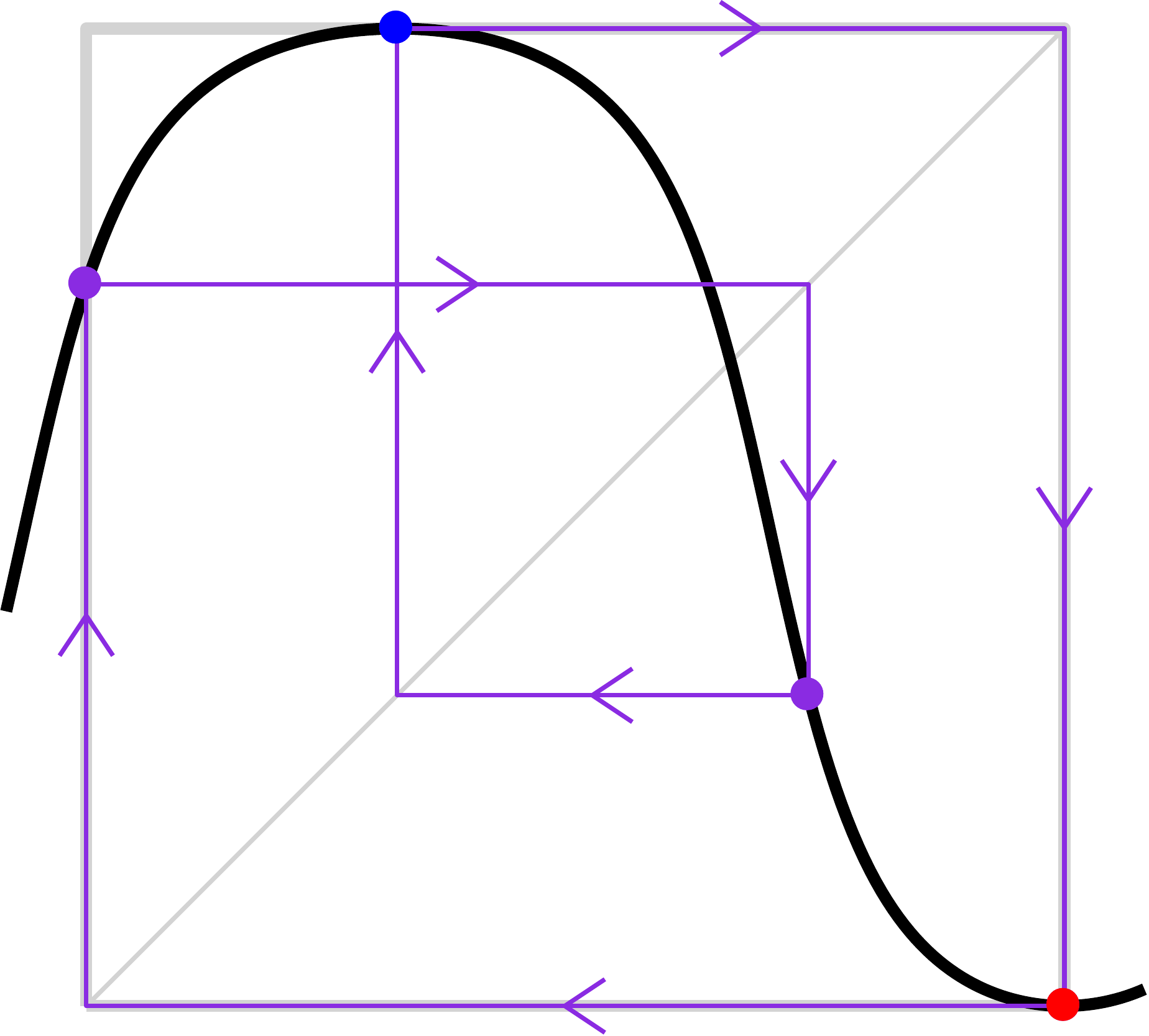}}
 \caption{ \label{f-aBn3}  Co-polynomial maps of topological  shape
   $+-$~. On the left are illustrations of a piecewise linear map and its
   corresponding rational map for the
  combinatorics $\(1,\,2,\,3,\,0\)$ with mapping pattern $\du{x_3}\mapsto x_0
  \mapsto x_1 \mapsto \du{x_2} \mapsto \du{x_3}$~. On the right are
  illustrations for  the combinatorics
  $\(2,\,3,\,1,\,0\)$ with mapping pattern $\du{x_3}\mapsto x_0 \mapsto x_2
  \mapsto \du{x_1}\mapsto \du{x_3}$~.} 
\end{figure}

\begin{figure}[!htb]
  \centerline{%
    \includegraphics[height=\figHt]{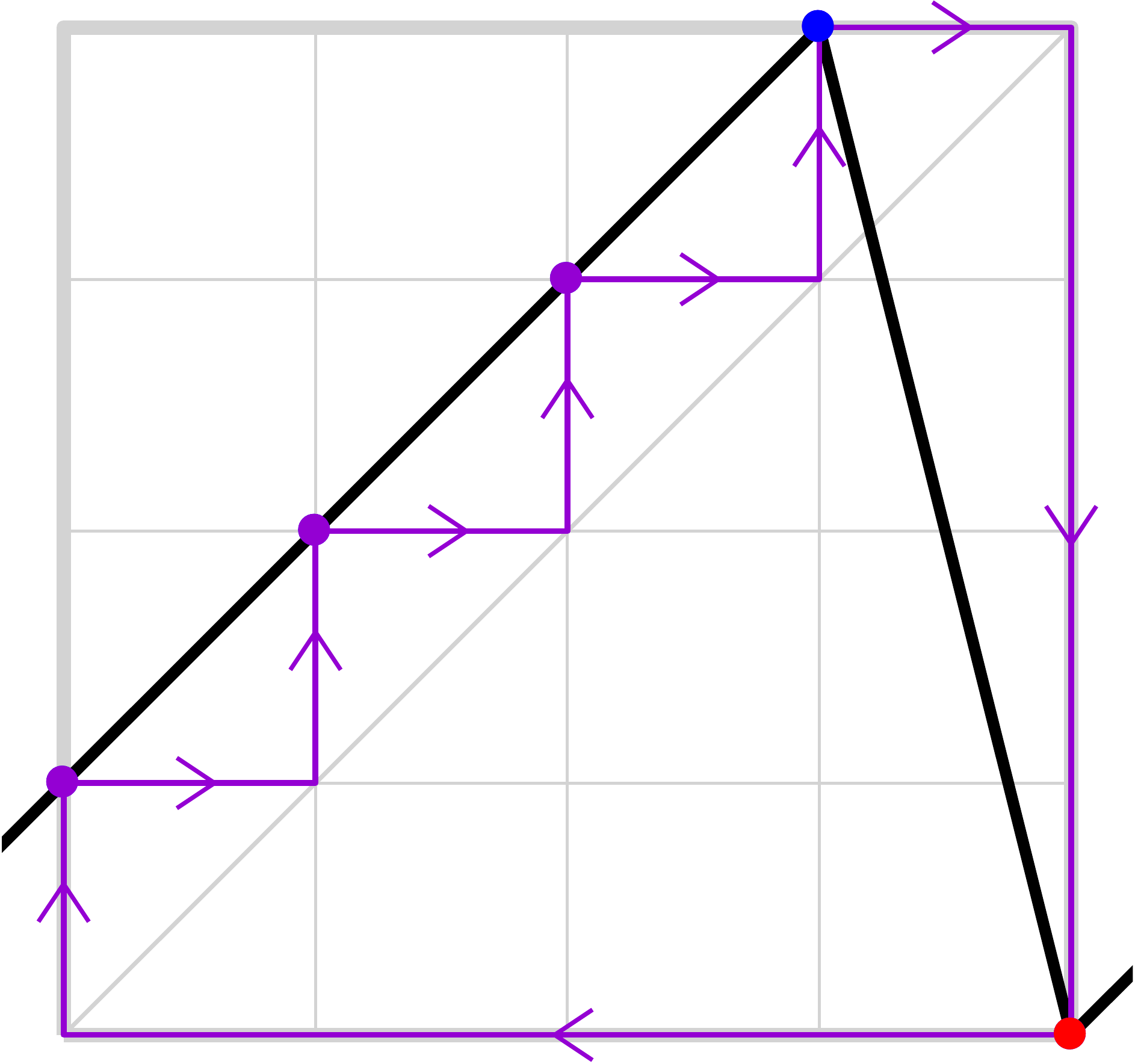} \,
    \includegraphics[height=\figHt]{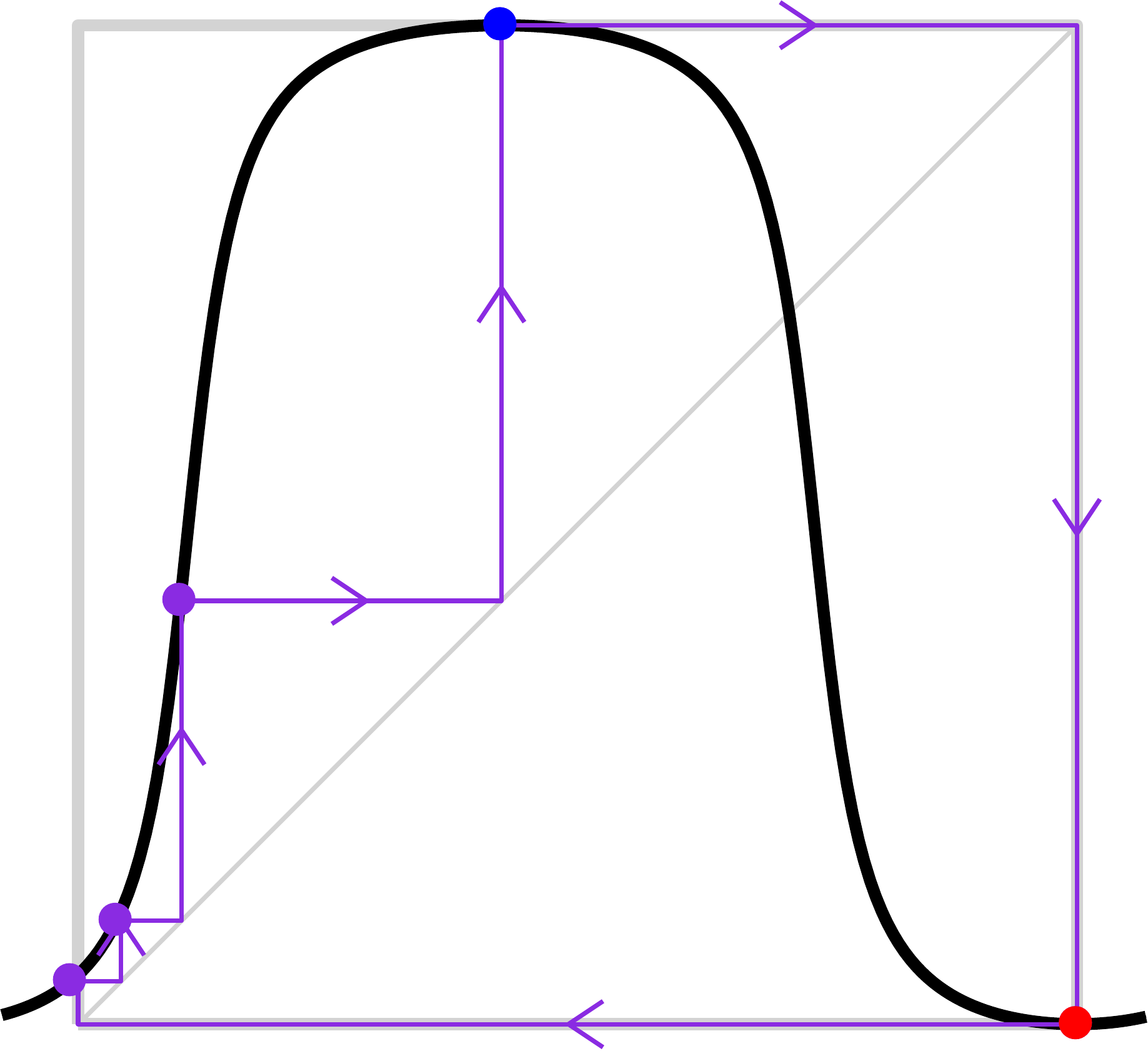}\hfill
    \includegraphics[height=\figHt]{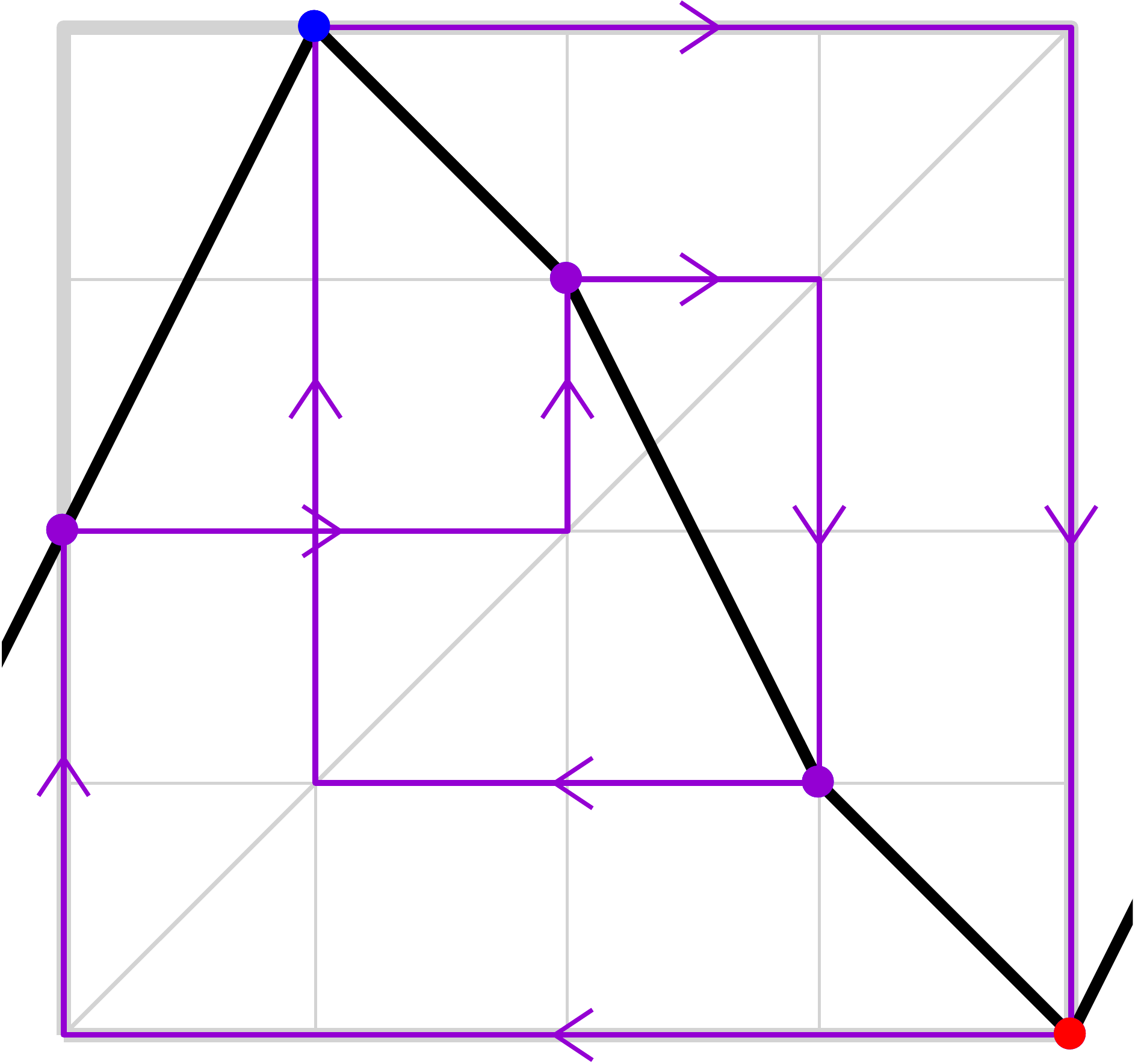} \,
    \includegraphics[height=\figHt]{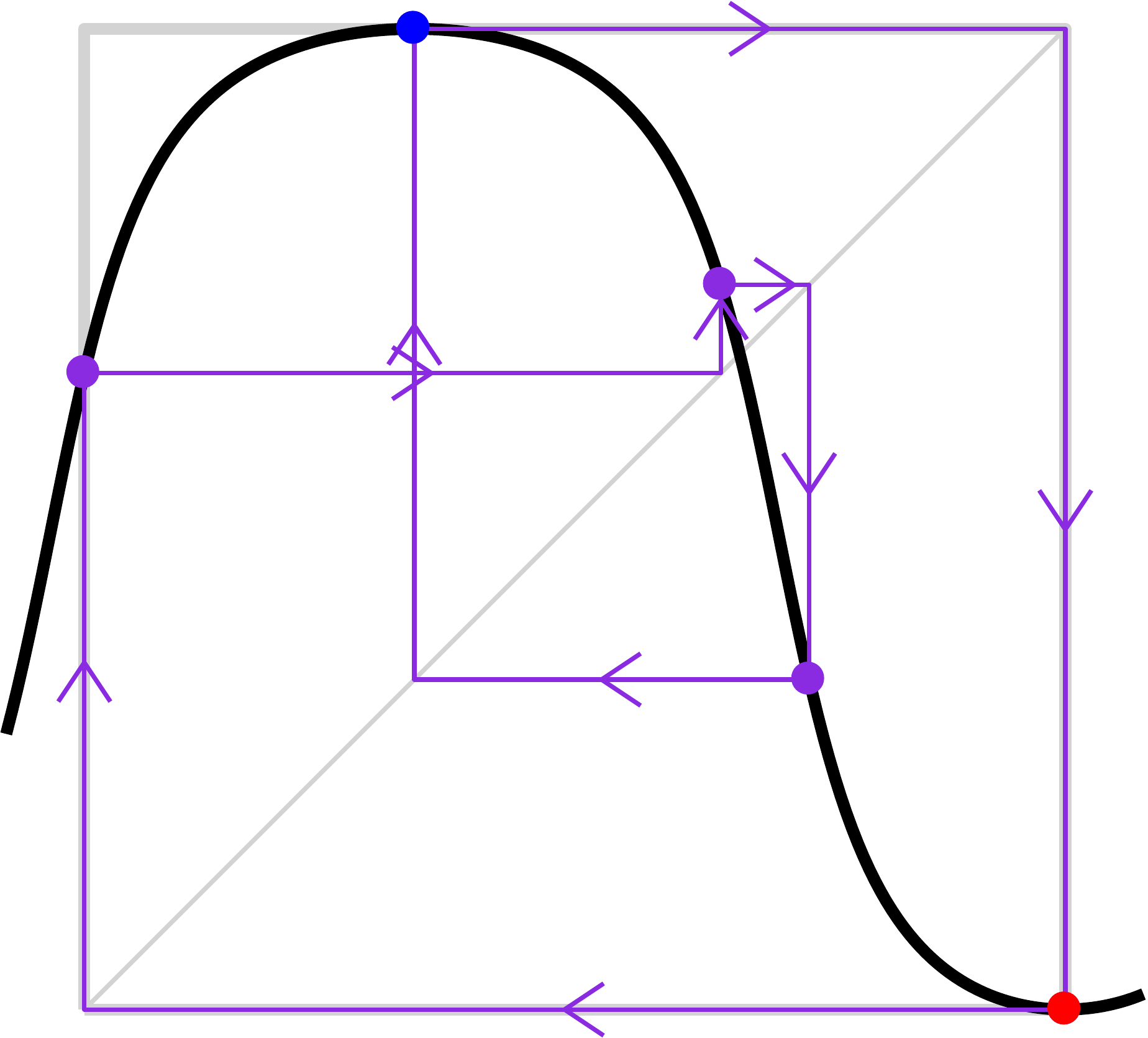}}
  \caption{\label{f-aBn4}   Co-polynomial maps of topological 
    shape $+-$~. On the left are illustrations of a piecewise linear map
    and its corresponding  rational map for the  combinatorics
    $\(1,\,2,\,3,\,4,\,0\)$ with mapping pattern {$\du{x_4}
      \mapsto x_0\mapsto x_1 \mapsto x_2 \mapsto \du{x_3} \mapsto \du{x_4}$}.
    On the right are illustrations for the
    combinatorics $\(2,\,4,\,3,\,1,\,0\)$ with mapping pattern
 $\du{x_4} \mapsto x_0 \mapsto x_2 \mapsto x_3 \mapsto \du{x_1} \mapsto \du{x_4}$~.}
\end{figure} 

\begin{figure}[!htb]
   \centerline{%
    \includegraphics[height=\figHt]{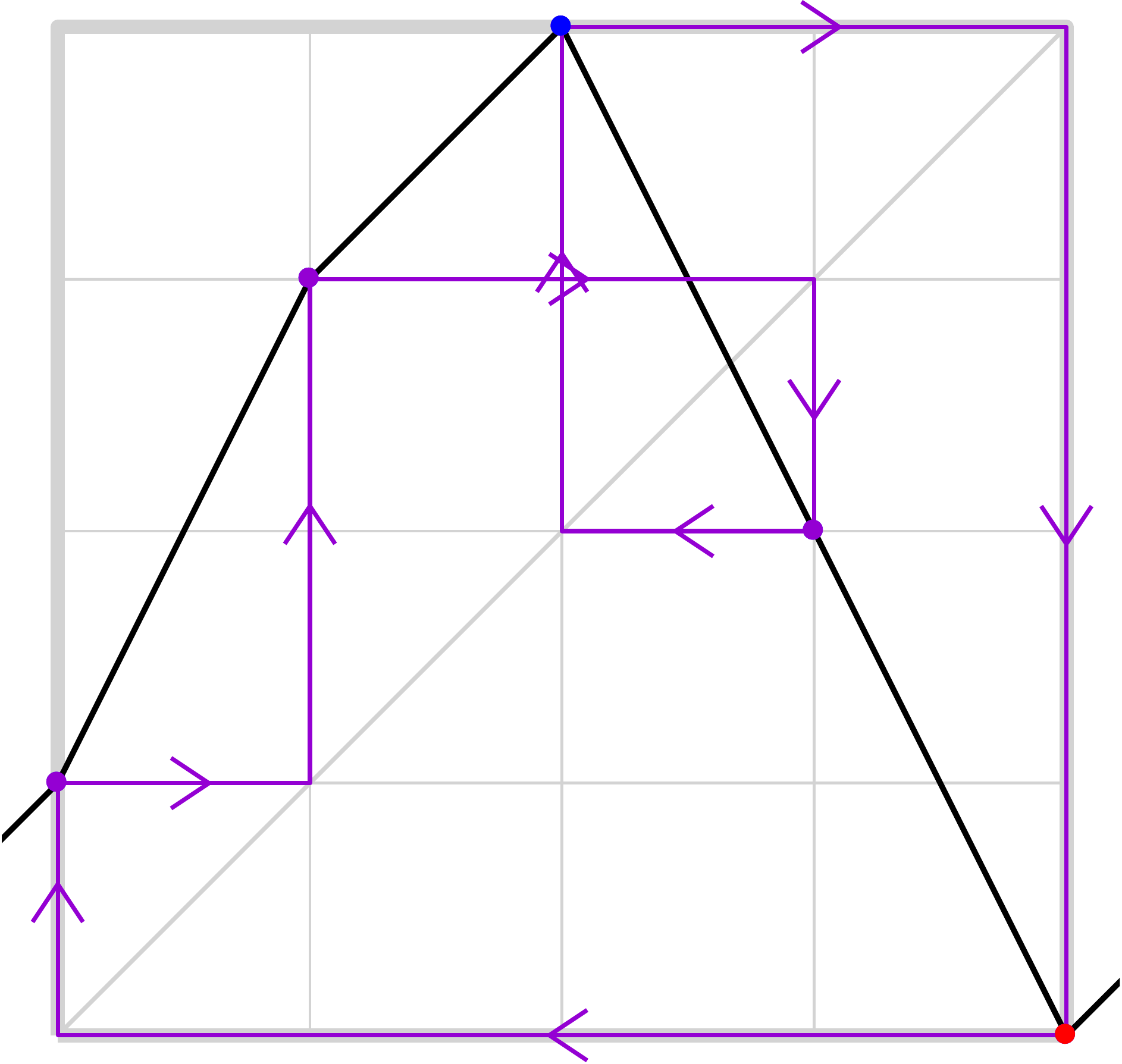} \,
    \includegraphics[height=\figHt]{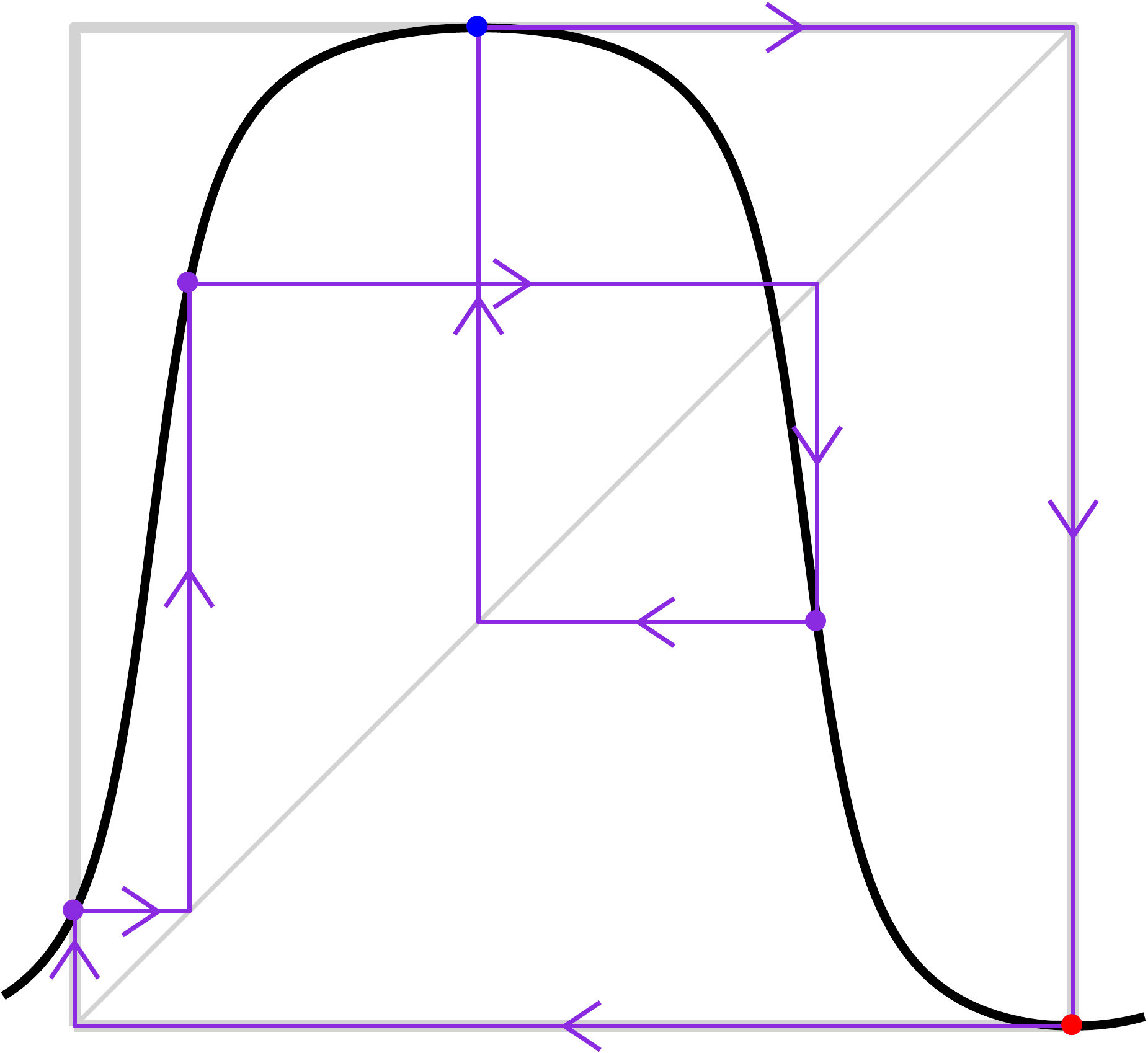}\hfill
    \includegraphics[height=\figHt]{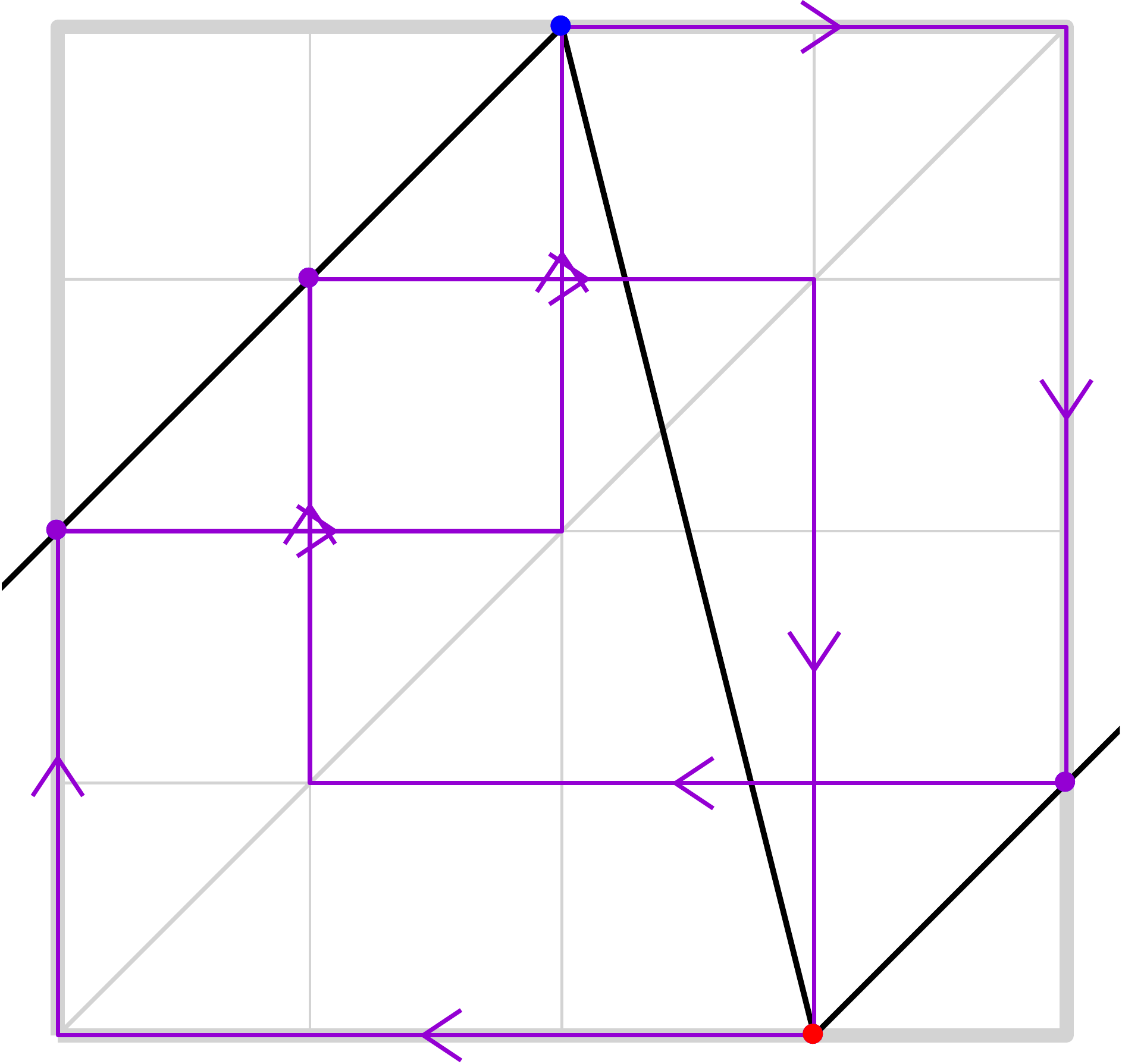} \,
    \includegraphics[height=\figHt]{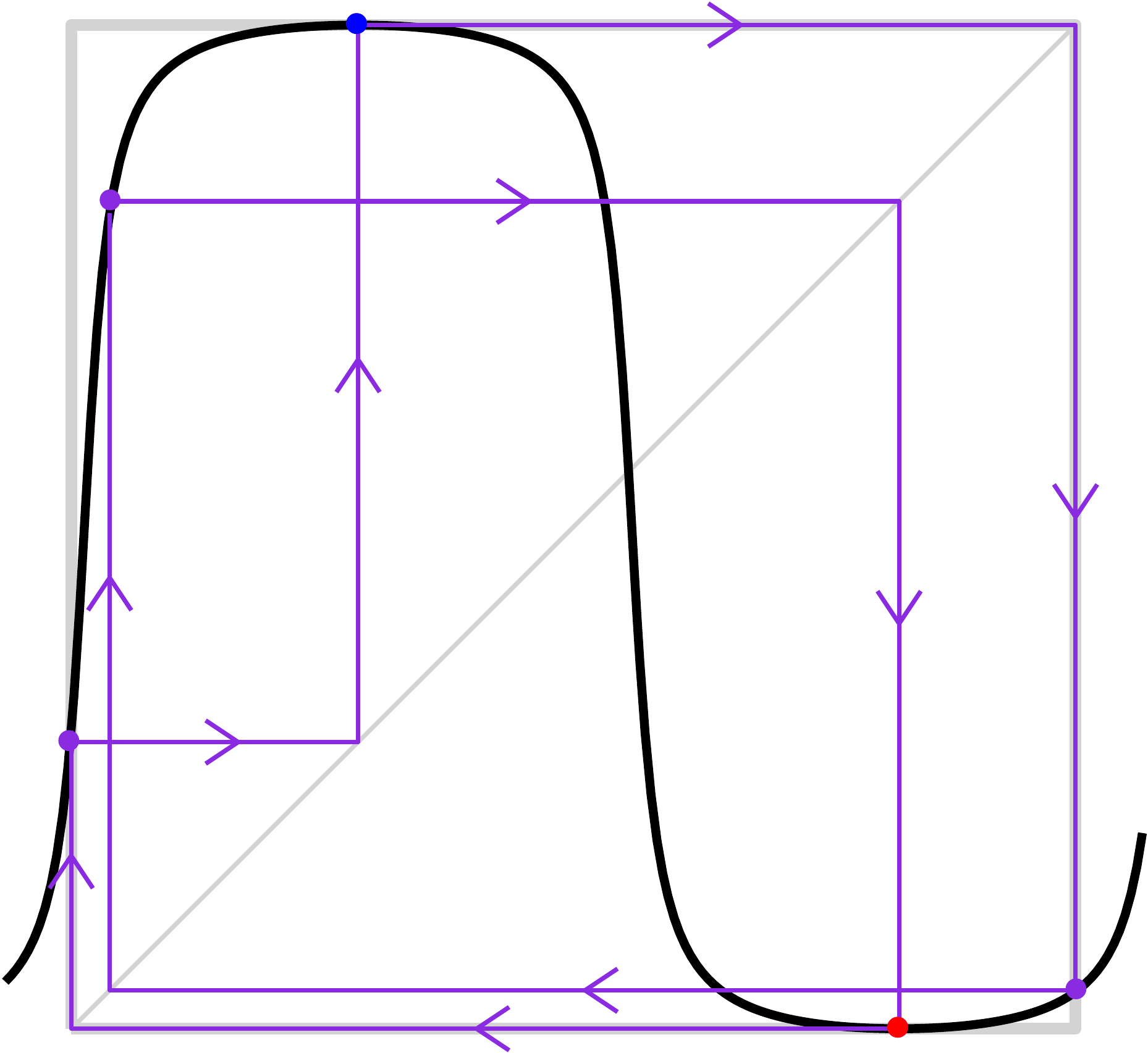}}
  \caption{\label{f-aBn4a}  Left: the co-polynomial with combinatorics
    $\(1,\,3,\,4,\,2,\,0\)$, mapping pattern 
    \hbox{$\du{x_4}\mapsto x_0 \mapsto x_1 \mapsto x_3 \mapsto \du{x_2} \mapsto
      \du{x_4}$} and topological shape $+-$  .
    Right: combinatorics $\(2,\,3,\,4,\,0,\,1\)$,
    mapping pattern $\du{x_3}\mapsto x_0\mapsto \du{x_2} \mapsto x_4
    \mapsto x_1 \mapsto \du{x_3}$, and shape  $+-+$~. }
\end{figure}

\begin{figure}[!htb]
\centerline{
    \includegraphics[height=\figHt]{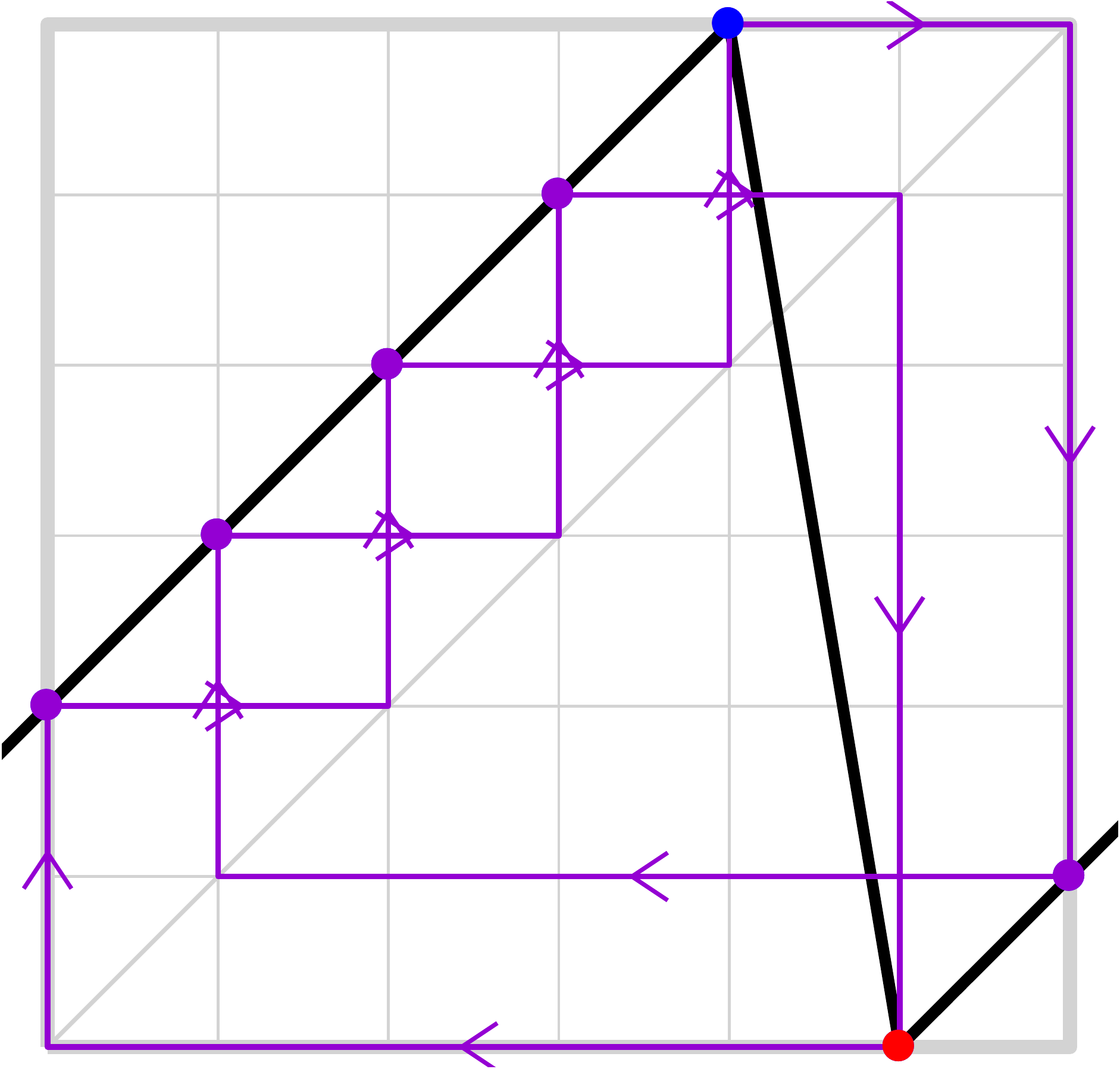}
    \includegraphics[height=\figHt]{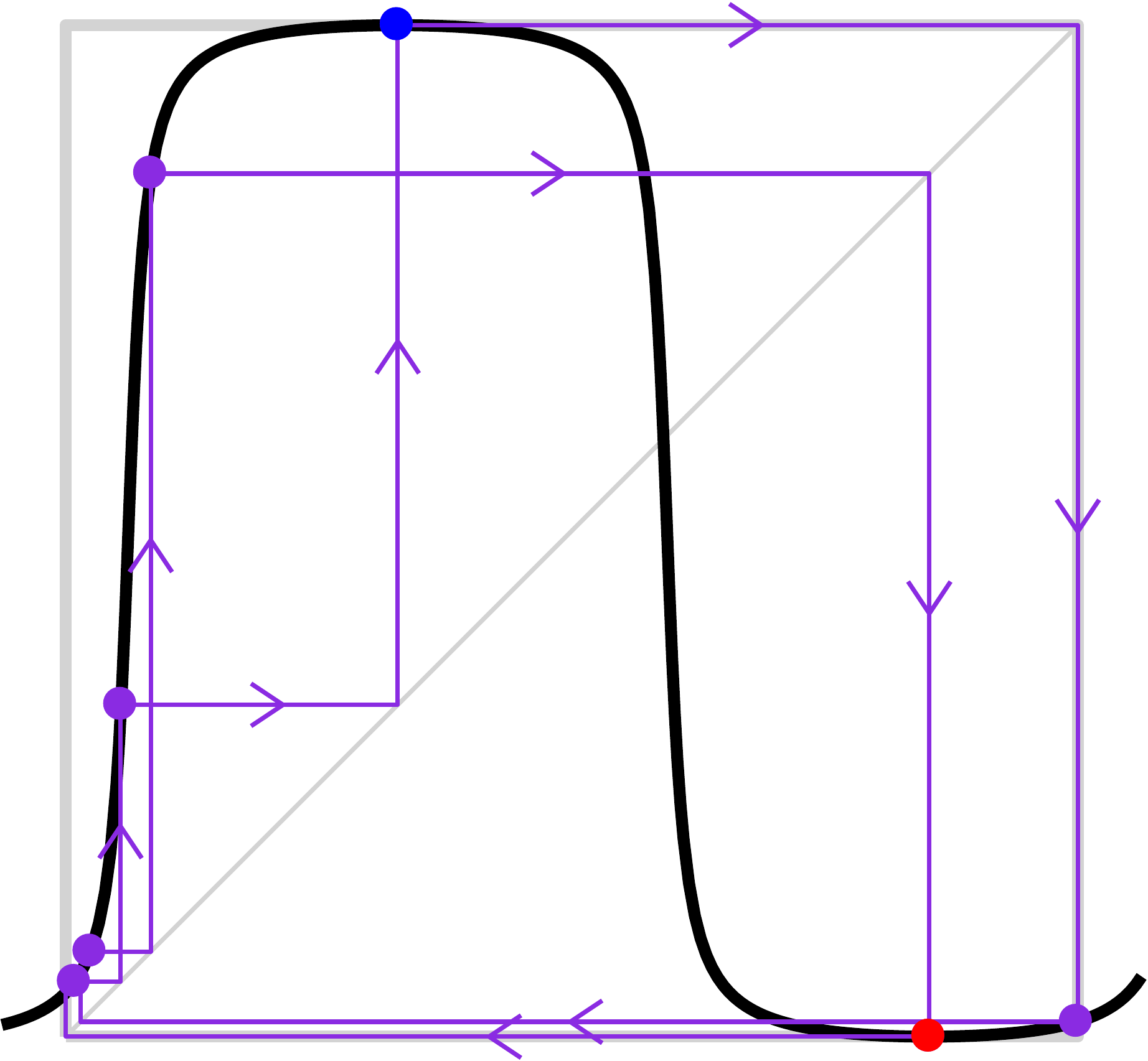}}
\caption{\label{f-aBn5a}  An illustration of the  piecewise
  linear map and its corresponding rational map for the 
combinatorics $\(2,\,3,\,4,\,5,\,6,\,0,1\)$ with  shape $~+-+~$ and 
  mapping pattern
    \hbox{$\du{x_5}\mapsto x_0 \mapsto x_2 \mapsto \du{x_4} \mapsto x_6 \mapsto
      x_1\mapsto x_3\mapsto   \du{x_5}$}~.}
\end{figure}

\FloatBarrier
\phantom{menace}
\ifthenelse{\IsThereSpaceOnPage{.3\textheight}}{\clearpage}{\relax} 
\subsubsection*{ Type C (Capture): 
 One critical orbit lands in a cycle containing the  other.}

\begin{figure}[!htb]
  \centerline{%
    \includegraphics[height=1.1in]{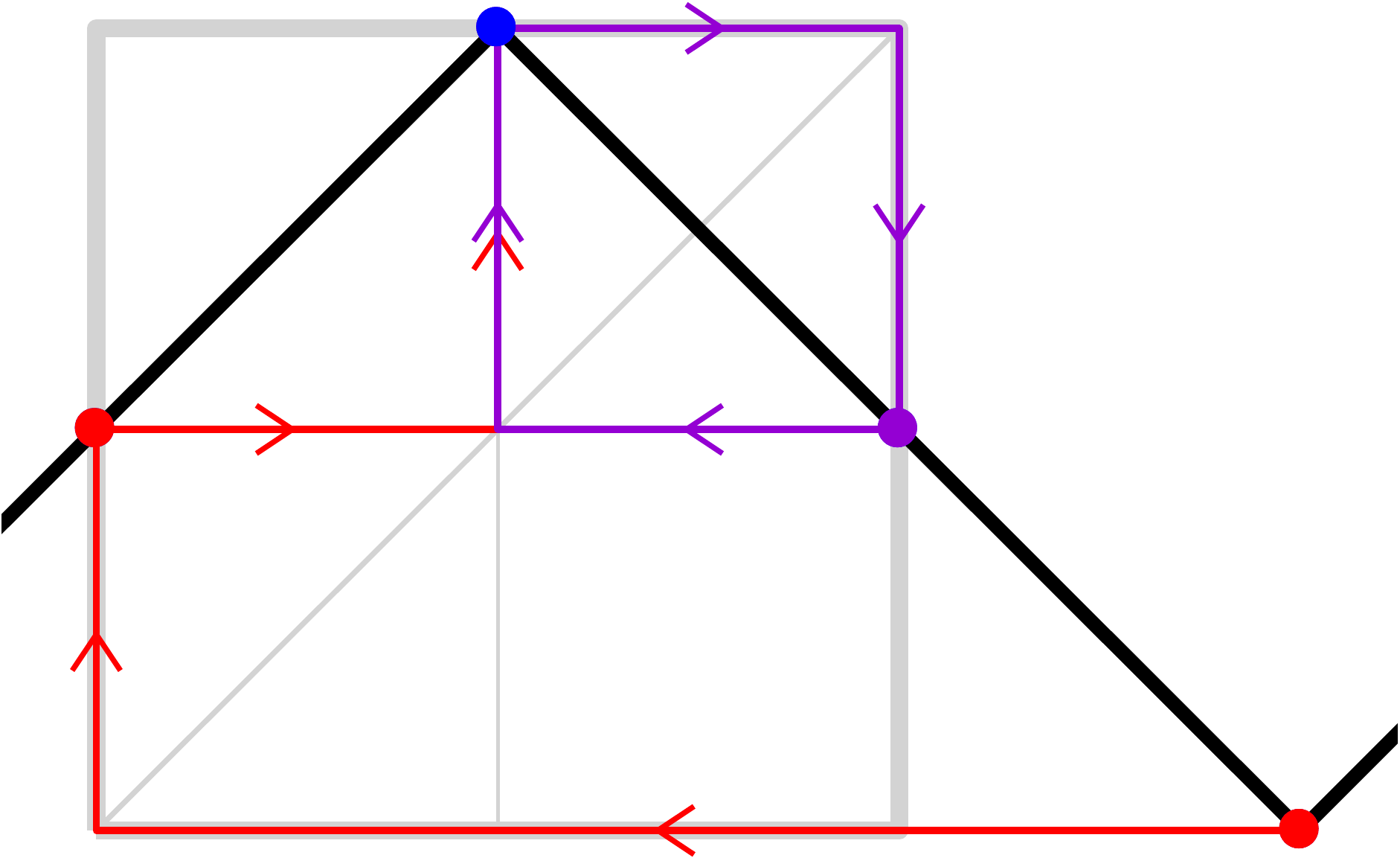} \,
    \includegraphics[height=1.1in]{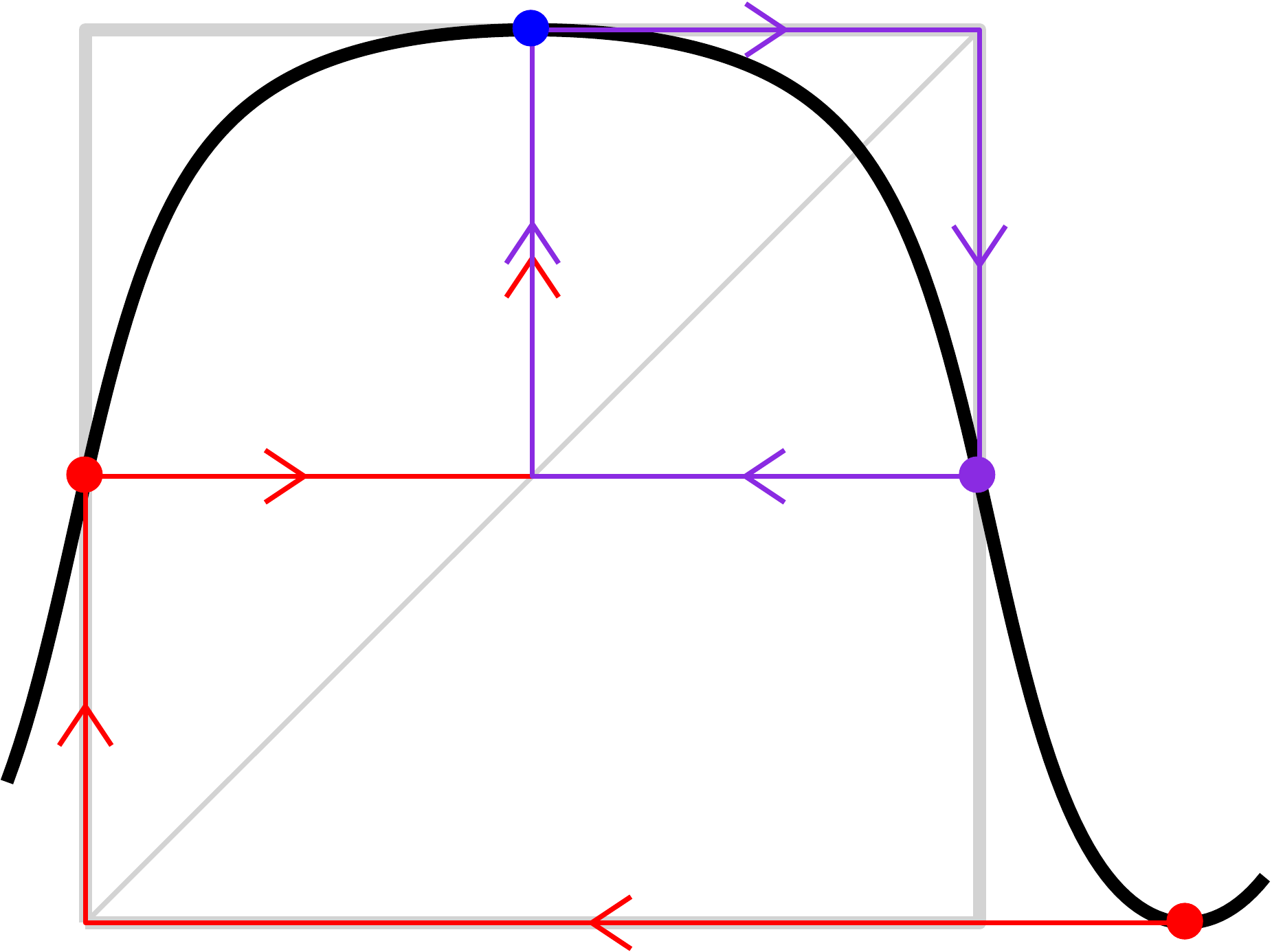}\hfill
    \includegraphics[height=1.1in]{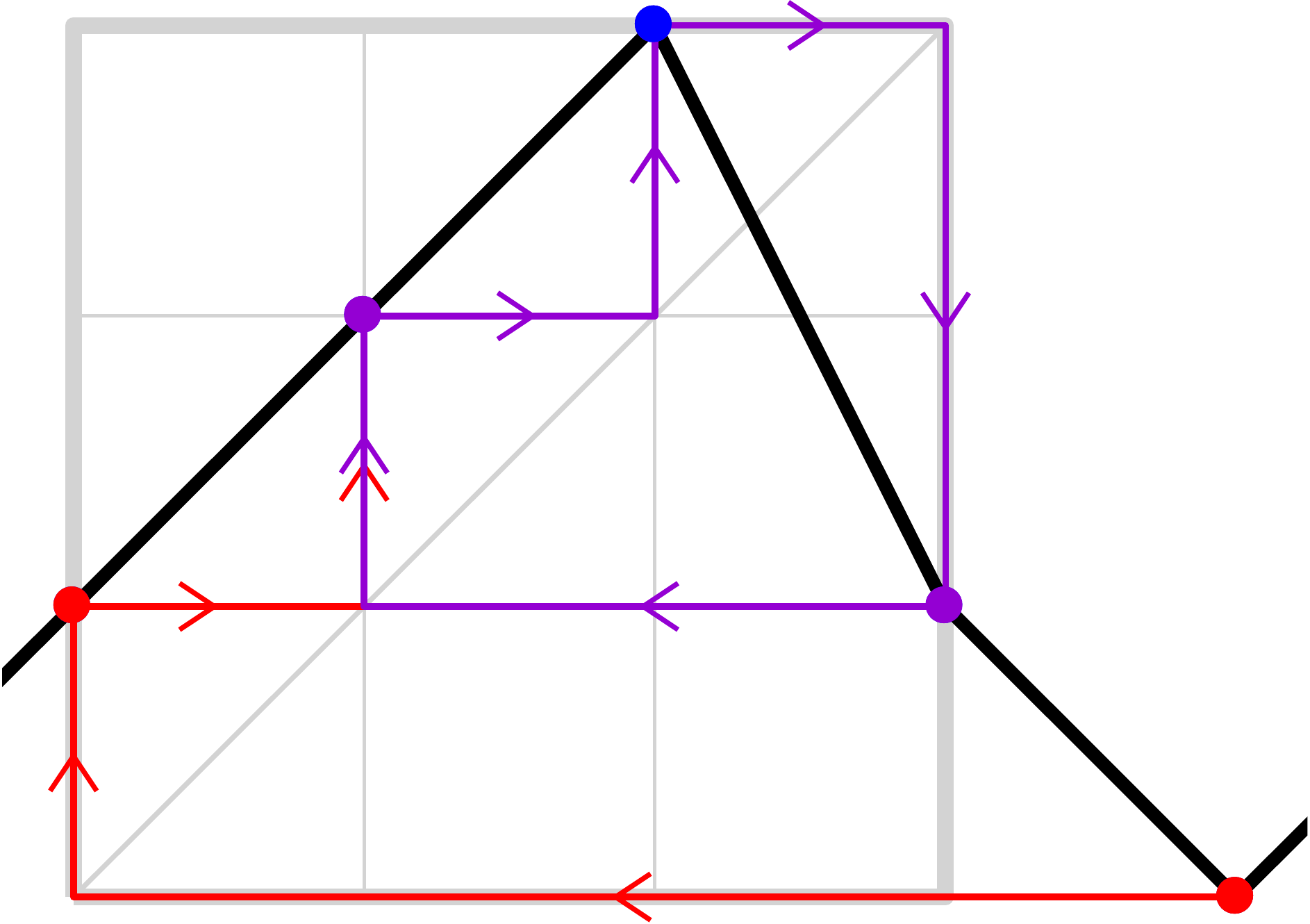} \,
    \includegraphics[height=1.1in]{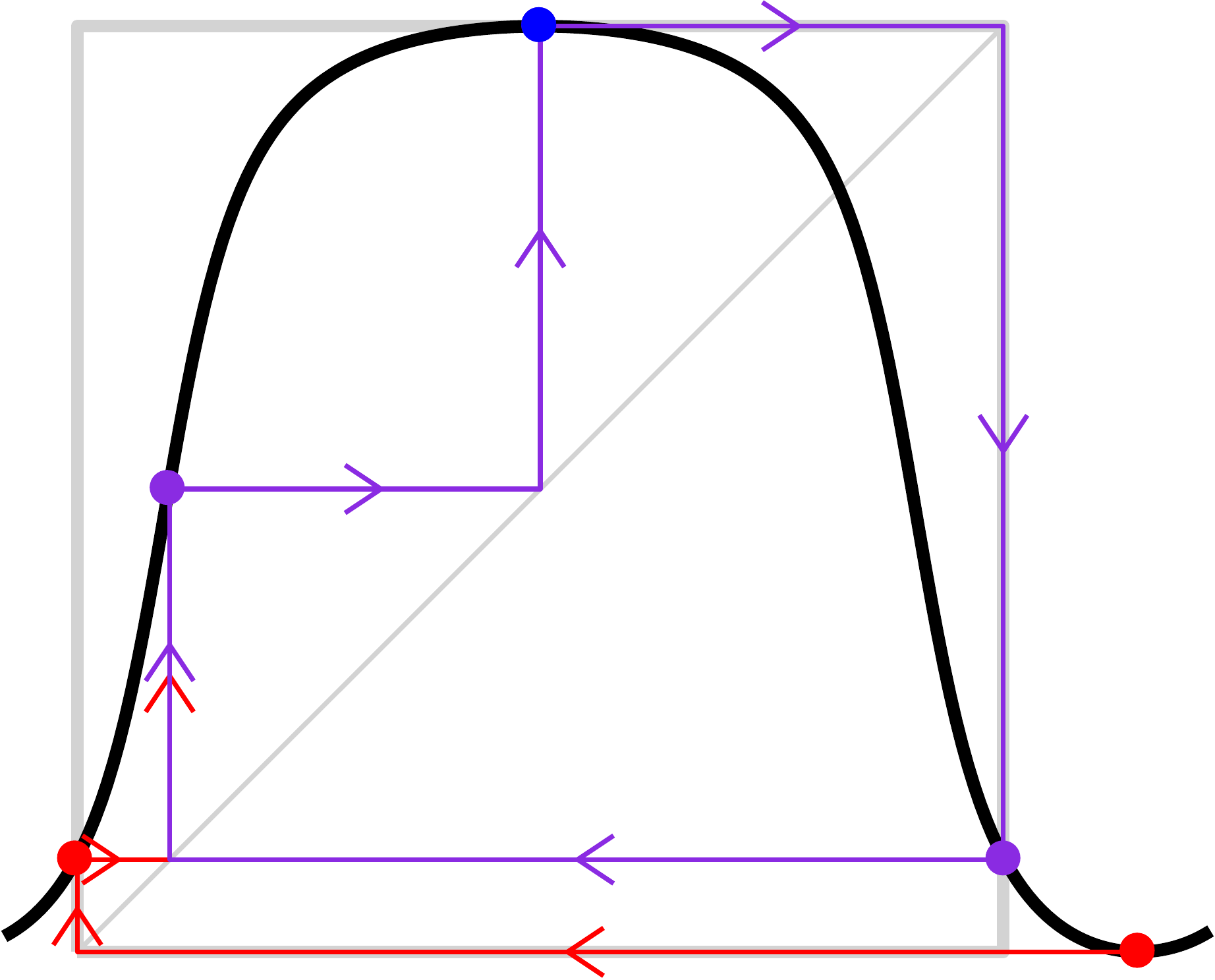}}
\caption{\label{f-aC1}   Strictly unimodal maps on $f(\Rhat)$ of
    topological shape $+-$~. On the left are illustrations of a piecewise
    linear map and its corresponding rational map with $n=3$ and combinatorics
    $\(1,\,2,\,1,\,0\)$. This is a
    period 2 capture case with mapping pattern $\du{x_3}\mapsto x_0 \mapsto
    \du{x_1}\leftrightarrow x_2$. On the right are illustrations of  a map with
    $n=4$ and  combinatorics  $\(1,\,2,\,3,\,1,\,0\)$. This is capture with
    eventual period three. The mapping pattern is
    \hbox{$\du{x_4}\mapsto x_0 \rightarrow x_1 \mapsto \du{x_2}
    \mapsto x_3 \mapsto x_1$}.}
\end{figure}

\begin{figure}[!htb]
\centerline{%
    \includegraphics[height=1.3in]{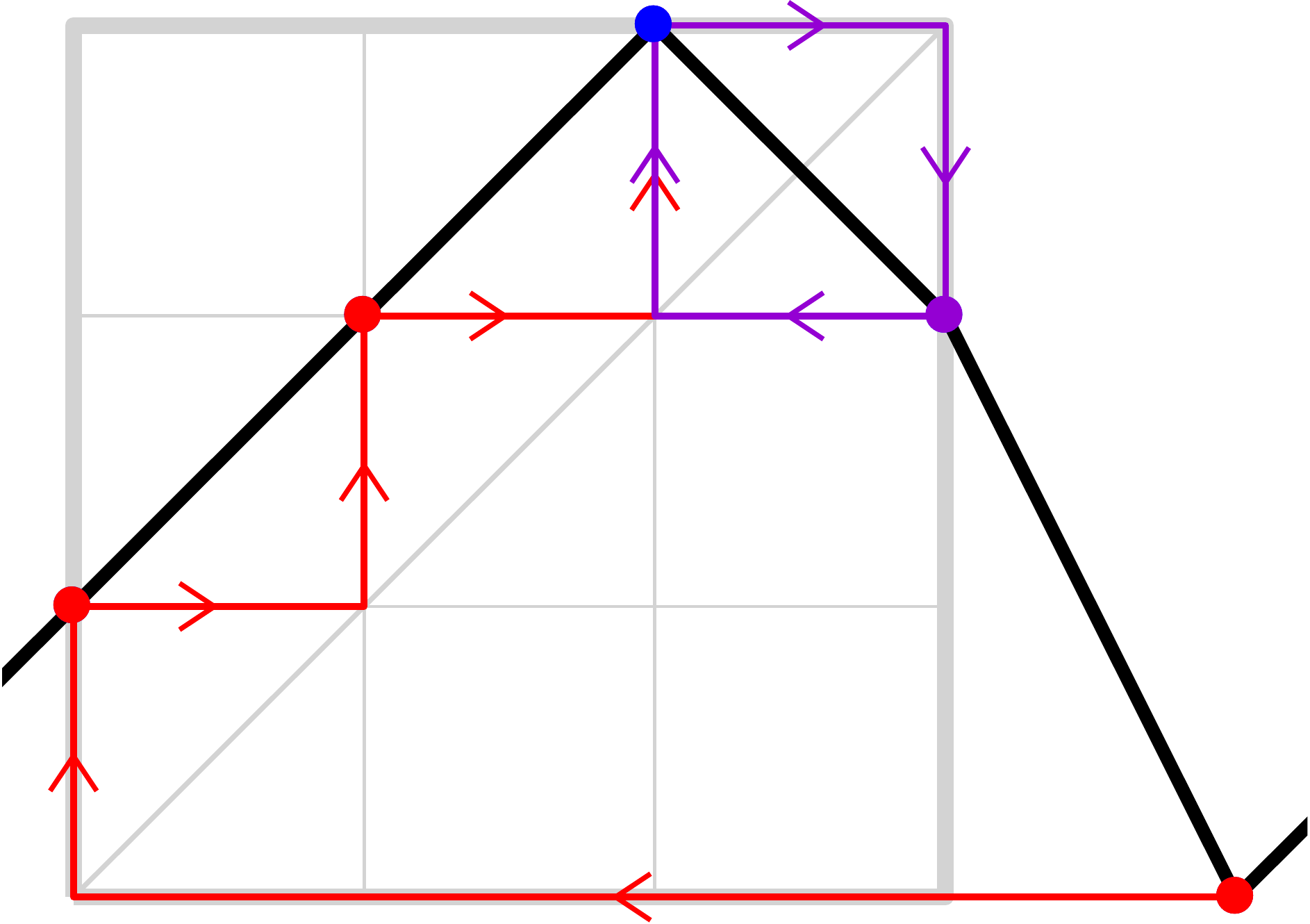} \,
    \includegraphics[height=1.3in]{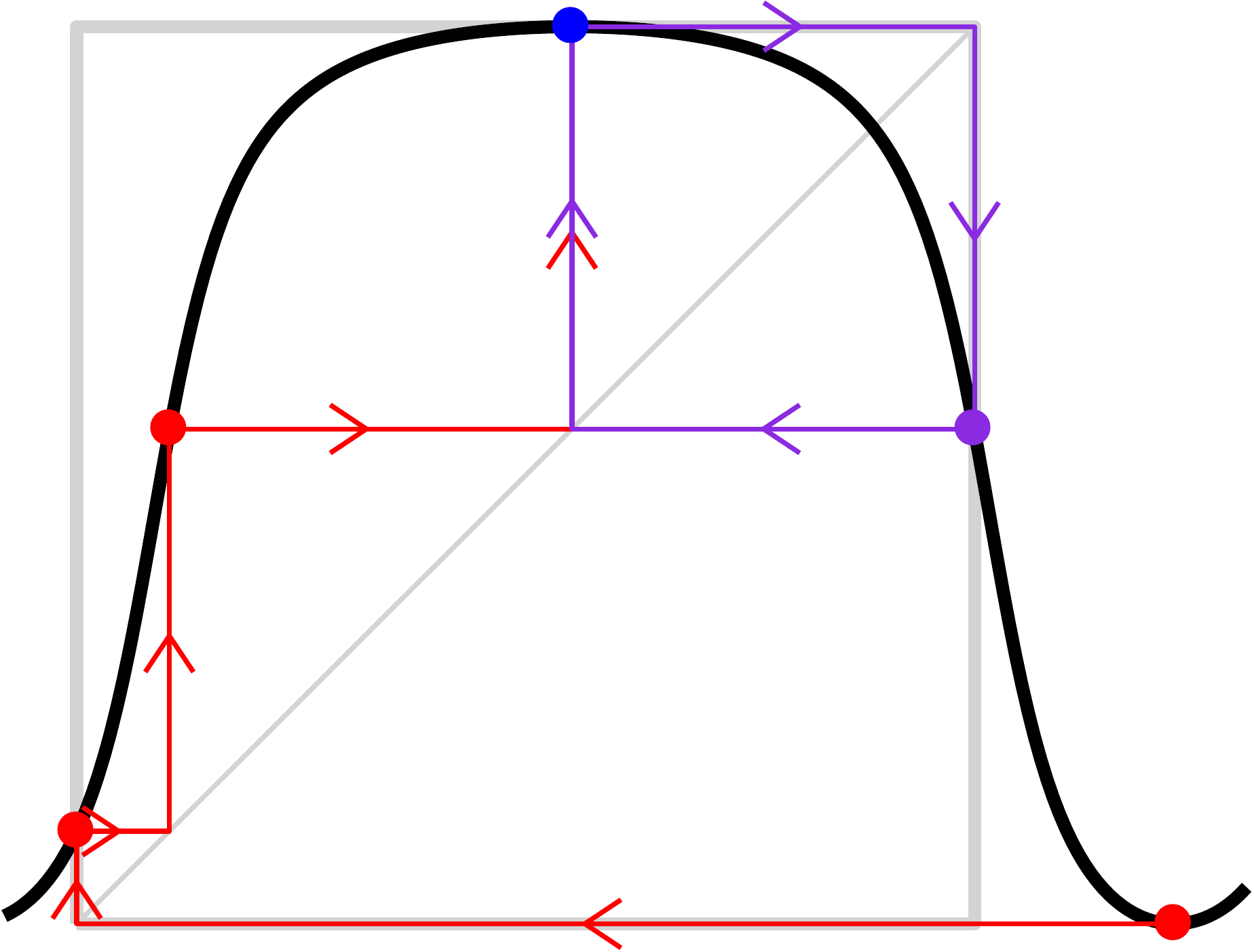}\hfill
    \includegraphics[height=1.3in]{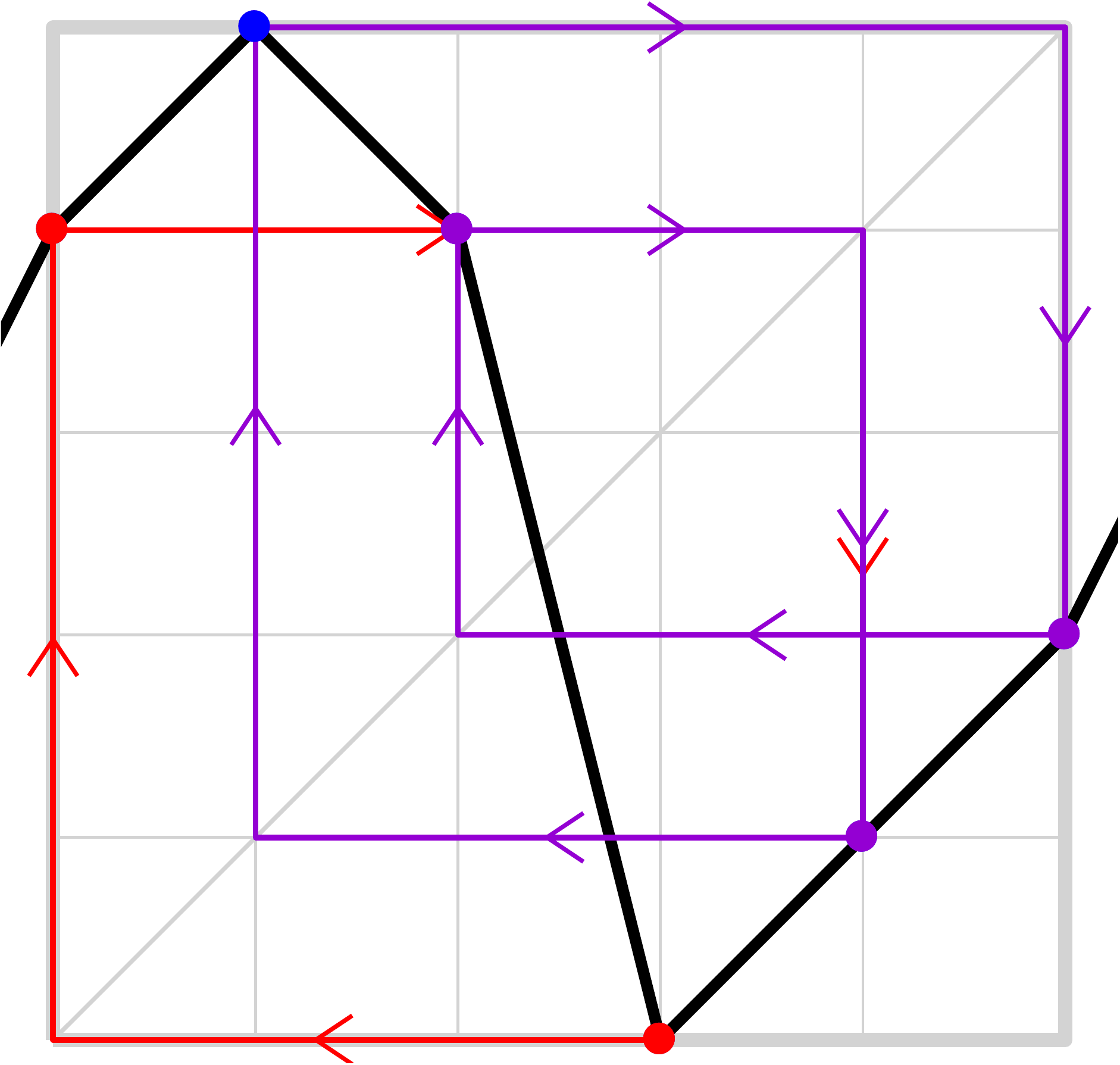} \,
    \includegraphics[height=1.3in]{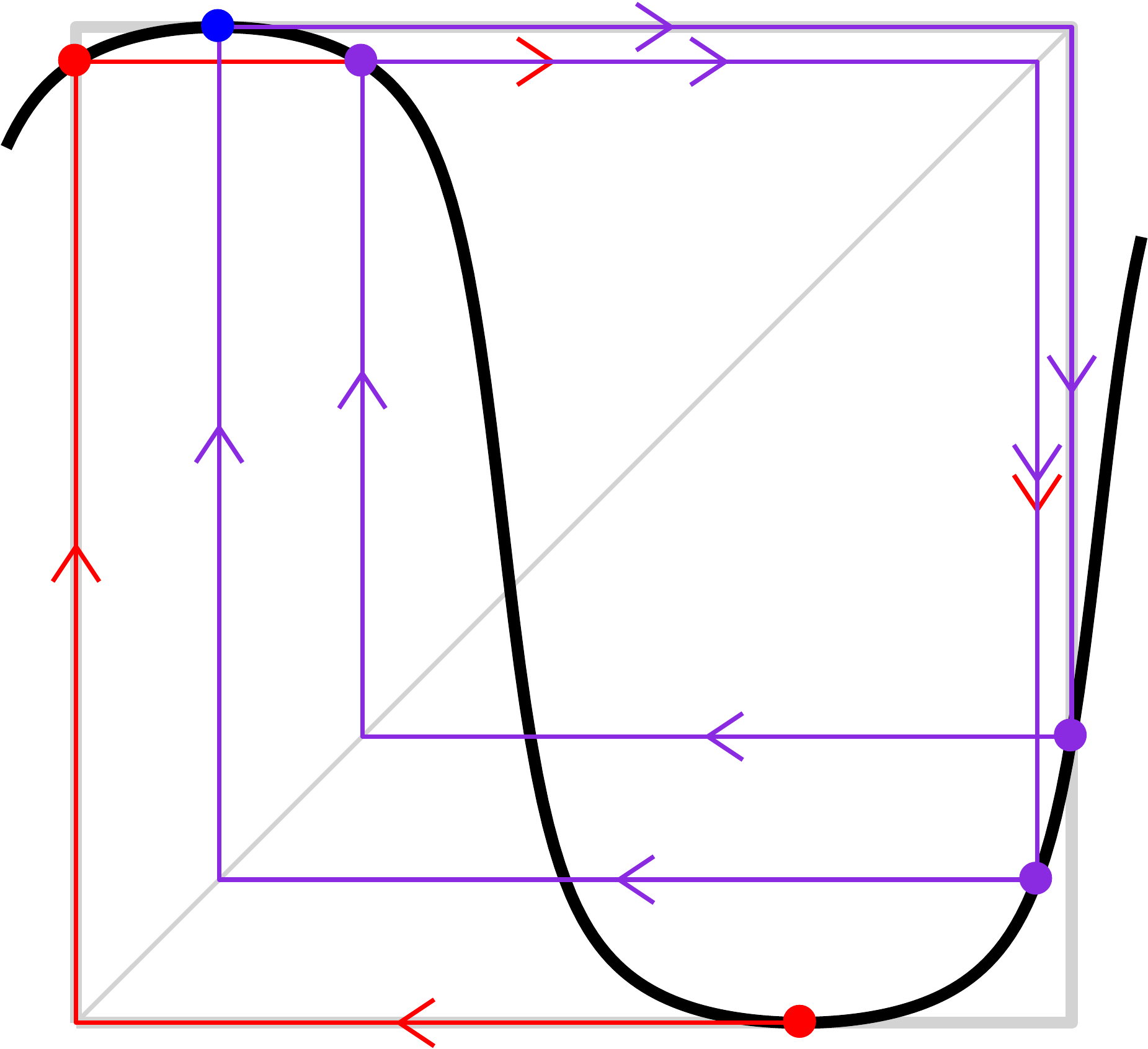}}
  \caption{\label{f-aC2}   On the left are illustrations of 
  a period two capture case  of a strictly unimodal map of  topological shape
  $+-$~,  with combinatorics $\(1,\,2,\,3,\,2,\,0\)$ and  with mapping pattern
  $\du{x_4}\mapsto x_0 \mapsto  x_1\mapsto \du{x_2}\leftrightarrow x_3$.
  As in the previous two cases, one critical point lies outside of $f(\Rhat)$.
  On the right are illustrations of a  capture component for $n=5$ with both
  critical points in
  $f(\widehat{\R})$,  of topological shape $+-+$~,  combinatorics
  $\(4,\,5,\,4,\,0,\,1,\,2\)$,
and mapping pattern \hbox{$\du{x_3} \mapsto x_0 \mapsto x_4 \mapsto \du{x_1}
  \mapsto x_5 \mapsto x_2 \mapsto x_4$}  with a periodic critical orbit of
period 4.}            
\end{figure}

\FloatBarrier
\phantom{menace}
\ifthenelse{\IsThereSpaceOnPage{.3\textheight}}{\clearpage}{\relax}
\subsubsection*{Type D: Two disjoint critical orbits.}

 The only example of unobstructed combinatorics for $n\le 6$ (up to
orientation reversal) is the Wittner example (See figures~\ref{f1} and
\ref{f5}).
Furthermore, as a consequence of \autoref{C-per2obs}, there can be no
unobstructed Type~D examples with a period~2 orbit. An example 
of Type~D with $n=8$ is shown in  \autoref{f-typeD}.

\begin{figure}[!htb]
 \centerline{\includegraphics[height=\bigfigHt]{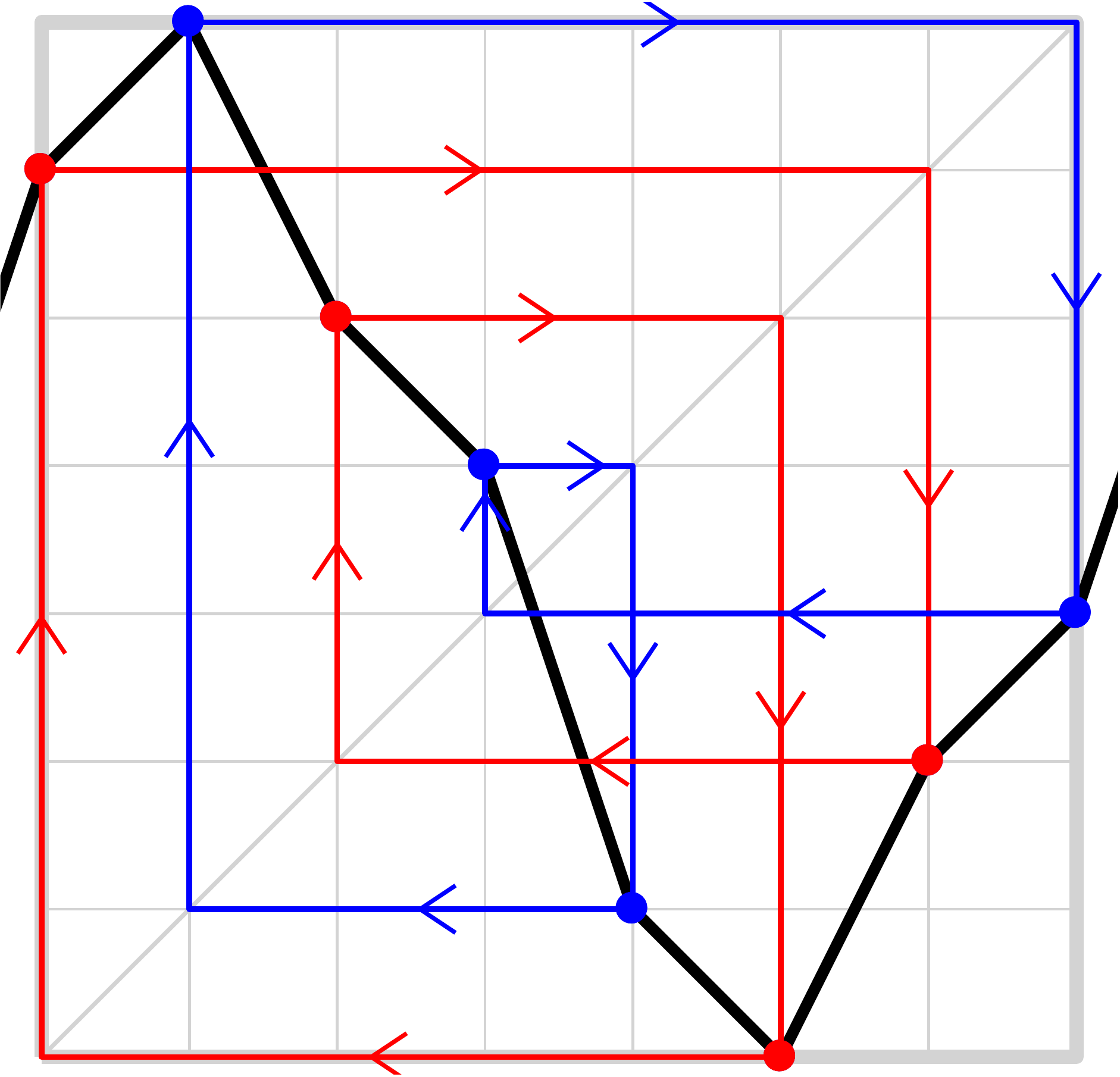} \qquad
             \includegraphics[height=\bigfigHt]{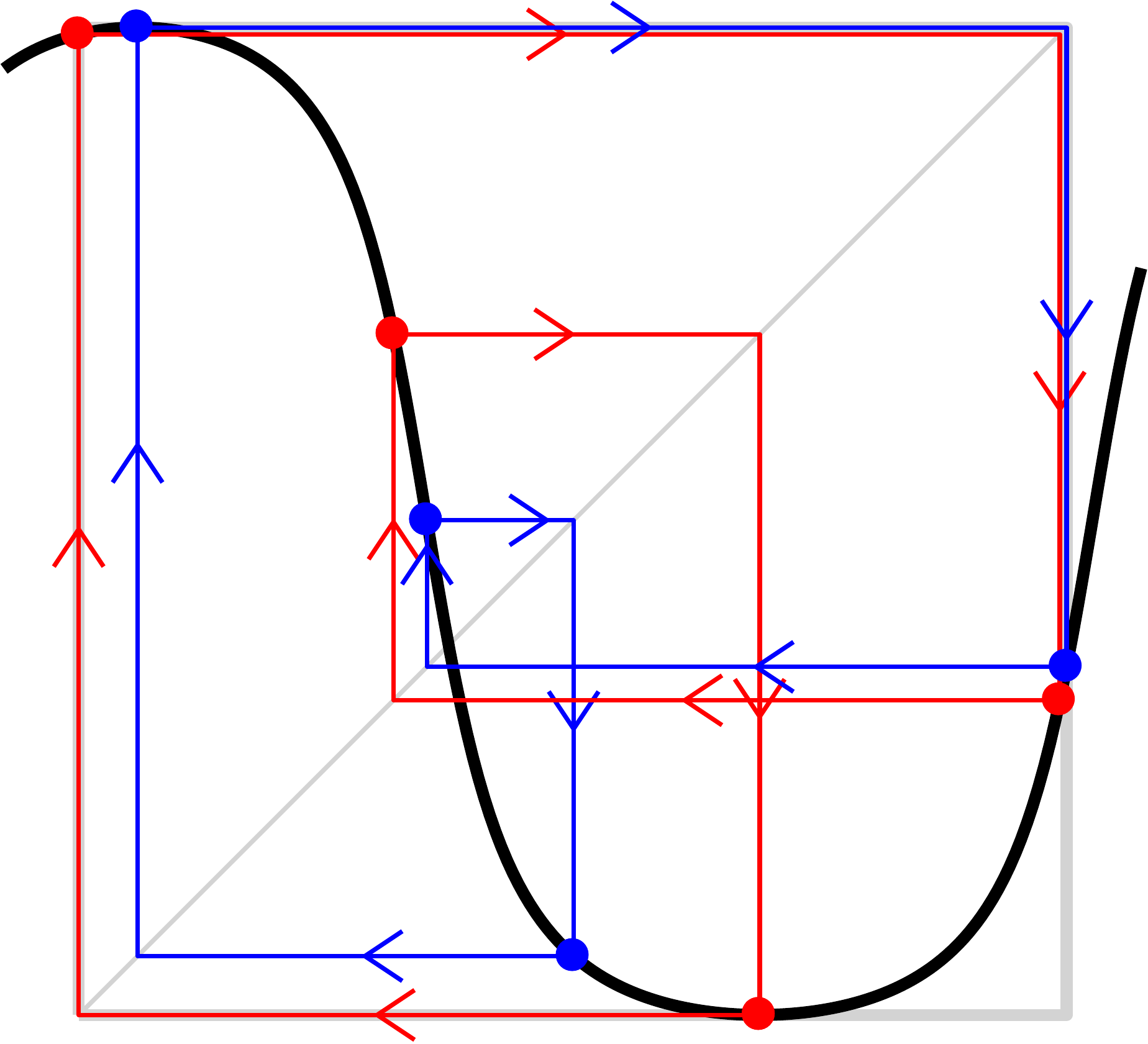}}
 \caption{\label{f-typeD}  
   On the left, the PL model for the type~D map of shape $+-+$ and
   combinatorics
   $\(6,\,7,\,5,\,4,\,1,\,0,\,2,\,3\)$, which has two period four critical
   orbits.  The mapping pattern is
   \hbox{$\du{x_1}\mapsto x_7 \mapsto x_3 \mapsto x_4 \mapsto \du{x_1}$} and
   \hbox{$\du{x_5}\mapsto x_0 \mapsto x_6 \mapsto x_2 \mapsto \du{x_5}$}.
   Its realization as a lifted map is shown on the right. }
 \end{figure}

\FloatBarrier
\phantom{menace}
\ifthenelse{\IsThereSpaceOnPage{.3\textheight}}{\clearpage}{\relax}
\subsubsection*{Half-Hyperbolic: One critical orbit is periodic, 
  the other is eventually repelling. }

\begin{figure}[!h]
  \centerline{%
    \includegraphics[height=1.25in]{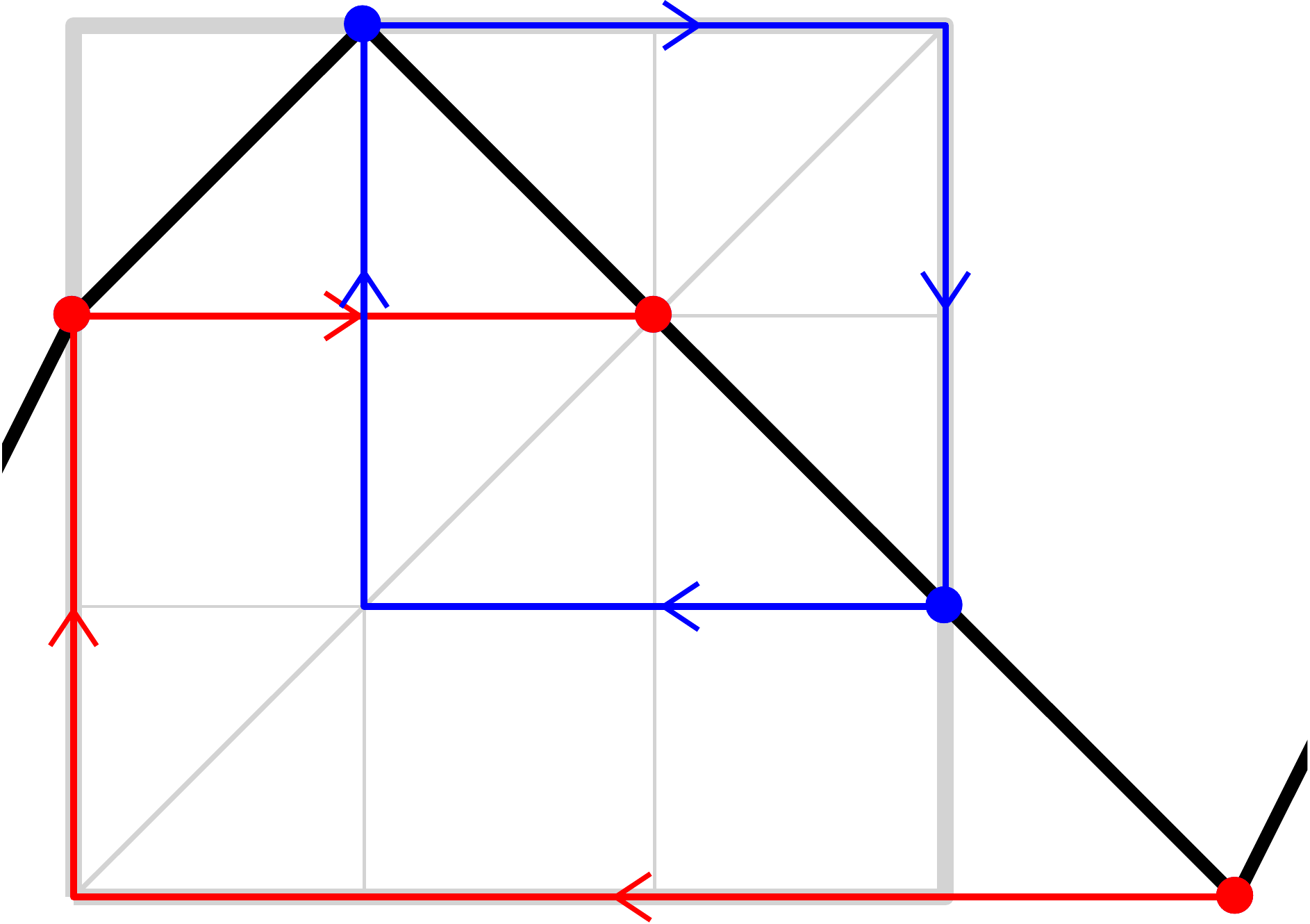} 
    \includegraphics[height=1.25in]{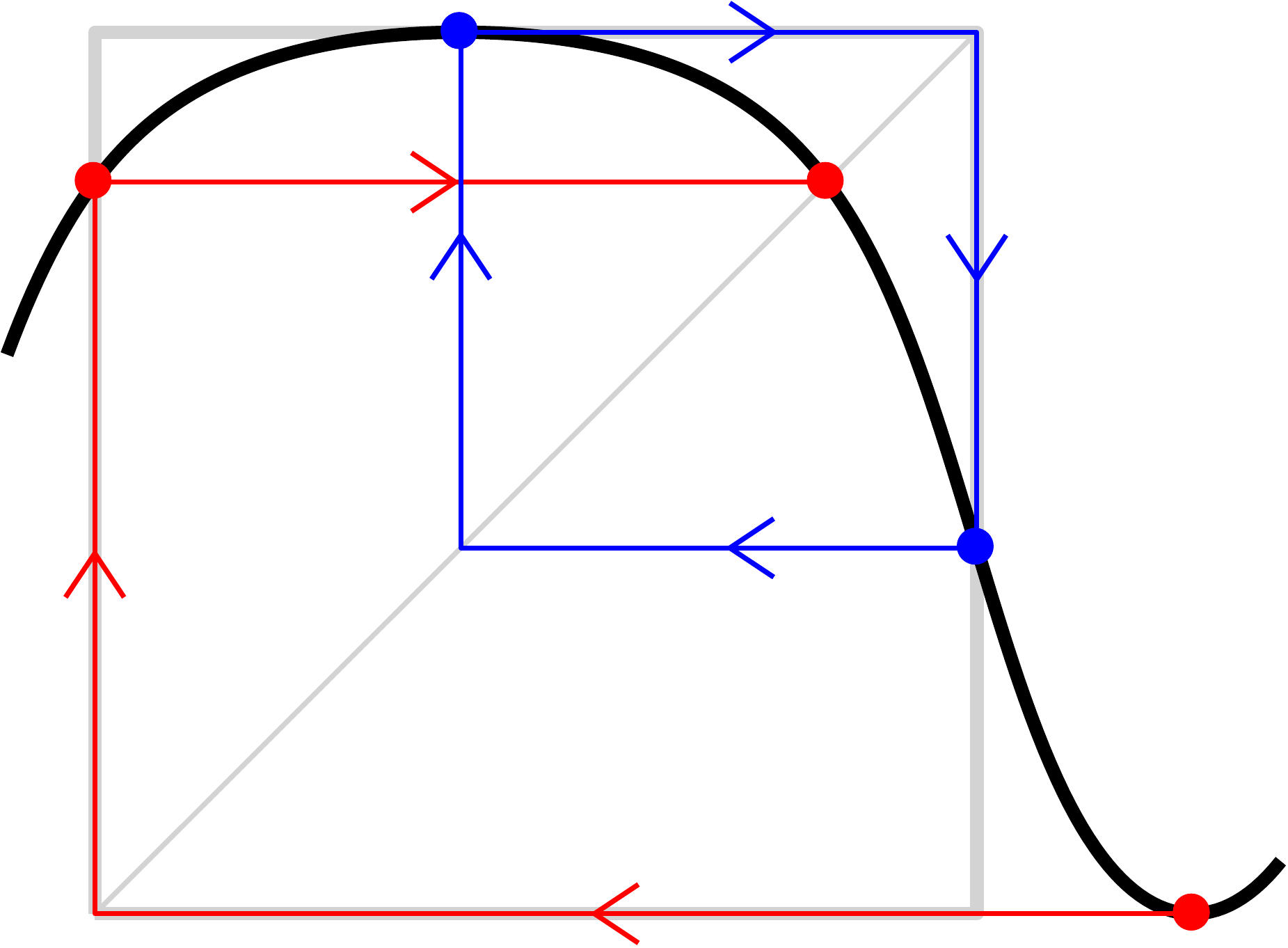}\hfill
    \includegraphics[height=1.25in]{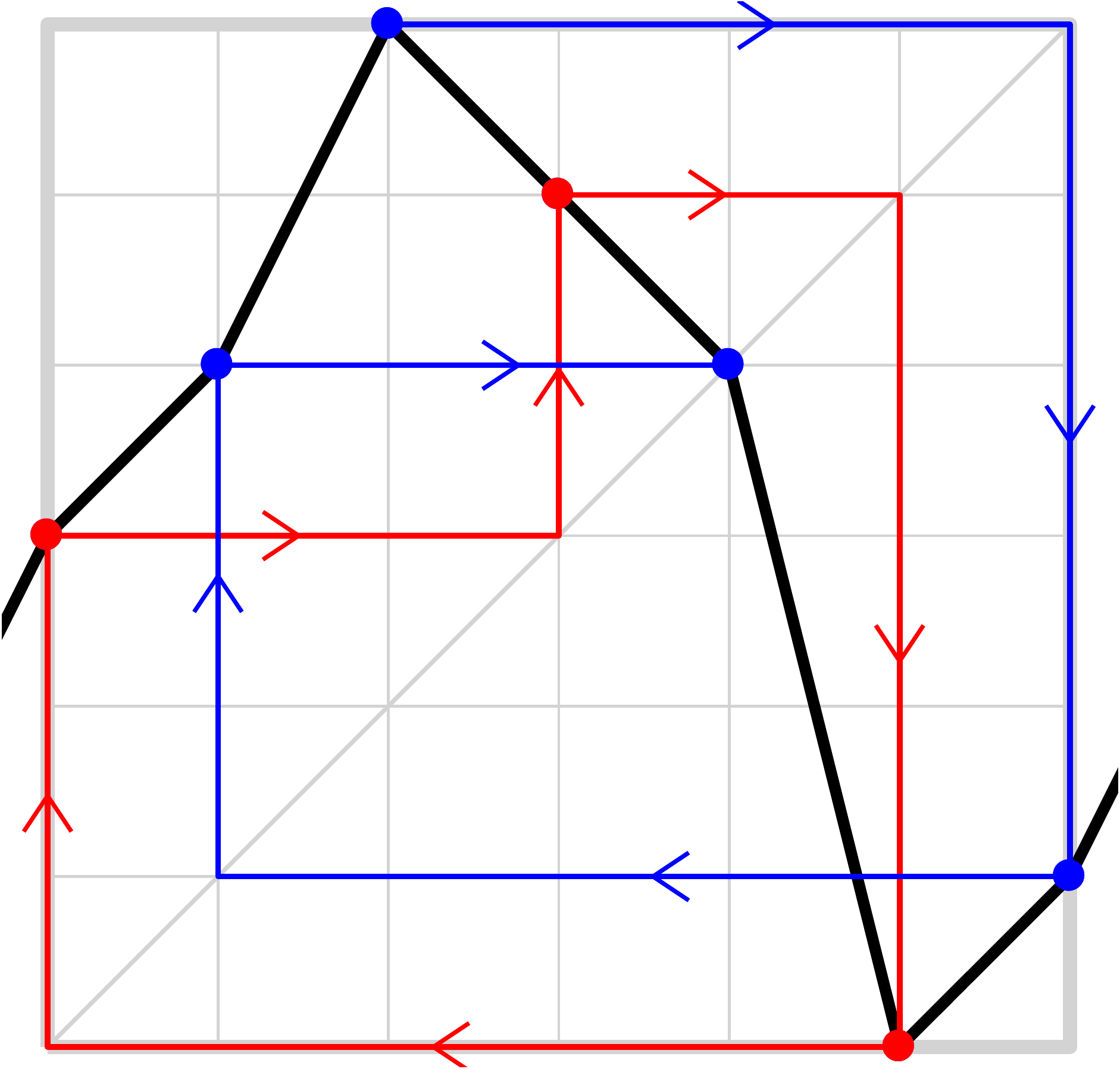} 
    \includegraphics[height=1.25in]{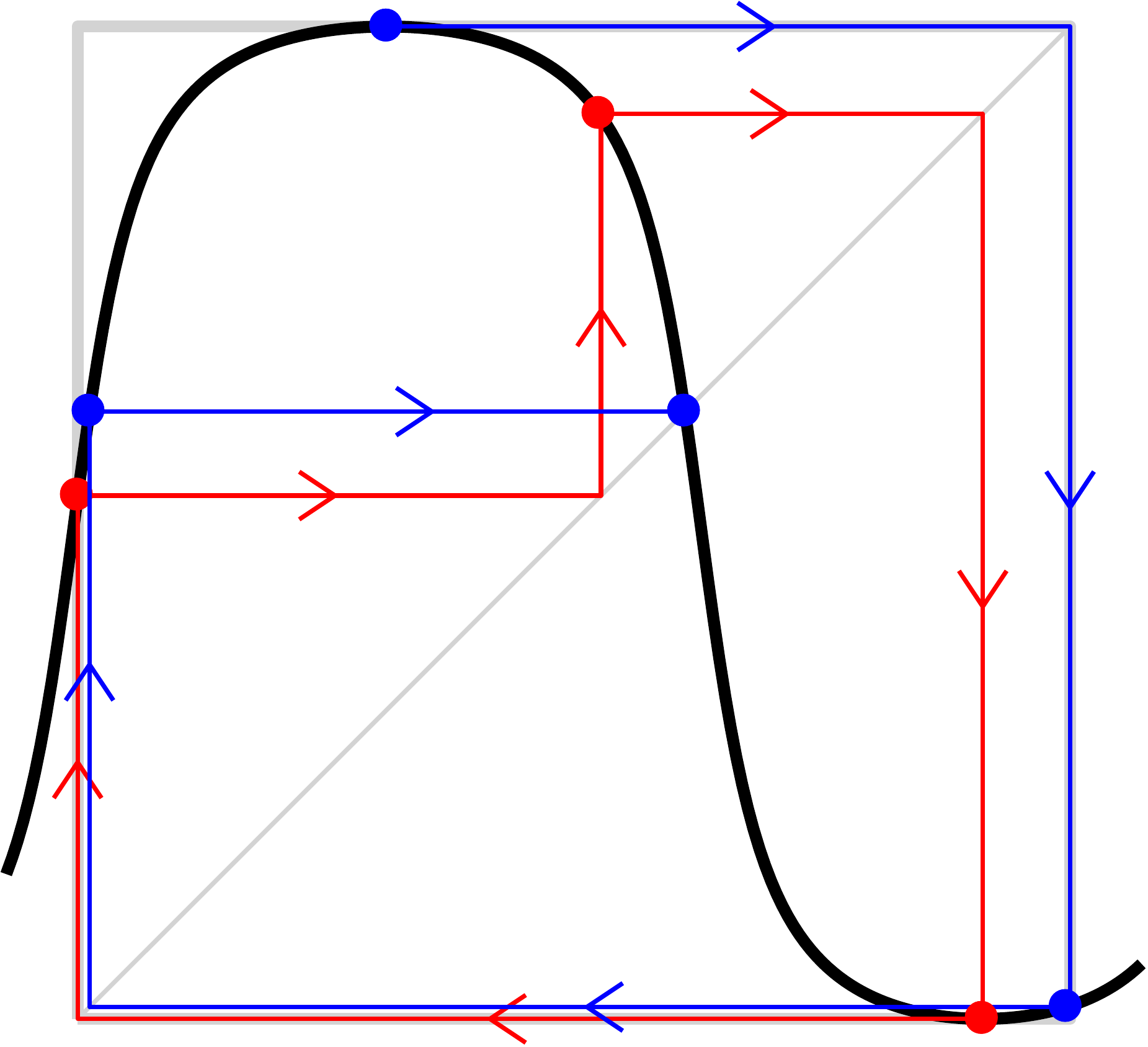}}
  \caption{\label{f-semihyp1}   On the left:
    A strictly unimodal map of topological shape $+-$~ and 
    combinatorics  $\(2,\,3,\,2,\,1,\,0\)$, 
    with a periodic critical orbit $\du{x_1}\leftrightarrow x_3$ and an
    eventually  repelling critical orbit 
    $~\du{x_4}\mapsto x_0 \mapsto x_2 \mapstoself$.    
    Here one critical point lies outside of $f(\Rhat)$.  On the
    right:  topological shape $+-+$ and combinatorics
    $\(3,\,4,\,6,\,5,\,4,\,0,\,1\)$, with
    periodic orbit $\du{x_5}\mapsto x_0 \mapsto x_3 \mapsto \du{x_5}$
    and repelling critical orbit
    {$\du{x_2}\mapsto x_6\mapsto x_1 \mapsto x_4 \mapstoself$}.
    Note that both critical points lie in $f(\Rhat)$.
  }
  \end{figure}

\FloatBarrier
\phantom{menace}
\vspace{0pt plus 2\baselineskip minus 2\baselineskip}
\ifthenelse{\IsThereSpaceOnPage{.3\textheight}}{\clearpage}{\relax}
\subsubsection*{Totally Non-Hyperbolic: Every postcritical cycle is repelling.}
In these cases,  since both critical orbits are eventually repelling, it
follows  that the complex Julia set is   the entire Riemann sphere. 

\noindent
\begin{minipage}{\textwidth}
  \centerline{%
    \includegraphics[height=\figHt]{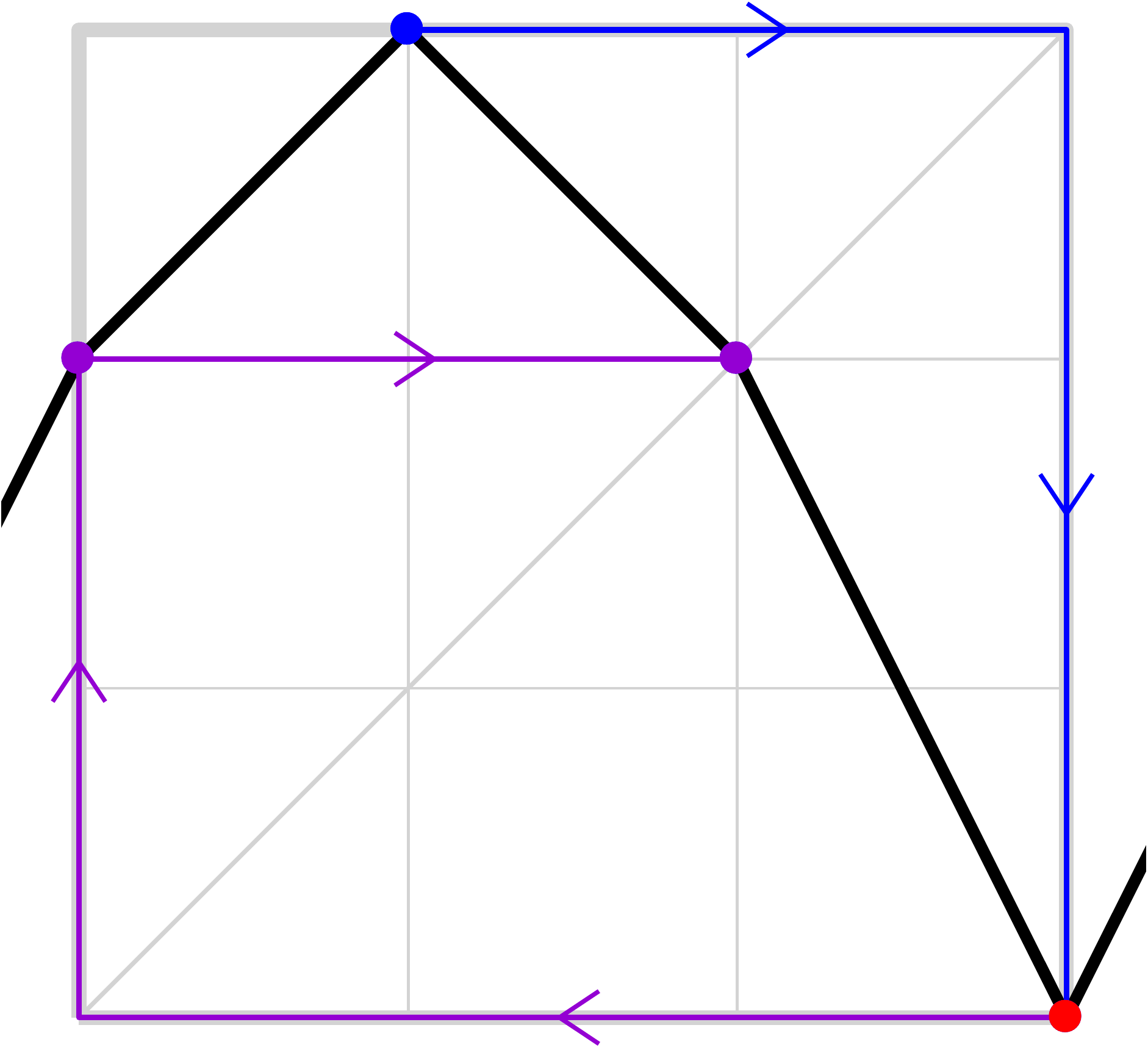} \,
    \includegraphics[height=\figHt]{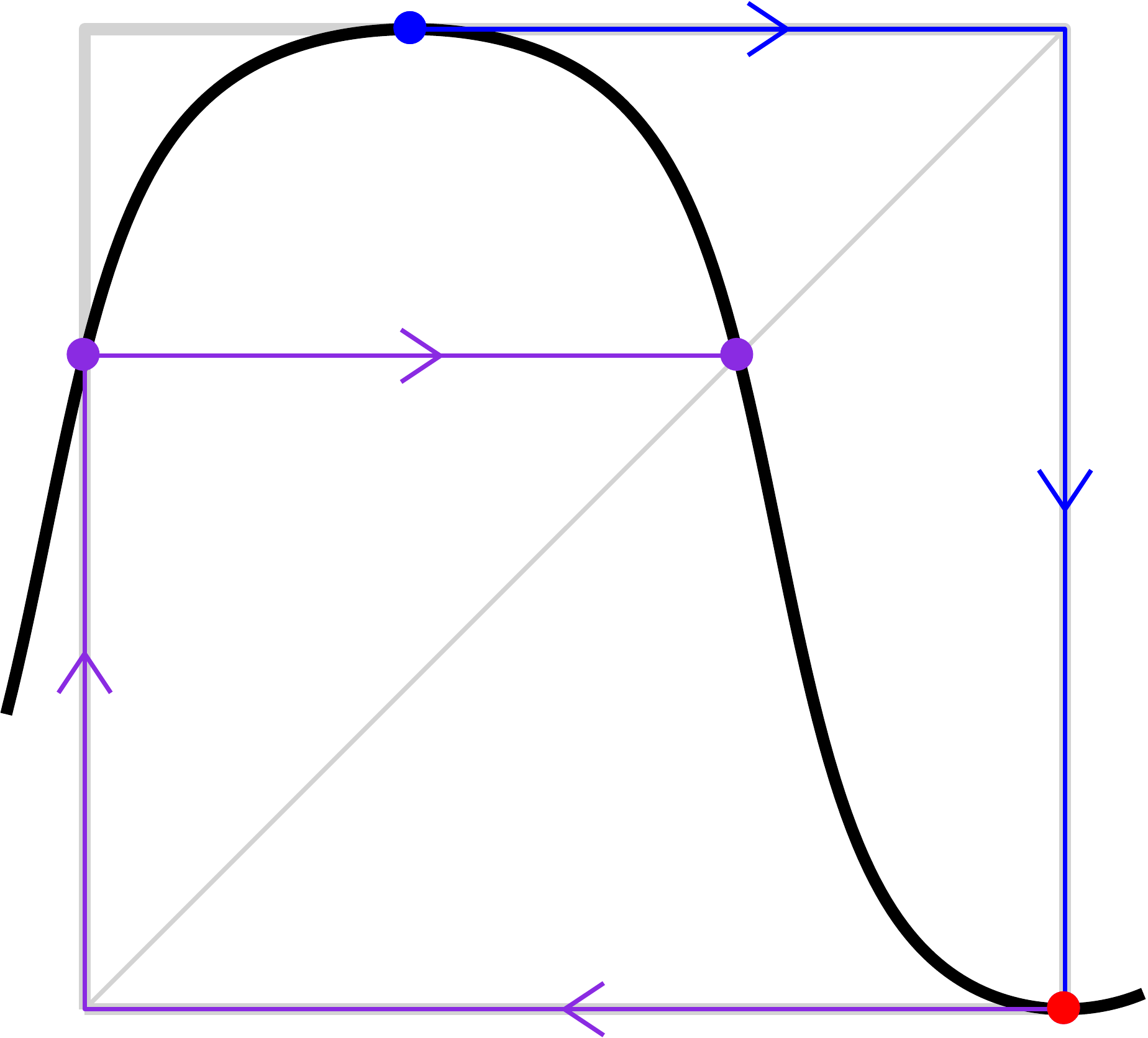}\hfill
    \includegraphics[height=\figHt]{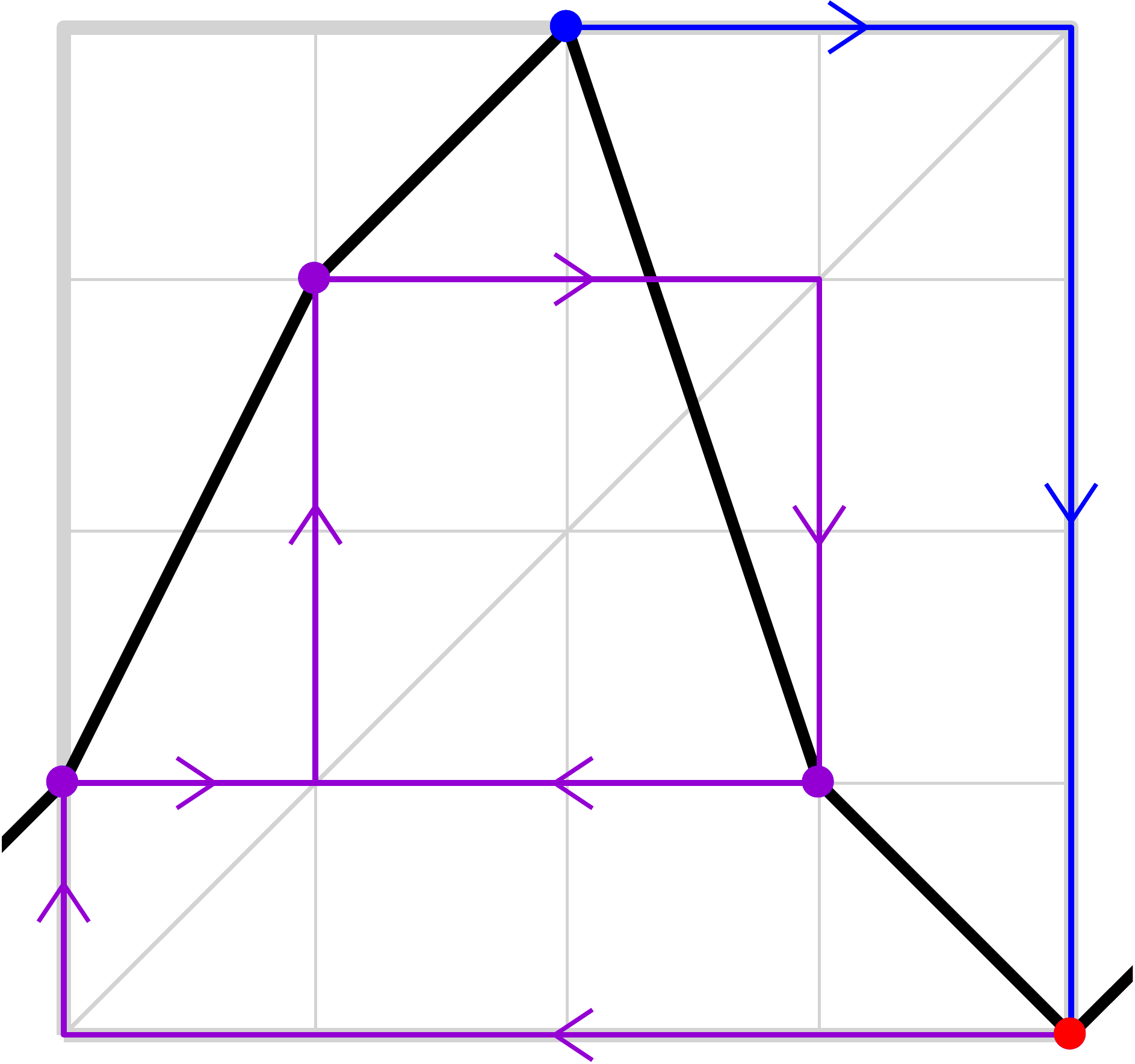}\,
    \includegraphics[height=\figHt]{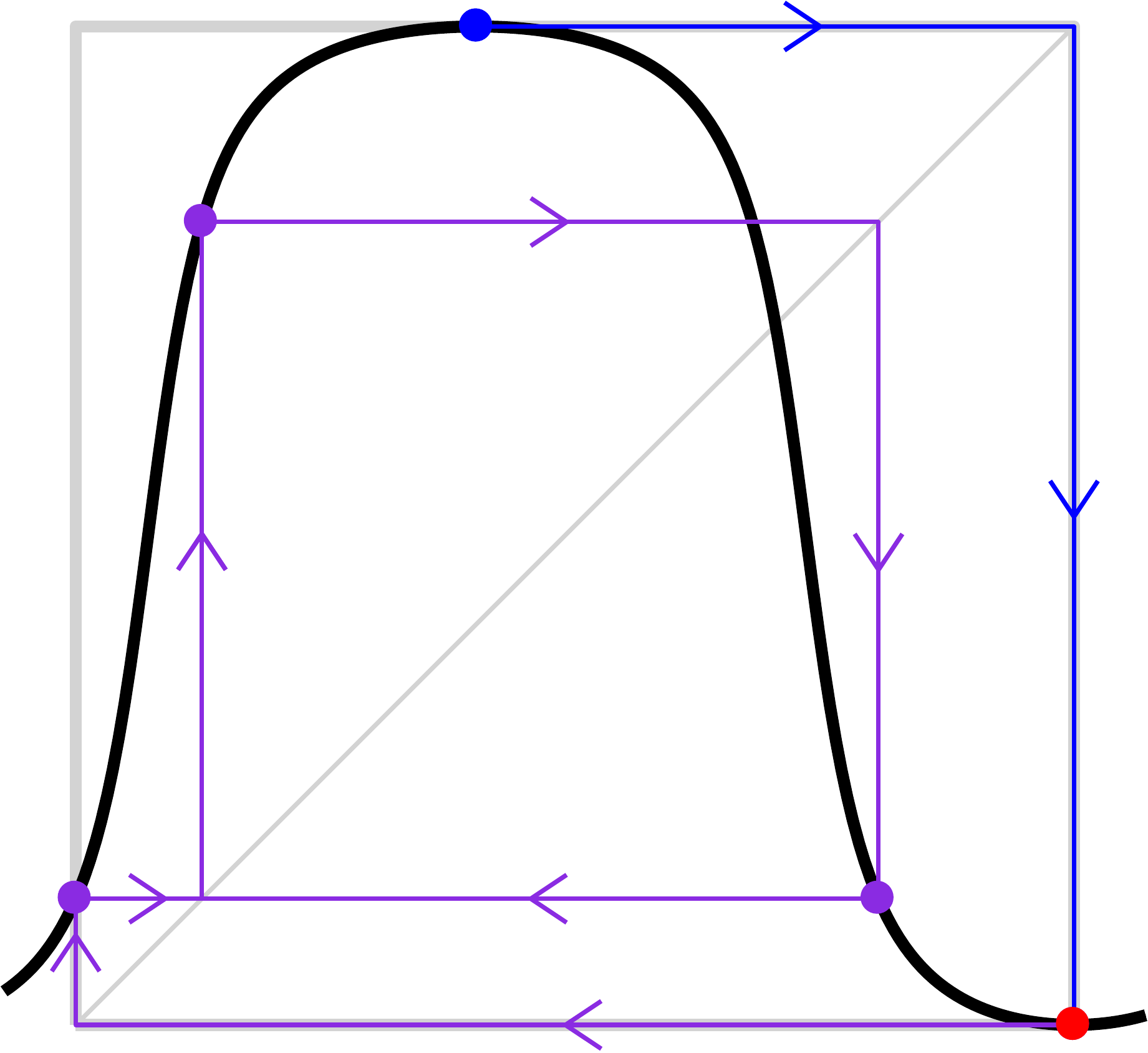}}
\captionof{figure}{\label{f-nonhyp}   Maps of co-polynomial shape. 
    On the left are  illustrations of the piecewise linear map and its
    corresponding
    rational map for the combinatorics
    $\(2,\, 3,\, 2,\, 0\)$ with dynamical pattern 
    \hbox{$\du{x_1}\mapsto \du{x_3} \mapsto x_0 \mapsto x_2 \mapstoself$}. On
    the right are illustrations of the
   combinatorics  $\(1,\,3,\,4,\,1,\,0\)$ with  mapping pattern 
    $\du{x_2}\mapsto \du{x_4} \mapsto x_0 \mapsto x_1\leftrightarrow x_3$~.}
\vspace{2ex plus 10ex }
  \centerline{%
    \includegraphics[height=\figHt]{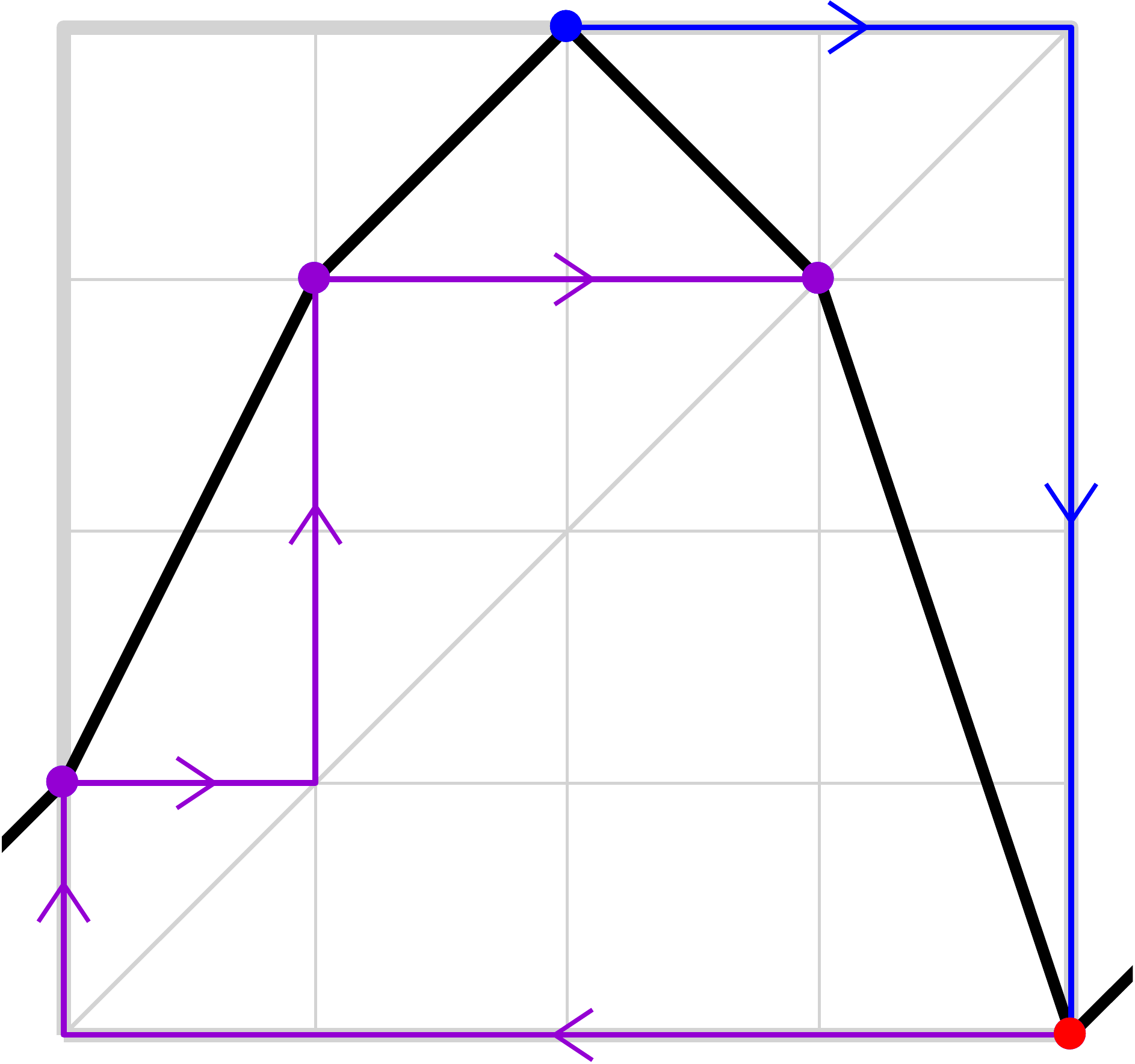} \,
    \includegraphics[height=\figHt]{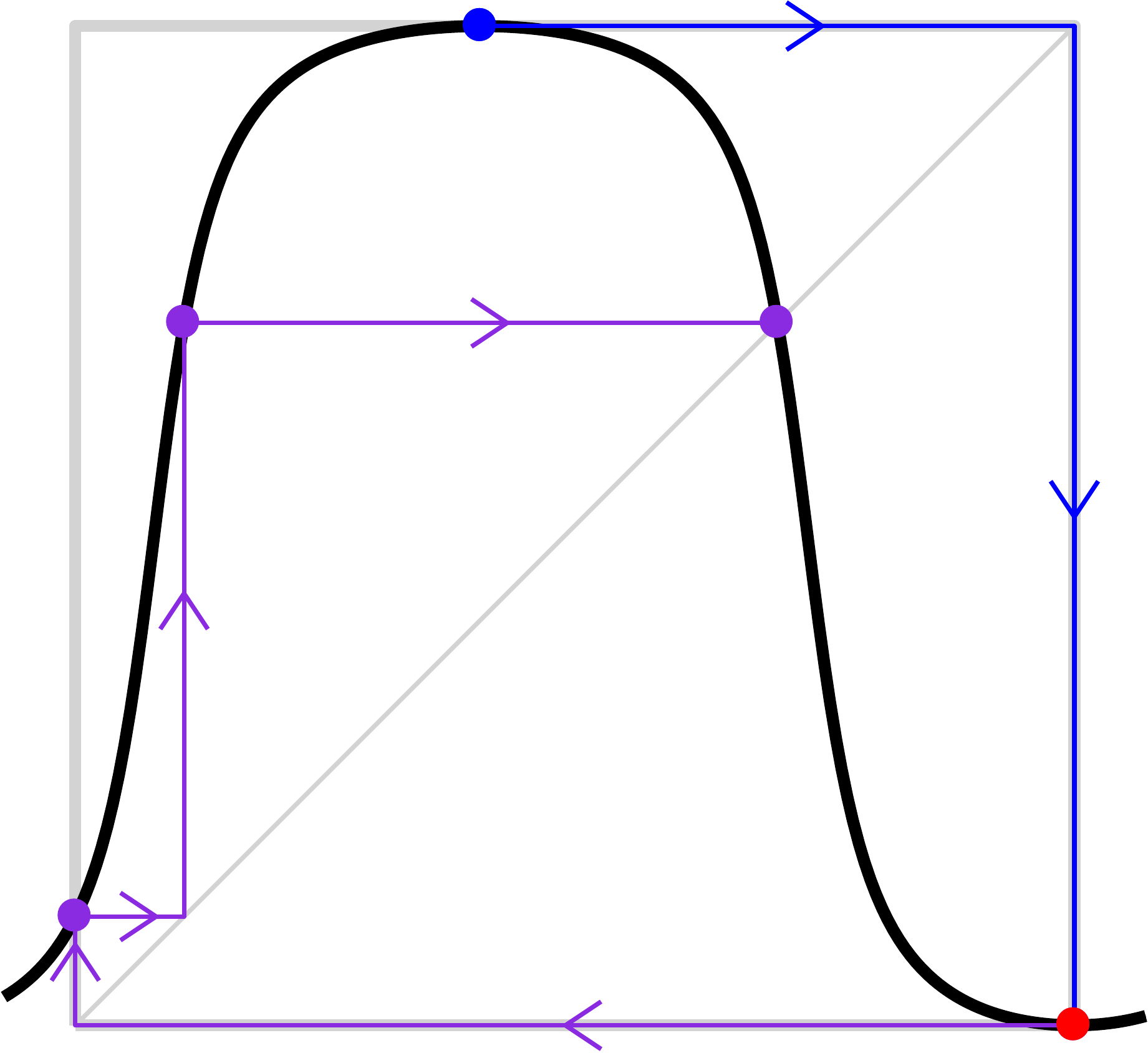}\hfill
    \includegraphics[height=\figHt]{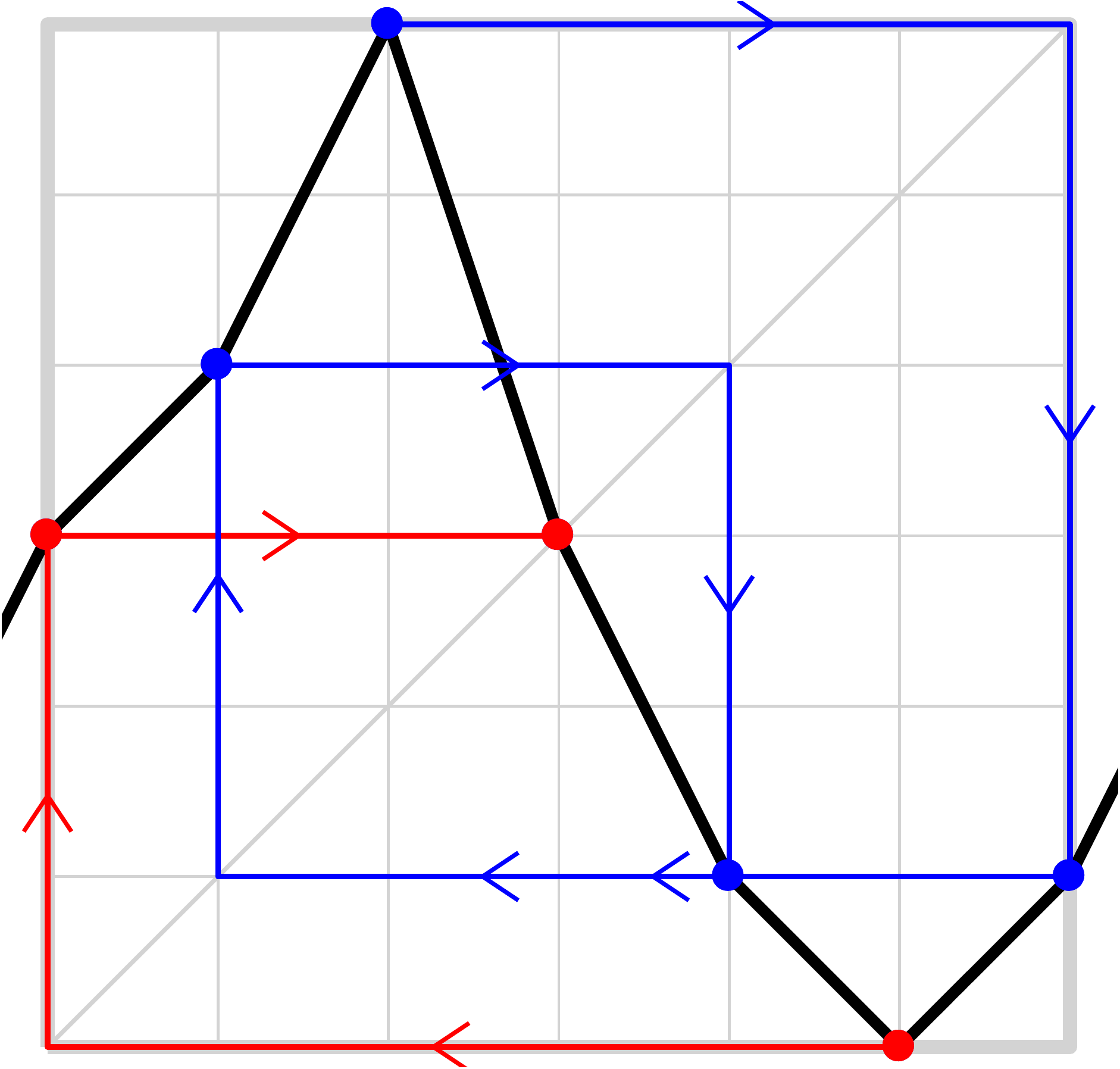}\,
    \includegraphics[height=\figHt]{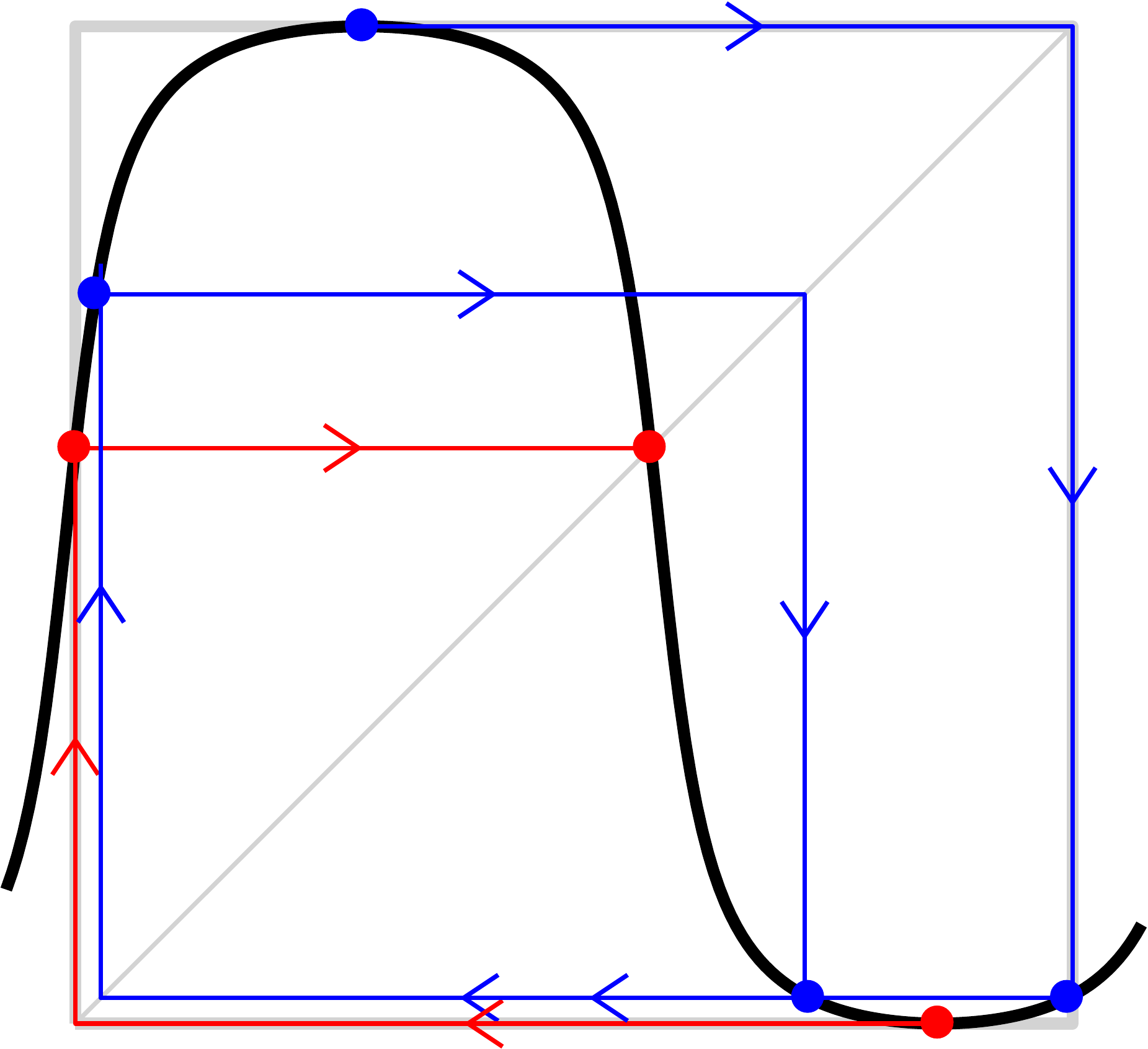}}
  \captionof{figure}{\label{f-nonhyp1}  On the left are illustrations of
    a map of
    co-polynomial shape with combinatorics $\(1,\, 3,\, 4,\, 3,\, 0\)$ and
    dynamical pattern
    \hbox{$\du{x_2}\mapsto \du{x_4} \mapsto x_0 \mapsto x_1\mapsto x_3
      \mapstoself$}~.  On the right are illustrations of a map of
    topological
    shape $+ - +$, with combinatorics $\(3,\,4,\,6,\,3,\,1,\,0,\,1\)$.
    The mapping pattern is $\du{x_2}\mapsto x_6 \mapsto x_1
  \leftrightarrow x_4$ and $\du{x_5}\mapsto x_0 \mapsto x_3 \mapstoself$}~.
\end{minipage}

\FloatBarrier
\phantom{menace}
\ifthenelse{\IsThereSpaceOnPage{.3\textheight}}{\clearpage}{\relax}
\subsection*{Strongly Obstructed Combinatorics}\label{s-obst} 

Since for strongly obstructed combinatorics the limit cannot be realized as
a rational map, the images of limiting maps will use the convention of
\autoref{F-algExamp} in \autoref{s4}, indicating the marked points by
disks and their images under the map by open squares.  In some cases, the
square will appear to be filled, because the image is another marked point.
\renewcommand{\figHt}{.22\textwidth}

\subsubsection*{Type B: Both critical points in a common periodic orbit.}

For Type~B, strongly obstructed examples seem to be uncommon.  The smallest
example  we were able to find has $n=6$.

\begin{figure}[!htb]
  \centerline{%
    \includegraphics[height=\figHt]{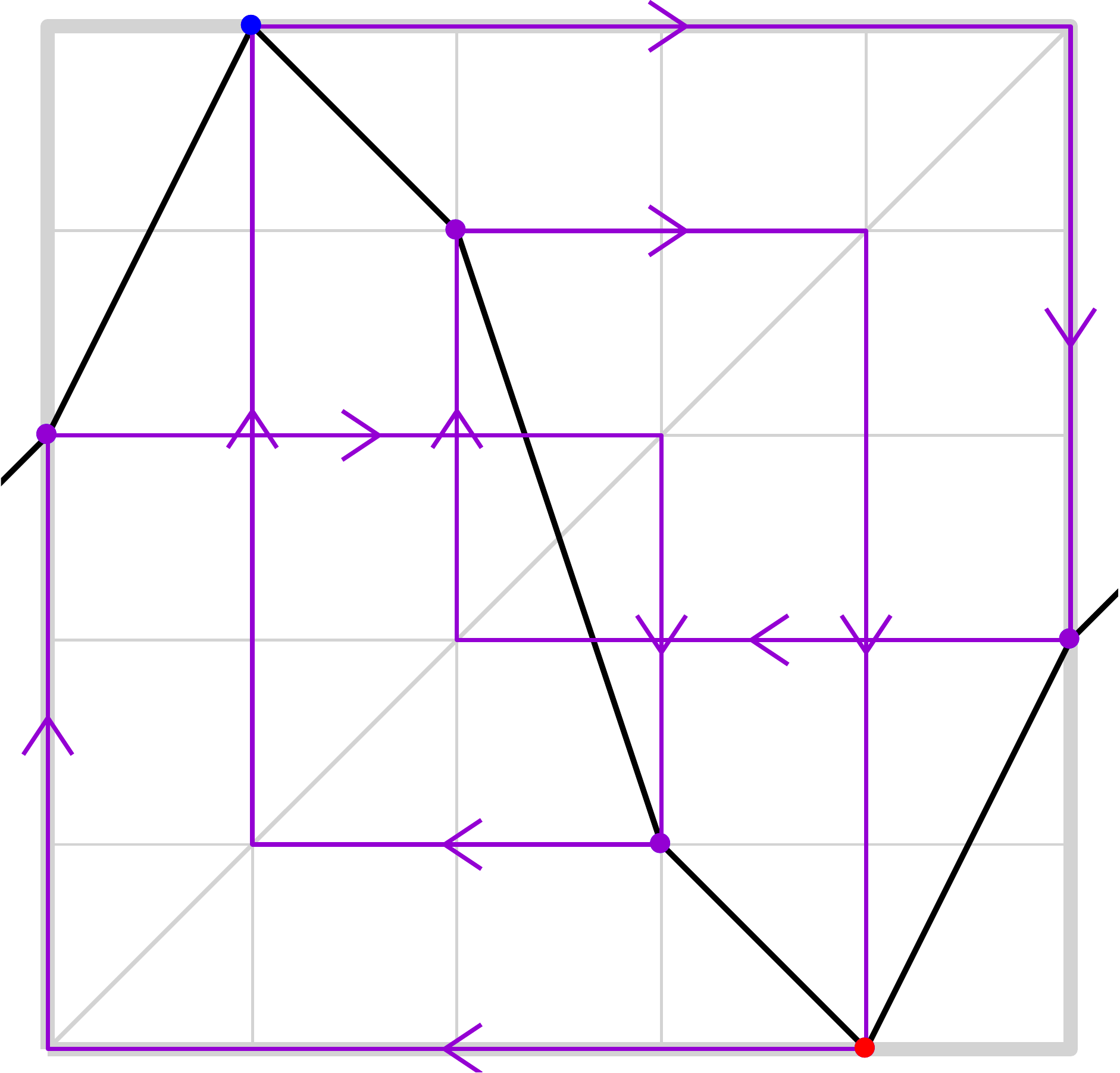}\qquad
    \includegraphics[height=\figHt]{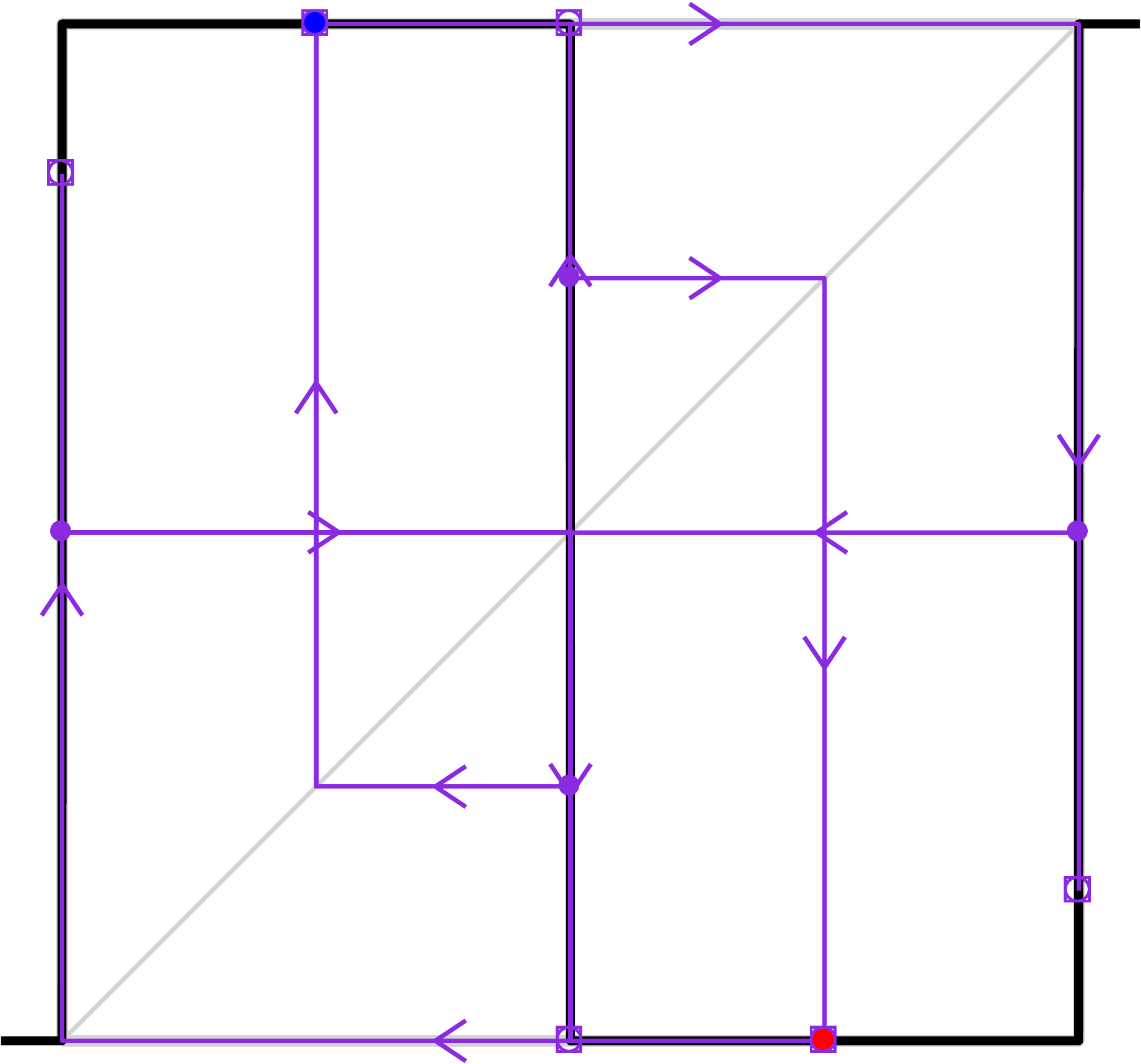}}
\caption { \label{f-aBobst}  
  The piecewise linear map and the the limit of the iterated
  pull-back  maps are shown for the strongly obstructed
  combinatorics $\(3,\,5,\,4,\,1,\,0,\,2\)$  of topological shape $+-+$~ and
  mapping pattern
  \hbox{$\du{x_4}\mapsto x_0\mapsto x_3 \mapsto \du{x_1} \mapsto x_5 \mapsto x_2
    \mapsto \du{x_4}$}~. }
\end{figure} 

\FloatBarrier
\phantom{menace}
\ifthenelse{\IsThereSpaceOnPage{.15\textheight}}{\clearpage}{\relax} 
\subsubsection*{Type C: One critical orbit lands in a cycle containing the  other.}

\begin{figure}[!htb]
  \centerline{
   \includegraphics[height=\figHt]{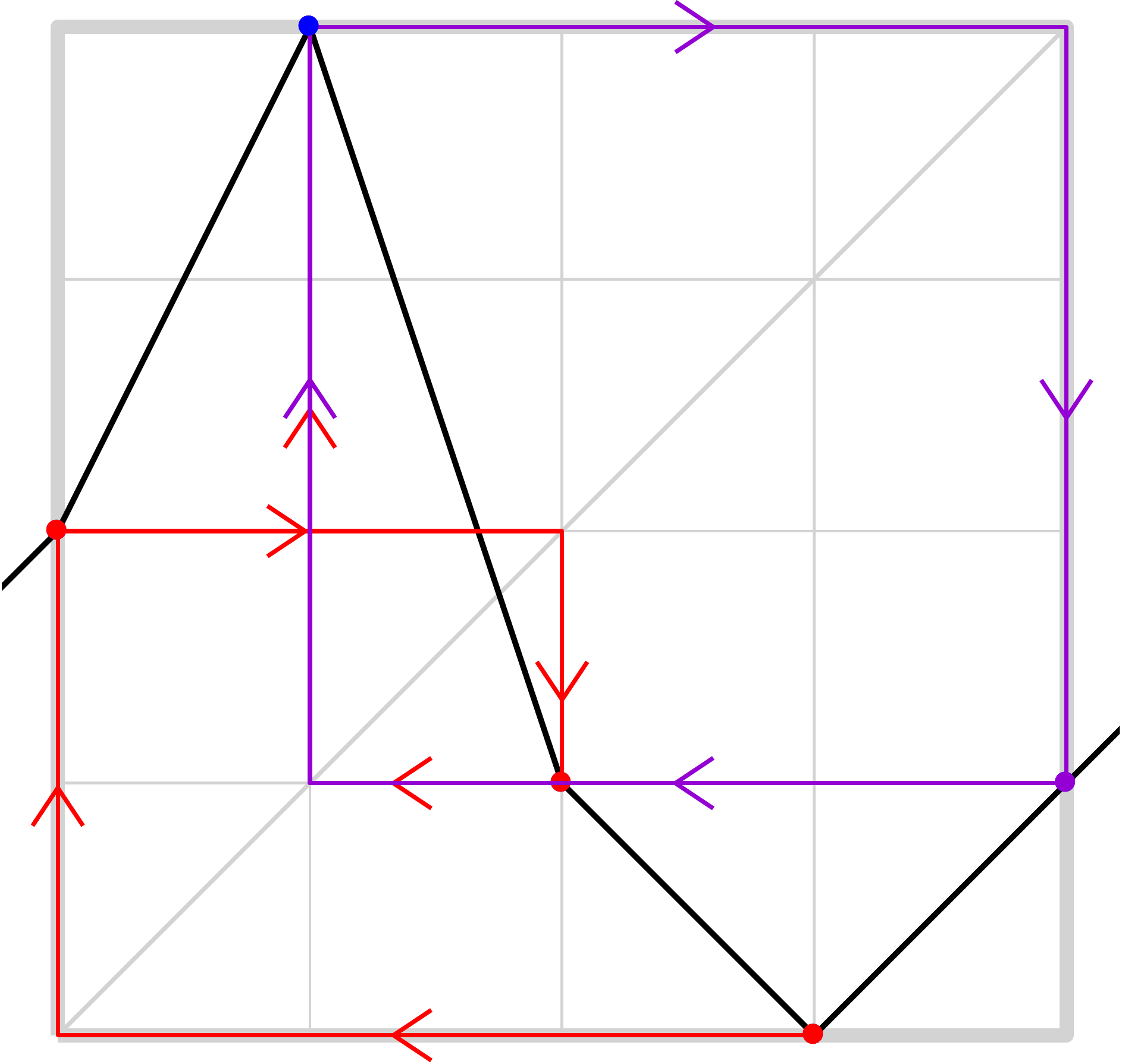}
   \includegraphics[height=\figHt]{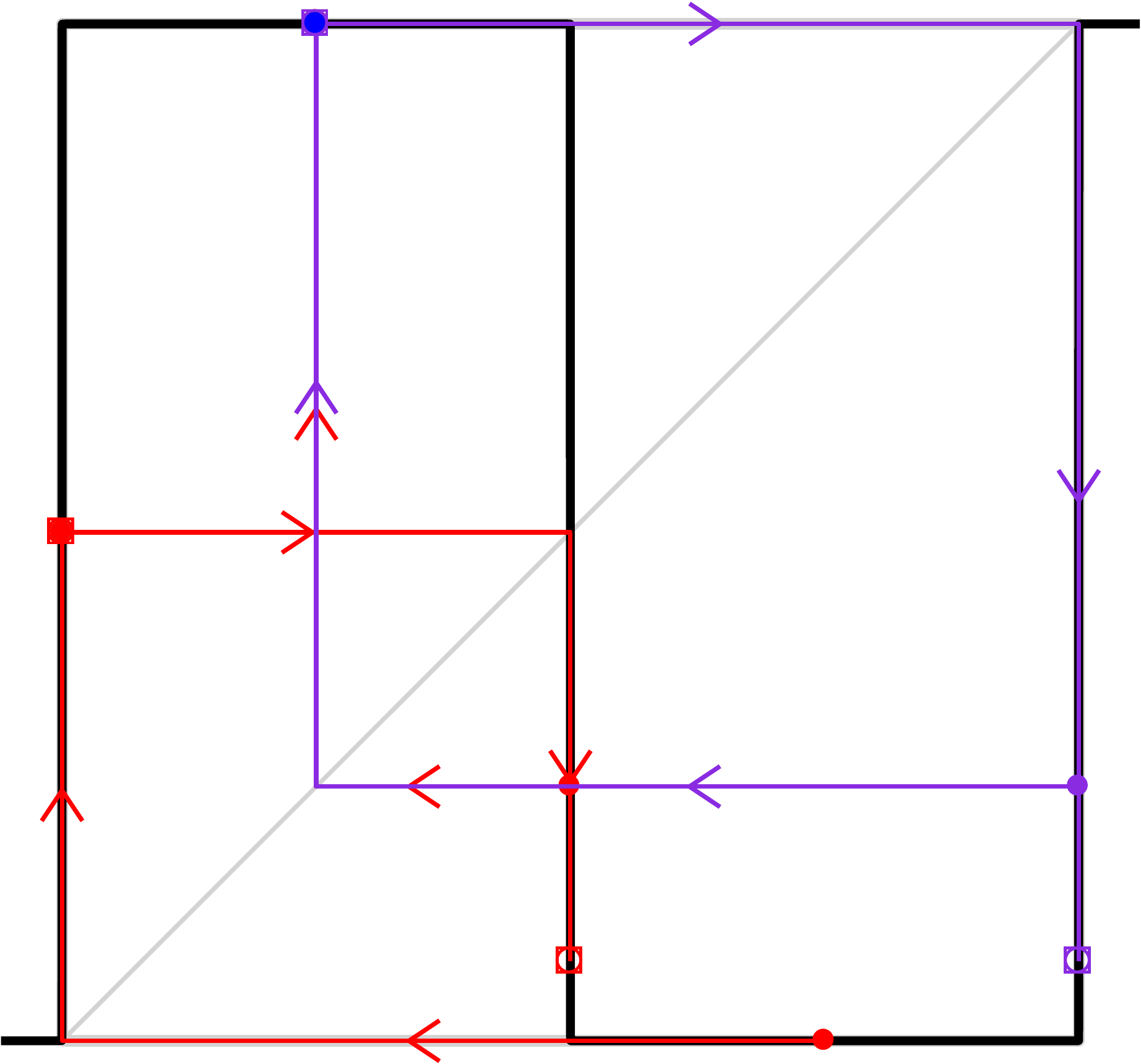} \quad
   \includegraphics[height=\figHt]{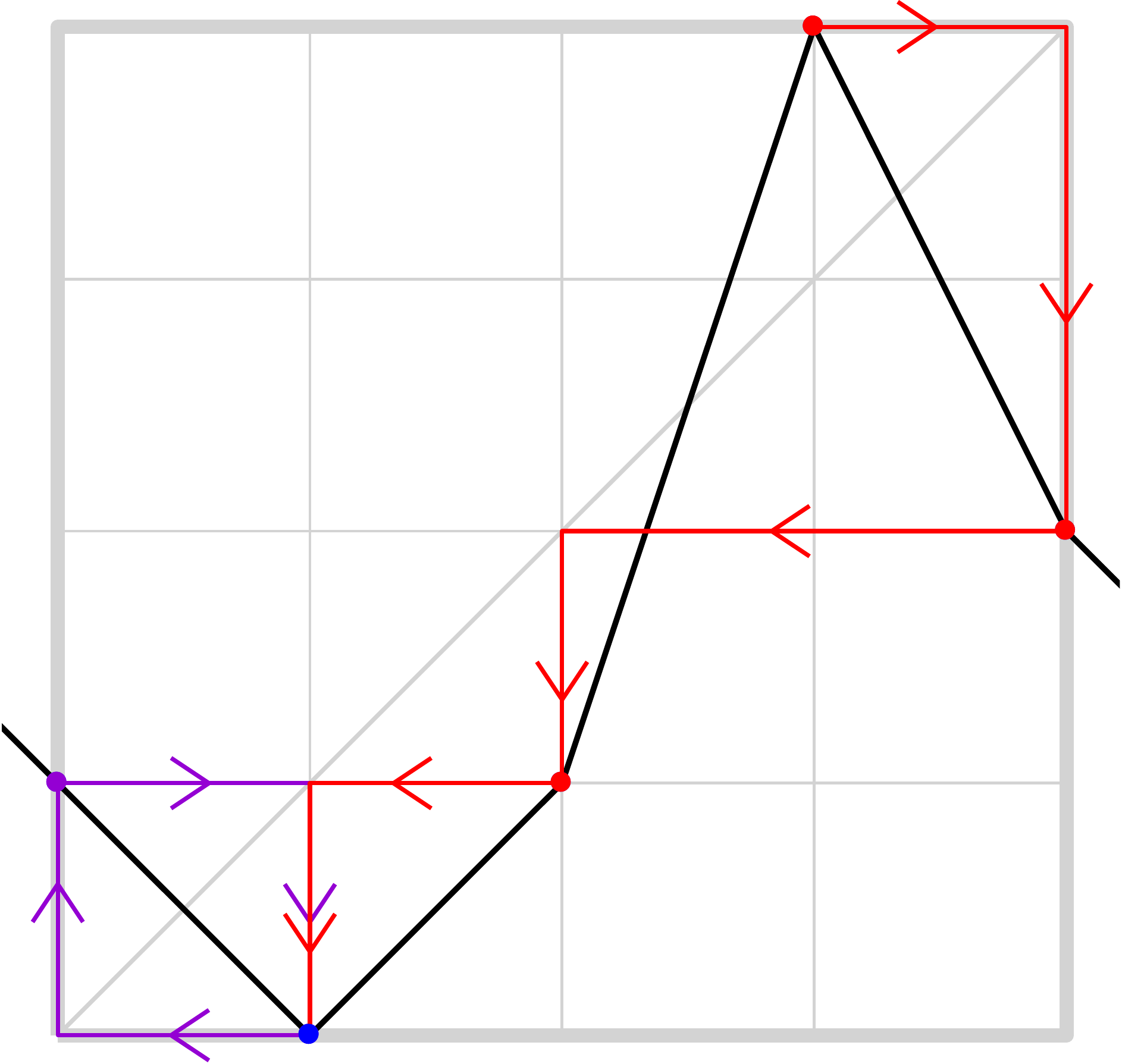}
   \includegraphics[height=\figHt]{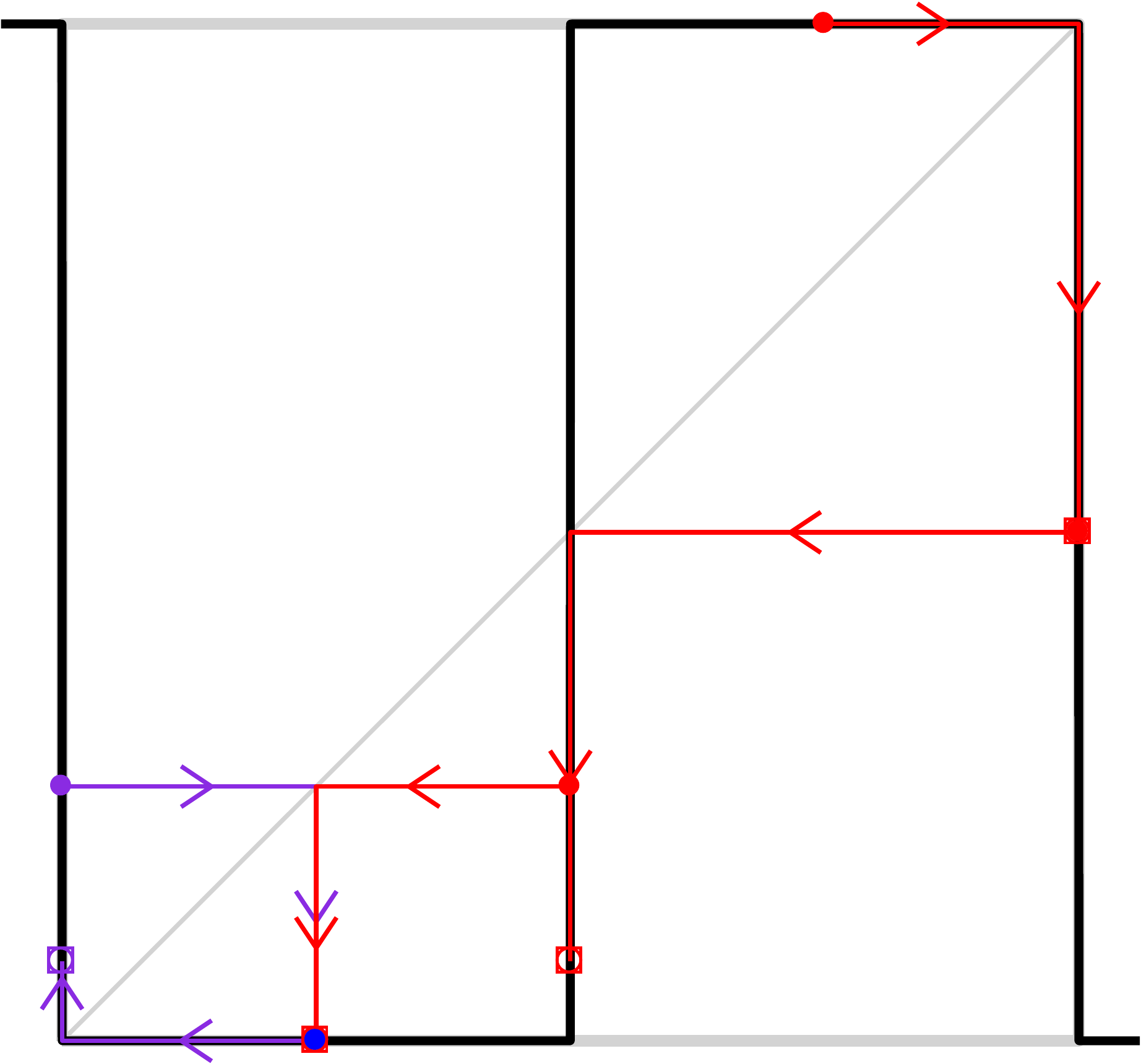}}
 \caption{\label{f-aC3}   On the left an illustration of a strongly
   obstructed
   capture case with combinatorics $\(2,\,4,\,1,\,0,\,1\)$,  topological shape
   $+-+$~,  and  mapping pattern
 \hbox{$\du{x_3}\mapsto x_0 \mapsto x_2 \mapsto \du{x_1}\leftrightarrow x_4$}. 
  On the right,  the ``left-right reflected''
  version of the  figure on the left,  of topological shape $-+-$  with
  very similar dynamics,
  and very similar parameters. Like the figure on the left, it is the illustration
of a strongly obstructed period two capture component.
The combinatorics is  $\(1,\,0,\,1,\,4,\,2\)$ and  the  mapping pattern is
$\du{x_3}\mapsto x_4 \mapsto x_2 \mapsto \du{x_1}\leftrightarrow x_0$.
}  
\end{figure}

\phantom{menace}
\ifthenelse{\IsThereSpaceOnPage{.3\textheight}}{\clearpage}{\relax} 
\subsubsection*{Type D: Disjoint periodic critical orbits.}
 Unlike the Type B case, strongly obstructed dynamics of Type~D are
common. 
\begin{figure}[!htb]
  \centerline{
    \includegraphics[height=\figHt]{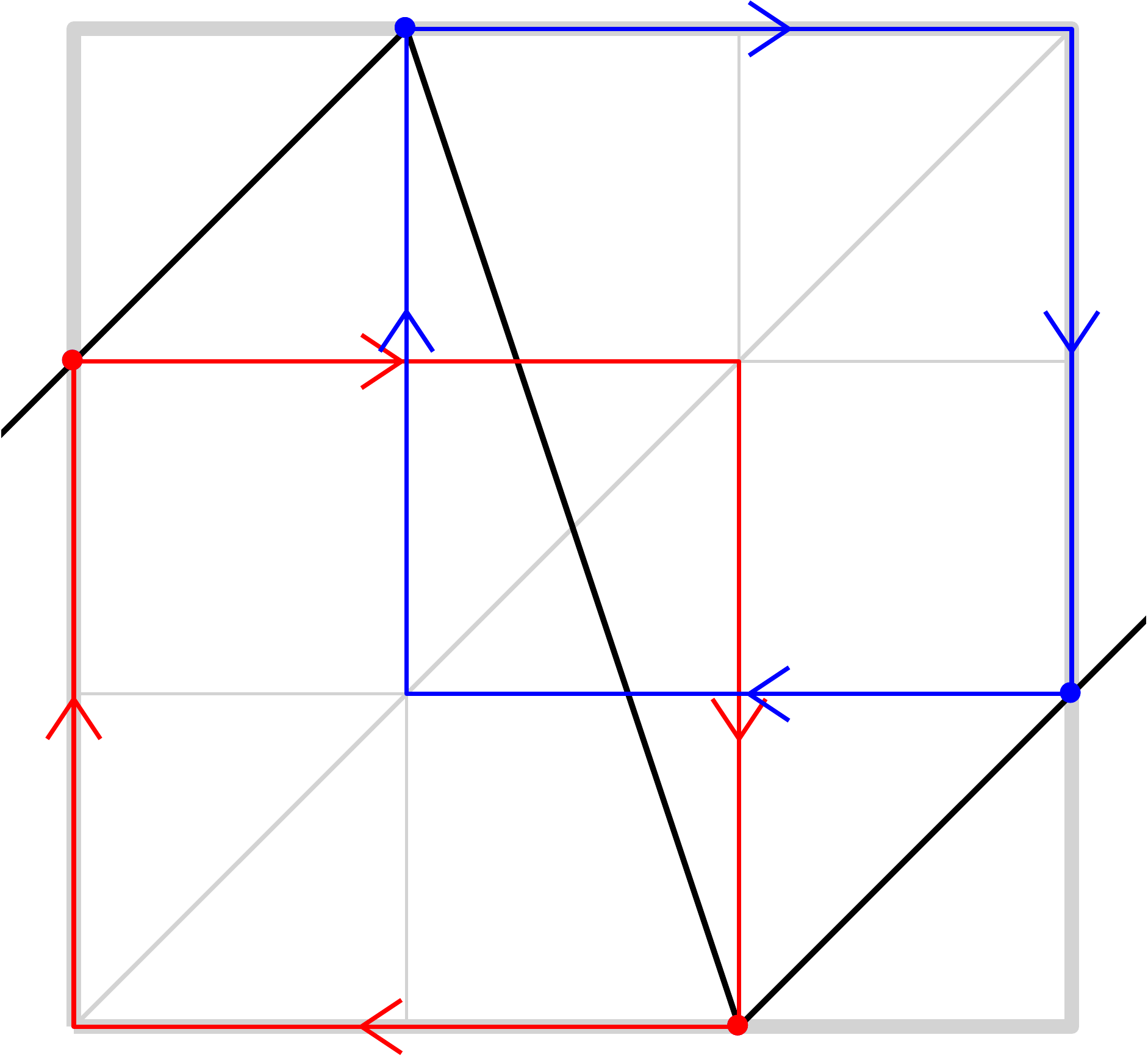}
    \includegraphics[height=\figHt]{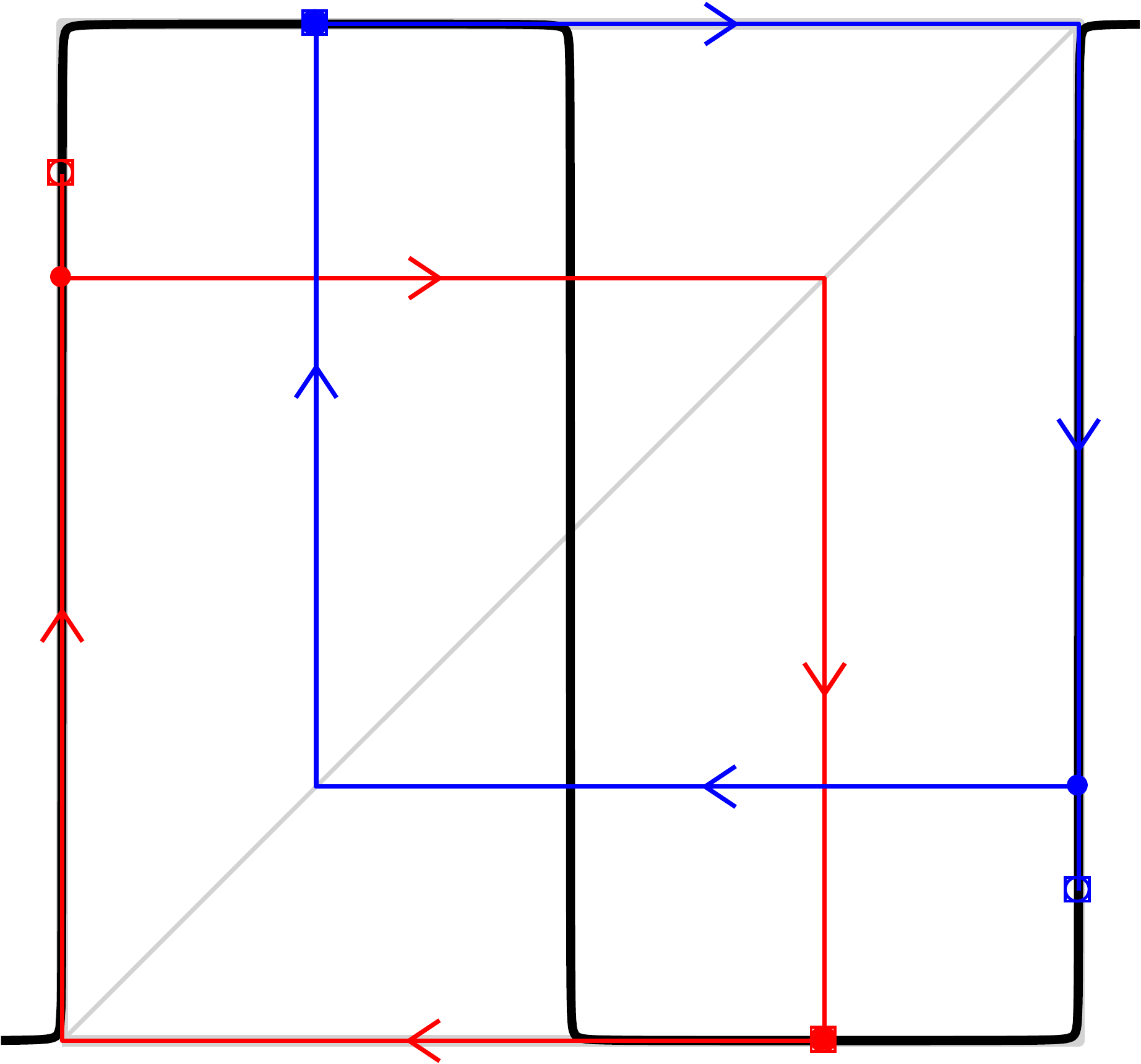}\quad
    \includegraphics[height=\figHt]{qr-1032-pl.pdf}
    \includegraphics[height=\figHt]{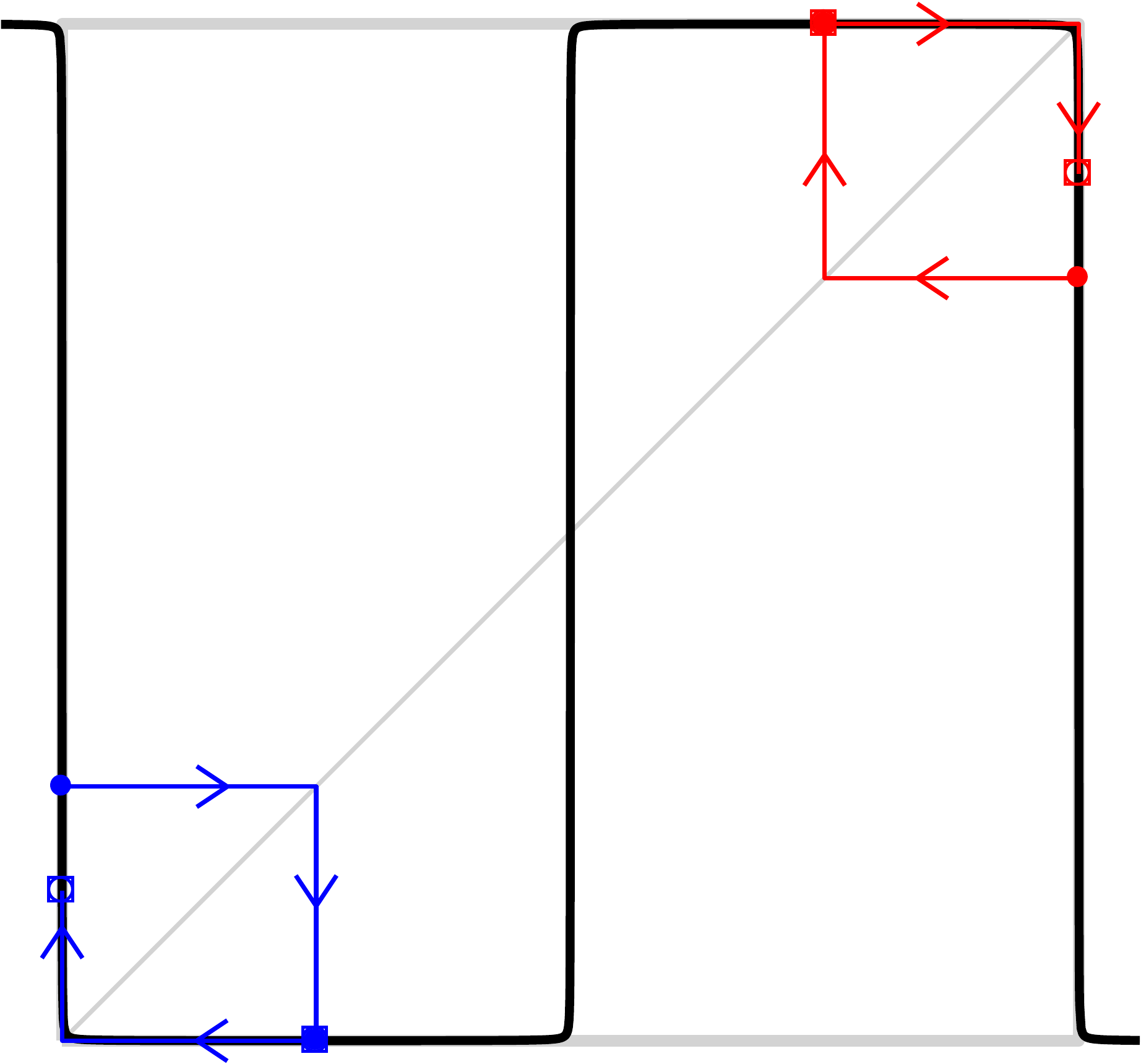}}
\caption{\label{f-aBn3a}
  On the left are illustrations of a piecewise linear map and limiting graph of
  a map of
  topological shape $+-+$~ with combinatorics $\(2,\,3,\,0,\,1\)$, mapping
  pattern $\du{x_2}\leftrightarrow x_0$, $\du{x_1} \leftrightarrow x_3$.
  On the right is the  ``left-right reflected version''
  which is of topological shape
  $-+-$~ with combinatorics $\(1,\,0,\,3,\,2\)$ and mapping pattern
  $\du{x_1}\leftrightarrow x_0$, $\du{x_2}\leftrightarrow x_3$.
  }
\end{figure}

\begin{figure}[!htb]
\centerline{
    \includegraphics[height=\figHt]{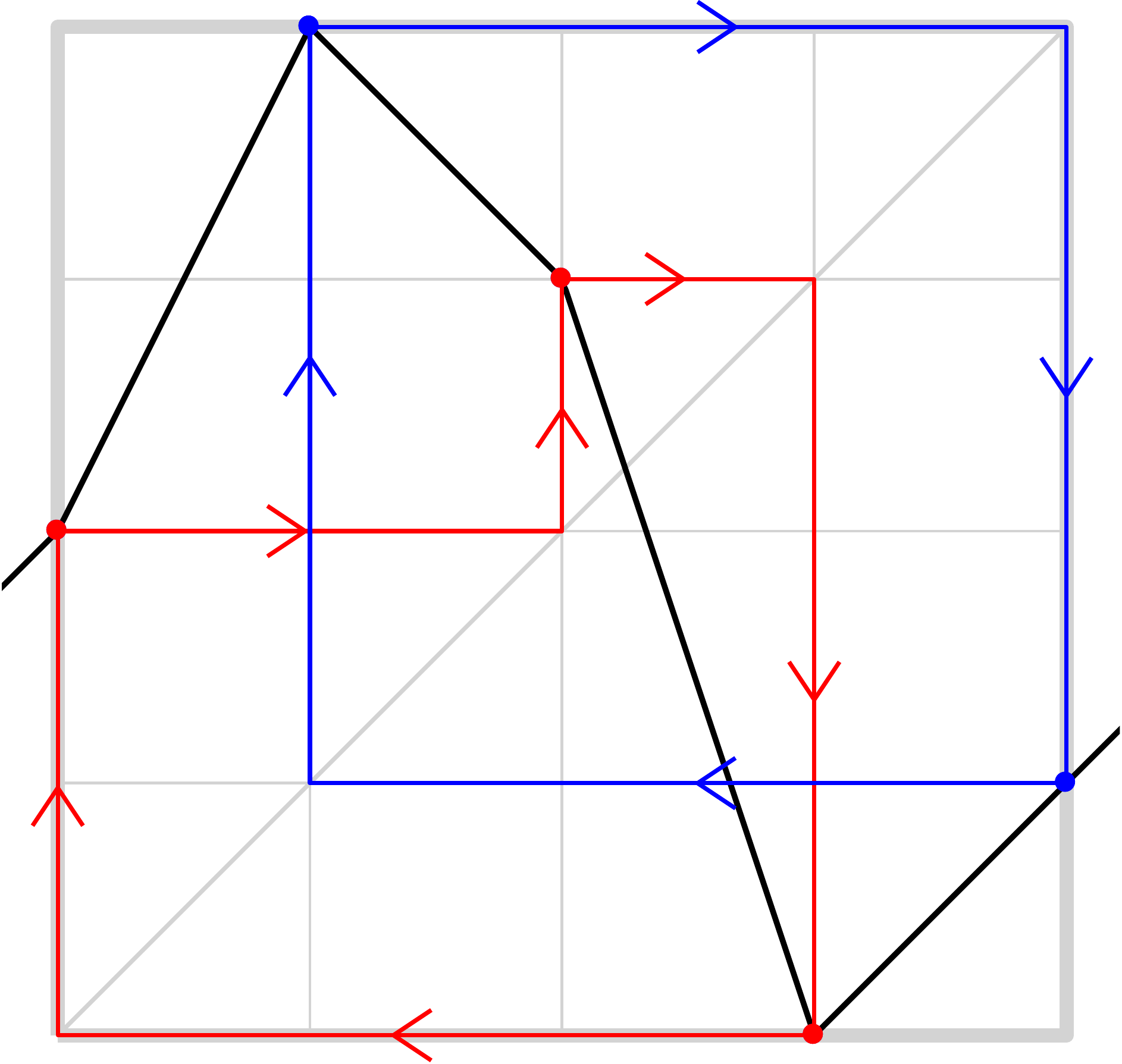}
    \includegraphics[height=\figHt]{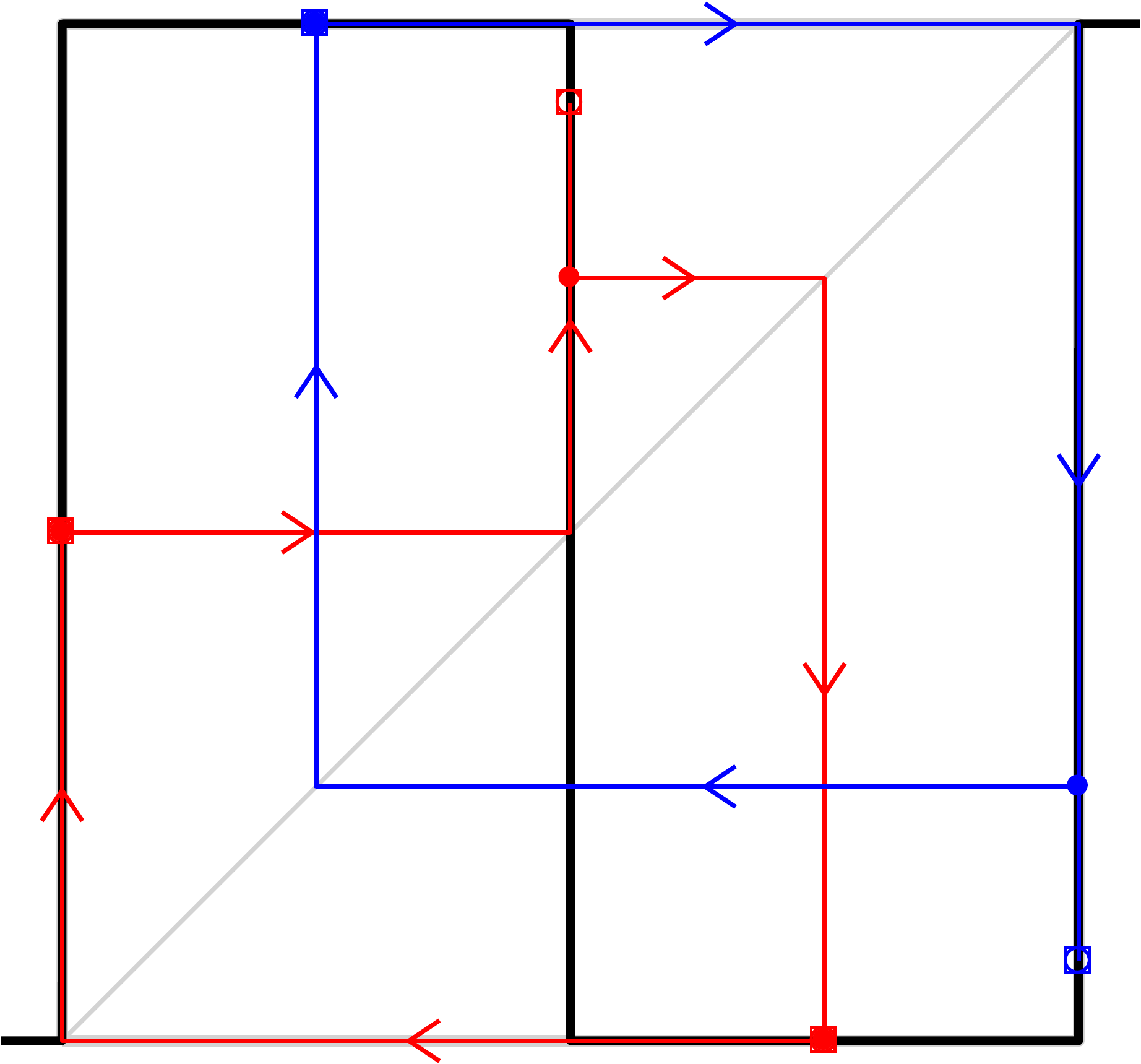} \quad
    \includegraphics[height=\figHt]{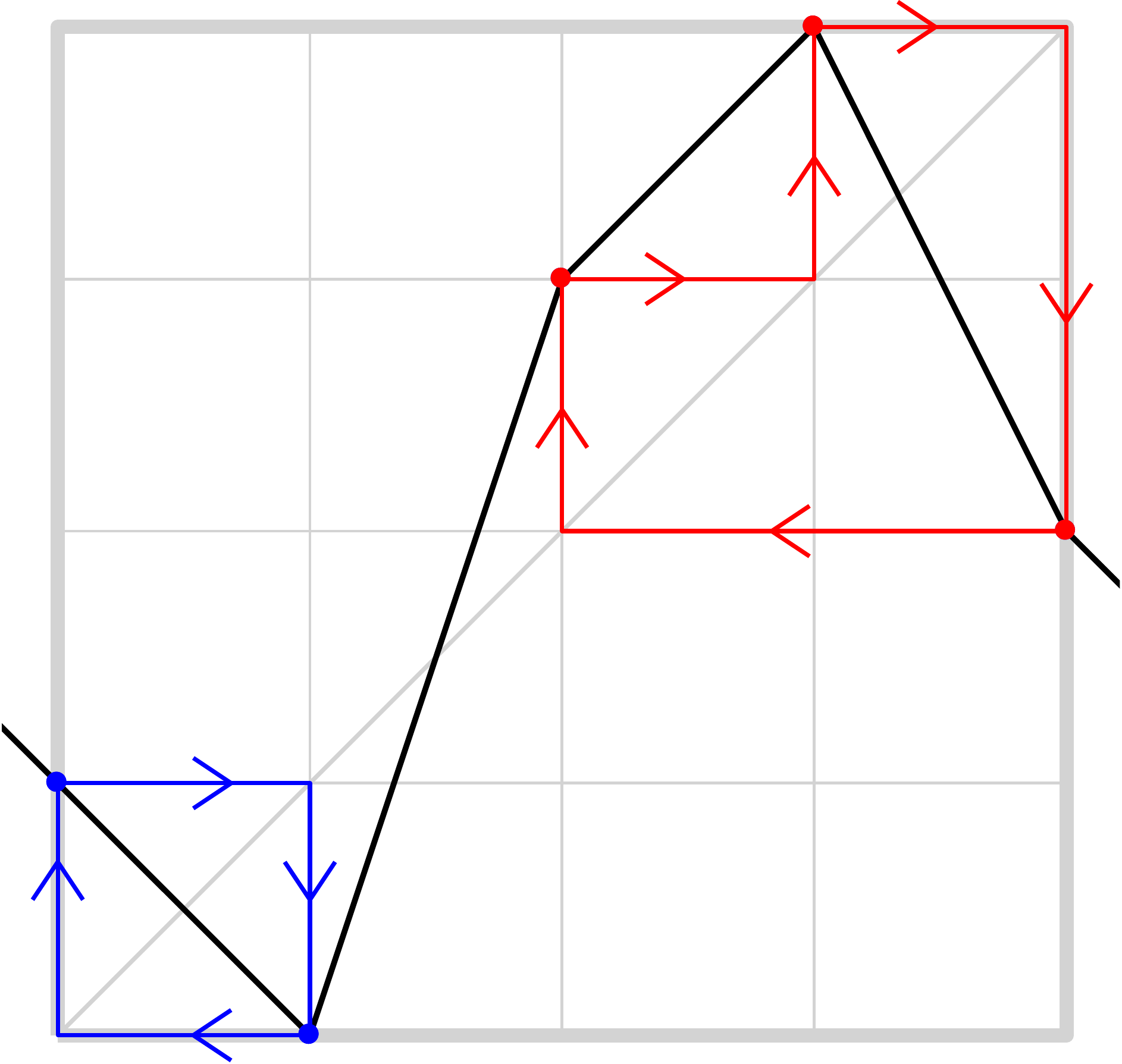}
    \includegraphics[height=\figHt]{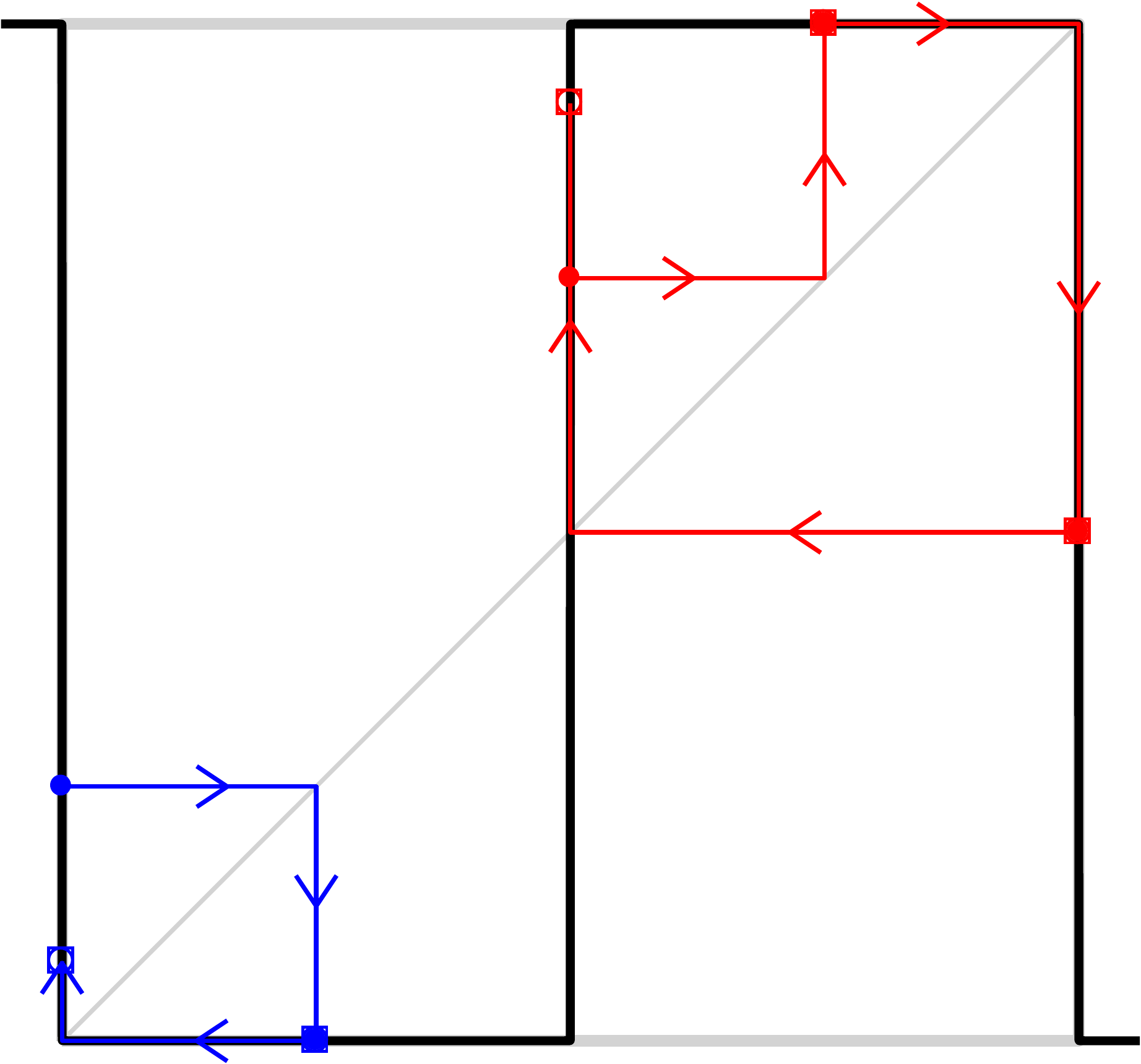}}   
  \caption{\label{f-aDn4a}   On the left are shown the  piecewise
    linear map and limiting graph for the 
    obstructed combinatorics
    $\(2,\,4,\,3,\,0,\,1\)$  of topological shape $+-+$~,
    with mapping pattern
    $\du{x_1}\leftrightarrow x_4$,
    $\du{x_3}\mapsto x_0 \mapsto x_2 \mapsto \du{x_3}$.
    On the right is the ``left-right reflected version''
    with 
    combinatorics $\(1,\,0,\,3,\,4,\,2\)$,  topological shape $-+-$~
    and mapping pattern  $\du{x_1}\leftrightarrow x_0$,
    $\du{x_3}\mapsto x_4 \mapsto x_2 \mapsto \du{x_3}$~.}
\end{figure}

\begin{figure}[!htb]
\centerline{
    \includegraphics[height=\figHt]{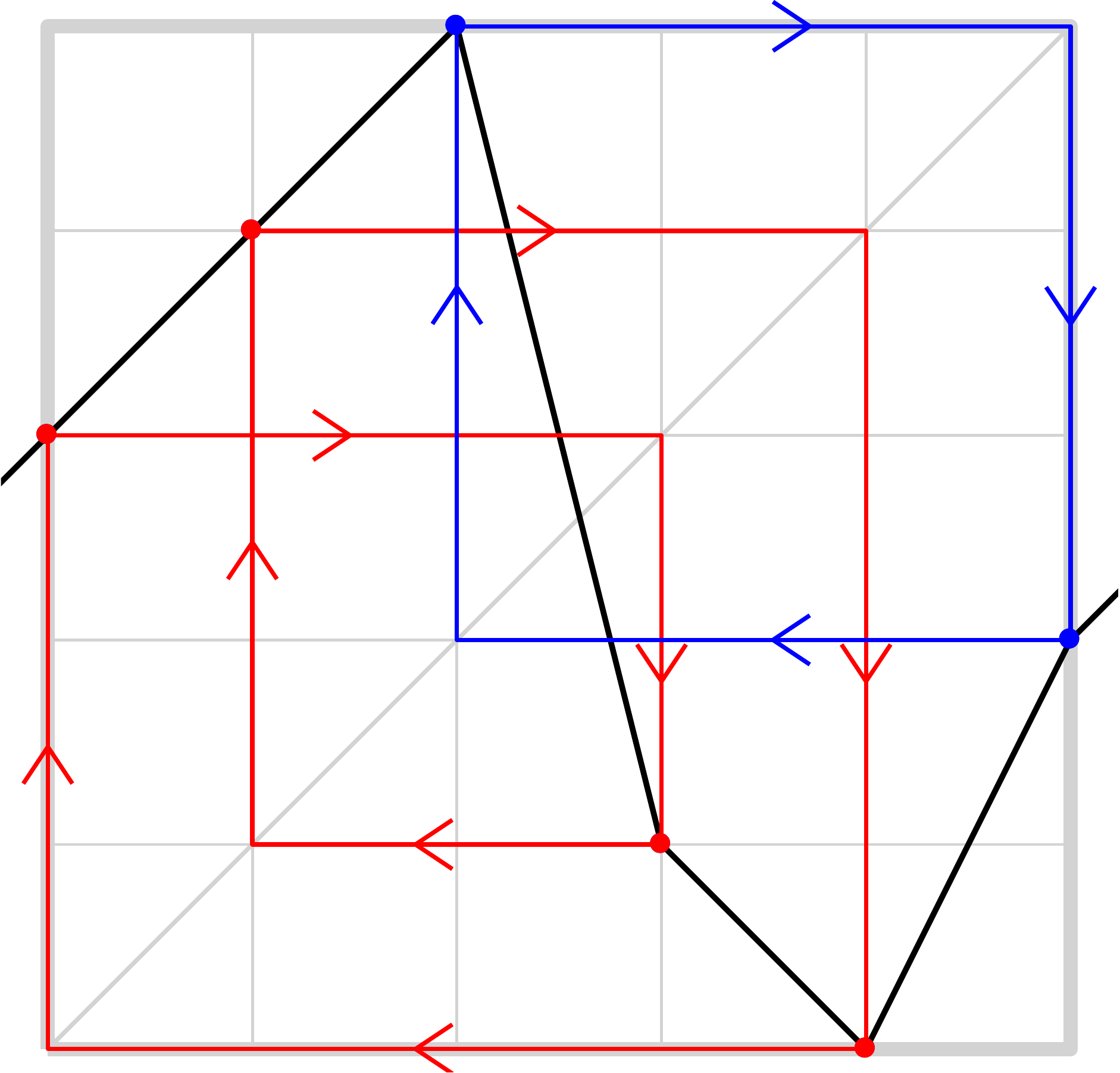}
    \includegraphics[height=\figHt]{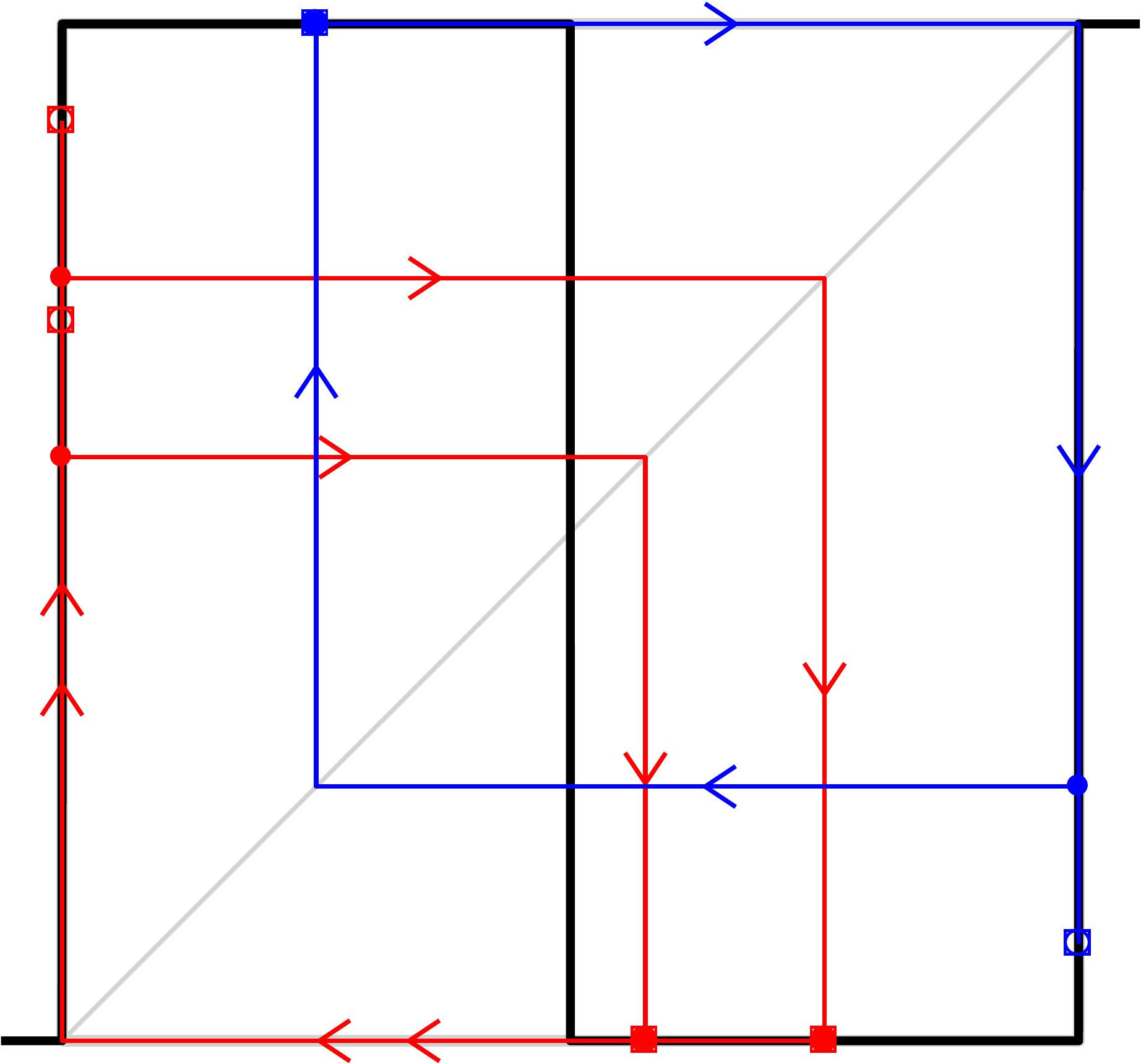} \quad
    \includegraphics[height=\figHt]{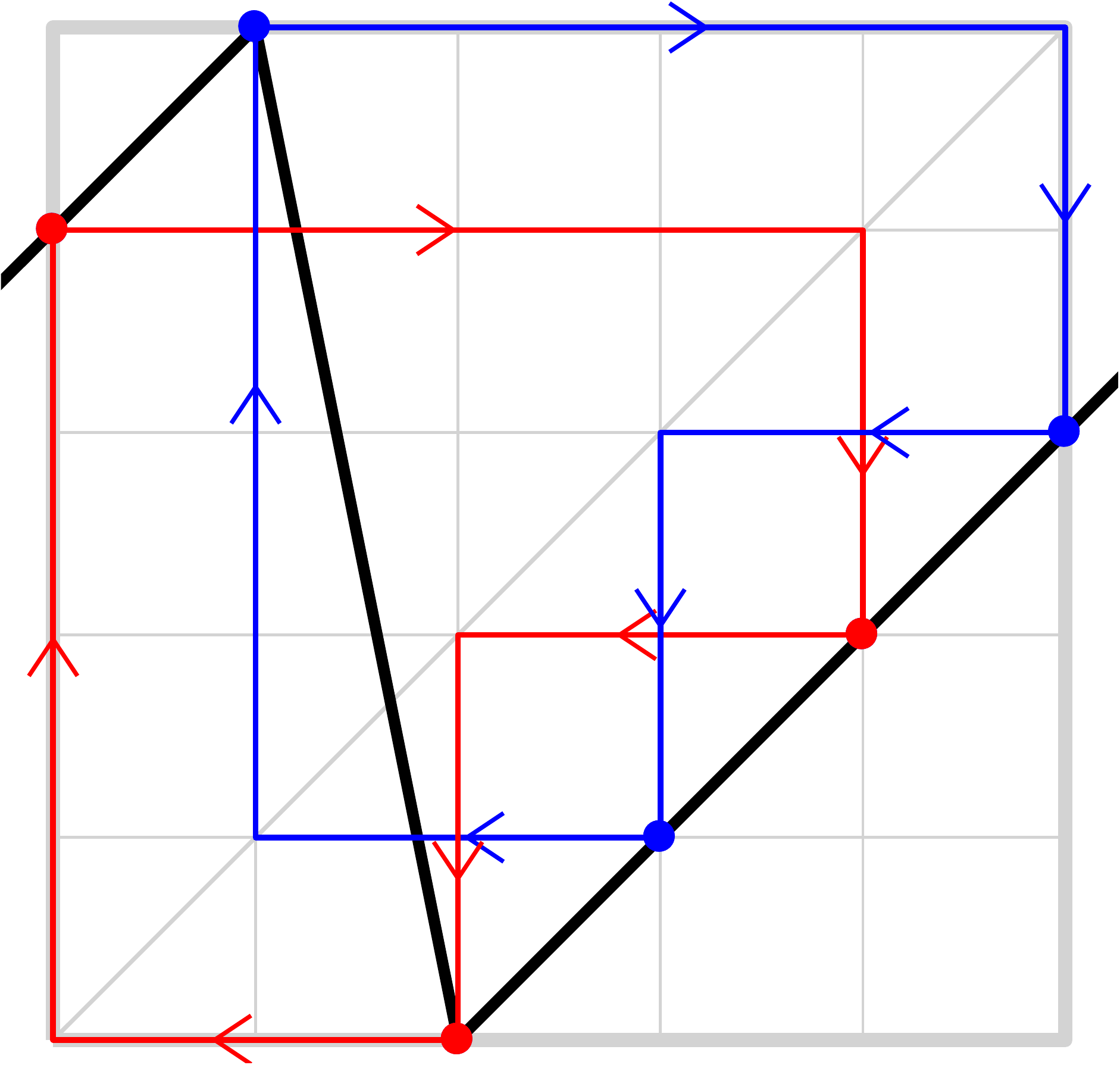}
    \includegraphics[height=\figHt]{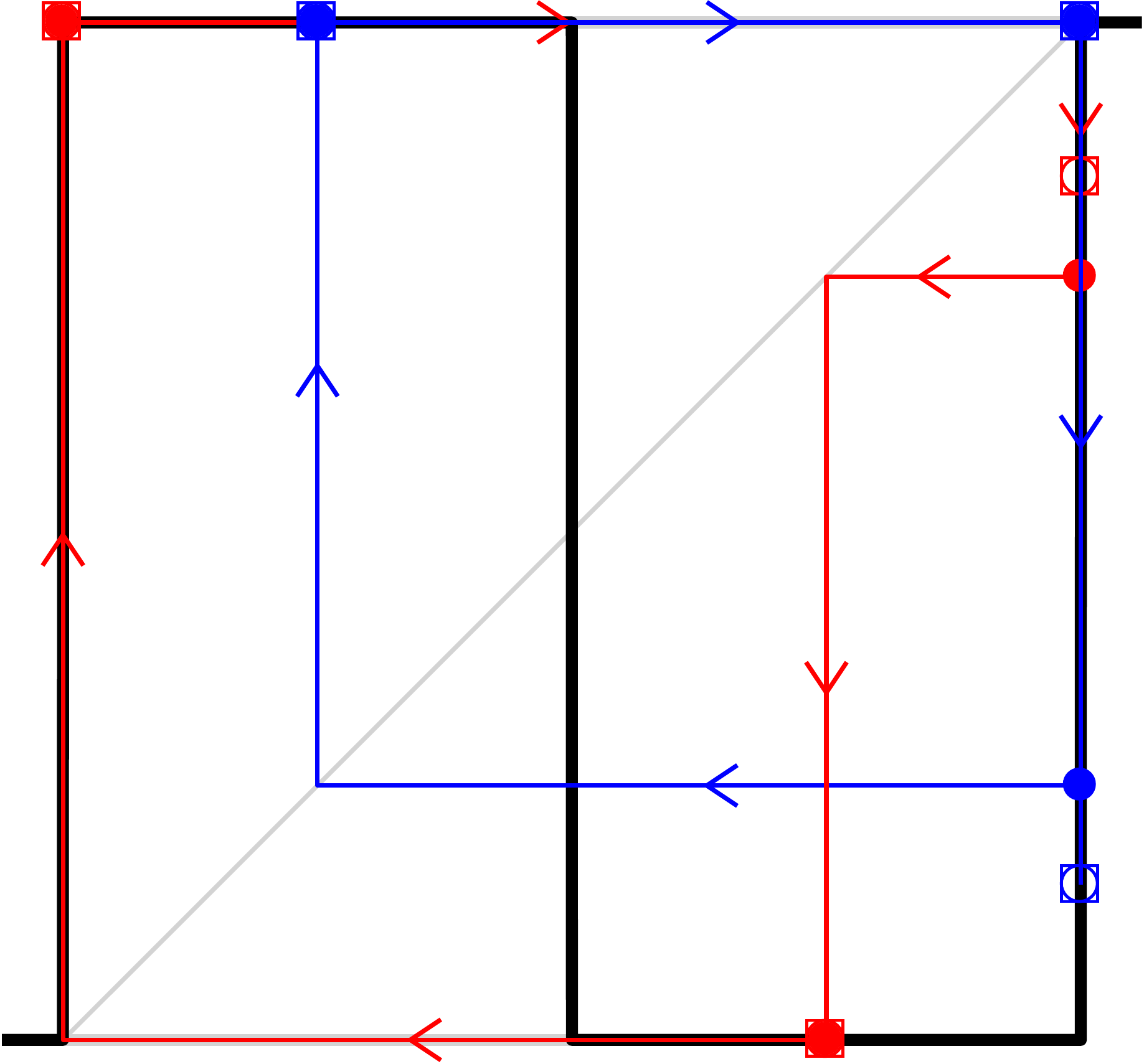} 
    }
 \caption{\label{f-aDn4b}   
   On the left is 
   shown obstructed combinatorics $\(3,\,4,\,5,\,1,\,0,\,2\)$,
   with mapping pattern $\du{x_4}\mapsto x_0\mapsto x_3 \mapsto x_1
   \mapsto\du{x_4}$  and  $\du{x_2}\leftrightarrow x_5$.  
   On the right is shown
   obstructed combinatorics $\(4,\,5,\,0,\,1,\,2,\,3\)$
   with mapping pattern 
   $\du{x_2}\mapsto x_0\mapsto x_4 \mapsto {x_2}$
   and $\du{x_1}\mapsto x_5 \mapsto x_3 \mapsto \du{x_1}$.
  Both have shape $+-+$.
 }
\end{figure}

\FloatBarrier
\phantom{menace}
\ifthenelse{\IsThereSpaceOnPage{.3\textheight}}{\clearpage}{\relax} 
\subsubsection*{Half-Hyperbolic: One critical orbit is periodic,  but
  the other is eventually repelling. }

\begin{figure}[!htb]
  \centerline{%
    \includegraphics[height=\figHt]{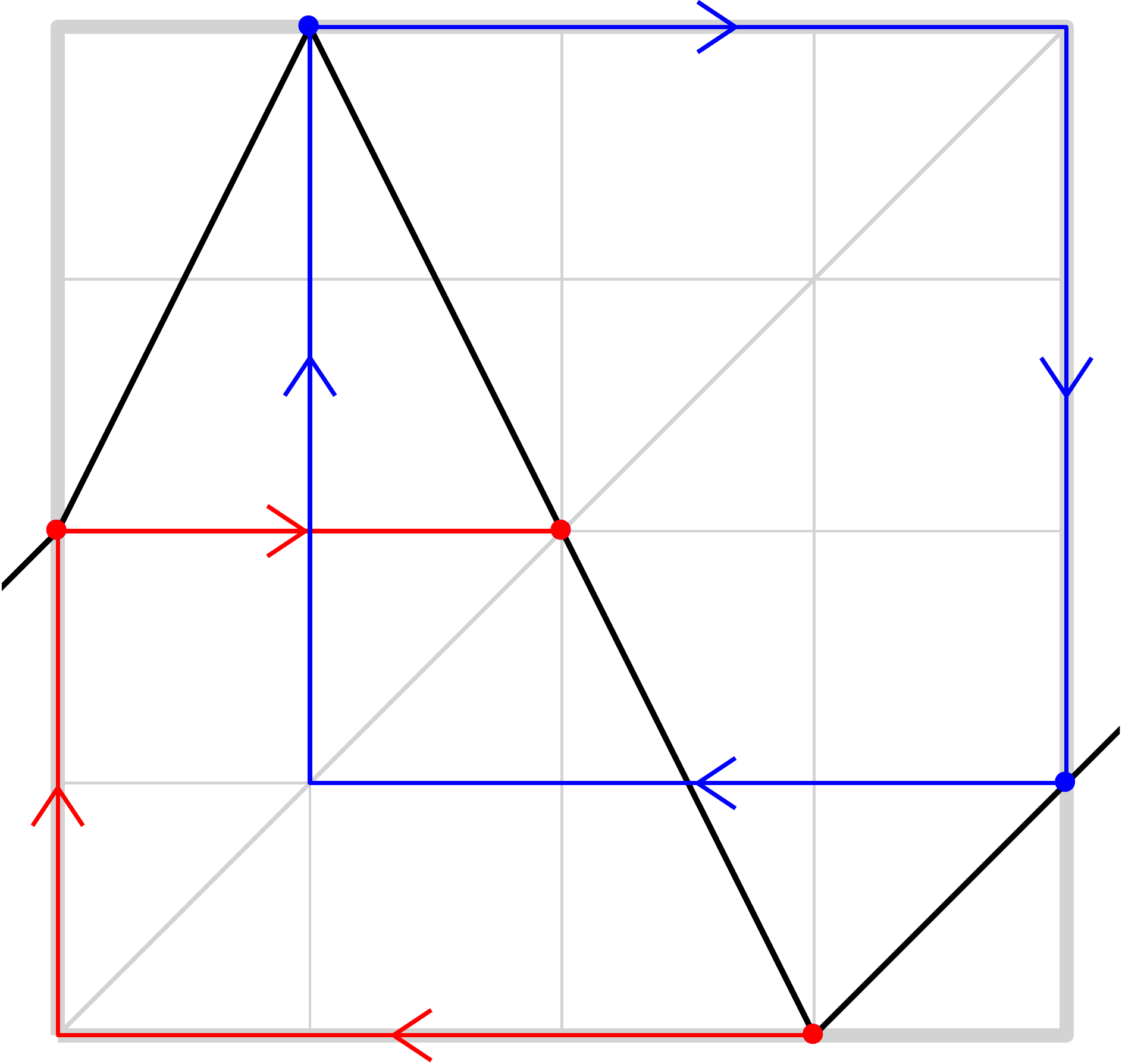}
    \includegraphics[height=\figHt]{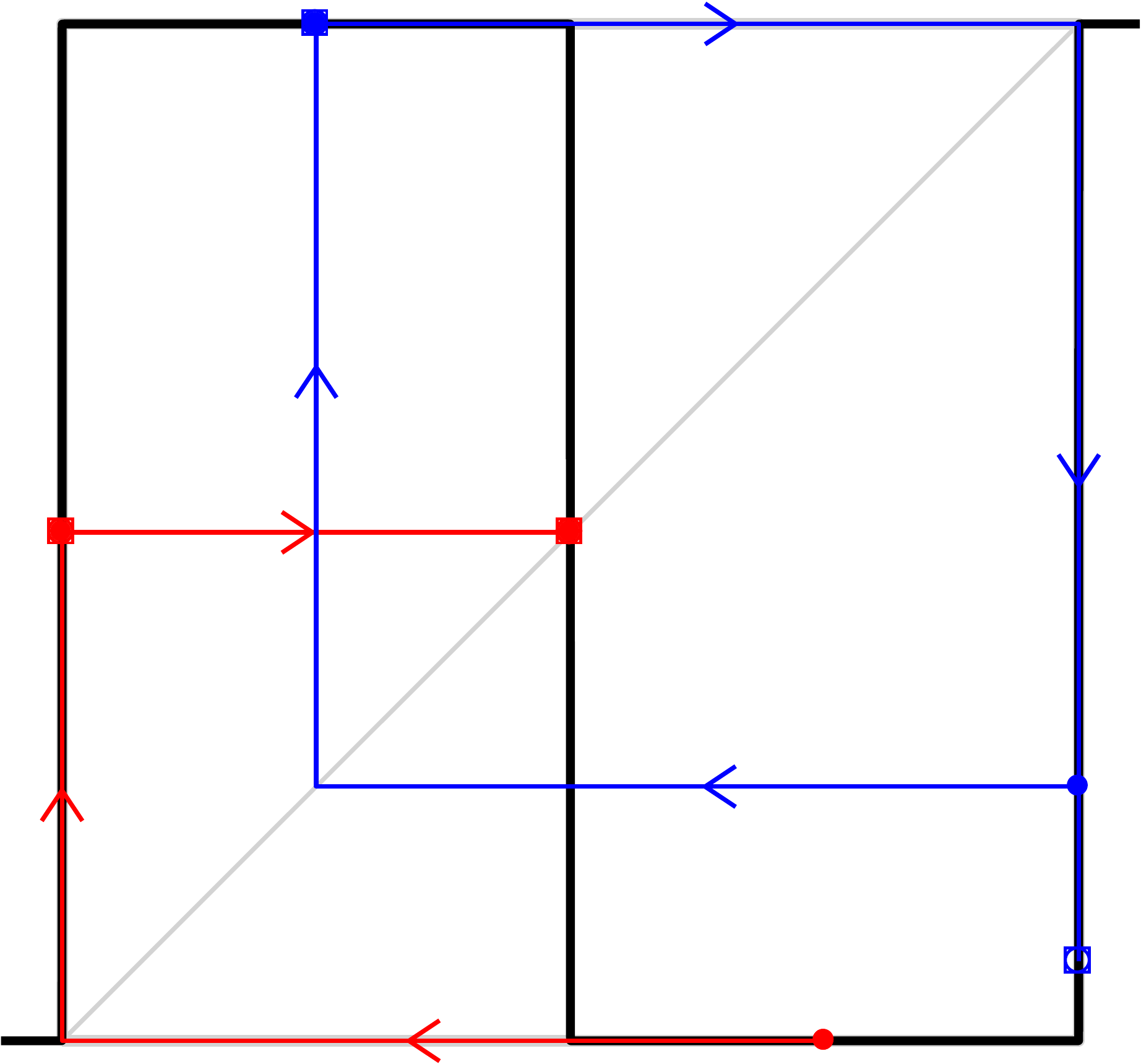} \hfil
    \includegraphics[height=\figHt]{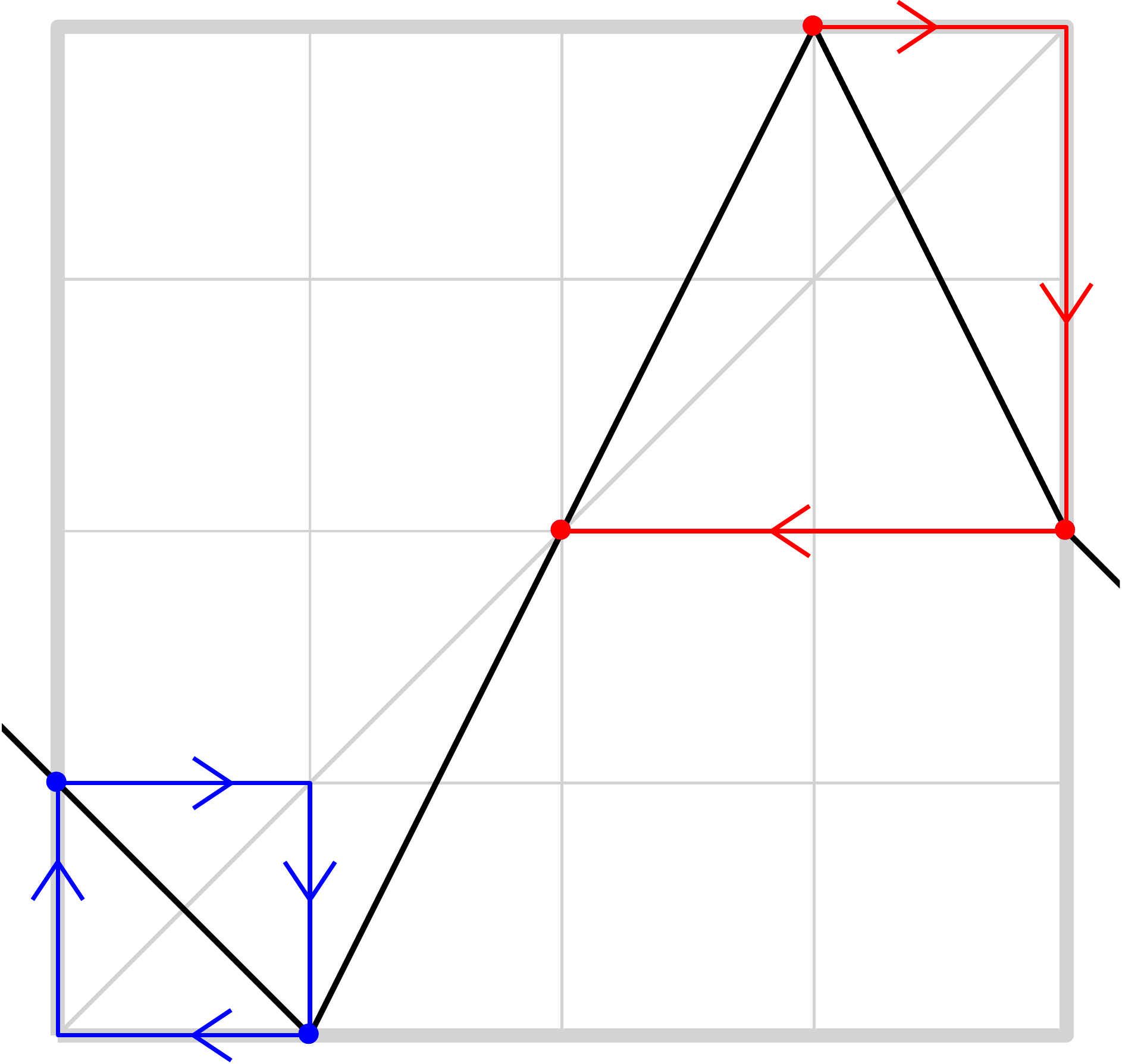}
    \includegraphics[height=\figHt]{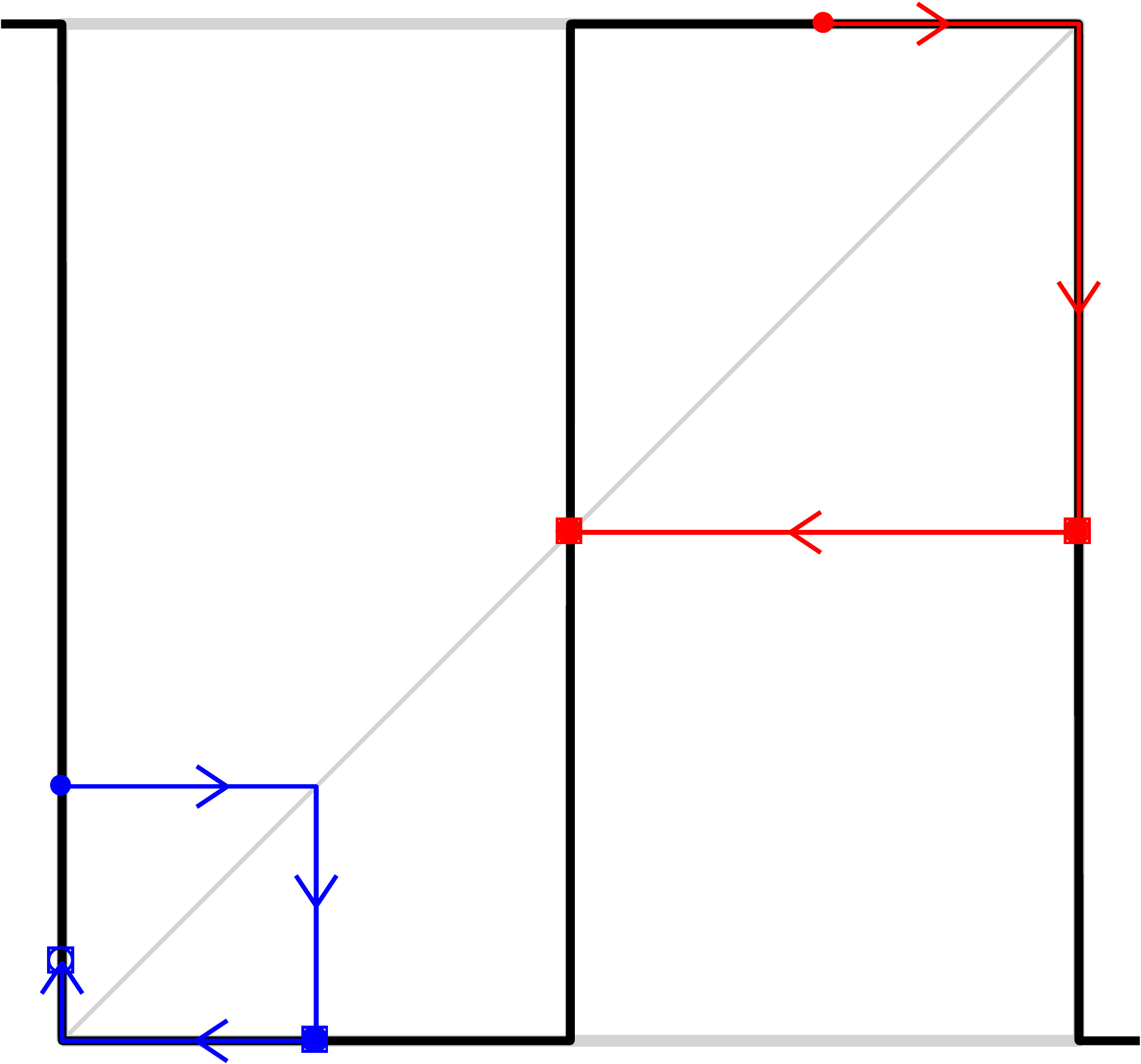}}
  \caption{\label{f-semihyp}   On the left, a strongly
    obstructed case  of topological shape $+-+$ with combinatorics
    $\(2,\,4,\,2,\,0,\,1\)$.
   The mapping pattern is $\du{x_1} \leftrightarrow x_4$; and
   {$\du{x_3}\mapsto x_0\mapsto x_2 \mapstoself$}. On the right,
    combinatorics $\(1,\,0,\,2,\,4,\,2\)$, 
  topological shape  $-+-$,   
   and mapping pattern $\du{x_1}\leftrightarrow x_0$ and $\du{x_3}\mapsto x_4
 \mapsto x_2\mapstoself$.  This is the
 ``left-right reflected'' version of the picture on the left.}
\end{figure}

\FloatBarrier
\ifthenelse{\IsThereSpaceOnPage{.3\textheight}}{\clearpage}{\relax} 
\subsubsection*{Totally Non-Hyperbolic: Every postcritical cycle is repelling.}

\begin{figure}[!htb]
  \centerline{%
    \includegraphics[height=\figHt]{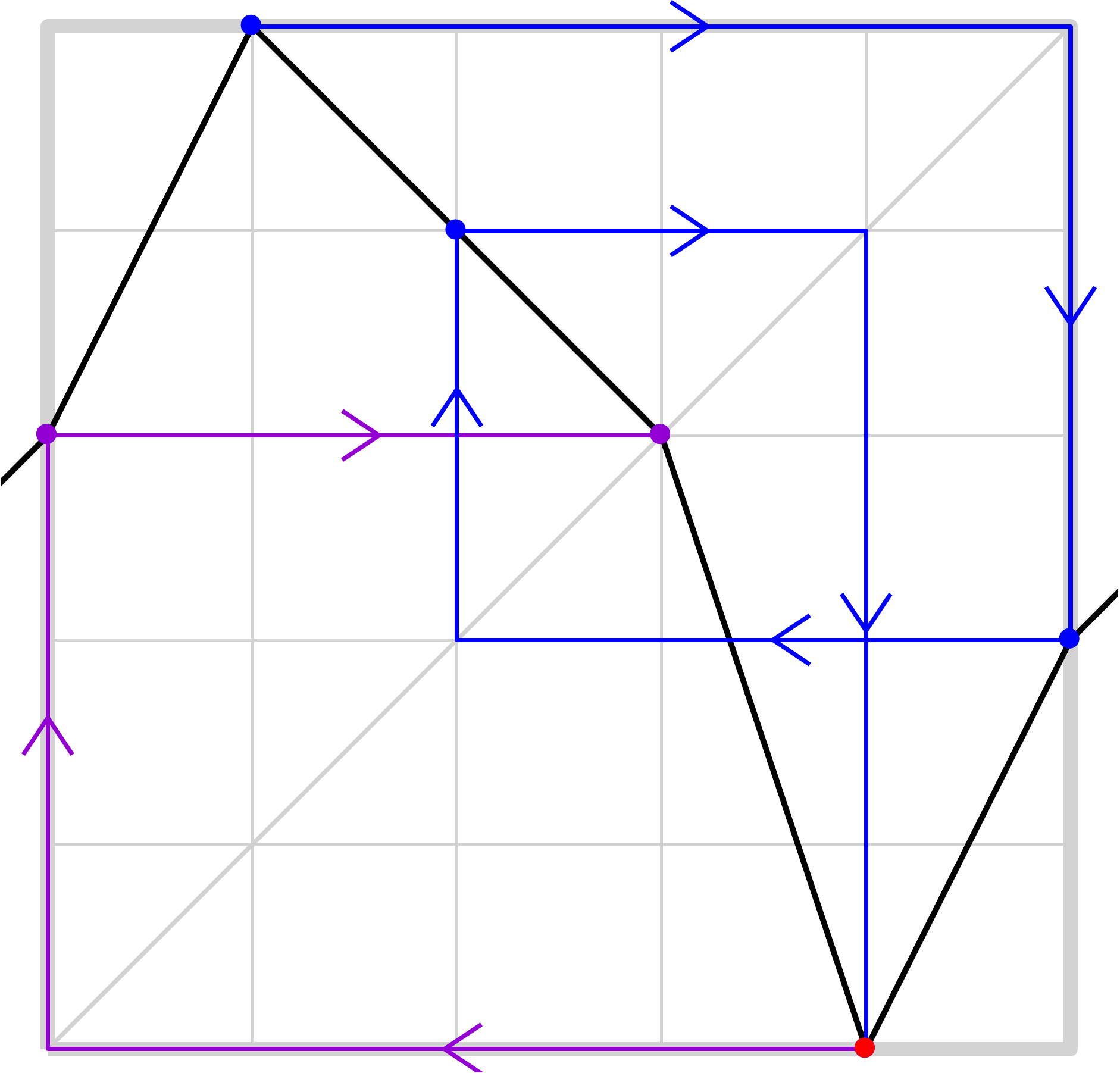} \qquad
    \includegraphics[height=\figHt]{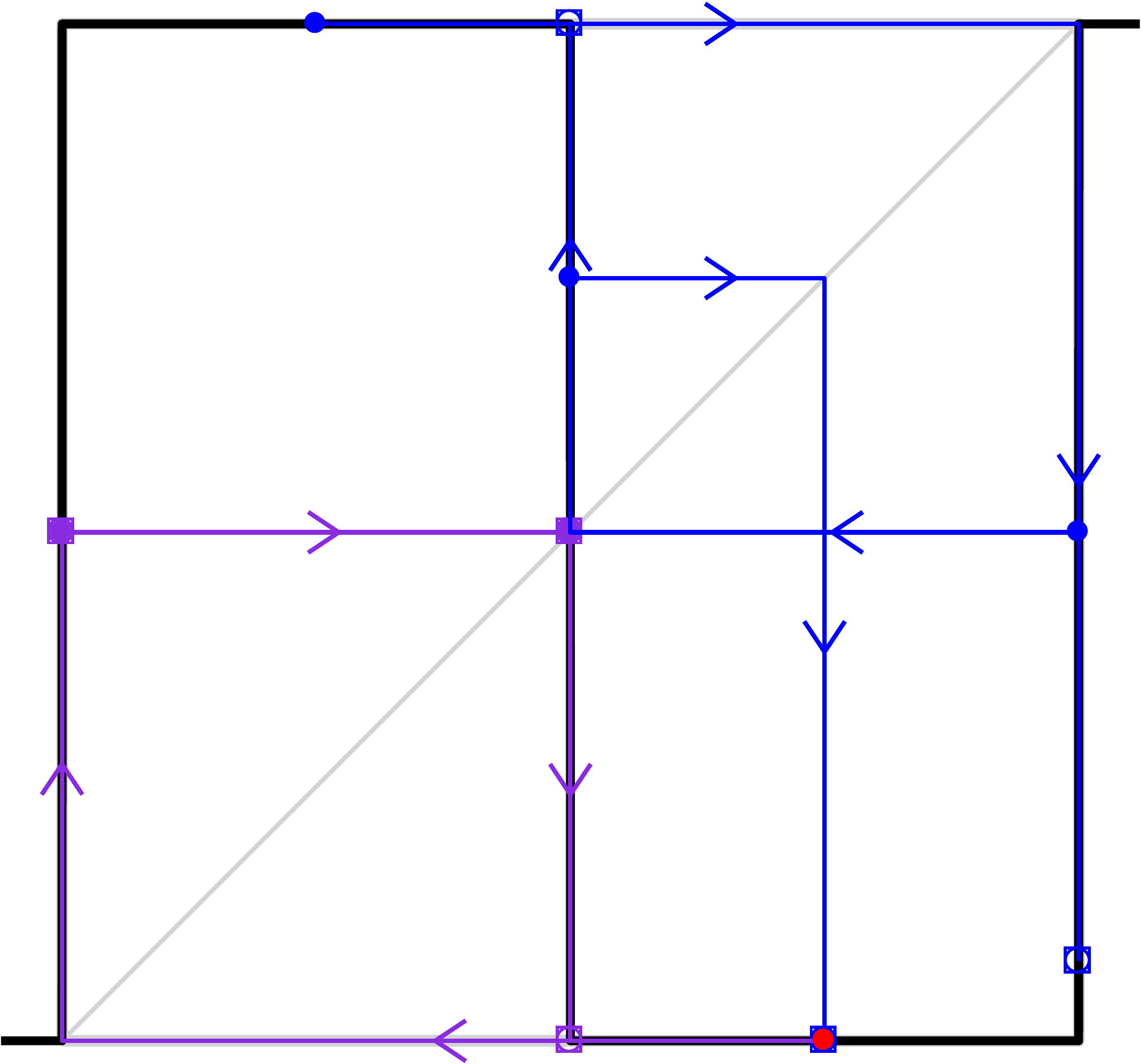}}
    \caption{\label{f-nonhyp-obst}  An illustration of a strongly
      obstructed map  of topological shape $+-+$~ with combinatorics
    $\(3,\,5,\,4,\,3,\,0,\,2\)$ and mapping pattern 
    $\du{x_1}\mapsto x_5 \mapsto x_2 \mapsto \du{x_4}\mapsto x_0\mapsto x_3
      \mapstoself$.}
 \end{figure}

 \begin{figure}[!htb]
   \centerline{%
     \includegraphics[height=\figHt]{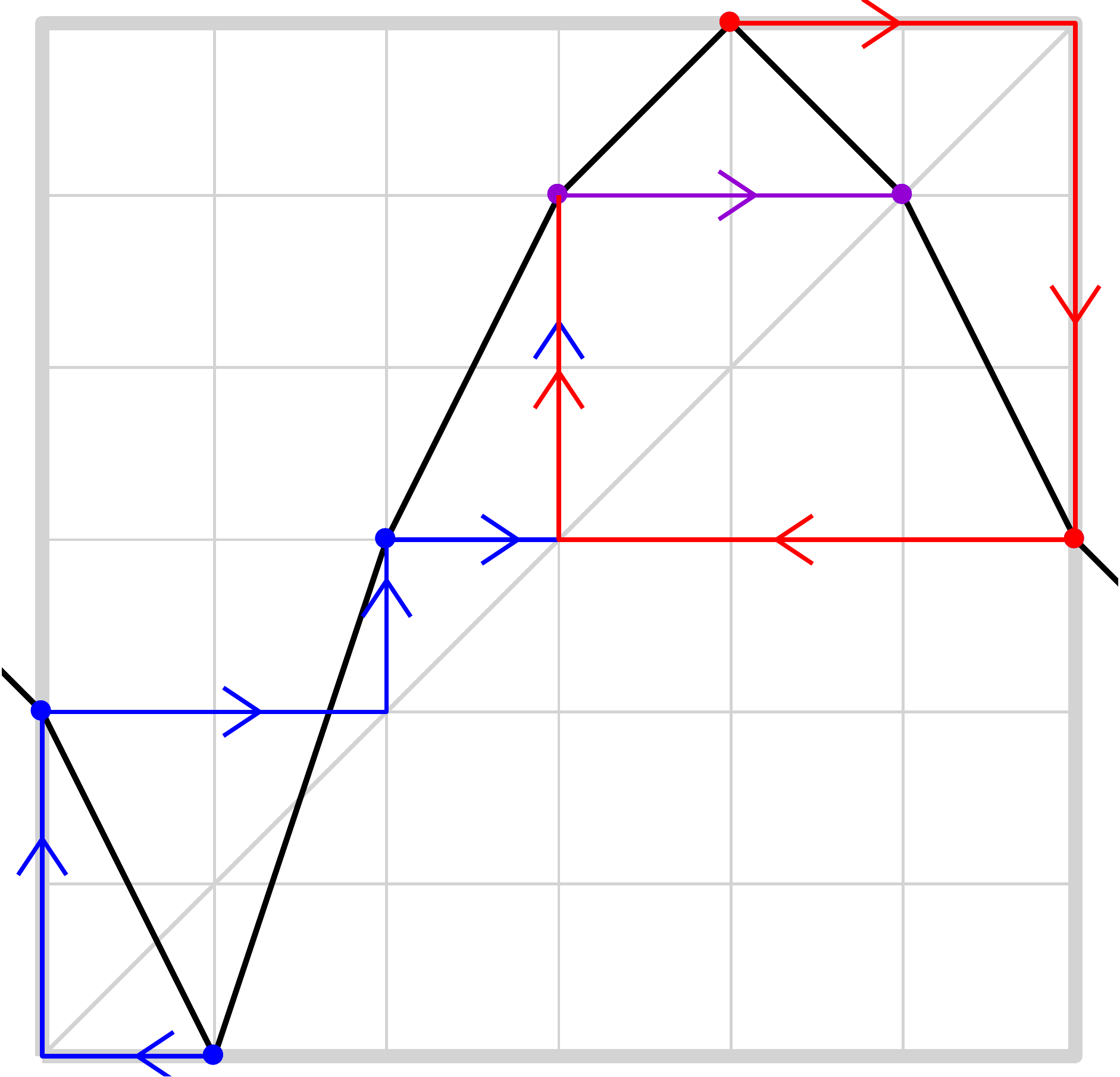}
     \includegraphics[height=\figHt]{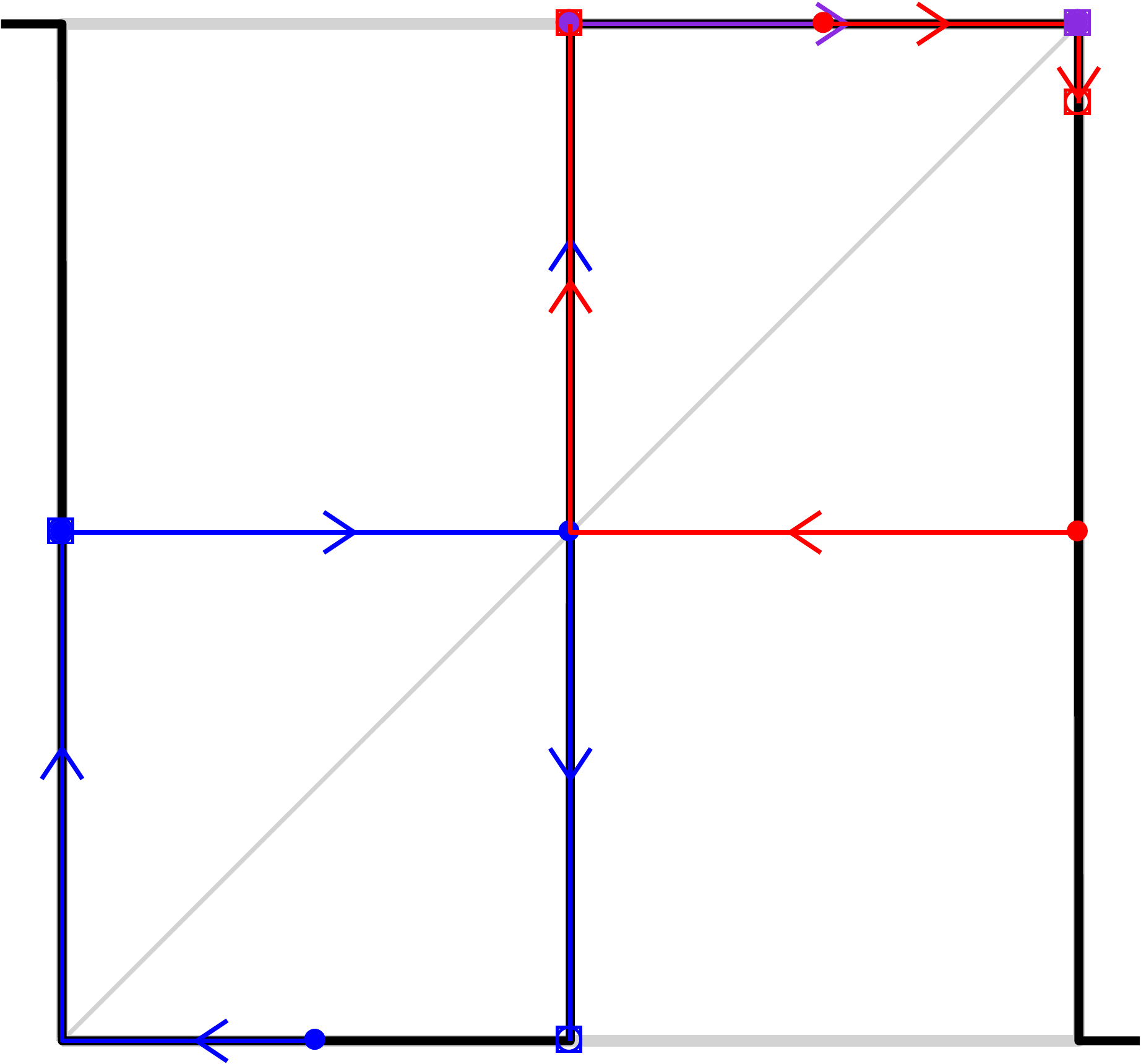} \hfil
     \includegraphics[height=\figHt]{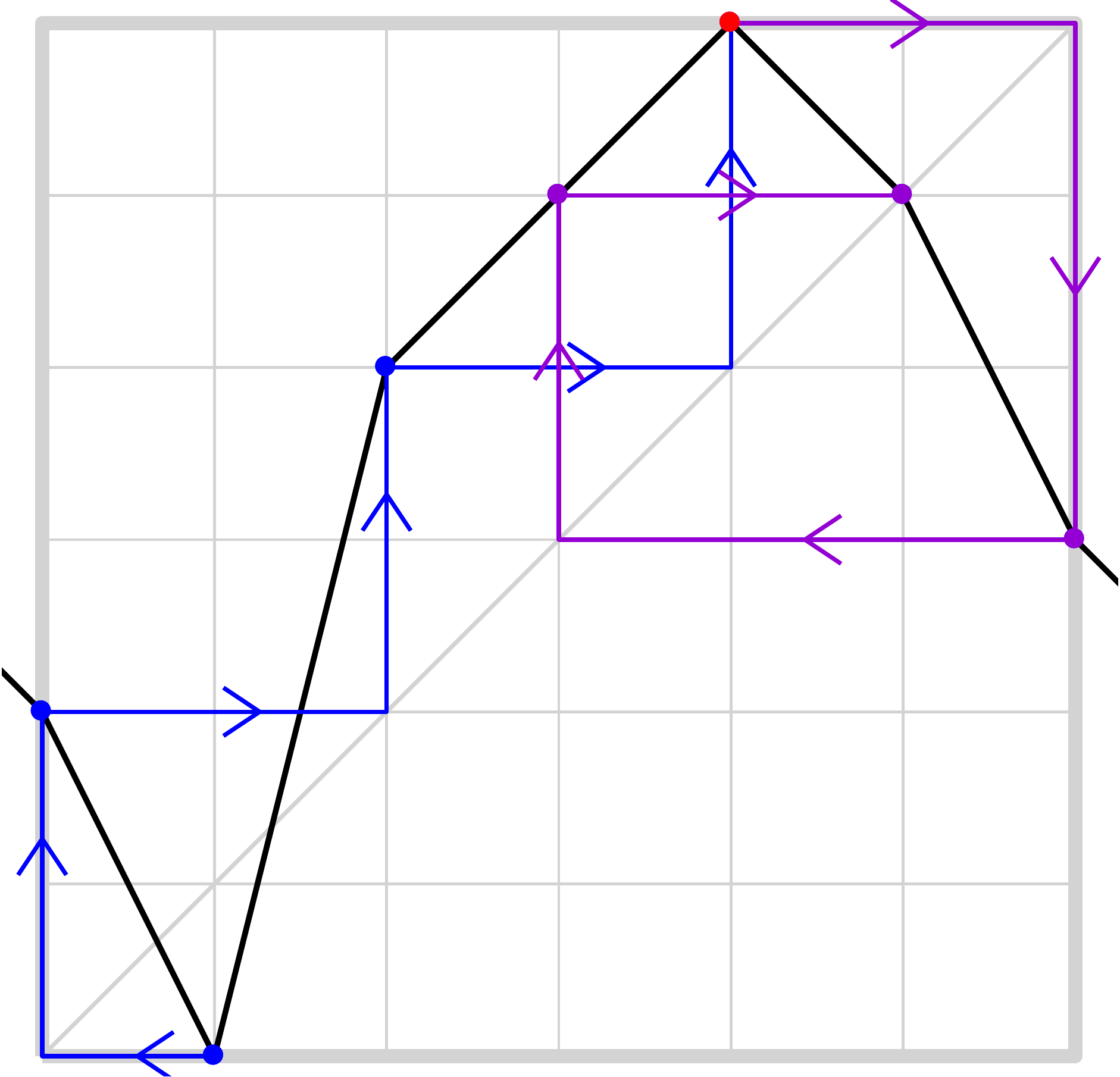}
     \includegraphics[height=\figHt]{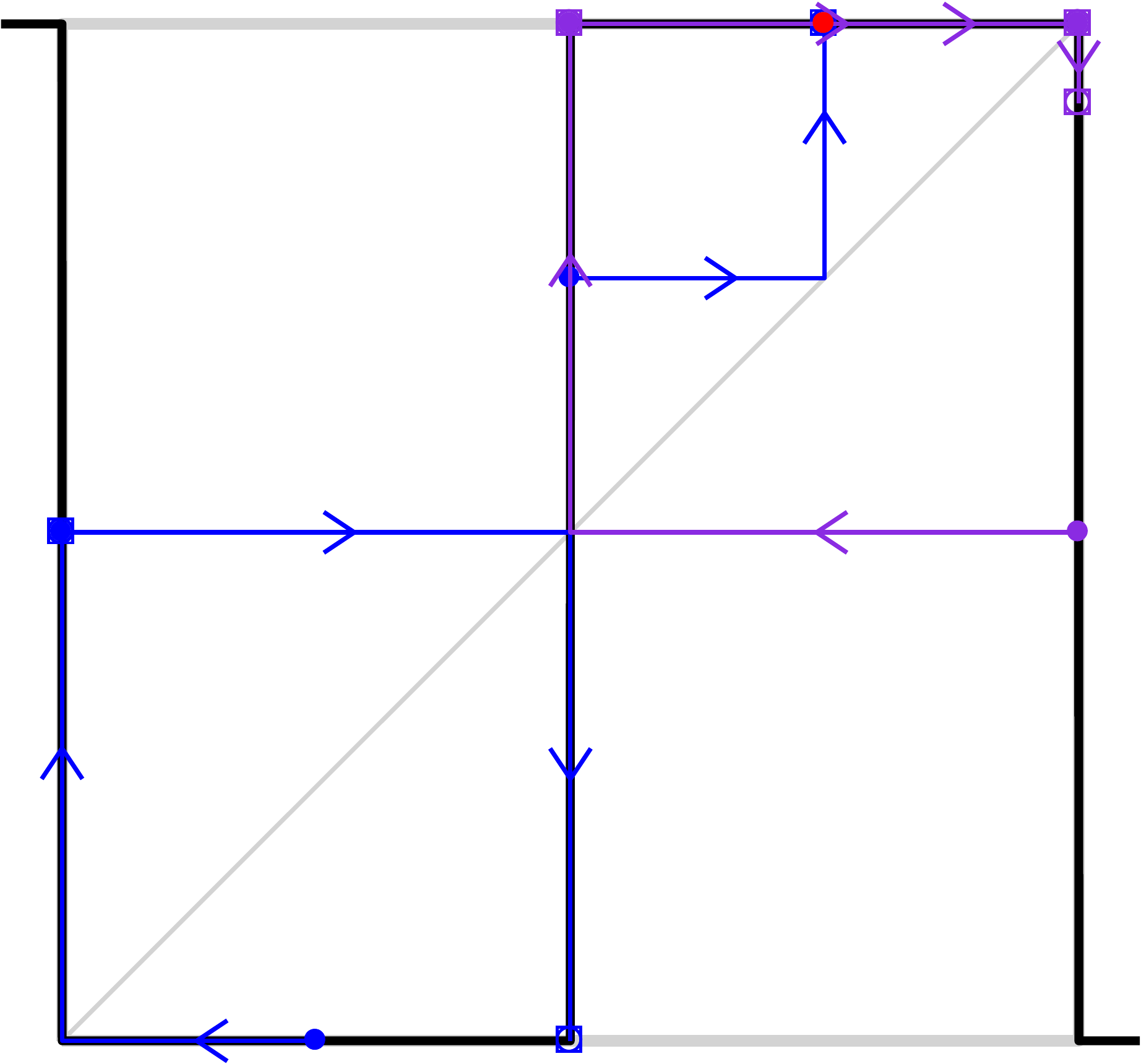}}
    \caption{\label{f-nonhyp-obst1}  
      Two cases of shape $-+-$.
      Left: combinatorics $\(2,\,0,\,3,\,5,\,6,\,5,\,3\)$ and mapping
      pattern  
    {$\du{x_1}\mapsto x_0 \mapsto x_2 \mapsto x_3 \mapsto x_5\mapstoself$}
    and {$\du{x_4}\mapsto x_6 \mapsto x_3\mapsto x_5 \mapstoself$}.
    Right:  combinatorics $\(2,\,0,\,4,\,5,\,6,\,5,\,3\)$ and mapping
    pattern   {$\du{x_1}\mapsto x_0 \mapsto x_2 \mapsto \du{x_4} \mapsto x_6
      \mapsto x_3 \mapsto x_5\mapstoself$}.
    }
 \end{figure}

\FloatBarrier
\section{Further Information about the Figures}\label{ap-b} 
After a preliminary remark, this appendix will consist of a brief \autoref{t-1}
which classifies figures according to their dynamic type and topological shape,
and the more extensive 
\autoref{t-2} which lists their combinatorics and parameters.\msk

\begin{rem}[Constraints on the parameters]\label{R-constr}
  Note that:
  \begin{eqnarray*}
     \textrm{Shape}~~ +-+~~ \textrm{or~ co-polynomial}&\Longrightarrow&\quad
    \Sigma\,\in\,[-1,\,0]~,\\
    \textrm{Shape}~~ -+-~~ \textrm{or~ polynomial}&\Longrightarrow&\quad               
     \Sigma\,\in\,[0,\,1]~.
  \end{eqnarray*}
  On the other hand,
  \begin{eqnarray*}
  \textrm{for the}~~ -+~~ \textrm{combinatorics,} &\Sigma& \in\; [-0.5,\, 0.5]\,\\
  \textrm{while for}~~ +-~~ \textrm{combinatorics,} &\Sigma& \in\;
  [-1,\,-0.5]\,\cup\,  [0.5,\,1]~.
  \end{eqnarray*}
  (Compare \autoref{F-canmod}.)
 Note that $\mu$ and $\Delta$  are invariant under orientation 
  reversal. However, $\kappa$ changes sign and~$\Sigma$ is replaced by
  $\pm 1-\Sigma~$.
  \end{rem}
\ssk

Recall that
all combinatorics of topological shape $-+-$ are obstructed;  
but that combinatorics of unimodal  shape are unobstructed, except
possibly in the totally non-hyperbolic case. (Compare 
\autoref{R-rel}, as well as \autoref{L2} and  \autoref{C-no-obs}.) \msk

\stepcounter{equation} 
\captionof{table}{\label{t-1}  A list of figure numbers
  classified according to their dynamic  type and topological shape.
  It includes figures
  which illustrate minimal and non-polynomial combinatorics,      
  We provide at least one example for each combination 
 which can actually occur, with
  the possible exception of the totally non 
      hyperbolic unimodal  case where we conjecture that there are no 
      examples. Note that Figures~\ref{F-exc}, \ref{f-aBn2}L, and
    \ref{f-nonhyp}L  are the only examples
    with Euclidean orbifold; and that
    \autoref{F-exc} is the only example where
    the rational function exists but the Thurston algorithm
    usually doesn't converge (compare  \autoref{f-pullback-134310}).}
\begin{center}
\small
\newcommand{\rb}[1]{\raisebox{-1.5ex}{#1}}
\newcommand{\trb}[1]{\rb{\textrm{#1}}}
\begin{tabular}{|c|c|c|c|c|c|}
   \hline
   \multicolumn{6}{|c|}{\sf Unobstructed}\Tstrut \Bstrut\\\hline \hline
 &\rb{\sf B}&\rb{\sf C}&\rb{\sf D}&{\sf Half}&{\sf Totally Non-}\Tstrut\\[-1ex]
 &      &      &      &{\sf Hyperbolic}&{\sf Hyperbolic} \Bstrut\\ \hline
\rb{co-poly}&\ref{f-aBn2}LR, \ref{f-aBn3}LR, & \trb{none} & \trb{none} &
         \trb{none} & \rb{\ref{f-nonhyp}LR, \ref{f-nonhyp1}L}\Tstrut\\[-1ex]
        & \ref{f-aBn4}LR, \ref{f-aBn4a}L & & & & \Bstrut\\ \hline
 $+-+$   & \ref{f-aBn4a}R, \ref{f-aBn5a} & \ref{f-aC2}R&
        \ref{f1},\,\ref{f-typeD}& \ref{f-semihyp1}R& \ref{f-nonhyp1}R\Tstrut\Bstrut\\\hline
      {\rm unimodal}& {\rm none} &\ref{f2}, \ref{f6}, \ref{f-aC1}LR, \ref{f-aC2}L& {\rm none}& \ref{f-semihyp1}L & \ref{F-exc}\Tstrut\Bstrut\\ \hline
\multicolumn{6}{c}{\relax}\Tstrut\Bstrut\\[-2.75ex]  \hline
\multicolumn{6}{|c|}{\sf Obstructed}\Tstrut \Bstrut\\\hline \hline
 $+ - +$& \ref{f-aBobst}&\ref{f-aC3}L&\ref{f-aBn3a}L, \ref{f-aDn4a}L, \ref{f-aDn4b}LR& \ref{f-semihyp}L &\ref{f-nonhyp-obst}\Tstrut\Bstrut\\
      \hline
      $- + -$&{\rm none} &\ref{f-aC3}R&\ref{f-aBn3a}R, \ref{f-aDn4a}R& \ref{f-semihyp}R & \ref{f-nonhyp-obst1}LR\Tstrut\Bstrut\\\hline
     {\rm unimodal}&{\rm none}&{\rm none}&{\rm none}&{\rm none}&{\rm none} ?\Tstrut\Bstrut\\\hline
    \end{tabular}
\end{center}
{ 
\def\+{\phantom{-}}
\newcommand\half{1/2} 
\newcommand{\corpoly}{\rule[-2ex]{0pt}{0pt} 
  \raisebox{-1ex}{$\stackrel{\hbox{corresp.}}{\hbox{poly}}$}}
\newcommand{\rb}[1]{\raisebox{1.5ex}[0pt][0pt]{#1}}
\begin{longtable}{|c|c:r :c:c:c:c|}
  \caption{ \label{t-2}
This table lists the parameters for all of the figures which
illustrate some  
combinatorics.  Parameters for the ``corresponding
polynomials'' of \autoref{t-kninv} are also listed. 
Here $N$  denotes the number of iterations that were used in producing
the figures,  starting with the points of
$\vecx$ equally spaced within the
appropriate intervals and continuing until $(\Sigma,\Delta)$ was within
$10^{-7}$ of the limiting value. 
However, two examples require special mention:
The $*$ in the ``N'' column for \autoref{F-exc} indicates that the
data for this exceptional case had to be computed directly,
without any iteration. The entries for \autoref{f-pullback-134310} show
what actually happens when we iterate with this combinatorics. (In fact
the initial conditions for \autoref{f-pullback-134310} were 
specially chosen, but any generic choice would yield similar behavior.)
In the unobstructed cases
the parameters $\mu$ and $\kappa$ (or $\Sigma$ and $\Delta$) 
of the limit map uniquely determine the combinatorics; but this is not true
for  strongly obstructed cases (characterized by $\mu=\pm\infty$ or by
$\Delta=1$). 
The converse statement that the combinatorics determines
the parameters is true in all cases. 
The parameter $\Sigma\in\R/2\Z$ is always  specified 
by its representative in the interval $(-1,\,1]$. \bsk
}\\
 \hline
 \sf figure &\sf combinatorics& $N$& $\mu$& $\kappa$& $\Sigma$&
 $\Delta$\Tstrut\Bstrut\\ \hline \hline 
\endfirsthead
  \captionsetup{singlelinecheck=off}
  \caption[]{\textit{Continued from previous page}}\\
  \hline  \sf figure &\sf combinatorics& $N$& $\mu$& $\kappa$& $\Sigma$&
     $\Delta$\Tstrut\Bstrut\\ \hline \hline 
\endhead
  \hline
  \multicolumn{7}{r}{\textit{\sffamily \autoref{t-2} continues on next page}}\Tstrut\\
\endfoot 
  \hline
\endlastfoot
\ref{f1}, \ref{f5} & $\(5,6,4,1,0,2,3\)$& 87& $-7.2407034$& $\+1.097305$&$-0.338652$&$0.836756$\Tstrut\Bstrut\\ \hline
\ref{f2}& $\(1,2,3,2,0\)$&63& $-1.270048$& $-1.351520$&$\+0.893946$&$0.561904$\Tstrut\Bstrut\\  \hline
\rb{\ref{F-nonex}}&$\stackrel{\hbox{$\(4,2,1,0,1\)$}\Tstrut\Bstrut}{\hbox{$\(3,1,0,1\)$}}$&
   $\stackrel{\hbox{$12$}\Bstrut}{\hbox{$\phantom{1}2$}}$&\rb{$-4$}&\rb{$\+1$}&$\rb{-0.271699}$&$\rb{0.728301}$\Tstrut\Bstrut\\\hline
\ref{f-bad}&$\(3,5,3,2,0,2\)$&6&$-\infty$&$\+0$&$-\half$&$1$\Tstrut\Bstrut\\\hline
\ref{f6}&$\(3,2,1,2\)$&43&$-1.295581$&$\+1.191483$&$\+0$&$\half$\Tstrut\Bstrut\\\hline
\ref{F-algExamp}&$\(1,2,4,5,2,1,0\)$&42&$-2.594313$&$-1.637869$&$\+1$&$0.712071$\Tstrut\Bstrut\\ \hline
\ref{F-exc}&$\(1,3,4,3,1,0\)$&*&$-2$&$-\sqrt{2}$&$\+1$&$0.635940$\Tstrut\Bstrut\\\hline
 &  &odd&$-8/3$&$-5/3$&$\+1$&$0.719622$\Tstrut\\
\rb{\ref{f-pullback-134310}}&\rb{$\(1,3,4,3,1,0\)$}&even&$-3/2$&$-5/4$&$\+1$&$0.545629$\Bstrut\\\hline
\ref{f-weakobs}L&$\(3,5,3,2,1,0\)$&14&$-4$&$-1$&$-0.728301$&$0.728301$\Tstrut\Bstrut\\\hline
\ref{f-weakobs}R&$\(3,4,3,2,1,0\)$&48&$-4/3$&$-1$&$-0.892232$& $0.455511$\Tstrut\Bstrut\\ \hline
\ref{f-sq}&$\(2,0,1,2,5,3\)$&9&$\+\infty$&$-1$&$\half$&$1$\Tstrut\Bstrut\\\hline 
\ref{F-FP2} & $\FP(8,9)$ &$1000$&$-256.3050$&$-1.494245$&$-0.507393$&$0.995033$\Tstrut\Bstrut\\\hline
\ref{f-3201}L&$\(3,2,0,1\)$&14&$-3.796781$&$\+0.898403$&$-0.287930$&$0.712070$\Tstrut\Bstrut\\\hline
\ref{f-3201}LC&$\(3,2,0,1,4\)$&31&$\+3.498562$&$\+0.749281$&$\+0.258130$&$0.741869$\Tstrut\Bstrut\\ \hline 
\ref{f-43013}L&$\(4,3,0,1,3\)$&23&$-5.538584$&$\+1.769292$&$-0.182973$& $0.817027$\Tstrut\Bstrut\\ \hline
\ref{f-43013}LC&$\(4,3,0,1,3,5\)$&24&$\+3.890875$&$\+0.945437$&$\+0.223012$&$0.776988$\Tstrut\Bstrut\\ \hline 
\ref{f-aBn2}L&$\(1,0\)$&1&$-2$&$\+0$&$-\half$&$\+\half$\Tstrut\Bstrut\\
\corpoly&$\(0,2,1\)$&12&$\+3.236068$&$-0.618034$&$\+0.712071$&$0.712071$\Bstrut\\ \hline
\ref{F-min}, \ref{f-aBn2}R&$\(1,2,0\)$&12&$-4.649436$&$-1.324718$&$-0.772304$&$0.772304$\Tstrut\Bstrut\\ 
\corpoly &$\(0,2,3,1\)$&12&$\+3.831874$&$-0.915937$&$\+0.772304$&$0.772304$\Bstrut\\ \hline
  \ref{f-aBn3}L&$\(1,2,3,0\)$&69&$-5.968584$&$-1.984292$&$-0.833802$&$0.833802$\Tstrut\Bstrut\\
\corpoly &$\(0,2,3,4,1\)$&12&$\+3.960270$&$-0.9801349$&$\+0.782264$&$0.782264$\Bstrut\\ \hline
  \ref{f-aBn3}R&$\(2,3,1,0\)$&12&$-3.796780$&$-0.898402$&$-0.712071$&$0.712071$\Tstrut\Bstrut\\ 
\corpoly &$\(0,3,4,2,1\)$&31&$\+3.498562$&$-0.749281$&$\+0.741869$&$0.741869$\Bstrut\\ \hline
\ref{f-aBn4}L&$\(1,2,3,4,0\)$&90&$-6.656438$&$-2.328219$&$-0.855761$&$0.855761$\Tstrut\Bstrut\\ 
\corpoly  &$\(0,2,3,4,5,1\)$&13&$\+3.990267$&$-0.995134$&$\+0.784470$&$0.784470$\Bstrut\\ \hline
\ref{f-aBn4}R&$\(2,4,3,1,0\)$&13&$-4.044724$&$-1.022362$&$-0.731704$&$0.731704$\Tstrut\Bstrut\\ 
\corpoly &$\(0,3,5,4,2,1\)$&24&$\+3.738915$&$-0.869457$&$\+0.764527$&$0.764527$\Bstrut\\ \hline
\ref{f-aBn4a}L&$\(1,3,4,2,0\)$&44&$-5.628255$&$-1.814128$&$-0.820750$&$0.820750$\Tstrut\Bstrut\\ 
\corpoly &$\(0,2,4,5,3,1\)$&12&$\+3.905709$&$-0.952855$&$\+0.778136$&$0.778136$\Bstrut\\ \hline
\ref{f-aBn4a}R&$\(2,3,4,0,1\)$&88&$-18.46037$&$-1.433732$&$-0.592613$&$0.932596$\Tstrut\Bstrut\\\hline 
\ref{f-aBn5a}&$\(2,3,4,5,6,0,1\)$&100&$-25.104560$&$-3.261693$&$-0.655923$&$0.952150$\Tstrut\Bstrut\\\hline
\ref{f-aC1}L&$\(1,2,1,0\)$&61&$-1.295598$&$-1.191488$&$1$&$\+\half$\Tstrut\Bstrut\\ \hline  
\ref{f-aC1}R&$\(1,2,3,1,0\)$&62&$-2.340346$&$-1.539254$&$1$&$0.682963$\Tstrut\Bstrut\\ \hline
\ref{f-aC2}L&$\(1,2,3,2,0\)$&63&$-1.270048$&$-1.351520$&$\+0.893945$&$0.561904$\Tstrut\Bstrut\\ \hline
\ref{f-aC2}R&$\(4,5,4,0,1,2\)$&25&$-8.401378$&$\+0.770864$&$-0.399338$&$0.854152$\Tstrut\Bstrut\\ \hline
\ref{f-typeD}&$\(6,7,5,4,1,0,2,3\)$&38&$-5.333358$&$\+0.972001$&$-0.318689$&$0.784399$\Tstrut\Bstrut\\ \hline
\ref{f-semihyp1}L&$\(2,3,2,1,0\)$&22&$-1.333333$&$-1$&$-0.8929233$&$0.455512$\Tstrut\Bstrut\\ \hline
\ref{f-semihyp1}R&$\(3,4,6,5,4,0,1\)$&68&$-6.704389$&$-1.109249$&$-0.673413$&$0.825583$\Tstrut\Bstrut\\ \hline
\ref{f-nonhyp}L&$\(2,3,2,0\)$&2&$-4$&$-1$&$-0.728301$&$0.728301$\Tstrut\Bstrut\\ 
\corpoly &$\(0,3,4,3,1\)$&40&$\+3.678573$&$-0.839287$&$\+0.759201$&$0.759201$\Bstrut \\ \hline
\ref{f-nonhyp}R&$\(1,3,4,1,0\)$&23&$-5.538584$&$-1.769292$&$-0.817021$&$0.817021$\Tstrut\Bstrut\\ 
\corpoly &$\(0,2,4,5,2,1\)$&24&$\+3.890873$&$-0.945436$&$\+0.776988$&$0.776988$\Bstrut\\ \hline
\ref{f-nonhyp1}L&$\(1,3,4,3,0\)$&53&$-5.678573$&$-1.839286$&$-0.822747$&$0.822747$\Tstrut\Bstrut\\ 
\corpoly &$\(0,2,4,5,4,1\)$&15&$\+3.927737$&$-0.963869$&$\+0.779808$&$0.779808$\Bstrut\\ \hline
\ref{f-nonhyp1}R&$\(3,4,6,3,1,0,1\)$&99&$-8.266304$&$-1$&$-0.632072$&$0.853973$\Tstrut\Bstrut\\\hline
\ref{f-aBobst}&$\(3,5,4,1,0,2\)$&6&$-\infty$&$\+0$&$-\half$&$1$\Tstrut\Bstrut\\ \hline
\ref{f-aC3}L&$\(2,4,1,0,1\)$&13&$-\infty$&$-1$&$-\half$&$1$\Tstrut\Bstrut\\ \hline
\ref{f-aC3}R&$\(1,0,1,4,2\)$&13&$\+\infty$&$-1$&$\+\half$&$1$\Tstrut\Bstrut\\ \hline
\ref{f-aBn3a}L&$\(2,3,0,1\)$&22&$-\infty$&$\+0$&$-\half$&$1$\Tstrut\Bstrut\\ \hline
\ref{f4}, \ref{f-aBn3a}R&$\(1,0,3,2\)$&22&$\+\infty$&$\+0$&$\+\half$&$1$\Tstrut\Bstrut\\ \hline
\ref{f-aDn4a}L&$\(2,4,3,0,1\)$&12&$-\infty$&$-1$&$-\half$&$1$\Tstrut\Bstrut\\ \hline
\ref{f-aDn4a}R&$\(1,0,3,4,2\)$&13&$\+\infty$&$-1$&$\+\half$&$1$\Tstrut\Bstrut\\ \hline
\ref{f-aDn4b}L&$\(3,4,5,1,0,2\)$&40&$\+\infty$&$-0.618034$&$-\half$&$1$\Tstrut\Bstrut\\ \hline
\ref{f-aDn4b}R&$\(4,5,0,1,2,3\)$&69&$-\infty$&$\+\infty$&$-\half$&$1$\Tstrut\Bstrut\\ \hline
\ref{f-semihyp}L&$\(2,4,2,0,1\)$&13&$-\infty$&$-1$&$-\half$&$1$\Tstrut\Bstrut\\ \hline
\ref{f-semihyp}R&$\(1,0,2,4,2\)$&13&$\+\infty$&$-1$&$\+\half$&$1$\Tstrut\Bstrut\\ \hline
\ref{f-nonhyp-obst}&$\(3,5,4,3,0,2\)$&5&$-\infty$&$-1$&$-\half$&$1$\Tstrut\Bstrut\\ \hline
\ref{f-nonhyp-obst1}L&$\(2,0,3,5,6,5,3\)$&8&$\+\infty$&$\+1$&$\+\half$&$1$\Tstrut\Bstrut \\ \hline
\ref{f-nonhyp-obst1}R&$\(2,0,4,5,6,5,3\)$&9&$\+\infty$&$\+1$&$\+\half$&$1$\Tstrut\Bstrut \\ \hline
\end{longtable}
%
} 

\FloatBarrier

\bigskip

{\sc Araceli Bonifant: Mathematics Department, University of Rhode Island,
  Kingston, R.I.,  02881.}
\textsf{email: bonifant@uri.edu}

\bigskip

{\sc John Milnor: Institute for Mathematical Sciences, Stony Brook University,
Stony Brook, NY. 11794-3660.}
\textsf{email: jack@math.stonybrook.edu}

\bigskip

{\sc Scott Sutherland: Institute for Mathematical Sciences, Stony Brook
  University, Stony Brook, NY. 11794-3660.}
\textsf{email: scott@math.stonybrook.edu}
\end{document}